\documentclass[9pt,reqno]{amsart}  
\newlength\figH
\newlength\figW

\usepackage[T1]{fontenc} 
\usepackage[english]{babel}
\usepackage[p,osf]{cochineal}
\usepackage[varqu,varl,var0]{inconsolata}
\usepackage[scale=.95,type1]{cabin} 
\usepackage{mathalfa}
\usepackage[utf8]{inputenc} 
\usepackage{amsthm}
\usepackage{amsmath}
\usepackage{amssymb}
\usepackage{amsfonts}
\usepackage{amsaddr}
\usepackage{xspace}
\usepackage{enumitem} 
\usepackage{csquotes}
\usepackage{xcolor}
\usepackage[colorlinks]{hyperref}
\hypersetup{colorlinks, breaklinks, citecolor=teal,filecolor=black}
\usepackage{graphicx}
\usepackage{mathtools}
\mathtoolsset{centercolon}
\usepackage{bm}
\usepackage{pgfkeys}
\usepackage{tikz}
\usepackage{pgf,interval}
\usepackage{pgfplots}
\pgfplotsset{compat=newest}
\usepgfplotslibrary{groupplots}
\usepgfplotslibrary{dateplot}
\usetikzlibrary{positioning}
\intervalconfig{soft open fences}  

\DeclareMathOperator{\Tr}{Tr}

\DeclareMathOperator{\sgn}{sgn}

\DeclareMathOperator{\E}{\mathbf{E}}

\newcommand{\iso}{\mathrm{iso}}
\newcommand{\ii}{\mathrm{i}}
\newcommand{\cred}[1]{{\color{red} #1}}

\ifdefined\C
\renewcommand{\C}{\mathbf{C}}
\else
\newcommand{\C}{\mathbf{C}}
\fi
\newcommand{\HC}{\mathbf{H}}

\newcommand{\vx}{\bm{x}}

\newcommand{\bu}{\bm{u}}

\newcommand{\vy}{\bm{y}}

\newcommand{\R}{\mathbf{R}}
\newcommand{\N}{\mathbf{N}}
\newcommand{\Z}{\mathbf{Z}}

\newcommand{\cO}{\mathcal{O}}
\newcommand{\co}{{\scriptstyle\mathcal{O}}}

\newcommand{\dif}{\operatorname{d}\!{}}
\newcommand\restr[3]{{
  \left.\kern-\nulldelimiterspace 
  #1 
  \vphantom{\big|} 
  \right|_{#2}^{#3} 
  }}
\DeclarePairedDelimiter{\braket}{\langle}{\rangle}%
\DeclarePairedDelimiter{\abs}{\lvert}{\rvert}%

\DeclarePairedDelimiterX{\tuple}[1](){#1}
\DeclarePairedDelimiterX{\set}[1]\{\}{#1}
\DeclarePairedDelimiterXPP{\landauO}[1]{\cO}(){}{#1}
\DeclarePairedDelimiterXPP{\landauo}[1]{\co}(){}{#1}
\DeclarePairedDelimiterXPP{\landauOprec}[1]{\cO_\prec}(){}{#1}
\DeclarePairedDelimiterXPP{\landauOstd}[1]{\cO_\prec^2}(){}{#1}
\DeclarePairedDelimiterXPP{\landauOE}[1]{\cO_\prec^1}(){}{#1}
\DeclarePairedDelimiterXPP{\landauOd}[1]{\cO_\mathrm{m}}(){}{#1}

\usepackage{braket}
\usepackage{stackengine}
\usepackage{comment}

\usepackage{MnSymbol}

\usepackage{tikz, ifthen}
\usepackage{amsfonts}
\usepackage{pgfplots}
\usetikzlibrary{arrows,arrows.meta,intersections,calc}
\makeatletter
\long\def\ifnodedefined#1#2#3{%
	\@ifundefined{pgf@sh@ns@#1}{#3}{#2}%
}
\makeatother
\newcommand{\domain}[4]{
	\path[name path=kappa] (-3,#1^.333) -- (3,#1^.333); 
	\path[name intersections={of=kappa and rho,name=i#4}];;
	\draw[lightgray,fill=black!#3] (i#4-1 |- 0,1.2) -- (i#4-1 |- 0,#2) -- (i#4-2 |- 0,#2) -- (i#4-2 |- 0,1.2);
}
\newcommand{\domains}[4]{
	\path[name path=kappa] (-3,#1^.333) -- (3,#1^.333); 
	\path[name intersections={of=kappa and rho,name=i#4}];;
	\draw[lightgray,fill=black!#3] (i#4-1 |- 0,1.2) -- (i#4-1 |- 0,#2) -- (i#4-2 |- 0,#2) -- (i#4-2 |- 0,1.2);
	\draw[lightgray,fill=black!#3] (i#4-3 |- 0,1.2) -- (i#4-3 |- 0,#2) -- (i#4-4 |- 0,#2) -- (i#4-4 |- 0,1.2);
}

\newcommand{\rectangle}[3]{\fill [lightgray!50] (#1,-0.15) rectangle (#2,0.15) node[midway,below=2mm,black] {$J_#3$};}
\newcommand{\interv}[3]{\draw[color=black, {Bracket[width=5mm]}-{Bracket[width=5mm]}] 
	(#1,0) node[above=2mm] {$a_#3^{(\kappa)}$} -- (#2,0) node[above=2mm] {$b_#3^{(\kappa)}$};}
\pgfplotsset{compat=1.17, 
	/pgf/declare function={
		g(\e,\z) = -4*(-3 + \e^2 + 3*\z^2)^3 + (9*\e - 2*\e^3 + 18 * \e * \z^2)^2;
		h(\e,\z) = (9*\e - 2*\e^3 + 18*\e*\z^2 + g(\e,\z)^(1/2))^(1/3);
		f(\e,\z) = (6 - 2*\e^2 + 2^(1/3)*h(\e,\z)^2 - 6*\z^2)/(2^(5/3) * 3^(1/2)*h(\e, \z));
		m(\z) = (-1 + 20*\z^2 + 8*\z^4 + (1 + 8*\z^2)^(3/2))^(1/2)/(2^(3/2)*\z);
	},
}

\renewcommand{\Im}{\mathrm{Im}\,}
\renewcommand{\Re}{\mathrm{Re}\,}

\numberwithin{equation}{section}
\newtheorem{theorem}{Theorem}[section]
\newtheorem{assumption}[theorem]{Assumption}
\newtheorem{lemma}[theorem]{Lemma}
\newtheorem{proposition}[theorem]{Proposition}
\newtheorem{definition}[theorem]{Definition}
\newtheorem{example}[theorem]{Example}

\newtheorem{remark}[theorem]{Remark}

\allowdisplaybreaks

\newcommand{\eps}{\epsilon}
\newcommand{\D}{\mathrm{d}}

\newcommand{\I}{\mathrm{i}}
\newcommand{\Defo}{{\Lambda}}

\author{Giorgio Cipolloni}
\address{Princeton Center for Theoretical Science, Princeton University, Princeton, NJ 08544, USA}
\author{L\'aszl\'o Erd\H{o}s\(^\#\) and Joscha Henheik\(^\#\)}
\address{IST Austria, Am Campus 1, 3400 Klosterneuburg, Austria}
\author{Dominik Schr\"oder\(^{\ast}\)}
\address{ETH Zurich, R\"amistrasse 101, 8092 Zurich, Switzerland}
\email{gc4233@princeton.edu} 
\email{lerdos@ist.ac.at}
\email{joscha.henheik@ist.ac.at}
\email{dschroeder@ethz.ch}
\thanks{\(^\#\)Supported by ERC Advanced Grant ``RMTBeyond'' No.~101020331}
\thanks{\(^\ast\)Supported by the SNSF Ambizione Grant \texttt{PZ00P2\char`_209089}}
\title[Eigenvector overlaps for  non-Hermitian random matrices]{
	Optimal lower bound on eigenvector overlaps for non-Hermitian  random matrices 
}
\date{\today} 
\begin{document}
\begin{abstract}
	We consider large non-Hermitian $N\times N$ matrices with an additive independent, identically distributed (i.i.d.)
	noise for each matrix elements.
	We show that already a small noise of variance $1/N$ completely thermalises the bulk singular vectors,  in particular
	they satisfy the strong form of Quantum Unique Ergodicity (QUE) with an optimal speed of convergence. In physics
	terms, we thus extend the Eigenstate Thermalisation Hypothesis, formulated originally by Deutsch \cite{deutsch}
	and proven for Wigner matrices in~\cite{ETHpaper}, to arbitrary non-Hermitian matrices with an i.i.d. noise.
	As a consequence we  obtain an optimal {\it lower} bound on the diagonal overlaps of the corresponding non-Hermitian eigenvectors.
	This quantity, also known as the (square of the) eigenvalue condition number  measuring the sensitivity of the eigenvalue to
	small perturbations, 	has notoriously escaped  rigorous treatment beyond the explicitly computable Ginibre ensemble
	apart from the  very recent {\it upper} bounds given in~\cite{2005.08930}
	and~\cite{2005.08908}. As a key tool, we develop a new systematic decomposition of general observables
	in random matrix theory that governs the size of  products of resolvents with deterministic matrices in between.
\end{abstract}
\maketitle
\section{Introduction} \label{sec:intro}
Traditional  random matrix theory focuses on statistics of eigenvalues, where spectacular
universality phenomena arise: the local spectral statistics tend to become universal as the dimension goes to infinity
with new distributions arising; most importantly the celebrated
	{\it Wigner-Dyson-Mehta bulk statistics} and the {\it Tracy-Widom edge statistics} in
the {\it Hermitian} spectrum
and the {\it Ginibre statistics} in the {\it non-Hermitian} spectrum.
More recently eigenvectors of {\it Hermitian} ensembles   received considerable attention.
They   also become universal, albeit in a more conventional way: they tend to be entirely randomised, i.e.
Haar distributed \cite{BourgadeYau1312.1301,BourgadeYauYin1807.01559,MarcinekYau2005.08425,BenigniLopatto2103.12013,
	normalfluc, A2, CipolloniBenigni2022}. In this paper we study two related questions:
how do  eigenvectors and singular vectors  of a typical {\it non-Hermitian} random matrix  in high dimension look like?
To answer them, we introduce a new decomposition of general observables
that identifies correlations of the  Hermitised
resolvents as {\it entire matrices}  at different spectral parameters. This captures
correlations of  the singular vectors
well beyond correlations of traces of resolvents that govern only the singular \emph{values}.
Somewhat surprisingly, we are then able to transfer information on singular vectors to the
non-Hermitian eigenvectors.  \normalcolor

\subsection{Non-Hermitian eigenvector overlaps}
To be specific, we consider non-Hermitian $N\times N$ matrices of the form $\Lambda+X$, where
$\Lambda$ is an arbitrary deterministic matrix
and $X$ is random.
We assume that the norm of $\Lambda$ is bounded independently of $N$
and  $X$ has
independent, identically distributed (i.i.d.) centred matrix elements with variance  $\E |x_{ij}|^2 =\frac{1}{N}$
with some further moment conditions.
This normalisation guarantees that $\| X\|\le 2+o(1)$ and the spectrum of $X$  lies
essentially in the unit disk ({\it circular law}) with very high probability, hence $\Lambda$ and $X$ remain
of comparable size as $N$ increases. Note that $X$ perturbs each matrix elements of
$\Lambda$  by a small random amount of order $1/\sqrt{N}$, however the spectra of $\Lambda$ and  $\Lambda+X$
substantially differ.

The analysis of  non-Hermitian random matrices is typically much harder than that of the Hermitian ones.
Non-Hermitian matrices have two different sets of spectral data: eigenvalues/vectors and singular values/vectors
which cannot be directly related. In particular,
the study of singular vectors and eigenvectors substantially differ: while singular vectors can  still be
understood from a Hermitian theory, there is no such route for eigenvectors.
Unlike for non-Hermitian {\it eigenvalues}, where Girko's formula translates their linear statistics into a Hermitian problem,
no similar "Hermitisation" relation is known for non-Hermitian  {\it eigenvectors}.
Furthermore, left and right eigenvectors
differ and their relation is very delicate. Assuming that each eigenvalue $\mu_i$  of $\Lambda+X$ is simple, we denote
the corresponding left and right eigenvectors by $\boldsymbol{l}_i $, $\boldsymbol{r}_i$, i.e.
$$
	(\Lambda+X) \boldsymbol{r}_i=\mu_i\boldsymbol{r}_i\,, \qquad \boldsymbol{l}_i^t (\Lambda+X) =\mu_i\boldsymbol{l}_i^t\,,
$$
under the standard bi-orthogonality relation
$\langle \bar{\boldsymbol{l}}_j, \boldsymbol{r}_i \rangle = \boldsymbol{l}^t_j \boldsymbol{r}_i = \delta_{i,j}$.
Note that this relation  leaves a large freedom in choosing the normalisation of each eigenvector.
The key invariant quantity is the {\it eigenvector overlap}
\[
	\mathcal{O}_{ij} := \langle \boldsymbol{r}_j, \boldsymbol{r}_i \rangle \langle \boldsymbol{l}_j, \boldsymbol{l}_i \rangle\,,
\]
which emerges  in many problems where non-Hermitian eigenvectors are concerned,
see e.g.~\cite{Akemann, ChalkerMehlig, ChalkerMehlig2, Belinschi, 1801.01219, Fyodorov}. Two prominent examples
are
\begin{itemize}
	\item[(i)] in numerical linear algebra; where $\sqrt{\mathcal{O}_{ii}}$ is the {\it eigenvalue condition number}
		determining how fast $\mu_i$ moves under small perturbation in the worst case
		using the formula
		\begin{equation}\label{condnum}
			\sqrt{\mathcal{O}_{ii}}  = \lim_{t\to 0} \;\sup\Big\{ \Big| \frac{\mu_i(\Lambda+X+tE) -\mu_i(\Lambda+X)}{t} \Big| \; : \; E\in \C^{N\times N},
			\| E\|=1\Big\}
		\end{equation}
		(see, e.g.~\cite{2005.08930});
	\item[(ii)] in the theory of  the {\it Dyson Brownian motion} for non-Hermitian matrices; where $\mathcal{O}_{ij}$ gives the correlation
		of the martingale increments for the stochastic evolution of the eigenvalues $\mu_i$ and $\mu_j$
		as the matrix evolves by the  natural
		Ornstein-Uhlenbeck flow (see~\cite{grela}, \cite[Appendix A]{1801.01219}).
\end{itemize}

The main result of this paper is an almost
optimal {\it lower} bound of order $N$ on the diagonal overlap $\mathcal{O}_{ii}$, with very high probability.
In the context of numerical linear algebra this means that non-Hermitian eigenvalues of $\Lambda+X$ still
move at a speed of order $\sqrt{N}$ under the "worst" perturbation $E$ in~\eqref{condnum},
despite having added a random smoothing component $X$
to $\Lambda$. Note that in numerics  one typically views the random smoothing as a tool to reduce
the overlap of $\Lambda$ in order to enhance  the stability of its eigenvalues; our result shows a natural limitation for  such reduction.
It still does not exclude the possibility that a very specially  chosen
$X$ reduces the eigenvalue condition numbers much more than a typical random one does, in particular it does not disprove the
Davidson-Herrero-Salinas conjecture (see~\cite[Problem 2.11]{DHS}). However, our
$N$-dependent lower bound on  $\mathcal{O}_{ii}$ shows that a naive randomisation argument
is not sufficient for resolving this conjecture.
Complementary  {\it upper} bounds  on $\mathcal{O}_{ii}$ have recently been proven in~\cite{2005.08930}
and~\cite{2005.08908}. These hold only in expectation sense, as  $\mathcal{O}_{ii}$ has a fat-tail,
and they are off by a factor $N$. Very recently this factor was removed in \cite{JiErd}.
We remark, however, that  $N$ is the most relevant parameter of the problem
only from our random matrix theory point of view.
Works motivated by numerical analysis, such as~\cite{2005.08930, 2005.08908} and references therein,
often  focus on tracking the $\gamma$-dependence for the
problem $\Lambda+\gamma X$ in the small noise regime $\gamma\ll 1$ in order to reduce
the effect of the random perturbation.
In this setup  the non-optimality of the  $N$-power may be
considered less relevant.\footnote{As long as $\gamma$ is $N$ independent, one may set
	$\gamma=1$ by a simple rescaling  so we refrain from carrying
	this extra factor in the current paper. We remark that our methods would allow to trace the polynomial
	$\gamma$-dependence in all our main estimates as well, albeit not with an optimal power.
	\label{foot}}

In the context of the Dyson Brownian motion,
our lower bound on $\mathcal{O}_{ii}$ implies a diffusive lower bound on the
eigenvalues of the Ornstein-Uhlenbeck (OU) matrix flow,
generalizing the analogous result
of Bourgade and Dubach~\cite[Corollary 1.6]{1801.01219} from Ginibre ensemble
to arbitrary i.i.d.~ensemble (see~\eqref{diff} later).

\subsection{Thermalisation of singular vectors}
The key step to our lower bound on $\mathcal{O}_{ii}$ is a {\it thermalisation} result on the singular vectors that is
of independent interest.
Namely, we show that  singular vectors of $\Lambda+X$
are fully randomised
in the large $N$ limit in the sense that
their quadratic forms with arbitrary test matrices have a  deterministic limit with an optimal
$N^{-1/2}$ speed of convergence.
This holds with very high probability which enables us to make such statement
for matrices of the form $(\Lambda-z)+X$   {\it simultaneously} for any shift parameter $z$, even for random ones.
We will use this for $z=\mu$, an eigenvalue of $\Lambda+X$. This allows us to gain access to
eigenvectors of $\Lambda+X$, by noticing
that singular vectors and eigenvectors are unrelated in general with an obvious exception: if $\mu$ is an eigenvalue of
$\Lambda+X$,  then any vector in the kernel of $\Lambda+X-\mu$ is an eigenvector of $\Lambda+X$
with eigenvalue $\mu$, and a singular vector of $\Lambda+X-\mu$ with singular value $0$.
Hence high probability statements for singular vectors can be converted into similar statements
for eigenvectors -- this key idea  may be viewed as the eigenvector version  of the transfer principle between eigenvalues
and singular values encoded in Girko's formula.

Our thermalisation result for singular vectors
may be viewed as the non-Hermitian analogue of the \emph{Quantum Unique Ergodicity}
(QUE)  for Hermitian Wigner matrices proven in~\cite{ETHpaper}.
We now briefly explain the QUE  phenomenon  and its physics background in the  simplest Hermitian context
before we consider the singular vectors of $\Lambda+X$. In fact, via a standard Hermitisation procedure
we will  turn the singular vector problem to a Hermitian eigenvector problem.

For Hermitian random matrices $H$, that can be considered as the Hamilton operator of a disordered quantum system,
a major motivation comes from physics, where the randomisation of the eigenvectors is interpreted
as a   {\it thermalisation} effect. The {\it Eigenstate Thermalisation Hypothesis (ETH)} by Deutsch~\cite{deutsch}
and Srednicki~\cite{Srednicki} (see also~\cite{DAlessio, deutsch2})  asserts that  any deterministic Hermitian matrix  $A$ (observable),
becomes essentially diagonal in the eigenbasis of a "sufficiently chaotic" Hamiltonian, where chaos
may come from an additional randomness or from the ergodicity of the underlying classical dynamics.
In other words,
\begin{equation}\label{uAu}
	\langle \bu_i, A \bu_j \rangle - \delta_{ij} \langle\!\langle A \rangle\!\rangle_i \to 0\,, \qquad \mbox{ as $N\to\infty$}\,,
\end{equation}
where
$\{ \bu_i\}$ is a orthonormal eigenbasis of  $H$ and the deterministic "averaged" coefficient
$\langle\!\langle A \rangle\!\rangle_i$ is  to be computed from the statistics of $H$.

%
%
%

In the mathematics literature the same problem is known as the Quantum (Unique) Ergodicity, 
originally formulated for  the Laplace-Beltrami operator on surfaces with ergodic geodesic flow,
see~\cite{Shnirelman, ColinDeVerdiere, Zelditch}, on regular graphs~\cite{AnantLeMasson}
and on special arithmetic surfaces~\cite{RudnickSarnak, BrooksLindenstrauss, Lindenstrauss, Soundararajan}.
In~\cite{ETHpaper} we proved QUE in the strongest form with an optimal speed of convergence
for the eigenvectors of Wigner matrices that, by E. Wigner's vision,  can be viewed as the "most random"  Hamiltonian.
In this case,
the diagonal limit $\langle\!\langle A \rangle\!\rangle_i$ in \eqref{uAu} is independent of $i$  and  given by the normalised
trace $\langle A\rangle:= \frac{1}{N}\Tr A$.
In fact, in subsequent papers~\cite{normalfluc, A2} (see also \cite{BenigniLopatto2103.12013}) even the normal fluctuation
of $\sqrt{N}\big[  \langle \bu_i, A \bu_i \rangle - \langle A\rangle \big]$  was proven, followed by the proof
of joint Gaussianity of finite many overlaps in~\cite{CipolloniBenigni2022}.
Previously QUE results were proven for rank one observables (see~\cite{KnowlesYin13, TVEigenvector}
under four moment matching and \cite{BourgadeYau1312.1301} in general) and
finite rank observables~\cite{MarcinekYau2005.08425}, see also~\cite{Benigni} for deformed Wigner matrices
and~\cite{BourgadeYauYin1807.01559} for band matrices.
The proofs crucially used that $H$ is Hermitian,
heavily relying on sophisticated Hermitian techniques (such as {\it local laws}
and {\it Dyson Brownian Motion}) developed in the last decade for eigenvalue universality
questions.

Back to our non-Hermitian context,
we  consider the singular vectors $\{\bu_i, {\bm v}_i\}_{i=1}^N$ of $\Lambda + X$,
$$
	(X+\Lambda)(X+\Lambda)^* {\bm u}_i = \sigma_i^2 {\bm u}_i\,, \qquad  (X+\Lambda)^*(X+\Lambda) {\bm v}_i = \sigma_i^2 {\bm v}_i\,,
$$
belonging to the singular value $\sigma_i$.
We view them as the two $N$-dimensional
components of the eigenvectors ${\bm w}_i = (\bu_i, {\bm v}_i)$ of the $2N$-dimensional
{\it Hermitisation} of $\Lambda + X$, defined as
\begin{equation} \label{eq:herm1}
	H= H^\Defo:= W +  \hat{\Defo}\,, \quad W:= \left( \begin{matrix}
			0   & X \\
			X^* & 0
		\end{matrix}\right)\,, \qquad \hat{\Defo}: =\left( \begin{matrix}
			0         & \Lambda \\
			\Lambda^* & 0
		\end{matrix} \right)\,.
\end{equation}
In particular, from the overlaps  $\langle {\bm w}_i, A {\bm w}_j\rangle$ of eigenvectors for the Hermitised problem with a general $(2N)\times (2N)$
matrix $A$ one may read off all the singular vector overlaps
of the form $\langle \bu_i, B \bu_j\rangle$, $\langle {\bm v}_i, B {\bm v}_j\rangle$ and $\langle \bu_i, B {\bm v}_j\rangle$ with any $N\times N$ matrix $B$.
Therefore our goal is to show the general thermalisation phenomenon, the convergence of $\langle{\bm w}_i, A {\bm w}_j\rangle$ (cf.~\eqref{uAu}),
for the Hermitised matrix $H^\Lambda$ thus generalizing the ETH proven in~\cite{ETHpaper}
beyond Wigner matrices and with an additional arbitrary matrix $\Lambda$. Unlike in the Wigner case, the
limit $\langle\!\langle A \rangle\!\rangle_i$ genuinely depends on the index $i$ and part of the task is to determine
its precise form.
Note that due to the large zero blocks, $W$ has  about half as many random degrees of freedom
as a Wigner matrix  of the same dimension has, moreover the block structure gives rise to potential instabilities, thus
the ETH for $H^\Lambda$ is considerably more involved than for Wigner matrices.
In the next section we explain the main new method of this paper
that systematically handles all  these instabilities.

\subsection{Structural decomposition of observables}
We introduce a  new concept for splitting general observables into "regular" and "singular" components;
where  the singular component gives the leading
contribution and the regular component is estimated.
In the case of Wigner matrices $H$ in~\cite{ETHpaper, 2012.13218}
we used the decomposition $A=\langle A\rangle + \mathring{A}$,
where   the traceless part of $A$, $\mathring{A}:=A-\langle A\rangle$,  is the regular component
and the projection\footnote{We equip the space of matrices with the usual normalised
	Hilbert-Schmidt scalar product, $\langle A, B\rangle: = \frac{1}{N}\Tr A^*B = \langle A^*B\rangle$.}
of $A$ onto the one dimensional space spanned by the identity matrix is the singular component.
This gave rise to the following
decomposition of  resolvent $G=G(w) =(H-w)^{-1}$ for any $w\in\C\setminus \R$:
\begin{equation}\label{Gdecomp}
	\langle GA\rangle = m \langle A\rangle + \langle A\rangle \langle G-m \rangle  + \langle G\mathring{A}\rangle,
\end{equation}
where $m=m(w)$ is the Stieltjes transform of the semicircle law. The second term in~\eqref{Gdecomp} is asymptotically
Gaussian of size  $\langle G-m \rangle\sim (N\eta)^{-1}$ \cite{MR3678478}
and the last term
is also Gaussian, but of much smaller size
$\langle G\mathring{A}\rangle\sim \langle \mathring{A}\mathring{A}^*\rangle^{1/2}/(N\eta^{1/2})$
in the interesting regime of small $\eta:=|\Im w|\ll 1$ \cite{2012.13218}.

Similar decomposition governs the traces of longer resolvent chains of Wigner matrices, for
example
$$
	\langle GAG^*B \rangle = \langle GG^* \rangle = \frac{1}{\eta} \langle \Im G \rangle \sim \frac{1}{\eta}\gg 1
$$
if $A=B=I$, i.e. both observable matrices are purely singular, while for regular (and bounded) observables $A=\mathring{A}$,
$B=\mathring{B}$ we have
\begin{equation}\label{GAGB}
	\langle GAG^*B \rangle \sim 1.
\end{equation}
Both examples indicate the {\it $\sqrt{\eta}$-rule}
(see \eqref{eq:correlator} and Remark \ref{rmk:sqrteta} later), informally asserting that each regular observable renders
the size of a resolvent chain smaller by a factor $\sqrt{\eta}$ than its singular counterpart.
In~\cite{multiG,A2} we obtained the deterministic leading terms and optimal error estimates on
the fluctuation for resolvent chains of arbitrary length
\begin{equation}\label{chain}
	\langle G(w_1)A_1 G(w_2) A_2 \ldots\rangle
\end{equation}
with arbitrary observables in between. The answer followed the {\it $\sqrt{\eta}$-rule}
hence it heavily depended on the
$A_i = \langle A_i\rangle + \mathring{A}_i$ decomposition for each observable.

In particular, in order to estimate $\langle \bu_i, A\bu_j \rangle - \delta_{ij}\langle A\rangle = \langle \bu_i, \mathring{A}\bu_j \rangle$
for ETH in~\eqref{uAu}, we had
$$
	N |\langle \bu_i, \mathring{A}\bu_j \rangle|^2 \lesssim \langle \Im G(w_1) \mathring{A} \Im G(w_2) \mathring{A}  \rangle \lesssim 1\,,
$$
where we first used spectral decomposition of both $G$'s
and then used a version of~\eqref{GAGB}. Here the spectral parameters $w_k=e_k+\ii \eta$ are chosen
such that $e_1$ and $e_2$ be close to the eigenvalues  corresponding to $\bu_i$ and $\bu_j$, respectively,
and $\eta\sim N^{-1}$ in order to resolve the spectrum on the fine scale of the individual eigenvalues.\footnote{Strictly speaking
	we used $\eta= N^{-1+\xi}$ with any small $\xi>0$, and all estimates held up to an $N^\xi$ factor
	but we ignore these technicalities in the introduction.}

The key point in all these  analyses for Wigner matrices was that the
regular/singular concept was {\it independent} of the spectral parameter: the same universal decomposition
into tracial and traceless parts worked in every instance along the proofs.
One  consequence is the $i$-independence of the limiting overlap $\langle\!\langle A \rangle\!\rangle_i:=\langle A\rangle$
in~\eqref{uAu}.\footnote{A quick direct way to see this independence is the special case of Gaussian Wigner matrices
	(GUE or GOE),  where the eigenvectors are Haar distributed, independently of their eigenvalue.}

For more complicated ensembles, like $H^\Lambda$ in~\eqref{eq:herm1}, especially if
an arbitrary  matrix $\Lambda$ is involved, the correct decomposition depends on the
location  in the spectrum of $H$ where we work.
To guess it, first we recall the {\it single resolvent local law} (Theorem~\ref{thm:singleG})
for the resolvent $G=G^\Lambda(w)= (H^\Lambda-w)^{-1}$,
asserting that $\langle GA\rangle \approx \langle MA\rangle$, where  $M=M^\Lambda(w)$ solves
a nonlinear deterministic equation, the {\it Matrix Dyson Equation (MDE)}, see~\eqref{eq:MDE} later.
Then a heuristic calculation (see Appendix~\ref{subsec:variance}) shows that
for $w=e+\ii\eta\in \C_+$ we have
\begin{equation}\label{2term}
	\E \big| \langle (G-M)A\rangle\big|^2 \approx \frac{ \big| \langle \Im M A\rangle\big|^2}{(N\eta)^2}
	+ \frac{  \big| \langle \Im M A E_-\rangle\big|^2}{N^2\eta(|e|+\eta)} + \mathcal{O} \Big(\frac{1}{N^2\eta}\Big)\,,
	\quad E_- :=\left(\begin{matrix}
			1 & 0 \\ 0 & -1	\end{matrix}\right)\,,
\end{equation}
indicating that the singular component of $A$ is {\it two dimensional}, depends on $w$,
and for any $A$ orthogonal to the two singular directions
$\Im M$ and $E_- \Im M$ the size of $\langle (G-M)A\rangle$
is smaller by a factor $\sqrt{\eta}$.
The first singular direction is always present. The second singular direction
is a consequence of the block structure of $H$ and it is manifested only for $w$  near
the imaginary axis. For energies $|e|\sim 1$, only the first  singular direction, namely the one involving $\Im M$ plays a role.

What about longer chains~\eqref{chain}? Each matrix $A_i$ is sandwiched between two resolvents
with different spectral parameters $w_i$, $w_{i+1}$. We find that the correct decomposition of any $A$
between two resolvents in a chain $\ldots G(w)AG(w')\ldots $ depends only on $w, w'$ and it has
the  form
\begin{equation}\label{Adec}
	A = \langle V_+, A\rangle U_+ + \langle V_-, A\rangle U_- + \mathring{A}\,, \qquad V_\pm = V_\pm^{w,w'}\,, \quad
	\mathring{A}=  \mathring{A}^{w,w'}\,,
\end{equation}
where the first two terms form the singular component of $A$, and $\mathring{A}$,
defined by this equation, is the regular component. We will establish that both $V_+$ and $V_-$ are
the right eigenvectors of a certain {\it stability operator} $\mathcal{B}$ acting on $\C^{2N\times 2N}$
that
corresponds to the Dyson equation, and $U_\pm$ will be explained later.
For example, if $\Im w$ and $\Im w'$ have opposite signs then
$V_+$ is  the  right eigenvector  of
$$
	\mathcal{B}[\cdot ] = 1- M(\bar w) \mathcal{S} [\cdot ]M(\bar w')\,,
$$
where
$\mathcal{S}$ is covariance operator for the matrix $W$  in~\eqref{eq:herm1}
(see~\eqref{Sop}). $V_\pm$ with other sign combinations are defined
very similarly (in Appendix~\ref{subsec:exactform} we present
all cases). Actually, the special directions $\Im M$ and $E_- \Im M $ that we found by
direct variance calculation in~\eqref{2term} also emerge {\it canonically} as eigenvectors
of a certain stability operator!
Similar variance calculation for longer chains
would reveal the same consistency: the variance of the chain~\eqref{chain}
is the smallest if each $A_i$ is regular with respect to the two neighboring
spectral parameters $w_i, w_{i+1}$.

Note that the  choice of $V_\pm$ is basically dictated  by variance calculations
like~\eqref{2term}. However,  the matrices $U_\pm$ in~\eqref{Adec} can still be chosen
freely up to their linear independence and the
normalisation requirement $\langle V_\sigma, U_\tau\rangle=\delta_{\sigma, \tau}$.
The latter
guarantees that the sum of the singular terms
in~\eqref{Adec} is  actually a (non-orthogonal) projection $|U_+\rangle\langle V_+| + |U_-\rangle\langle V_- | $ acting on $A$.
Since $V_\pm$ are the right eigenvectors of a stability operator, one may be tempted to choose $U_\pm$
as certain left eigenvectors but we did not find this guiding principle helpful. Instead,
we use  this freedom   to simplify the calculation of the singular terms. Substituting the singular part of $A$
into $\ldots G(w)A G(w') \ldots$, we need to compute $G(w)U_\pm G(w')$ and quite pragmatically
we choose $U_\pm$ such that the entity could be applied
and thus reduce the length of the chain.  Thanks to the spectral
symmetry of $H=H^\Lambda$, for its resolvent we have $E_- G(-w) E_-=-G(w)$, and we find that  $U_+=I$, $U_-=E_-$
do the job, which accidentally coincide with the left eigenvectors of the stability operator
for the special case of i.i.d. matrices.

In Appendix~\ref{subsec:exactform} we present the  canonical choices
of $V_\pm$ and $U_\pm$ in a more general situation and explain at which stage
of the proof their correct choice emerges. In our current application only $V_\pm$ are nontrivial (in particular energy dependent),
while $U_\pm$ are very simple. This is due to the fact that
the chain~\eqref{chain} consists of  resolvents of the {\it same} operator. In more general
problems one may take resolvents with two different $\Lambda$'s in the chain,
in which case  $U_\pm$ are also nontrivial.

\medskip

This decomposition scheme is the really novel ingredient of our proofs.
Several other tools we use, such as {\it recursive Dyson equations}, hierarchy of {\it master inequalities} and  {\it reduction inequalities}
have been introduced before (especially in our related works on Wigner matrices~\cite{ETHpaper, 2012.13218}),
but the dependence of the decomposition on the spectral parameters in the  current setup
requires quite different new estimates along the arguments.
We informally explain the prototype of such an estimate at the beginning of Section~\ref{sec:basic}.

\subsection{Notations}
We define the $2N \times 2N$ matrices $E_{\pm} := E_1 \pm E_2$, where
\begin{equation}
	\label{eq:defe1e2}
	E_1 := \left(\begin{matrix}
			1 & 0 \\ 0 & 0
		\end{matrix}\right)\quad \text{and} \quad
	E_2 := \left(\begin{matrix}
			0 & 0 \\ 0 & 1
		\end{matrix}\right)\,.
\end{equation}
Each entry of the matrix is understood as a multiple of the $N \times N$--identity. By $\lceil \cdot \rceil$, $\lfloor \cdot \rfloor$ we denote the upper and lower integer part, respectively, i.e.~for $x \in \R$ we define $\lceil x \rceil := \min\{ m \in \Z \colon m \ge x \}$ and $\lfloor x \rfloor := \max\{ m \in \Z \colon m \le x \}$.
We denote $[k] := \{1, ... , k\}$ for $k \in \mathbf{N}$ and $\langle A \rangle := d^{-1} \mathrm{Tr}(A)$, $d \in \N$, is the normalised trace of a $d \times d$-matrix.
For positive quantities $A, B$ we write $A \lesssim B$ resp.~$A \gtrsim B$ and mean that $A \le C B$ resp.~$A \ge c B$ for some $N$-independent constants $c, C > 0$. We denote vectors by bold-faced lower case Roman letters $\boldsymbol{x}, \boldsymbol{y} \in \mathbf{C}^{2N}$, for some $N \in \mathbf{N}$, and define
\begin{equation*}
	\langle \boldsymbol{x}, \boldsymbol{y} \rangle := \sum_i \bar{x}_i y_i\,,
	\qquad A_{\boldsymbol{x} \boldsymbol{y}} := \langle \boldsymbol{x}, A \boldsymbol{y} \rangle\,.
\end{equation*}

Matrix entries are indexed by lower case Roman letters $a, b, c , ... $ from the beginning of the alphabet and unrestricted sums over $a, b, c , ... $ are always understood to be over $\{ 1 , ... , N, N+1, ... , 2N\}$. Analogously, unrestricted sums over lower case Roman letters $i,j,k, ...$ from the middle of the alphabet are always understood to be over $\{ -N , ... , -1, 1, ... , N\}$. Finally,
the lower case Greek letters $\sigma$ and $\tau$ as indices indicate a sign, i.e.~$\sigma, \tau \in \{ + , - \}$, and unrestricted sums over $\sigma, \tau$ are understood to be over $\{+ , -  \}$.

We will use the concept of `with very high probability', meaning that any fixed $D > 0$, the probability of an $N$-dependent event is bigger than $1 - N^{-D}$ for all $N \ge N_0(D)$. Also, we will use the convention that $\xi > 0$ denotes an arbitrarily small constant, independent of $N$. Moreover, we introduce the common notion of \emph{stochastic domination} (see, e.g., \cite{semicirclegeneral}): For two families
\begin{equation*}
	X = \left(X^{(N)}(u) \mid N \in \N, u \in U^{(N)}\right) \quad \text{and} \quad Y = \left(Y^{(N)}(u) \mid N \in \N, u \in U^{(N)}\right)
\end{equation*}
of non-negative random variables indexed by $N$, and possibly a parameter $u$, then we say that $X$ is stochastically dominated by $Y$, if for all $\varepsilon, D >0$ we have
\begin{equation*}
	\sup_{u \in U^{(N)}} \mathbf{P} \left[X^{(N)}(u) > N^\epsilon Y^{(N)}(u)\right] \le N^{-D}
\end{equation*}
for large enough $N \ge N_0(\epsilon, D)$. In this case we write $X \prec Y$. If for some complex family of random variables we have $\vert X \vert \prec Y$, we also write $X = O_\prec(Y)$.

\medskip

{\it Acknowledgement:} The authors are grateful to Oleksii Kolupaiev for valuable discussions, especially about
the choice of contours in Lemma~\ref{lem:intrepG^2}.
We  thank Nikhil Srivastava for bringing~\cite{DHS} to our attention.

\section{Main results} \label{sec:results}
We consider \emph{real or complex i.i.d.~matrices} $X$ , i.e.~$N \times N$ matrices whose entries are independent and identically distributed as $x_{ab} \stackrel{\rm d}{=} N^{-1/2}\chi$ for some real or complex random
variable $\chi$ satisfying the following assumptions:
\begin{assumption} \label{ass:iid}
	We assume that $\E \chi  = 0$ and $\E |\chi|^2 = 1$. Furthermore, we assume the existence of high moments, i.e., that there exist constants $C_p > 0$, for any $p \in \N$, such that
	\begin{equation*} 
		\E |\chi|^p \le C_p\,.
	\end{equation*}
	Additionally, in the complex case, we assume that $\E \chi^2 = 0$.
\end{assumption}
For definiteness, in the sequel we perform our entire analysis
for the complex case; the real case being completely analogous and hence omitted.

\subsection{Non-Hermitian singular vectors and eigenvectors} \label{sec:first}
Fix a deterministic matrix $\Defo \in \C^{N \times N}$, with $N$-independent norm bound, $\| \Lambda\| \lesssim 1$.
Let $\{\sigma_i\}_{i \in [N]}$ be the singular values of $X+\Defo$ with corresponding (normalised) left- and right-singular vectors $\{\boldsymbol{u}_i\}_{i \in [N]}$ and $\{\boldsymbol{v}_i\}_{i \in [N]}$, respectively, i.e.
\begin{equation} \label{eq:singvec}
	(X+\Defo)\boldsymbol{v}_i = \sigma_i \boldsymbol{u}_i \qquad \text{and} \qquad (X+\Defo)^* \boldsymbol{u}_i = \sigma_i \boldsymbol{v}_i\,.
\end{equation}
All these objects naturally depend on $\Defo$, but we will omit this fact from the notation.

Let $\nu_i$, $i\in [N]$, be the increasingly ordered singular values of $\Lambda$.
Define the {\it Hermitisation} of $\Defo$ 
as
\begin{equation} \label{eq:Defohat}
	\hat{\Defo}:= \begin{pmatrix}
		0 & \Defo \\ \Defo^* & 0
	\end{pmatrix}. 
\end{equation}
Due to its block structure,
the spectrum of $\hat{\Defo}$ is symmetric with respect to zero, with eigenvalues  $\{\nu_i\}_{0 \neq |i| \le N}$
such that
$\nu_{-i} = -\nu_{i}$ for all $i \in [N]$.
The empirical density of states of $\hat{\Defo}$ is denoted by
\begin{equation*} \label{eq:muLambda}
	\mu_{\hat{\Defo}} := \frac{1}{2N} \sum_{0 \neq |i| \le N} \delta_{\nu_i}\,.
\end{equation*}


Let $\mu_{\rm sc}$ be the Wigner semicircle distribution with density $\rho_{\rm sc}(x):= (2 \pi)^{-1} \sqrt{[4 - x^2]_+}$,
where $[\cdots]_+$ is the positive part of a real number.  Define
the free additive convolution
\begin{equation}\label{freec}
	\mu = \mu^{\Defo} := \mu_{\rm sc} \boxplus \mu_{\hat{\Defo}},
\end{equation}
which
is a probability distribution on $\R$.
We now briefly recall basic facts about the free convolution with
the semicircle density (see, e.g. the classical paper by   P. Biane \cite{biane}).
Most conveniently $\mu$
may be defined by inverting its Stieltjes transform
\begin{equation*}
	m(w) = m^\Defo(w) := \int_\R \frac{\mu(\D e) }{e-w} , \qquad w \in \C \setminus \R,
\end{equation*}
where $m$ satisfies the implicit equation
\begin{equation} \label{eq:selfconsm}
	m(w) = \int_{\R} \frac{\mu_{\hat{\Defo}}(\D e)}{e- (w + m(w))}  \,.
\end{equation}
With the additional constraint $\Im m(w) \cdot \Im w >0$ this equation
has a unique solution that is analytic away from the real axis with $m(\overline{w}) = \overline{m}( w)$.
Since $\mu_{\hat{\Defo}}$ is symmetric to zero with bounded support,
$\mu$ is also symmetric with  support bounded independently of $N$. Moreover $\mu$
is absolutely continuous with respect to Lebesgue measure with density  denoted by $\rho = \rho^\Lambda$.
The density $\rho$ may be obtained\footnote{ For orientation of the reader we mention that $\rho$ is the deterministic approximation, the so-called \emph{self-consistent
		density of states (scDos)}, for the empirical eigenvalue density of the Hermitisation of $X+\Defo$.  This connection will
	be explained in the next Section~\ref{subsec:eth}.}
as the boundary value of $\Im m$ at any $e$ on the real line, i.e.
\begin{equation} \label{eq:scDos}
	\rho(e):= \lim\limits_{\eta \downarrow 0} \rho(e+\I \eta)\,, \qquad \rho(w) := \frac{1}{\pi} |\Im m(w)|\,.
\end{equation}
In fact $m$ itself has a continuous extension to the real axis from the upper half plane
$m(e) := \lim_{\eta \downarrow 0}m(e+ \I \eta)$. Proving
the existence of these limits is standard from~\eqref{eq:selfconsm}.

%


%
Next, for any (small) $\kappa > 0$, we define the \emph{$\kappa$-bulk} of the density $\rho$ as
\begin{equation} \label{eq:bulk}
	\mathbf{B}_\kappa = 	\mathbf{B}_\kappa^\Defo := \{ x \in \R : \rho(x) \ge \kappa^{1/3} \}\,
\end{equation}
which is symmetric to the origin.
Finally, we denote a (modified) $i^{\mathrm{th}}$ quantile of the density $\rho$ by $\gamma_i$, i.e.
\begin{equation} \label{eq:quantiles}
	\frac{i+N}{2N} = \int_{-\infty}^{\gamma_i} \rho(e) \, \D e\,, \qquad |i| \le N\,,
\end{equation}
and we immediately conclude by symmetry of $\rho$ that $\gamma_i = - \gamma_{-i}$ for every $|i| \le N$.

Our first main result establishes the {\it thermalisation of singular vectors} of $X+\Defo$ in the bulk, i.e.
for indices $i,j$ with quantiles $\gamma_i, \gamma_j$ uniformly in the \emph{bulk} of the density $\rho$.

\begin{theorem} \label{thm:main singvec} {\rm (Thermalisation of Singular Vectors)} \\
	Fix a bounded $\Defo \in \C^{N \times N}$ and $\kappa > 0$ independent of $N$.
	Let $\{\boldsymbol{u}_i\}_{i \in [N]}$ and $\{\boldsymbol{v}_i\}_{i \in [N]}$ be the (normalised) left- and right-singular vectors of $X+\Defo$, respectively, where $X$ is an i.i.d.~matrix satisfying Assumption~\ref{ass:iid}. Then, for any deterministic matrix $B \in \C^{N \times N}$ with $\Vert B \Vert \lesssim 1$ it holds that\footnote{The deterministic terms following the Kronecker symbol $\delta_{j,i}$ in \eqref{eq:uuvvuv} will be shown to be bounded.} 
	\begin{subequations} \label{eq:uuvvuv}
		\begin{equation}
			\label{eq:uu}
			\max_{i,j} \left| \langle {\bm u}_i, B {\bm u}_j \rangle - \delta_{j, i} \frac{\left\langle \Im\left[ \frac{\gamma_j + m(\gamma_j)}{\Defo \Defo^* - (\gamma_j + m(\gamma_j))^2}\right] \, B \right\rangle}{\pi \rho(\gamma_j)} \right|   \prec \frac{1}{\sqrt{N}} \,,
		\end{equation}
		\begin{equation}
			\label{eq:vv}
			\max_{i,j} \left| \langle {\bm v}_i, B {\bm v}_j \rangle - \delta_{j, i} \frac{\left\langle \Im\left[ \frac{\gamma_j + m(\gamma_j)}{\Defo^* \Defo - (\gamma_j + m(\gamma_j))^2}\right] \, B \right\rangle}{\pi \rho(\gamma_j)}  \right| \prec \frac{1}{\sqrt{N}} \,,
		\end{equation}
		\begin{equation}
			\label{eq:uv}
			\max_{i,j} \left| \langle {\bm u}_i, B {\bm v}_j \rangle - \delta_{j, i} \, \frac{\left\langle \Defo \, \Im\left[  \big( \Defo^* \Defo - (\gamma_j + m(\gamma_j))^2  \big)^{-1} \right] \, B \right\rangle}{\pi \rho(\gamma_j)}  \right| \prec \frac{1}{\sqrt{N}}\,,
		\end{equation}
	\end{subequations}
	where the maximum is taken over all $i,j \in [N]$ such that the quantiles $\gamma_i, \gamma_j \in \mathbf{B}_\kappa$ are in the
	$\kappa$-bulk of the density $\rho$.
\end{theorem}

The thermalisation of singular vectors will be a simple
corollary to the \emph{Eigenstate Thermalisation Hypothesis (ETH)} for the Hermitisation $H^\Defo$ of $X+\Defo$, which is formulated in Theorem~\ref{thm:main} below. The proof of Theorem \ref{thm:main singvec} will be given in Section \ref{sec:proofmain}.


Our second main result concerns the bi-orthonormal left and right eigenvectors $\{\boldsymbol{l}_i\}_{i \in [N]}$ and $\{\boldsymbol{r}_i\}_{i \in [N]}$, respectively, of $X+\Lambda$, with corresponding eigenvalues $\{\mu_i\}_{i \in [N]}$, i.e.
\begin{equation} \label{eq:leftrightev}
	(X+\Defo) \boldsymbol{r}_i = \mu_i \boldsymbol{r}_i \,, \qquad \boldsymbol{l}^t_i (X+\Defo) = \mu_i \boldsymbol{l}^t_i\,,
\end{equation}
where $t$ denotes the transpose of a vector.
More precisely, the following theorem provides a lower bound on the diagonal part of the \emph{overlaps matrix}
\begin{equation} \label{eq:overlap}
	\mathcal{O}_{ij} := \langle \boldsymbol{r}_j, \boldsymbol{r}_i \rangle \langle \boldsymbol{l}_j, \boldsymbol{l}_i \rangle\,,
\end{equation}
defined subject to the standard normalisation
\begin{equation} \label{eq:orthonorm}
	\langle \bar{\boldsymbol{l}}_j, \boldsymbol{r}_i \rangle = \boldsymbol{l}^t_j \boldsymbol{r}_i = \delta_{i,j}\,.
\end{equation}
We restrict our results to eigenvalues $\mu_i$ in the \emph{bulk} of $X + \Defo$ in the following sense.
\begin{definition} \label{def:bulknonherm}
	We say that $z \in \C$ is in the \emph{bulk} of $X + \Defo$ if and only if there exists an $N$-independent $\kappa> 0$ for which
	\begin{equation*}
		0 \in \mathbf{B}_\kappa^{\Defo - z} = \{x \in \R : \rho^{\Defo - z}(x) \ge \kappa^{1/3}  \}\,.
	\end{equation*}
\end{definition}
There is no simple characterisation of the bulk of  $X + \Defo$ in terms of the spectrum of $\Defo$.
However, taking the imaginary part of~\eqref{eq:selfconsm} at $w=0+\ii 0$ shows that $0\in \mathbf{B}_\kappa^{\Defo - z}$
is equivalent to
$$
	\frac{1}{N}\sum_{i=1}^N \frac{1}{ \nu_i(\Defo -z)^2 + \kappa^{2/3}} \ge 1\,,
$$
where $\nu_i(\Defo -z)$ are the singular values of $\Lambda-z$.
\begin{theorem}
	\label{thm:overlap}
	Consider $X+\Lambda$, with $\Lambda$ being a deterministic matrix as in \eqref{eq:Defohat} and with $X$ being an i.i.d.~matrix satisfying Assumption \ref{ass:iid}.
	Then
	\begin{equation}
		\label{eq:Oiilb}
		\mathcal{O}_{ii} \succ N\,,
	\end{equation}
	where the index $i \in [N]$ is such that $\mu_i$ is in the bulk of $X+\Lambda$.
\end{theorem}
In the introduction we already mentioned the consequence of this result on the sensitivity of an eigenvalue
of $X+\Lambda$ under small perturbations. Now we explain its other consequence on
the diffusivity of the Dyson-type eigenvalue dynamics.
Let  each entry of $X=X(t)$ evolve as an independent complex  OU process,
$$
	\D X_{ij} = \frac{\D B_{ij}}{\sqrt{N}} - \frac{1}{2} X_{ij} \D t,
$$
where $B_{ij}$ are independent standard complex Brownian motions
and the initial condition $X(0)$ satisfies Assumption \ref{ass:iid}.
A direct calculation~\cite[Proposition A.1]{1801.01219} shows that
the eigenvalues $\mu_i=\mu_i(t) $ follow the Dyson-type stochastic dynamics
\begin{equation}\label{DBM}
	\D \mu_i = \D M_i - \frac{1}{2} \mu_i \D t, \qquad \{\mu_i(0)\} = \mbox{Spec} X(0), \qquad 1\le i\le N,
\end{equation}
where the martingales $M_i$ have brackets $\langle M_i, M_j\rangle=0$ and
$\D \langle M_i, M_j\rangle_t = \frac{1}{N}\mathcal{O}_{ij} (t) \D t$. In particular,
we immediately obtain, for any $\epsilon>0$ that
\begin{equation}\label{diff}
	\E \big[ |\mu_i(t) -\mu_i(0)|^2 {\bf 1}(\mu_i(0)\in \mathbf{B}_\kappa)\big]\ge tN^{-\epsilon}
\end{equation}
up to some time $t\le T(\kappa)$, where $\mathbf{B}_\kappa$ denotes the $\kappa$-bulk of $X(0)$.
For Ginibre initial condition $X(0)$ \eqref{diff} was established in~\cite[Corollary 1.6]{1801.01219},
we now generalise it to  i.i.d. initial conditions.
We remark that~\eqref{DBM} is similar to its Hermitian counterpart, the
standard Dyson Brownian motion (DBM) on the real line, with some notable differences.
In particular, in the latter process the eigenvalues cannot cross each other, hence they
are quite rigid and  confined to an interval of size essential $1/N$, so they are not diffusive
beyond a time-scale $1/N$. Along the evolution~\eqref{DBM} the non-Hermitian eigenvalues still repel
each other (encoded in the typically negative off-diagonal overlaps, see~\cite[Theorem 1.3]{1801.01219}
in the Gaussian case), but they still can pass by each other and not hindering the diffusive behavior~\eqref{diff}.

\begin{example} \label{exam:Defo=-z}
	The most prominently and extensively studied \cite{girko, bai, taovukrishna, BYYcirc1, BYYcirc2, Ycirc3, taovu, edgeuniv, iidcltcomplex, iidcltreal} deformation is $\Defo = -z$ with $z \in \C$, since it plays a key role in Girko's formula \cite{girko} expressing
	linear statistics of non-Hermitian eigenvalues of $X$ in terms of the Hermitisation of $X-z$.
	In this case, the self-consistent equation \eqref{eq:selfconsm} reduces to the well-known cubic relation
	\begin{equation*}
		- \frac{1}{m} = w + m - \frac{|z|^2}{w+m}\,.
	\end{equation*}
	As a consequence, the deterministic terms in \eqref{eq:uuvvuv} drastically simplify (e.g., the fractions in \eqref{eq:uu} and \eqref{eq:vv} are simply replaced by $\langle B \rangle$) and one also has  explicit formulas for the bulk \eqref{eq:bulk}
	in terms of solution of a cubic  equation.
	In particular, for $|z|< 1 - \epsilon_\kappa$, the $\kappa$-bulk $\mathbf{B}_\kappa$
	consists of a single interval, while for $|z| \ge 1 - \epsilon_\kappa$ it consists of two intervals, where $\epsilon_\kappa\sim \kappa^{2/3}$.
	In the former case $0\in \mathbf{B}_\kappa$.
	Consequently, Theorem~\ref{thm:overlap} gives the lower bound~\eqref{eq:Oiilb} for all  the diagonal overlaps $\mathcal{O}_{ii} $
	of eigenvectors of $X$ whose eigenvalue  $\mu_i$ lies in a disk of radius $1-\epsilon$ with some $\epsilon>0$ independent of $N$.
\end{example}
\begin{figure}[htbp]
	\centering
	\begin{tikzpicture}
		\begin{axis}[no markers,
				xlabel=$\Re w$,
				ylabel=$\rho(\Re w)$,
				width=\textwidth,
				height=0.4\textwidth,
				axis lines=middle,
				ymin=-.1,ymax=.85,
				xmin=-3.2,xmax=3.2,
				ytick={0},yticklabels={0},
				xtick={-2,0,2},
				xticklabels={-2,0,2},
				axis on top=true, color=black]
			\pgfmathsetmacro{\Z}{0.98}
			\pgfmathsetmacro{\min}{-m(\Z)+.001}
			\pgfmathsetmacro{\max}{m(\Z)}
			\path[name path=rho,save path=\saverho] (\min,0) -- plot[domain={\min}:\max,samples=301] ({\x},{f(abs(\x),\Z)}) -- (\max,0);
			\path[name path=kappa] (-3,.45) -- (3,.45);
			\path[name intersections={of=kappa and rho,name=i}];
			\draw[dashed,gray,fill=black!5] (i-1 |- 0,0) -- (i-1) -- (i-2) -- (i-2 |- 0,0);
			\draw[dashed,gray,fill=black!5] (i-3 |- 0,0) -- (i-3) -- (i-4) -- (i-4 |- 0,0);
			\draw[stealth-stealth,gray] (\min,.1) -- (i-1 |- 0,.1) node[right] {\footnotesize$\sim\kappa^{2/3}$};
			\draw[stealth-stealth,gray] (i-2 |- 0,.1) -- (i-3 |- 0,.1) node[right] {\footnotesize$\sim\kappa$};
			\draw[stealth-stealth,gray] (-.5,0) -- (-.5,0 |- i-1) node[midway,left] {\footnotesize$\sim\kappa^{1/3}$};
			\draw[|-|,thick] (i-1 |- 0,0) -- (i-2 |- 0,0);
			\draw[|-|,thick] (i-3 |- 0,0) -- (i-4 |- 0,0) node [midway, below] {
			};
			\draw[thick,solid,black,use path=\saverho];
		\end{axis}
	\end{tikzpicture}
	\caption{Depicted is the density $\rho$ for the deformation $\Defo = - z$ with $|z|$ slightly less than one.
		On the horizontal axis, we indicated the two components of the bulk $\mathbf{B}_\kappa$. The distance between $\mathbf{B}_\kappa$ and a regular edge scales like $\kappa^{2/3}$, while near an (approximate) cusp the distance between the two components scales linearly (see also \eqref{eq:bulk} and \eqref{kappabulkreg}).}
	\label{fig:Defo=-z Example}
\end{figure}
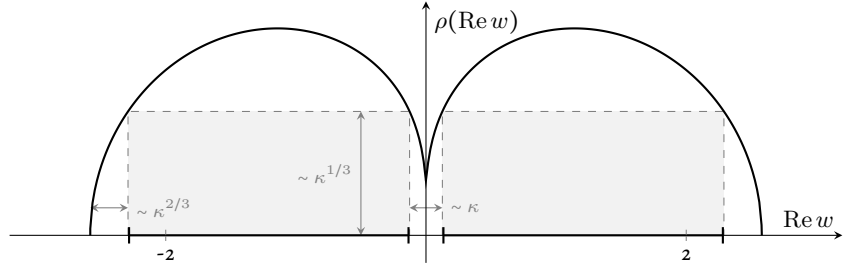
In the next section we explain the key technical result behind our main theorems, the eigenstate thermalisation
for the Hermitisation of $X+\Defo$.

\subsection{Eigenstate Thermalisation Hypothesis for the Hermitisation of $X+\Defo$} \label{subsec:eth}
The key to access the non-Hermitian singular vectors of $X+\Defo$  is to study its \emph{Hermitisation} \cite{girko}, which is defined as
\begin{equation} \label{eq:herm}
	H= H^\Defo:=\left( \begin{matrix}
			0           & X+\Defo \\
			(X+\Defo)^* & 0
		\end{matrix}\right) =:
	W + \hat{\Defo}\,,
\end{equation}
where $\hat{\Defo}^* = \hat{\Defo}$ was defined in \eqref{eq:Defohat} and can also be viewed  as the matrix of expectation
$\hat{\Defo}= \mathbf{E}H^\Defo$.

We denote by $\{{\bm w}_i\}_{|i|\le N}$ the (normalised) eigenvectors of $H$ and by $\{\lambda_i\}_{|i|\le N}$ the corresponding eigenvalues.\footnote{In the definition of the eigenvectors and eigenvalues, we omitted $0$ in the set of indices, i.e.~$|i| \le N$ really means $i \in \{-N, ... , -1, 1 , ... , N\}$. } By means of the singular value decomposition in \eqref{eq:singvec}, the eigenvalues and eigenvectors of $H$
are related to the singular values and singular vectors of $X+\Defo$ as follows:
\begin{equation*}
	{\bm w}_i = ({\bm u}_i, {\bm v}_i)^t \quad \text{and} \quad \lambda_i = \sigma_i \quad \text{for} \quad i \in [N]\,,
\end{equation*}
up to a normalisation, since now $\Vert {\bm u}_i \Vert^2 = \Vert {\bm v}_i \Vert^2 = \tfrac{1}{2}$.
Moreover, the block structure of $H$ induces a symmetric spectrum around zero, i.e.~$\lambda_{-i}=-\lambda_i$ for any $i\in[N]$. This symmetry for the eigenvalues is also reflected in the eigenvectors, which satisfy $\boldsymbol{w}_{-i} = E_- \boldsymbol{w}_i$ for any $i \in [N]$. By spectral decomposition, this immediately shows the \emph{chiral symmetry}
\begin{equation} \label{eq:chiral}
	E_- G(w) = - G(-w) E_-,  \qquad \mbox{with}\quad E_- = \left( \begin{matrix} 1 & 0\cr 0 & -1\end{matrix} \right),
\end{equation}
for the resolvent $G(w) = G^\Defo(w):=(H^\Defo-w)^{-1}$, with spectral parameter $w\in\mathbf{C}\setminus \mathbf{R}$.
We also have $\langle G(w) E_-\rangle=0$ for any $w$ since
$\langle  \boldsymbol{w}_i, E_- \boldsymbol{w}_i\rangle= \Vert {\bm u}_i \Vert^2 - \Vert {\bm v}_i \Vert^2  =0$.

A basic feature of a very large class of random matrices is that their resolvent becomes approximately deterministic
in the large $N$ limit, often even for any spectral parameter with $|\Im w|\ge N^{-1+\epsilon}$; these statements are called {\it local laws}.
In our case the deterministic approximation of the resolvent $G(w)$ is given by
\begin{equation} \label{eq:Mu}
	M(w) = M^\Defo(w):=\left( \begin{matrix}
			M_{11}(w)                             & \frac{\Defo M_{22}(w)}{w + m(w) } \\
			\frac{\Lambda^* M_{11}(w)}{w + m(w) } & M_{22}(w)
		\end{matrix} \right) \in \C^{2N \times 2N} \,,  \qquad w\in \C\setminus \R,
\end{equation}
with each block being understood as a matrix in $\C^{N \times N}$, where the diagonal entries are defined via
\begin{equation}	\label{eq:mde}
	M_{11}(w):=\frac{w + m(w)}{\Defo \Defo^* - (w + m(w))^2}\,, \qquad
	M_{22}(w):=\frac{w + m(w)}{\Defo^* \Defo - (w + m(w))^2}\,.
\end{equation}
Here we require $m(w) = \langle M(w) \rangle$, which is an implicit equation for the function $m(w)$. Simple
calculation shows that this implicit equation is exactly~\eqref{eq:selfconsm}. Moreover, one can easily check that $M(w)$ also satisfies the chiral symmetry \eqref{eq:chiral}, i.e.~
\begin{equation} \label{eq:chiralM}
	E_- M(w) = -M(-w) E_-\,.
\end{equation}

To derive these formulas systematically, we recall that the deterministic approximation to
$G(w)$ is obtained as the unique solution to the \emph{matrix Dyson equation (MDE)}
(introduced first in~\cite{Helton} and extensively studied  in \cite{1506.05095, firstcorr,1804.07752}).
The MDE corresponding to the random matrix $H$ is given by
\begin{equation} \label{eq:MDE}
	- \frac{1}{M(w)} =  w - \hat{\Defo} + \mathcal{S}[M(w)]
\end{equation}
under the constraint $\Im M(w)\cdot  \Im w > 0$, where $\Im M(w) := \frac{1}{2\I}\big[ M(w) - (M(w))^*\big]$. Here
$\mathcal{S}[\cdot]$, the  \emph{self-energy operator}, is defined  via
\[
	\mathcal{S}[T]:=\widetilde{\mathbf{E}} (\widetilde{H} - \mathbf{E} H) T (\widetilde{H} - \mathbf{E} H)
\]
for any $T \in \mathbf{C}^{2N \times 2N}$, where $\widetilde{H}$ denotes an independent copy of $H$.
In our case we can write $\mathcal{S}$ in the following two ways
\begin{equation}\label{Sop}
	\mathcal{S}[T]= 2E_1 \langle T E_2 \rangle + 2 \langle  E_1 T \rangle E_2 = \sum_{\sigma = \pm }\sigma \langle TE_\sigma \rangle E_\sigma,
\end{equation}
with $E_1,E_2$ defined as in \eqref{eq:defe1e2}, and $E_\pm=E_1\pm E_2$. Using $\langle M_{11}(w) \rangle = \langle M_{22}(w) \rangle$ that directly follows from~\eqref{eq:mde},
it is straightforward to check that $M(w)$ as defined in~\eqref{eq:Mu} satisfies the MDE~\eqref{eq:MDE}. Since the MDE has a unique
solution, we see that the density $\rho$  defined via free convolution in Section~\ref{sec:first}
coincides with the {\it self-consistent density of states (scDos)} corresponding to the MDE,
defined as the boundary value of $\frac{1}{\pi}\langle \Im M\rangle$ on the real axis in the theory of MDE~\cite{firstcorr,1804.07752}.

For the reader's convenience in Appendix~\ref{sec:A1} we will collect a few  facts about  $M$, in particular
we will show that it has a continuous extension as a matrix valued function to the real axis, i.e. the
limit $ M(e) := \lim_{\eta \downarrow 0} M(e+ \I \eta)$ exists for any $e\in\R$. This extends the similar result on its trace mentioned
in \eqref{eq:scDos}. Moreover, we will also
show that for spectral parameters $w \in \C\setminus \R$ with $\Re w \in \mathbf{B}_\kappa$,
we have
\begin{equation} \label{eq:Mboundedmain}
	\Vert M(w)\Vert \lesssim 1\,.
\end{equation}
In particular, together with~\eqref{eq:Mu}--\eqref{eq:mde}
this implies that the deterministic terms in~\eqref{eq:uuvvuv} are bounded.
Finally, we will also prove
an important regularity property of the $\kappa$-bulk, namely that
\begin{equation}\label{kappabulkreg}
	\mathrm{dist}(\partial \mathbf{B}_{\kappa'}, \mathbf{B}_\kappa) \ge \mathfrak{c} (\kappa - \kappa')
\end{equation}
for any small $0 < \kappa' < \kappa$ and some $N$-independent constant $ \mathfrak{c}= \mathfrak{c}(\Vert \Defo \Vert) > 0$.
In fact, for our proof it is sufficient if $ \mathfrak{c}= \mathfrak{c}(\kappa, \Vert \Defo \Vert)$, i.e. an additional
$\kappa$ dependence is allowed -- in Appendix~\ref{sec:A1} we will explain that this weaker result is considerably
easier to obtain (see Remark~\ref{rem:cdependence}).
We will also show that $\mathbf{B}_\kappa$ is a finite disjoint union of compact intervals; the number of  these \emph{components}
depends only on $\kappa$ and $\Vert \Defo \Vert$.

The above mentioned concentration of $G$ around $M$ is the content of the following single resolvent \emph{local law}, both in \emph{averaged} and \emph{isotropic form}, which we prove in Appendix \ref{app:locallaw}.
\begin{theorem} {\rm (Single resolvent local law for the Hermitisation $H$)} \label{thm:singleG}\\
	Fix a bounded deterministic $\Defo \in \C^{N \times N}$ and $\kappa > 0$
	independent of $N$. Then, for any $w \in \C \setminus \R$ with $|w| \le N^{100}$ and $\Re w \in \mathbf{B}_\kappa$, we have
	\begin{equation*}
		\left| \langle (G(w) - M(w)) B \rangle  \right| \prec \frac{1}{N \eta}\,, \qquad  \left| \langle \boldsymbol{x}, (G(w) - M(w)) \boldsymbol{y} \rangle \right| \prec \frac{1}{\sqrt{N \eta}}\,,
	\end{equation*}
	where $\eta := |\Im w|$, for any bounded deterministic matrix $\Vert B \Vert \lesssim 1$ and vectors $\Vert \boldsymbol{x} \Vert,  \Vert \boldsymbol{y}\Vert \lesssim 1$.
\end{theorem}

Our main result for the Hermitised random matrix $H$ is the \emph{Eigenstate Thermalisation Hypothesis (ETH)},
that in mathematical terms is the proof of an optimal convergence rate of the strong \emph{Quantum Unique Ergodicity (QUE)} for general observables $A$, uniformly in the \emph{bulk} \eqref{eq:bulk} of the spectrum of $H$, i.e.~in the bulk of the scDos $\rho$.
\begin{theorem} \label{thm:main} {\rm (Eigenstate Thermalisation Hypothesis for the Hermitisation $H$)} \\
	Fix some bounded $\Defo \in \C^{N \times N}$ 
	and $\kappa > 0$ independent of $N$. Let $\{{\bm w}_i\}_{|i| \le N}$ be the orthogonal eigenvectors of the Hermitisation $H$ of $X+\Defo$, where $X$ is an i.i.d.~matrix satisfying Assumption~\ref{ass:iid}. Then, for any deterministic matrix $ A \in \mathbf{C}^{2N \times 2N}$ with $\Vert A \Vert \lesssim 1$ it holds that
	\begin{equation} \label{eq:mainthm} 
		\max_{i,j} \left| \langle {\bm w}_i, A {\bm w}_j \rangle - \delta_{j, i} \frac{\langle  \Im M(\gamma_j) A \rangle}{\langle \Im M(\gamma_j) \rangle} -  \delta_{j, - i} \frac{\langle  \Im M(\gamma_j) E_-  A  \rangle}{\langle \Im M(\gamma_j)\rangle} \right| \prec \frac{1}{\sqrt{N}}\,,
	\end{equation}
	where the maximum is taken over all $|i|, |j| \le N$, such that the quantiles $\gamma_i, \gamma_j \in \mathbf{B}_\kappa$ defined in \eqref{eq:quantiles} are in the bulk of the scDos $\rho$.
\end{theorem}

The main technical result underlying Theorem~\ref{thm:main} is an averaged local law for \emph{two} resolvents
with different spectral parameters, which we will formulate in Theorem~\ref{thm:multiGll} later.

\begin{remark}
	Given the optimal bound \eqref{eq:mainthm}, following a Dyson Brownian Motion (DBM) analysis similar to \cite{normalfluc, A2}, it is possible to prove a CLT for single diagonal overlaps $ \langle {\bm w}_i, A {\bm w}_i\rangle $. However, for the sake of brevity, we do not present this argument here and defer the CLT analysis to future work.
\end{remark}

In the following Section \ref{sec:proofmain} we precisely define the regularisation
and we will prove our main results formulated above
assuming the key technical Proposition~\ref{prop:2Gav}.
This proposition  is obtained from a \emph{local law}, which we prove in Section \ref{sec:prooflocallaw}.
Local laws are proved by a hierarchy  of \emph{master and reduction inequalities}, that are derived in Sections \ref{sec:proofmaster} and \ref{sec:proofreduc}, respectively.  Appendix~\ref{app:motivation} contains two motivating calculations for the correct regularisation.
Several technical and auxiliary results are deferred to the other appendices.

\section{Proof of the main results} \label{sec:proofmain}

The key to understanding the eigenvector overlaps and showing
our main results is an improved bound on the averaged trace of \emph{two} resolvents with \emph{regular}
(see Section~\ref{sec:reg} below for the precise definition) deterministic matrices $A_1, A_2$ in between, i.e.~for
\begin{equation} \label{eq:2Gav}
	\langle G(w_1) A_1 G(w_2) A_2 \rangle\,.
\end{equation}
Naively, for arbitrary $A_1, A_2$, estimating \eqref{eq:2Gav} via a trivial Schwarz inequality and Ward identity yields the upper bound $|\langle G(w_1) A_1 G(w_2) A_2 \rangle| \prec 1/\eta$, where $\eta := \min_j |\Im w_j|$. However, this bound drastically improves, whenever the matrices $A_1, A_2$ are \emph{regular}, i.e.~orthogonal to certain \emph{critical} eigenvectors $V_\pm$ of the associated two-body stability operators \eqref{eq:stabop}, which is denoted as $A_j = \mathring{A}_j$; see \eqref{eq:circ} and Definitions~\ref{def:reg obs1} and~\ref{def:regobs}. In this case, in our key Proposition~\ref{prop:2Gav} we will show that
\begin{equation*}
	|\langle G(w_1) \mathring{A}_1 G(w_2) \mathring{A}_2 \rangle| \prec 1
\end{equation*}
even for very small $\eta \sim N^{-1+\epsilon}$ as a consequence of a more precise
\emph{local law} for \eqref{eq:2Gav}, which we present in Section \ref{sec:prooflocallaw}.
We find that (see Theorem \ref{thm:multiGll} and Remark \ref{rmk:sqrteta}) both the size of its deterministic approximation and the fluctuation around this mean heavily depend on whether (one or both of) the matrices $A_1, A_2$ are regular, i.e.~satisfy $\langle V_\pm, A_j \rangle = 0$, or not.

Therefore, the general rather \emph{structural} regularizing decomposition (or \emph{regularisation}) of a matrix $A$ shall be conducted as
\begin{equation} \label{eq:circ}
	\boxed{A^\circ \equiv \mathring{A} := A - \langle V_+, A \rangle U_+ - \langle V_-, A \rangle U_- }
\end{equation}
for $U_\sigma, V_\sigma \in \mathbf{C}^{2N \times 2N}$ satisfying $\langle V_\sigma, U_\tau \rangle = \delta_{\sigma, \tau}$ and the normalisation $\langle U_\sigma, U_\sigma \rangle  = 1$, where recall that
$\langle R,T \rangle := \langle R^* T \rangle$ denotes the (normalised) Hilbert-Schmidt scalar product.
The \emph{regularisation  map}
\begin{equation}
	(1-\Pi): \C^{2N \times 2N} \to \C^{2N \times 2N}\,, \quad A \mapsto \mathring{A}\,,
\end{equation}
where $\Pi$ is a \emph{two-dimensional} (non-orthogonal) projection,\footnote{The condition $\langle V_\sigma, U_\tau \rangle = \delta_{\sigma, \tau}$ guarantees that the regularisation is idempotent, i.e.~$(\mathring{A})^\circ = \mathring{A}$ and $\Pi^2 = \Pi$.} is closely related
to the built-in chiral symmetry \eqref{eq:chiral} of our model. Indeed, for other ensembles without this special structure only \emph{one} of the terms $\langle V_\sigma, A \rangle U_\sigma$ in~\eqref{eq:circ} would be present.

As mentioned above, the matrices $V_\pm$ are determined as \emph{critical} eigenvectors (with corresponding small eigenvalue) of naturally associated two-body stability operators with their precise form worked out in Appendix \ref{app:motivation} and given in \eqref{eq:Vpmexplicit}. In Appendix \ref{app:motivation} we also give two different calculations
that helped us guess  these formulas. However, for the directions $U_\pm$ there are \emph{a priori} no further constraints (apart from orthogonality and normalisation).
Hence, as it turns out to be convenient for our proofs, we will \emph{choose} the matrices $U_\sigma$ (in principle, even allowing for two different deformations $\Defo_1 \neq \Defo_2$) in such a way, that a resolvent identity
\begin{equation} \label{eq:resolventid}
	G^{\Defo_1}(w_1) U_\sigma G^{\Defo_2}(w_2) \approx\big(G^{\Defo_1}(w_1) - G^{\Defo_2}(\sigma w_2)\big) U_\sigma\,,
\end{equation}
can be applied (here, the symbol `$\approx$' neglects lower order terms). This is used to reduce the number of resolvents in a chain. Note that, again due to the eminent chiral symmetry \eqref{eq:chiral} for the resolvents, there are in fact \emph{two} matrices $U_\sigma$ for which a resolvent identity \eqref{eq:resolventid} can be applied.

Although the regularisation \eqref{eq:circ} shall be motivated for arbitrary deformations $\Defo_1, \Defo_2$ in Appendix~\ref{app:motivation}, we will henceforth choose a single bounded deformation $\Defo \in \C^{N \times N}$, which remains fixed with the just mentioned exception in Appendix~\ref{app:motivation}. For a single deformation $\Defo$, this restricts the matrices $U_\pm$ satisfying \eqref{eq:resolventid} to be given by $E_\pm$.

In case that the spectral parameters $(w_1, w_2)$ appearing in \eqref{eq:2Gav} (with a single fixed deformation $\Defo$) are such that \emph{none} of the eigenvectors of the stability operator is \emph{critical} (cf.~Appendix \ref{app:stabop}), we consider \emph{every} matrix $A$ as regular. The distinction between these two scenarios is regulated by cutoff functions $\mathbf{1}_\delta^\pm$
introduced in \eqref{eq:case regulation} below.

\subsection{Regular observables: A bound on \eqref{eq:2Gav}}
\label{sec:reg}

As already mentioned above, our main result for the Hermitised random matrix, Theorem~\ref{thm:main}, shall be derived from a bound on \eqref{eq:2Gav}, where we assume the (real parts of the) spectral parameters $w_1, w_2$ to be in the \emph{bulk} of the scDos $\rho$ (recall \eqref{eq:bulk}).

We now specify the concept of regularisation \eqref{eq:circ} to our setting. The eigenvectors $V_\pm$ will be computed in Appendix \ref{app:motivation}, the matrices $U_\pm$ are simply chosen as $E_\pm$.

\begin{definition}{\rm (Regular observables)} \label{def:reg obs1}
	Given $\kappa > 0$, let\footnote{Note that the parameter $\delta > 0$ is independent of the matrix size $N$.}
	\begin{equation} \label{eq:delta}
		\delta = \delta(\kappa, \Vert \Defo \Vert) > 0
	\end{equation}
	be sufficiently small (to be chosen below, see \eqref{eq:deltachoice}) and let $w, w' \in \C\setminus \R$ with $\Re w, \Re w' \in \mathbf{B}_\kappa$ be spectral parameters. Furthermore, we introduce the (symmetric) cutoff functions
	\begin{equation} \label{eq:case regulation}
		\mathbf{1}_\delta^{\pm}(w, w') := \phi_\delta(\Re w \mp \Re w' ) \ \phi_\delta(\Im w ) \ \phi_\delta(\Im w')\,,
	\end{equation}
	where $0 \le \phi_\delta \le 1$ is a smooth symmetric bump function on $\R$ satisfying $\phi_\delta(x) = 1$ for $|x|\le \delta/2$ and $\phi_\delta(x) = 0$ for $|x| \ge \delta$.
	\begin{itemize}
		\item[(a)] We define the \emph{$(w,w')$-regular component} or \emph{$(w,w')$-regularisation} $\mathring{A}^{w,w'}$ of a matrix $A$  as\footnote{\label{ftn:alternreg}Putting the summation parameter $\tau$ at the second spectral parameter $w'$ (and not at $w$) is simply a free choice, which we made here. More precisely, defining the regularisation as
				\begin{equation*}
					\tilde{\mathring{A}}^{{w,w'}}:= A -  \sum_{\tau = \pm}\mathbf{1}_\delta^{\tau \mathfrak{s} }(w, w') \frac{\langle M(\Re w + \tau \I \Im w) A M(\Re w' + \I \Im w') E_{\tau \mathfrak{s} } \rangle}{\langle M(\Re w + \tau \I \Im w) E_{\tau\mathfrak{s} } M( \Re w' + \I \Im w') E_{\tau \mathfrak{s} } \rangle} E_{\tau\mathfrak{s} }
				\end{equation*}
				would equally work in our proofs (see Appendices \ref{app:motivation} and \ref{app:stabop} for details).}
			\begin{equation}  \label{eq:reg A1A2}
				\boxed{	\mathring{A}^{{w,w'}}:= A -  \sum_{\tau = \pm}\mathbf{1}_\delta^{\tau \mathfrak{s} }(w, w') \frac{\langle M(\Re w + \I \Im w) A M(\Re w' + \tau \I \Im w') E_{\tau \mathfrak{s} } \rangle}{\langle M(\Re w + \I \Im w) E_{\tau \mathfrak{s}} M(\Re w' + \tau \I \Im w') E_{\tau \mathfrak{s} } \rangle} E_{\tau\mathfrak{s}}\,,}
			\end{equation}
			where the relative sign of the imaginary parts is defined as
			\begin{equation} \label{eq:sign}
				\mathfrak{s} \equiv \mathfrak{s}_{w,w'}:= - \sgn (\Im w \,  \Im w')\,.
			\end{equation}
		\item[(b)] We say that $A$ is \emph{$(w,w')$-regular} if and only if $A =  \mathring{A}^{{w,w'}}$.
	\end{itemize}
\end{definition}
The regularisation shall be revisited in Definition \ref{def:regobs}, where we tailor it to certain averaged \eqref{eq:av} or isotropic \eqref{eq:iso} resolvent chains.
\begin{remark} \label{rmk:regular}
	We have several comments concerning the above definition.
	\begin{itemize}
		\item[(i)] In case that at least one of the spectral parameters is away from the imaginary axis, say $|\Re w| > \delta$ w.l.o.g., then the regularisation in \eqref{eq:reg A1A2} contains at most one summand: If $\mathbf{1}_\delta^{+}(w, w') = 1$, i.e.~$\Re w$ is close to $\Re w'$, then we have that
			\begin{equation*}
				\mathring{A}^{{w,w'}}:= A -   \frac{\langle M(w) A M(\Re w' + \mathfrak{s} \I \Im w')  \rangle}{\langle M(w)  M(\Re w' + \mathfrak{s} \I \Im w')  \rangle} E_{+}\,,
			\end{equation*}
			whereas if $\mathbf{1}_\delta^{-}(w, w') = 1$, i.e.~if $\Re w$ is close to $- \Re w'$, then we have that
			\begin{equation*}
				\mathring{A}^{{w,w'}}:= A -   \frac{\langle M(w) A E_-M(-\Re w' + \mathfrak{s} \I \Im w')  \rangle}{\langle M(w)  M(- \Re w' + \mathfrak{s} \I \Im w')  \rangle} E_{-}\,,
			\end{equation*}
			where we used that $M(w)E_- = - E_- M(-w)$ (see \eqref{eq:chiralM})
		\item[(ii)] The cutoff functions in \eqref{eq:case regulation} satisfy the basic symmetry properties
			\begin{equation*}
				\mathbf{1}^\pm_\delta(w,w') = \mathbf{1}^\pm_\delta(\bar{w},w') = \mathbf{1}^\pm_\delta(w,\bar{w}') = \mathbf{1}^\pm_\delta(\bar{w},\bar{w}')\,.
			\end{equation*}
			However, $\mathring{A}$ is \emph{not} symmetric in its two spectral parameters, i.e.~$\mathring{A}^{w,w'} \neq \mathring{A}^{w',w}$ in general
		\item[(iii)] For spectral parameters satisfying $\mathbf{1}^\pm_\delta(w,w') > 0$, it will be shown in Appendix \ref{app:stabop} that the respective denominators  in \eqref{eq:reg A1A2} are bounded away from zero. In particular, the linear map $A \mapsto \mathring{A}$ is bounded with a bound depending only on $\delta $ and $\Vert \Defo \Vert$.
		\item[(iv)] Whenever it holds that $\mathbf{1}_\delta^\pm(w,w') = 0$ then also $\mathbf{1}_{\delta'}^\pm(w,w') = 0$ for every $0 < \delta'< \delta$. Conversely, whenever it holds that $\mathbf{1}_\delta^\pm(w,w') = 1$ then also $\mathbf{1}_{\delta'}^\pm(w,w') = 1$ for every $0 <\delta< \delta'$.
		\item[(v)] We point out that the notion of regularity implicitly depends on $\kappa$ and $\delta$ and hence also on the (norm of the) deformation $\Defo$.
	\end{itemize}
\end{remark}


Moreover, the regularisation defined above satisfies the following elementary properties. The identities in \eqref{eq:regularbar2} and \eqref{eq:regularbasic} are immediate from the definition, the perturbative statements are proven in Appendix \ref{app:stabop}.
\begin{lemma} \label{lem:regularbasic}
	Fix a bounded deterministic deformation $\Defo \in \C^{N \times N}$ and let $A \in \C^{2N \times 2N}$ be an arbitrary bounded deterministic matrix. 
	\begin{itemize}
		\item[(i)] Let $w, w'\in \C\setminus \R$ with $\Re w,  \Re w' \in \mathbf{B}_\kappa$. Then, we have the identities
			\begin{equation}
				\big(	\mathring{A}^{{w,w'}} \big)^* = 	\mathring{(A^*)}^{{\bar{w}',\bar{w}}}\,, \quad
				\mathring{A}^{w,w'}E_- = \mathring{(AE_-)}^{{w,-w'}} \,, \quad 	E_- \mathring{A}^{w,w'} = \mathring{(E_- A)}^{{-w,w'}}\,. \label{eq:regularbasic}
			\end{equation}
		\item[(ii)] Moreover, by definition it holds that
			\begin{equation}
				\mathring{A}^{{{w},\bar{w}'}} = \mathring{A}^{{{w},{w}'}} \,,  \label{eq:regularbar2}
			\end{equation}
			and setting $\mathfrak{s} := - \sgn(\Im w \Im w')$, we have the perturbative estimate\footnote{Note that the asymmetry between \eqref{eq:regularbar1} and \eqref{eq:regularbar2} stems from the asymmetry imposed in the definition of the regularisation, namely by placing the summation index $\tau$ in \eqref{eq:reg A1A2} at the second spectral parameter.}
			\begin{equation}
				\mathring{A}^{{\bar{w},{w}'}} = \mathring{A}^{{{w},{w}'}} +  \mathcal{O}(|w - \mathfrak{s}  \bar{w}'|\wedge 1) E_\mathfrak{s}  +\mathcal{O}(|w +\mathfrak{s}  {w}'|\wedge 1) E_{-\mathfrak{s} } \label{eq:regularbar1}\,.
			\end{equation}
		\item[(iii)] 	Let $w_1, w_1', w_2,w_2' \in \C\setminus \R$ with $\Re w_1,  \Re w_1', \Re w_2, \Re w_2'\in \mathbf{B}_\kappa$ as well as $\Im w_1 \cdot \Im w_2 >0 $ and $\Im w_1' \cdot \Im w_2' >0 $ be spectral parameters. 
			Then
			we have that
			\begin{align}
				\mathring{A}^{{w_2,w'_1}} & = 	\mathring{A}^{{w_1,w_1'}} +  \mathcal{O}\big(|w_1 - w_2|  \wedge 1\big) E_+ + \mathcal{O}\big(|w_1 - w_2|  \wedge 1\big) E_-\,, \label{eq:regularperturb1}         \\[1mm]
				\mathring{A}^{{w_1,w'_2}} & = 	\mathring{A}^{{w_1,w_1'}} +  \mathcal{O}\big(|w'_1 - w'_2|  \wedge 1\big) E_+ +  \mathcal{O}\big(|w'_1 - w'_2|  \wedge 1\big) E_-   \,. \label{eq:regularperturb2}
			\end{align}
	\end{itemize}
\end{lemma}

We can now state the bound on \eqref{eq:2Gav} for regular observables,
which shall be proven in Section \ref{sec:prooflocallaw} as an immediate corollary to a local \cred{law} for \eqref{eq:2Gav} given in Theorem \ref{thm:multiGll} and the bound from Lemma~\ref{lem:Mbound}.
\begin{proposition} \label{prop:2Gav} Fix a bounded deterministic $\Defo \in \C^{N \times N}$,
	$\epsilon >0 $, $\kappa > 0$, and let $w_1, w_2 \in \C$ with $|w_1|, |w_2| \le N^{100}$, $\Re w_1, \Re w_2 \in \mathbf{B}_\kappa$,  and $|\Im w_1|, |\Im w_2| \ge N^{-1+\epsilon}$. Moreover, let $A_1 \in \C^{2N \times 2N}$ be a \emph{$(w_1, w_2)$-regular} 
	and $A_2 \in \C^{2N \times 2N}$ a \emph{$(w_2, w_1)$-regular}
	deterministic matrix, both satisfying $\Vert A_1 \Vert , \Vert A_2 \Vert \lesssim 1$.
	Then we have
	\begin{equation} \label{eq:2Gbound}
		\big|  \big\langle G(w_1) \mathring{A}_1^{{w_1, w_2}} G(w_2) \mathring{A}_2^{{w_2, w_1}}\big\rangle  \big| \prec 1 \,.
	\end{equation}
\end{proposition}

\subsection{Estimating \eqref{eq:2Gav} for general observables}
\label{sec:general}
Armed with the correct regularisation, we can now present a systematic analysis of
$\langle G(w_1) A_1 G(w_2) A_2\rangle$
from~\eqref{eq:2Gav} for \emph{arbitrary} bounded deterministic
matrices $A_1, A_2$.  Decomposing $A_1, A_2$ according to Definition \ref{def:reg obs1} as
\begin{equation} \label{eq:A1A2decomp}
	\begin{split}
		A_1 &= \mathring{A}_1^{w_1, w_2} + \langle \! \langle A_1 \rangle \! \rangle^+_{w_1, w_2} E_+ + \langle \! \langle A_1 \rangle \! \rangle^-_{w_1, w_2} E_-\,, \\[1mm]
		A_2 &= \mathring{A}_2^{w_2, w_1} + \langle \! \langle A_2 \rangle \! \rangle^+_{w_2, w_1} E_+ + \langle \! \langle A_2 \rangle \! \rangle^-_{w_2, w_1} E_-\,,
	\end{split}
\end{equation}
(where $\langle \! \langle \cdot \rangle \!\rangle^\sigma_{w,w'}$ can be read off as the coefficients in \eqref{eq:reg A1A2}) and plugging \eqref{eq:A1A2decomp} into \eqref{eq:2Gav}, we find that
\begin{equation} \label{eq:correlator}
	\begin{split}
		\big\langle G(w_1)A_1 G(w_2) A_2 \big\rangle = &\sum_{\sigma, \tau} \langle \! \langle A_1 \rangle \! \rangle^\sigma_{w_1, w_2} \langle \! \langle A_2 \rangle \! \rangle^\tau_{w_2, w_1}\big\langle G(w_1)E_\sigma G(w_2) E_\tau \big\rangle \\
		&+ \sum_\sigma \langle \! \langle A_1 \rangle \! \rangle^\sigma_{w_1, w_2} \big\langle G(w_1)E_\sigma G(w_2) \mathring{A}_2^{w_2, w_1} \big\rangle \\
		&+ \sum_\tau \langle \! \langle A_2 \rangle \! \rangle^\tau_{w_2, w_1} \big\langle G(w_1)\mathring{A}_1^{w_1, w_2}  G(w_2)E_\tau \big\rangle \\
		&+ \big\langle G(w_1) \mathring{A}_1^{{w_1, w_2}} G(w_2) \mathring{A}_2^{{w_2, w_1}}\big\rangle\,.
	\end{split}
\end{equation}

Which terms
in~\eqref{eq:correlator} are effectively present depends on the coefficients $\langle \! \langle \cdot \rangle \!\rangle^\sigma_{w,w'}$, i.e.
on the singular components of $A_1, A_2$. For terms with nonzero coefficients the following rule of thumb applies.
Denoting $\eta := \min \big( |\Im w_1|, |\Im w_2| \big)\ge N^{-1+\epsilon}$, the
terms $\langle GEGE\rangle$ in the first line of \eqref{eq:correlator} are bounded by $1/\eta$,
the terms $\langle GEG\mathring{A}\rangle$ in the two middle lines of \eqref{eq:correlator}
are bounded by $1/\sqrt{\eta}$,
and $\langle G\mathring{A}G\mathring{A}\rangle$ in the last line is of order one (Proposition~\ref{prop:2Gav}).
This is in perfect agreement with the \emph{$\sqrt{\eta}$-rule}
mentioned in the Introduction (see also Remark \ref{rmk:sqrteta} below).
Some of these bounds are actually sharp for special values of $w_1, w_2$, for example
$$
	\langle G(w) E_+ G(\bar w) E_+ \rangle = \frac{\langle\Im G(w)\rangle}{\eta}\sim\frac{1}{\eta}, \qquad
	\mbox{or}\quad \langle G(w) E_- G(-\bar w) E_- \rangle
	= - \frac{\langle \Im G(w)\rangle}{\eta},
$$
where we used the chiral symmetry~\eqref{eq:chiral}. In fact, two terms with  $\sigma\tau=-1$ in the first
line of~\eqref{eq:correlator} are identically zero by applying the chiral symmetry, followed by the resolvent
identity and  $\langle G E_-\rangle=0$.
For a middle term in~\eqref{eq:correlator} we have
$$
	\langle G(w) E_+ G(\bar w)  \mathring{A}^{{\bar w, w}} \rangle  = \frac{1}{\eta}\langle \Im G(w)  \mathring{A}^{{\bar w, w}}\rangle
	\lesssim 1+  \frac{1}{N\eta}\frac{1}{\sqrt{\eta}}.
$$
In the very last relation we treated $\langle G(w)  \mathring{A}^{{\bar w, w}}\rangle $ and
$\langle G(\bar w)  \mathring{A}^{{\bar w, w}}\rangle $ separately. In both cases we first used Lemma \ref{lem:regularbasic} to
adjust the regularisation to  $\mathring{A}^{{w, w}}$ and $\mathring{A}^{{\bar w, \bar w}}$, respectively, to match
the new single-resolvent setup and then we applied the corresponding single-resolvent local law with regular observable
(see Theorem~\ref{thm:singleGopt} below).

Note that the most critical estimate concerns
the last line of~\eqref{eq:correlator}, i.e.~the regular part for both observable matrices.
The bound \eqref{eq:2Gbound} is obtained from a local law with \emph{two} resolvents and \emph{two} regular matrices, while the first and
the middle terms in~\eqref{eq:correlator} can be understood already from an improved local law for \emph{one} resolvent and
\emph{one} regular matrix
(see Theorem~\ref{thm:singleGopt} below) after applying resolvent identities and adjusting the
regularisation by Lemma \ref{lem:regularbasic}.
Furthermore, observe that
the sizes of the first three lines  in~\eqref{eq:correlator} are sensitive to $w_1, w_2$ via the usual resolvent identity (see \eqref{eq:resolid} below) and the chiral symmetry \eqref{eq:chiral}, for example
\begin{equation*}
\begin{split}
	&|\langle G(w_1) E_+ G(w_2) E_+ \rangle| = \left| \frac{\langle G(w_1)- G(w_2) \rangle}{w_1-w_2} \right| 
\lesssim \frac{1}{|w_1-w_2|}\,, \\
\quad \mbox{or}\quad  &|\langle G(w_1) E_- G(w_2) E_- \rangle | = |\langle G(w_1) G(-w_2) \rangle |\lesssim \frac{1}{|w_1+w_2|},
\end{split}
\end{equation*}
while the last line in~\eqref{eq:correlator} is typically order one.

Summarizing, the singular parts of  $\langle G(w_1) A_1 G(w_2) A_2\rangle$
can be explicitly computed (using single-resolvent local laws) as explicit functions of $w_1, w_2$,
while the regular part remains of order one. A combination of our  decomposition~\eqref{eq:reg A1A2},
the perturbation formulas from Lemma~\ref{lem:regularbasic}, and our single- and two-resolvent
local laws  together
with their explicit deterministic terms from the subsequent Section~\ref{sec:prooflocallaw}
provide an effective recipe to compute $\langle G(w_1) A_1 G(w_2) A_2\rangle$  with high precision
in all cases. We refrain from formulating it as a comprehensive theorem due to the large number of cases.

\subsection{Proof of the main results} \label{subsec:proofmain}
We will first focus on the proof of Theorem \ref{thm:main} and turn to the proofs of Theorem \ref{thm:main singvec}  and Theorem \ref{thm:overlap} afterwards.

\subsubsection{Proof of Theorem \ref{thm:main}} As a first step towards the proof of Theorem \ref{thm:main}, we show that \eqref{eq:mainthm} indeed follows from a bound similar to \eqref{eq:2Gbound}, where $G$ is replaced by $\Im G$. The proof of the following simple lemma is given after completion of the proof of Theorem \ref{thm:main}.
\begin{lemma} \label{lem:overlap2G}
	Fix a bounded deterministic $\Defo \in \C^{N \times N}$,
	$\epsilon >0 $, $\kappa > 0$,  and let $B \in \C^{2N \times 2N}$. Then, for any bulk indices $|i|, |j| \le N$, i.e.~with $\gamma_i, \gamma_j \in \mathbf{B}_\kappa$, and $\eta \ge N^{-1+\epsilon}$, we have
	\begin{equation} \label{eq:overlap2G}
		N\left| \langle \boldsymbol{w}_i, B \boldsymbol{w}_j \rangle\right|^2 \prec  (N \eta)^2 \langle \Im G(\gamma_i+ \I \eta) B \Im G(\gamma_j+ 2\I \eta) B^* \rangle\,.
	\end{equation}
\end{lemma}

The same bound still holds without the factor of two {in the argument of the second resolvent} in \eqref{eq:overlap2G}. However, we chose to have it, in order to ensure that the spectral parameters of the two resolvents are always forced to be different.

\begin{proof}[Proof of Theorem \ref{thm:main}]
	Having Lemma \ref{lem:overlap2G} at hand, we are left with estimating the rhs.~of \eqref{eq:overlap2G} for
	\begin{equation} \label{eq:B}
		B = A - \frac{\langle \Im M(\gamma_j) A \rangle }{\langle \Im M(\gamma_j)\rangle } E_+ - \frac{\langle \Im M(\gamma_j) E_- A \rangle }{\langle \Im M(\gamma_j) \rangle } E_-
	\end{equation}
	using Proposition \ref{prop:2Gav}. Note that the two terms in \eqref{eq:mainthm} carrying a $\delta$-symbol arise from the orthogonality relations $\langle \boldsymbol{w}_i, E_\pm \boldsymbol{w}_j \rangle = \delta_{j, \pm i}$, following from the spectral symmetry described around \eqref{eq:chiral}.

	We now write out $\Im G(w) = (G(w) - G(\bar{w}))/(2\I)$, such that \eqref{eq:overlap2G} leaves us with four different terms, each of which can be bounded individually. Since their treatment is completely analogous, we focus on the exemplary term
	\begin{equation} \label{eq:exemp term}
		\langle G(\gamma_i+ \I \eta) B G(\gamma_j - 2\I \eta) B^* \rangle
	\end{equation}
	with the deterministic matrix $B$ being defined in \eqref{eq:B}. We rely on the following simple perturbative lemma, which follows from Lemma \ref{lem:regularbasic} by invoking Lemma~\ref{lem:Mbasic}.

	\begin{lemma} \label{lem:perturb mainthm}Using the notation introduced in \eqref{eq:reg A1A2}, the matrix $B \in \C^{2N \times 2N}$ from \eqref{eq:B} satisfies
		\begin{equation} \label{eq:Bperturb}
			\begin{split}
				B &= \mathring{A}^{{\gamma_i + \I \eta, \gamma_j - 2\I \eta}} + \mathcal{O}\big(|\gamma_i - \gamma_j|+ \eta\big) E_+ + \mathcal{O}\big(|\gamma_i + \gamma_j| + \eta\big) E_-\,, \\
				B^* &= \mathring{(A^*)}^{{\gamma_j - 2 \I \eta, \gamma_i + \I \eta}} + \mathcal{O}\big(|\gamma_i - \gamma_j| + \eta\big) E_+ + \mathcal{O}\big(|\gamma_i + \gamma_j|+ \eta\big) E_-\,.
			\end{split}
		\end{equation}
	\end{lemma}
	Hence, plugging \eqref{eq:Bperturb} into \eqref{eq:exemp term}, we get a sum of several terms, which can all be estimated separately. For the `leading term', we use Proposition \ref{prop:2Gav} to get that
	\begin{equation*}
		\big\vert \big\langle G(\gamma_i + \I \eta) \mathring{A}^{{\gamma_i + \I \eta, \gamma_j - 2\I \eta}} G(\gamma_j - 2\I \eta) \mathring{(A^*)}^{{\gamma_j - 2\I \eta, \gamma_i + \I \eta}} \big\rangle \big\vert \prec 1\,.
	\end{equation*}

	Two further representative terms are given by
	\begin{equation*} 
		\mathcal{O}\big(|\gamma_i \mp \gamma_j| + \eta\big) \, \langle G(\gamma_i + \I \eta) E_\pm G(\gamma_j - 2\I \eta) C \rangle\,,
	\end{equation*}
	where $C \in \C^{2N \times 2N}$ is some generic bounded matrix. Now, by using \eqref{eq:chiral}, these terms can be rewritten as
	\begin{equation*}
		\mathcal{O}\big(|\gamma_i \mp \gamma_j| + \eta\big) \, \langle G(\gamma_i + \I \eta)  G(\pm(\gamma_j - 2\I \eta)) E_\pm C \rangle\,.
	\end{equation*}
	For either sign choice (due to the factor two), we can now employ a simple resolvent identity 
	\begin{equation} \label{eq:resolid}
	G(w_1) G(w_2) = \frac{G(w_1) - G(w_2)}{w_1 - w_2}\,, 
	\end{equation}
	 leaving us with
	 	\begin{equation*}
		\frac{\mathcal{O}\big( |\gamma_i - \gamma_j| + \eta \big)}{(\gamma_i \mp \gamma_j) + (1 \pm2) \I \eta} \langle [G(\gamma_i + \I \eta) - G(\pm (\gamma_j - 2\I \eta))] C \rangle \,,
	\end{equation*}
	which is surely stochastically dominated by one by means of Theorem \ref{thm:singleG}.
	Thus, collecting all the terms, we find that $\vert \eqref{eq:exemp term}\vert \prec 1$.

	Finally, we choose $\eta = N^{-1+\xi}$ for an arbitrarily small $\xi > 0$, such that Lemma \ref{lem:overlap2G} with $B$ as in \eqref{eq:B} yields Theorem~\ref{thm:main}.
\end{proof}
We conclude with giving a proof of Lemma \ref{lem:overlap2G}.
\begin{proof}[Proof of Lemma \ref{lem:overlap2G}]

	By spectral decomposition we write
	\begin{equation*}
		\begin{split}
			\langle \Im G(\gamma_i + \I \eta) B \Im G(\gamma_j + 2\I \eta) B^* \rangle&=\frac{1}{2N}\sum_{k,l} \frac{2\eta^2| \langle \boldsymbol{w}_k, B \boldsymbol{w}_l \rangle|^2}{[(\lambda_k-\gamma_i)^2+\eta^2][(\lambda_l-\gamma_j)^2+4\eta^2]} \\
			&\succ \frac{| \langle \boldsymbol{w}_i, B \boldsymbol{w}_j \rangle|^2}{N\eta^2}\,,
		\end{split}
	\end{equation*}
	which proves \eqref{eq:overlap2G}. We point out that in the last inequality we used rigidity of the eigenvalues \cite{firstcorr, slowcorr}:
	\begin{equation}
		\label{eq:rigidity}
		|\lambda_i-\gamma_i|\prec \frac{1}{N}\,,
	\end{equation}
	which holds for bulk indices as a standard consequence of the single-resolvent local law, Theorem~\ref{thm:singleG}.
\end{proof}

\subsubsection{Proof of Theorem \ref{thm:main singvec}}

The bounds in \eqref{eq:uu}, \eqref{eq:vv}, and \eqref{eq:uv} follow from Theorem~\ref{thm:main} by choosing
\[
	A =  \left(\begin{matrix}
			B & 0 \\
			0 & 0
		\end{matrix}\right)\,,  \qquad
	A = \left(\begin{matrix}
			0 & 0 \\
			0 & B
		\end{matrix}\right)\,, \qquad \text{and} \qquad A=\left(\begin{matrix}
			0 & 0 \\
			B & 0
		\end{matrix}\right)\,,
\]
respectively, and invoking \eqref{eq:Mu}--\eqref{eq:mde}. \qed

\subsubsection{Proof of Theorem \ref{thm:overlap}}

By the definition
\[
	H^z:=\left(\begin{matrix}
			0               & X+\Lambda-z \\
			(X+\Lambda-z)^* & 0
		\end{matrix}\right)
\]
it follows that $\mu\in \mathrm{Spec}(X+\Lambda)$ if and only if $\lambda_1^\mu=0$. Here by {$\{\lambda_i^z\}_{i\in [N]}$ we denoted the increasingly ordered non--negative} eigenvalues of $H^z$. We remark that $\Lambda$ is omitted by the notation since it is fixed throughout the proof. In particular, using the bound for products of two resolvents and two regular matrices in \eqref{eq:2Gbound}, we will now prove the lower bound in \eqref{eq:Oiilb} for the overlap of left and right eigenvectors corresponding to eigenvalues $\mu$ which lies in the bulk of the spectrum of $X+\Defo$.

\begin{proof}[Proof of Theorem \ref{thm:overlap}]

	Define
	\[
		F:=\left(
		\begin{matrix}
				0 & 0 \\
				1 & 0
			\end{matrix}
		\right) \in \C^{2N \times 2N}\,,
	\]
	then by \eqref{eq:2Gbound}, for $\eta\ge N^{-1}$, we conclude
	\begin{equation}
		\label{eq:bsup}
		\sup_{z\in \mathrm{bulk}} \langle\Im G^z(\ii\eta)F\Im G^z(\ii\eta)F^*\rangle  \prec 1\,,
	\end{equation}
	where the supremum is taken over the bulk as given in Definition \ref{def:bulknonherm}. {The fact that \eqref{eq:bsup} holds for the supremum over the $z$'s with very high probability follows by a standard grid argument together with the Lipschitz continuity of $z\mapsto \Im G^z$.}
	Here we used that $F$ is regular in the sense of \eqref{eq:reg A1A2}; this immediately follows from the fact that $F$ is (block) off--diagonal and $\Im M(\ii\eta)$ is (block) diagonal (see Lemma \ref{lem:MDE}). We now want to show that if we choose $z=\mu_i$ to be a bulk eigenvalue of $X+\Lambda$ the upper bound \eqref{eq:bsup} implies a lower bound on $\mathcal{O}_{ii}$. To make the notation simpler, from now on we denote $\mu=\mu_i$.

	Consider the non-Hermitian left/right--eigenvectors $ \boldsymbol{l}, \boldsymbol{r}$,
	with corresponding eigenvalue $\mu$, defined as in \eqref{eq:leftrightev}
	and set
	\begin{equation*}
		\mathcal{P}:=
		\left(\begin{matrix}
				\frac{\overline{\boldsymbol{l}} \, \overline{\boldsymbol{l}}^*}{\lVert \boldsymbol{l}\rVert^2} & 0                                                                    \\
				0                                                                                              & \frac{\boldsymbol{r}\boldsymbol{r}^*}{\lVert \boldsymbol{r}\rVert^2}
			\end{matrix}\right)\,.
	\end{equation*}
	Clearly $\mathcal{P}$ is a rank two orthogonal projection whose range  lies in
	the kernel of $H^\mu$, recalling that up to scalar multiples the non-Hermitian eigenvectors
	$\overline{\boldsymbol{l}}, \boldsymbol{r}$ coincide with some singular vectors ${\bm u}, {\bm v}$ of $X+\Lambda-\mu$, respectively,
	forming an eigenvector ${\bm w} = ({\bm u}, {\bm v})$ in the kernel of  $H^\mu$.
	Note that $\mathrm{Ker}(H^\mu)$ has dimension two if $\mu$ is a simple
	eigenvalue, but in general the multiplicity of $\mu$ and the multiplicity of $\lambda_1^\mu=0$ may differ.
	Let $\mathcal{Q}$ be the orthogonal projection onto the  kernel of $H^\mu$,
	then $\mathcal{P}\le \mathcal{Q}$.
	Then, almost surely, by spectral decomposition (and by the spectral symmetry of $H^\mu$)
	\begin{equation*}
		\Im G^\mu(\ii\eta)= \frac{\mathcal{Q}}{\eta} +
		\sum_{i: \lambda_i^\mu\ne 0 }\frac{\eta}{(\lambda_i^\mu)^2+\eta^2}
		\left(\begin{matrix}
				{\bm u}_i^\mu \\
				{\bm v}_i^\mu
			\end{matrix}\right)
		\left(\begin{matrix}
				{\bm u}_i^\mu \\
				{\bm v}_i^\mu
			\end{matrix}\right)^* \ge  \frac{\mathcal{P}}{\eta}\,.
	\end{equation*}
	By \eqref{eq:bsup} we thus obtain
	\begin{equation*}
		1\succ \sup_{z\in \mathrm{bulk}} \langle\Im G^z(\ii\eta)F\Im G^z(\ii\eta)F^*\rangle\succ \frac{1}{\eta^2}\langle \mathcal{P}F\mathcal{P}F^*\rangle =\frac{\big|\langle \overline{\boldsymbol{l}}, \boldsymbol{r}\rangle\big|^2}{N\eta^2\rVert \boldsymbol{r}\rVert^2\Vert \boldsymbol{l}\rVert^2}\,,
	\end{equation*}
	which, by \eqref{eq:orthonorm}, implies
	\begin{equation*}
		\mathcal{O}_{ii}=\rVert \boldsymbol{r}\rVert^2\Vert \boldsymbol{l}\rVert^2\succ \frac{1}{N\eta^2}\,.
	\end{equation*}
	Choosing $\eta=N^{-1+\epsilon/2}$, this concludes the proof.


\end{proof}

\section{Local laws with regular observables} \label{sec:prooflocallaw}
The goal of the present section is to establish the key Proposition \ref{prop:2Gav} by proving an averaged local law for a product of two resolvents (of the Hermitisation \eqref{eq:herm}) in the bulk of the scDos $\rho$ with \emph{regular} (recall Definition \ref{def:reg obs1} and see Definition~\ref{def:regobs} below) deterministic matrices $A_1, A_2$ in between. Throughout the rest of this paper, 
we consider the case of several spectral parameters $w_1, w_2, ...$ and fixed bounded deformations $\Defo_1 = \Defo_2= ... \equiv \Defo \in \C^{N \times N}$, 
which we continue to omit from the notation.

Using the abbreviations $G_i := G(w_i):= G^\Defo(w_i)$ (and analogously for $M_i$), the deterministic approximation to the resolvent chain
\begin{equation*}
	G_1 B_1  \,  \cdots \, B_{k-1} G_{k}
\end{equation*}
for arbitrary deterministic $B_1, ... , B_{k}$\footnote{We will use the the notational convention, that the letter $B$ denotes arbitrary (generic) matrices, while $A$ is reserved for \emph{regular} matrices, in the sense of Definition \ref{def:regobs}. } is denoted by
\begin{equation} \label{eq:Mdef}
	M(w_1, B_1,  ... ,  B_{k-1}, w_{k})
\end{equation}
and defined recursively in the length $k$ of the chain.
\begin{definition} \label{def:Mdef}
	Fix $k \in \N$ and let $w_1, ... , w_k \in \C \setminus \R$ be spectral parameters. As usual, the corresponding solutions to the MDE~\eqref{eq:MDE} are denoted by $M(w_j)$, $j \in [k]$. Then, for deterministic matrices $B_1, ... , B_{k-1}$ we recursively define
	\begin{align}
		M(w_1,B_1, ... B_{k-1} , & w_{k}) = \big(\mathcal{B}_{1k}\big)^{-1}\bigg[M(w_1) B_1 M(w_{2}, ...  , w_{k}) \label{eq:M_definitionapp}                                           \\
		                         & + \sum_{\sigma = \pm} \sum_{l = 2}^{k-1} \sigma M(w_1) \langle M(w_1,  ... , w_l) E_\sigma \rangle  E_\sigma M(w_l, ... , w_{k}) \bigg]\,, \nonumber
	\end{align}
	where we introduced the shorthand notation
	\begin{equation} \label{eq:stabopdef}
		\mathcal{B}_{mn} \equiv \mathcal{B}(w_m,w_n)= 1 - M(w_m) \mathcal{S}[\cdot] M(w_n)
	\end{equation}
	for the so-called stability operator, discussed later in Appendix \ref{app:stabop}.
\end{definition}
Note that the recursion \eqref{eq:M_definitionapp} is well defined, since on the rhs.~of \eqref{eq:M_definitionapp}, there are only $M(w_m, ... , w_n)$ appearing for which the number of spectral parameters is strictly smaller than on the lhs.~of \eqref{eq:M_definitionapp}, i.e.~$n-m +1< k$.
We may call these formulas \eqref{eq:M_definitionapp} {\it recursive Dyson equations} as they provide us with the correct deterministic quantity
for longer resolvent chains.
As an example, we have that
\begin{equation} \label{eq:Mexample}
	M(w_1, B_1, w_2) = \mathcal{B}_{12}^{-1}[M_1 B_1 M_2] = M_1 \mathcal{X}_{12}[B_1] M_2\,,
\end{equation}
where $\mathcal{B}_{12}^{-1}$ is the inverse stability operator \eqref{eq:stabopdef} and $\mathcal{X}_{12} = \big(1 - \mathcal{S}[M_1 \cdot M_2]\big)^{-1}$.
We remark that $M$ satisfies several different recursions besides~\eqref{eq:M_definitionapp};
they are presented in Lemma~\ref{lem:recurel} (see also~\cite[Lemma 5.4]{thermalisation} for a simpler setup of Wigner matrices).
The equivalence of these recursions will be proved via the so-called \emph{meta-argument}, see e.g.~\cite{metaargument}.

As already mentioned above, we are aiming at local laws for expressions of the form
\begin{equation} \label{eq:av}
	\langle G_1 A_1 \, \cdots \, G_k A_k \rangle
\end{equation}
in the averaged case, or
\begin{equation} \label{eq:iso}
	\big(G_1 A_1 \, \cdots \, A_k G_{k+1} \big)_{\boldsymbol{x}\boldsymbol{y}}
\end{equation}
in the isotropic case, where the deterministic matrices $A_1, ... , A_k$ are assumed to be \emph{regular}.

The general concept of {\it regularity} depending on
two spectral parameters $w$ and $w'$  has already been introduced
in Definition \ref{def:reg obs1}.  In the following definition we tailor this concept
to observables in chains \eqref{eq:av} and \eqref{eq:iso}. It basically says that
observable $A_j$, located between $G_j=G(w_j)$ and $G_{j+1}=G(w_{j+1})$
in these chains will naturally be regularised using the spectral parameters
$w_j$ and $w_{j+1}$.
\begin{definition} {\rm (Regular observables in chains)} \label{def:regobs}\\
	Fix $\kappa > 0$ and let $\delta = \delta(\kappa, \Vert \Defo\Vert) > 0$ be small enough (see \eqref{eq:delta} and \eqref{eq:deltachoice}). Consider one of the two expressions \eqref{eq:av} or \eqref{eq:iso} for some fixed length $k \in \N$ and bounded matrices $\Vert A_i \Vert \lesssim 1$ and let  $w_1, ... , w_{k+1} \in \C \setminus \R$ be spectral parameters with $\Re w_i \in \mathbf{B}_\kappa$.
	For any $j \in [k]$, analogously to \eqref{eq:case regulation}, we denote
	\begin{equation} \label{eq:case regulation2}
		\mathbf{1}_\delta^\pm(w_j, w_{j+1}) := \phi_\delta(\Re w_j \mp\Re w_{j+1} ) \ \phi_\delta(\Im w_j) \ \phi_\delta(\Im w_{j+1})
	\end{equation}
	and $\mathfrak{s}_j := - \sgn(\Im w_j \Im w_{j+1})$, where, here and in the following, in case of \eqref{eq:av}, the indices are understood cyclically modulo $k$.
	\begin{itemize}
		\item[(a)] For $i \in [k]$ we define the \emph{regular component} or \emph{regularisation} of $A_i$ from \eqref{eq:av} or \eqref{eq:iso} \emph{(w.r.t.~the pair of spectral parameters $(w_i, w_{i+1})$)} as
			\begin{equation} \label{eq:circ def}
				\mathring{A}_i := \mathring{A}_i^{{w_i,w_{i+1}}}\,.
			\end{equation}
		\item[(b)] Moreover, we call $A_i$ \emph{regular (w.r.t.~$(w_i, w_{i+1})$)} if and only if $\mathring{A}_i = A_i$.
	\end{itemize}
\end{definition}
For example, in case of \eqref{eq:av} for $k=1$ with spectral parameter $w_1 \in \C\setminus \R$ satisfying $\Re w_1 \in \mathbf{B}_\kappa$, $|\Re w_1| \le \delta/4$ and $|\Im w_1| \le \delta/2$ (recall \eqref{eq:delta} and \eqref{eq:case regulation2}), the regular component of $A_1$ is given by
\begin{equation} \label{eq:1G traceless}
	\mathring{A}_1 := A_1 - \frac{\langle \Im M_1 A_1 \rangle }{\langle \Im M_1 \rangle } E_+ - \frac{\langle M_1 A_1 M_1 E_- \rangle }{\langle M_1 E_- M_1 E_- \rangle}E_- \,.
\end{equation}
Here we used the short--hand notation $M_1:=M(w_1)$.

We emphasise, that our notation $\,\mathring{\cdot}\,$ for the regular component of $A_i$ does \emph{not} have an overall fixed meaning but depends on the spectral parameters of the resolvents `surrounding' the deterministic matrix $A_i$ under consideration, i.e.
\begin{equation*}
	\langle \ \cdots \ G_i A_i G_{i+1}\  \cdots \  \rangle \quad \text{or} \quad \big(\ \cdots \  G_i A_i G_{i+1} \ \cdots \  \big)_{\boldsymbol{x} \boldsymbol{y}}\,,
\end{equation*}
or in case of \eqref{eq:av} for $k=1$ the single spectral parameter involved.
However, if we aim at specifying the spectral parameters defining the operation $\,\mathring{\cdot}\,$, we add them (or their indices) as a subscript, i.e.~write
\begin{equation*} \label{eq:regularwiwi+1}
	\mathring{A}_i^{{w_i, w_{i+1}}} \equiv \mathring{A}_i^{{i,i+1}}  \equiv \mathring{A}_i \equiv A_i^\circ \equiv A_i^{\circ_{i,i+1}} \equiv A_i^{\circ_{w_i, w_{i+1}}}\,,
\end{equation*}
as done in Definition \ref{def:reg obs1}, and do not use imprecise notation $\mathring{A}_i$.

The just explained caveats are in stark contrast to the case of Wigner matrices \cite{ETHpaper,multiG, A2}, where the regular component of a matrix $A$ is simply its traceless part, i.e.~$\mathring{A} = A - \langle A \rangle$, irrespective of the spectral parameters involved. Apart from this independence of the location in the spectrum, there is a one further important difference to our case, which we already mentioned in Section \ref{sec:proofmain}: For Wigner matrices, the condition for $A$ being regular is one-dimensional and hence restricts $A$ to a $(N^2 -1)$-dimensional subspace of $\C^{N \times N}$ (the traceless matrices), whereas in our case, the regularity condition is two-dimensional (if $\mathbf{1}_\delta^\sigma(\cdot, \cdot) = 1$) and hence restricts a regular matrix $A$ to a $((2N)^2 - 2)$-dimensional subspace of $\C^{2N \times 2N}$, which depends on the `surrounding' spectral parameters.

We now give bounds on the size of the deterministic term $M(w_1, B_1, ... , B_{k-1}, w_{k})$ from \eqref{eq:Mdef}, where all $B_i$ are regular in the sense of Definition \ref{def:regobs}. The proof of this lemma is presented in Appendix~\ref{app:Mbound}.
\begin{lemma} {\rm (Bounds on $M$, see \cite[Lemma~2.4]{multiG})} \label{lem:Mbound} \\
	Fix $\kappa > 0$.  	Let $k \in [4]$ and  $w_1, ... , w_{k+1} \in \C \setminus \R$ be spectral parameters with $\Re w_j \in \mathbf{B}_\kappa$. Then, for bounded \emph{regular} deterministic matrices $A_1, ... , A_{k}$  (in the sense of Definition \ref{def:regobs}), we have the bounds
	\begin{align}
		\Vert M(w_1, A_1, ... , A_{k}, w_{k+1}) \Vert                      & \lesssim\begin{cases}
			                                                                             \frac{1}{\eta^{\lfloor k/2 \rfloor}} \hspace{11.5mm} & \text{if} \ \eta \le 1 \\
			                                                                             \frac{1}{\eta^{k+1}} \qquad                          & \text{if} \ \eta > 1
		                                                                             \end{cases} \,, \label{eq:Mboundnorm}  \\[2mm]
		\vert \langle  M(w_1, A_1, ... , A_{k-1}, w_{k})A_k  \rangle \vert & \lesssim \begin{cases}
			                                                                              \frac{1}{\eta^{\lfloor k/2\rfloor-1 }}\vee 1 \quad & \text{if} \ \eta \le 1 \\
			                                                                              \frac{1}{\eta^{k}} \quad                           & \text{if} \ \eta > 1
		                                                                              \end{cases}\,, \label{eq:Mboundtrace}
	\end{align}
	for the deterministic approximation \eqref{eq:Mdef} of a resolvent chain, where $\eta := \min_j |\Im w_j|$.
\end{lemma}
For the presentation of our main results, we would only need \eqref{eq:Mboundnorm} and \eqref{eq:Mboundtrace} for $k \in [2]$ from the previous lemma. However, the remaining bounds covered by Lemma \ref{lem:Mbound} will be instrumental in our proofs of Theorems \ref{thm:singleGopt} and \ref{thm:multiGll} below (see Sections \ref{sec:proofmaster} and \ref{sec:proofreduc}).

The main result of the present section and most important input for our proofs in Section~\ref{sec:proofmain} is the following averaged local law in the bulk of the spectrum for two resolvents and regular matrices.
\begin{theorem} \label{thm:multiGll} {\rm (Local laws with \emph{two} regular matrices)} \\
	Fix a bounded deterministic $\Defo \in \C^{N \times N}$,  $\epsilon >0 $ and $\kappa > 0$. 	Then, for spectral parameters $w_1, w_2, w_3 \in \C$ satisfying $\max_j |w_j| \le N^{100}$, $\Re w_j \in \mathbf{B}_\kappa$ and $\eta:= \min_j |\Im w_j| \ge N^{-1+\epsilon}$, deterministic vectors $\bm{x}, \bm{y}$ with $\Vert \bm{x} \Vert, \Vert \bm{y}\Vert \lesssim 1 $, and any \emph{regular} deterministic matrices $A_1, A_2 \in \C^{2N \times 2N}$ (cf.~Definition~\ref{def:regobs}), we have the \emph{averaged local law}
	\begin{subequations} \label{eq:2aviso}
		\begin{equation} 
			\left| \langle G_1 A_1 G_2 A_2 - M(w_1, A_1, w_2) A_2 \rangle  \right| \prec \begin{cases}
				\frac{1}{\sqrt{N \eta}} \quad & \text{if} \ \eta \le 1 \\
				\frac{1}{N \eta^3} \quad      & \text{if} \ \eta > 1
			\end{cases}
		\end{equation}
		and the \emph{isotropic law}
		\begin{equation}
			\left| \big\langle \bm{x}, \big(G_1 A_1 G_2 A_2 G_3 - M(w_1, A_1, w_2, A_2, w_3) \big) \bm{y} \big\rangle \right| \prec \begin{cases}
				\frac{1}{\eta} \quad            & \text{if} \ \eta \le 1 \\
				\frac{1}{\sqrt{N} \eta^4} \quad & \text{if} \ \eta > 1
			\end{cases}\,.
		\end{equation}
	\end{subequations}

\end{theorem}
Together with \eqref{eq:Mboundtrace} for $k=2$, this proves Proposition \ref{prop:2Gav}. Moreover, as a byproduct of our proof of Theorem~\ref{thm:multiGll}, we obtain the following optimal local laws with a single regular matrix.
\begin{theorem} \label{thm:singleGopt} {\rm (Optimal local laws with \emph{one} regular matrix)} \\
	Fix a bounded deterministic $\Defo \in \C^{N \times N}$, $\epsilon >0 $ and $\kappa > 0$.   Then, for spectral parameters $w_1, w_2 \in \C$ satisfying $\max_j |w_j| \le N^{100}$, $\Re w_j \in \mathbf{B}_\kappa$ and $\eta:= \min_j |\Im w_j| \ge N^{-1+\epsilon}$, deterministic vectors $\bm{x}, \bm{y}$ with $\Vert \bm{x} \Vert, \Vert \bm{y}\Vert \lesssim 1 $, and any \emph{regular} deterministic matrix $A_1$ (cf.~Definition \ref{def:regobs}), we have the \emph{optimal averaged local law}
	\begin{subequations} \label{eq:1aviso}
		\begin{equation}
			\left| \langle (G_1 - M_1) A_1 \rangle  \right| \prec \begin{cases}
				\frac{1}{N \eta^{1/2}} \quad & \text{if} \ \eta \le 1 \\
				\frac{1}{N \eta^2} \quad     & \text{if} \ \eta > 1
			\end{cases}
		\end{equation}
		and the \emph{optimal isotropic local law}
		\begin{equation}
			\left| \big\langle \bm{x}, \big(G_1 A_1 G_2 - M(w_1, A_1, w_2) \big) \bm{y} \big\rangle \right| \prec \begin{cases}
				\frac{1}{\sqrt{N \eta^2}} \quad & \text{if} \ \eta \le 1 \\
				\frac{1}{\sqrt{N} \eta^3} \quad & \text{if} \ \eta > 1
			\end{cases}\,.
		\end{equation}
	\end{subequations}
\end{theorem}

\begin{remark} \label{rmk:sqrteta} We have several comments.
	\begin{itemize}
		\item[(i)]  The above local laws are in agreement with the \emph{$\sqrt{\eta}$-rule} first established for Wigner matrices {for traceless matrices} in \cite[Theorem 2.5]{multiG}: Every regular deterministic matrix $A_i$ reduces both the size of the deterministic approximation and the error term by a factor $\sqrt{\eta}$.
		\item[(ii)] The error terms in Theorem \ref{thm:multiGll} dealing with two regular matrices can still be improved by a factor $1/\sqrt{N \eta}$, as shown in \cite{multiG}. A similar analysis could have been conducted here, but we refrain from doing so, as it is not needed for our main results from Section \ref{sec:results}. However, the error bounds in \eqref{eq:1aviso} with one regular matrix are in fact optimal.
		\item[(iii)] Given Theorem \ref{thm:singleG}, and Theorems~\ref{thm:multiGll}--\ref{thm:singleGopt}, it is possible to deduce similar bounds for averaged and isotropic chains as in \eqref{eq:2aviso}, where not both matrices $A_1, A_2$ are regular  (see \eqref{eq:correlator}).
	\end{itemize}
\end{remark}
In the rest of this paper, we give a detailed proof of Theorem \ref{thm:multiGll} in the much more involved $\eta \le 1$ regime. For $\eta > 1$, the bound simply follows by induction on the number of resolvents in chain by invoking the trivial $\Vert M(w) \Vert \lesssim 1/|\Im w|$. The detailed argument has been carried out in \cite[Appendix~B]{multiG} for the case of Wigner matrices. However, at a certain technical point (within the proof of the \emph{master inequalities} in Proposition \ref{prop:master} and the \emph{reduction inequalities} in Lemma~\ref{lem:reduction}), the proof uses Theorems \ref{thm:multiGll} and \ref{thm:singleGopt} (and even its analogues for longer chains) for the $\eta > 1$ regime. But the master and reduction inequalities are not needed for proving the above estimates in the $\eta > 1$ regime, hence the argument is not circular. With partial exception in Appendix \ref{app:Mbound}, where we prove Lemma \ref{lem:Mbound}, throughout the rest of this paper we assume that $\min_j |\Im w_j| =: \eta \le 1$.
\subsection{Basic control quantities and proof of Theorems \ref{thm:multiGll} and \ref{thm:singleGopt}}\label{sec:basic}
Our strategy for proving Theorem \ref{thm:multiGll} (and thereby Theorem \ref{thm:singleGopt} as a byproduct)
is to derive a system of {\it master inequalities}  (Proposition~\ref{prop:master})
for the errors in the local laws by cumulant expansion, then use an iterative scheme
to gradually improve their estimates. The cumulant expansion introduces longer resolvent chains, potentially leading
to an uncontrollable hierarchy, so our
master inequalities are complemented by a set of {\it reduction inequalities} (Lemma~\ref{lem:reduction})
to estimate longer chain in terms of shorter ones.
We have used a similar strategy  in~\cite{multiG, A2} for Wigner matrices, but now many new error
terms due to regularisations need to be handled.

Before entering the detailed proof, we explain the main mechanism of the new type of error terms. Cumulant expansions
applied  to chains $\ldots G_i A_i G_{i+1} \ldots$ with regular $A_i$'s introduce more resolvent factors, for example
$\ldots G_i  G_i A_i G_{i+1} \ldots$ or $\ldots G_i E_- G_i A_i G_{i+1} \ldots$
without introducing more $A$'s. Multiple $G$ factors without intermediate $A$'s appear
which we wish to reduce to fewer $G$ factors using {resolvent identities \eqref{eq:resolid}} or contour integral representations; in the
example above we will use
\begin{equation}\label{G2}
	G_iG_i = G(w_i)^2= \frac{1}{2\pi\ii}\int_\Gamma \frac{G(z)}{(z-w_i)^2} \D z,
\end{equation}
where $\Gamma$ is an appropriate contour (see Lemma~\ref{lem:intrepG^2}). When this formula is inserted into the chain,
we have $\ldots  G(z) A_i G_{i+1}\ldots$, i.e. $A_i$ is not regular any more with respect to the
neighboring spectral parameters $(z, w_{i+1})$ since $w_i$ has been changed to $z$. We need to regularise $A_i$
to the new situation. Fortunately, the
regularisation is Lipschitz continuous by Lemma~\ref{lem:regularbasic}, so roughly speaking we make an
error of order $|z-w_i|$ when we regularise $A_i$ from $(w_i, w_{i+1})$ to $(z, w_{i+1})$. This error exactly compensates
the higher power of $z-w_i$ in the denominator in~\eqref{G2}, making eventually the adjustment of regularisations
harmless in the estimates. We need to meticulously implement this strategy for longer chains and also taking into
account the chiral symmetry to reduce  $G_i E_- G_i $ in chains like $\ldots G_i E_- G_i A_i G_{i+1} \ldots$.
The precise form of the error terms in Lemma~\ref{lem:regularbasic} is essential.
It is remarkable that the signs appearing in~\eqref{eq:regularbar1}, \eqref{eq:regularperturb1}, and \eqref{eq:regularperturb2} exactly match those
that arise in the denominators of the contour integral formulas like~\eqref{G2}. We now start the actual proof.


As the basic control quantities in the sequel of the proof, we  introduce the normalised differences
\begin{align} \label{eq:Psi avk}
	\Psi_k^{\rm av}(\boldsymbol{w}_k, \boldsymbol{A}_k)                                        & := N \eta^{k/2} |\langle  G_1 A_1 \cdots G_{k} A_k  -  M(w_1, A_1, ... , w_{k}) A_k  \rangle |\,,                                                            \\
	\label{eq:Psi isok}
	\Psi_k^{\rm iso}(\boldsymbol{w}_{k+1}, \boldsymbol{A}_{k}, \boldsymbol{x}, \boldsymbol{y}) & := \sqrt{N \eta^{k+1}} \left\vert  \big(  G_1 A_1 \cdots A_{k} G_{k+1} - M(w_1, A_1, ... , A_{k}, w_{k+1}) \big)_{\boldsymbol{x} \boldsymbol{y}} \right\vert
\end{align}
for $k \in \N$, where we used the short hand notations
\begin{equation*} 
	G_i:= G(w_i)\,, \quad \eta := \min_i |\Im w_i|\,, \quad \boldsymbol{w}_k :=(w_1, ... , w_{k})\,, \quad \boldsymbol{A}_k:=(A_1, ... , A_{k})\,.
\end{equation*}
The deterministic matrices $\Vert A_i \Vert \le 1$, $i \in [k]$, are assumed to be \emph{regular} (i.e., $A_i = \mathring{A}_i$, see Definition~\ref{def:regobs}) and the deterministic counterparts
\begin{equation*} 
	M(w_1, A_1, ... , A_{k-1}, w_{k})
\end{equation*}
used in \eqref{eq:Psi avk} and \eqref{eq:Psi isok} (see also \eqref{eq:Mdef}) are given recursively in Definition \ref{def:Mdef}.

For convenience, we extend the above definitions to $k=0$ by
\begin{equation*} 
	\Psi_0^{\rm av}(w):= N \eta |\langle G(w) - M(w)\rangle |\,, \quad \Psi_0^{\rm iso}(w, \boldsymbol{x}, \boldsymbol{y}) := \sqrt{N \eta} \big| \big(G(w) - M(w)\big)_{\boldsymbol{x} \boldsymbol{y}} \big|
\end{equation*}
and observe that
\begin{equation} \label{eq:single G}
	\Psi_0^{\rm av} + \Psi_0^{\rm iso} \prec 1
\end{equation}
is the usual single-resolvent local law (in the bulk) from Theorem \ref{thm:singleG}, where here and in the following the arguments of $\Psi_k^{\rm av/iso}$ shall occasionally be omitted.
We remark that the index $k$ counts the number of regular matrices in the sense of Definition~\ref{def:regobs}.


Throughout the entire argument, let $\epsilon > 0$ and $\kappa > 0$ be \emph{arbitrary} but fixed, and let
\begin{equation} \label{eq:Omega}
	\mathbf{D}^{(\epsilon, \kappa)}:= \big\{  w \in \C :  \Re w \in \mathbf{B}_{\kappa}\,, \ N^{100} \ge |\Im w| \ge N^{-1+\epsilon}\big\}
\end{equation}
be the \emph{target spectral domain}, where the $\kappa$-bulk $\mathbf{B}_\kappa$ has been introduced in \eqref{eq:bulk}. This target spectral domain $\mathbf{D}^{(\epsilon, \kappa)}$ will be reached by shrinking a larger \emph{initial spectral domain}
\begin{equation} \label{eq:Omega00}
	\mathbf{D}^{(\epsilon_0, \kappa_0)} := \big\{  w \in \C :  \Re w \in \mathbf{B}_{\kappa_0}\,, \ N^{100} \ge |\Im w| \ge  N^{-1+\epsilon_0}\big\}
\end{equation}
many times, say $(L-1)$ times, during our whole argument, where $L=L(\epsilon)$ is an $N$-independent positive integer to be determined below (see Remark~\ref{rmk:L=L(epsilon)}).
In \eqref{eq:Omega00}, we set $\epsilon_0 := \epsilon/2$ and chose the initial bulk parameter
\begin{equation} \label{eq:kappa0}
	\kappa_0 = \kappa_0(\epsilon, \kappa) = \frac{\kappa}{L(\epsilon)}> 0
\end{equation}
The just mentioned shrinking of domains will be conducted alongside the decreasing family $(\mathbf{D}^{(\epsilon_0, \kappa_0)}_\ell)_{\ell \in [L]}$ of spectral domains, defined via
\begin{align} \label{eq:Omegal}
	\mathbf{D}^{(\epsilon_0, \kappa_0)}_\ell := \big\{  w \in \C :  \Re w \in \mathbf{B}_{\ell \kappa_0}\,, \ N^{100} \ge |\Im w| \ge \ell N^{-1+\epsilon_0}\big\} \subset \mathbf{D}^{\cred{(\epsilon_0, \kappa_0)}}\,.
\end{align}

\begin{figure}[htbp]
	\centering
	\begin{tikzpicture}
		\begin{axis}[no markers,
				xlabel=$\Re w$,
				ylabel=$\Im w$,
				width=\textwidth,
				height=0.4\textwidth,
				axis lines=middle,
				ymin=0,ymax=1.2,
				xmin=-2.7,xmax=3.3,
				ytick={0},yticklabels={0},
				xtick={-2,0,2},
				xticklabels={-2,0,2},
				axis on top=true, color=black]
			\pgfmathsetmacro{\Z}{0.93}
			\pgfmathsetmacro{\min}{-m(\Z)+.001}
			\pgfmathsetmacro{\max}{m(\Z)}
			\path[name path=rho,save path=\saverho] (\min,0) -- plot[domain={\min}:\max,samples=301] ({\x},{f(abs(\x),\Z)}) -- (\max,0);
			\domain{0.01}{0.02}{4}{1}
			\domain{0.04}{0.05}{6}{3}
			\domains{0.07}{0.08}{8}{5}
			\domains{0.1}{0.11}{10}{7}
			\domains{0.13}{0.14}{12}{8}
			\domains{0.18}{0.2}{18}{6}
			\draw[dashed] (\min,0) -- (\min,1.2);
			\node (int5m) at ($(i5-3)!0.5!(i5-4)$) {};
			\node (int6m) at ($(i6-1)!0.5!(i6-2)$) {};
			\node (dom6) at (int5m |- 0,1.1) {$\mathbf{D}^{(\epsilon,\kappa)}$};
			\draw[stealth-stealth] (int6m |- 0,0) -- (int6m |- 0,0.2) node[midway,left] {\footnotesize $N^{-1+\epsilon}$};
			\draw[stealth-stealth] (int5m |- 0,0) -- (int5m |- 0,0.14) node[midway,right] {\footnotesize $\sim N^{-1+\epsilon_0}$};
			\node (D1l) at (3.0,0.6) {$\mathbf{D}^{(\epsilon_0,\kappa_0)}$};
			\draw[-stealth] (D1l) -- (i1-2 |- D1l);
			\node (D2l) at (3.0,0.8) {$\mathbf{D}_2^{(\epsilon_0,\kappa_0)}$};
			\draw[-stealth] (D2l) -- (i3-2 |- D2l);
			\node (D5l) at (3.0,1) {$\mathbf{D}_3^{(\epsilon_0,\kappa_0)}$};
			\draw[-stealth] (D5l) -- (i5-4 |- D5l);
			\draw[stealth-stealth] (\min,1.1) -- (i6-1 |- 0,1.1) node[midway,below] {\footnotesize$\sim\kappa^{2/3}$};
			\draw[stealth-stealth] (0,.8) -- (i6-2 |- 0,.8) node[left] {\footnotesize$\sim\kappa$};
			\draw[thick,solid,black,use path=\saverho];
		\end{axis}
	\end{tikzpicture}
	\caption{Depicted are the target spectral domain \eqref{eq:Omega}, the initial spectral domain \eqref{eq:Omega00} and four intermediate domains from the family \eqref{eq:Omegal}. The solid black curve represents the symmetric scDos $\rho$ for the perturbation $\Defo = -z$ with $|z|$ slightly less than one (see Example \ref{exam:Defo=-z}). Close to a regular edge of the scDos, the horizontal distance between two domains scales like $\kappa^{2/3}$. Near an (approximate) cusp, the scaling agrees with the linear lower bound given in \eqref{kappabulkreg}. }
	\label{fig:domains}
\end{figure}
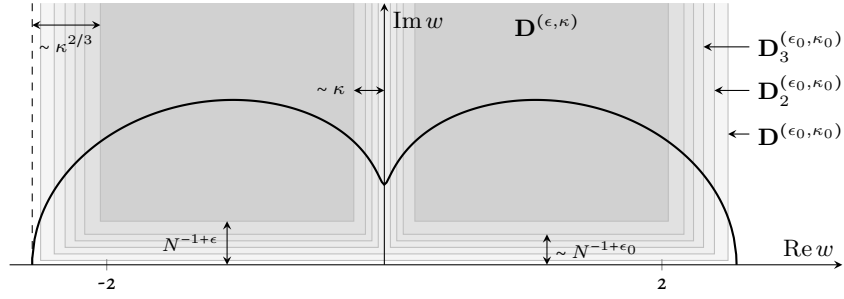

Finally, the cut-off parameter $\delta > 0$ used in the definition of the regular component of an observable (see \eqref{eq:delta} and \eqref{eq:circ def} in Definition \ref{def:regobs}) shall be chosen by the following two requirements: First, it has to be much smaller than the initial bulk-parameter $\kappa_0$ from \eqref{eq:kappa0}, i.e.
\begin{equation} \label{eq:deltachoice}
	0 < \delta \ll \mathfrak{c}\kappa_0\,,
\end{equation}
where $\mathfrak{c} > 0$ is the same constant as introduced in \eqref{kappabulkreg}. Second, it has to be small enough such that the denominators in \eqref{eq:circ def} (see also Appendix \ref{app:stabop}) as well as in Lemmas~\ref{lem:splittingstable 1av},~\ref{lem:splitting stable 1iso},~and~\ref{lem:splitting stable 2av} are uniformly bounded away from zero -- in case that $\mathbf{1}_\delta^\sigma(w_i,w_{i+1}) = 1$. Note that these requirements also depend on the deformation $\Defo\in \C^{N \times N}$ (but only via the norm $\Vert \Defo \Vert \lesssim 1$) as it determines the scDos~$\rho$.
\begin{definition} \label{def:epsi unif} {\rm (Uniform bounds in domains)} \\ Let $\epsilon > 0$ and $\kappa > 0$ as above and let $k \in \N$.
	We say that the bounds
	\begin{equation} \label{eq:epsi unif}
		\begin{split}
			\big\vert \langle G(w_1) B_1 \ \cdots \ G(w_k) B_k - M(w_1, B_1, ... , w_k) B_k \rangle \big\vert &\prec \mathcal{E}^{\rm av}\,, \\[2mm]
			\left\vert  \big( G(w_1) B_1 \ \cdots \ B_k G(w_{k+1}) - M(w_1, B_1, ... , B_k, w_{k+1})  \big)_{\boldsymbol{x} \boldsymbol{y}}  \right\vert &\prec \mathcal{E}^{\rm iso}
		\end{split}
	\end{equation}
	hold \emph{$(\epsilon, \kappa)$-uniformly} for some deterministic control parameters $\mathcal{E}^{\rm av/iso} = \mathcal{E}^{\rm av/iso}(N, \eta)$, depending only on $N$ and $ \eta:= \min_i |\Im w_i|$, if the implicit constant in \eqref{eq:epsi unif} are uniform in bounded deterministic matrices $\Vert B_j \Vert \le 1$, deterministic vectors $\Vert \boldsymbol{x} \Vert , \Vert \boldsymbol{y} \Vert \le 1$, and \emph{admissible} spectral parameters $w_j\in \mathbf{D}^{(\epsilon, \kappa)}$ satisfying $ 1 \ge \eta:= \min_j |\Im w_j|$.

	Similarly, we use the phrase that a bound holds \emph{$(\epsilon_0, \kappa_0,\ell)$-uniformly} (or simply \emph{$\ell$-uniformly}), if the above statement is true with $\mathbf{D}^{(\epsilon_0, \kappa_0)}_\ell$ in place of $\mathbf{D}^{(\epsilon, \kappa)}$.

	Moreover, we may allow for additional restrictions on the deterministic matrices. For example, we may talk about uniformity under the additional assumption that some (or all) of the matrices are \emph{regular} (in the sense of Definition \ref{def:regobs}).
\end{definition}
Note that \eqref{eq:epsi unif} is stated for a fixed choice of spectral parameters $w_j$ in the left hand side, but it is in fact equivalent to an apparently stronger statement, when the same bound holds with a supremum over the spectral parameters (with the same constraints). More precisely, if $\mathcal{E}^{\rm iso} \ge N^{-C}$ for some constant $C > 0$, then \eqref{eq:epsi unif} implies
\begin{equation} \label{eq:epsi unif2}
	\sup_{w_1, ... , w_{k+1}} 	\left\vert  \big( G(w_1) B_1 \ \cdots \ B_k G(w_{k+1}) - M(w_1, B_1, ... , B_k, w_{k+1})  \big)_{\boldsymbol{x} \boldsymbol{y}}  \right\vert \prec \mathcal{E}^{\rm iso}
\end{equation}
(and analogously for the averaged bound), where the supremum is taken over all choices of $w_j$'s in the admissible spectral domain $\mathbf{D}^{(\epsilon, \kappa)}$ or $\mathbf{D}^{(\epsilon_0, \kappa_0)}_\ell$. This bound follows from \eqref{eq:epsi unif} by a standard \emph{grid argument} (see, e.g., the discussion after \cite[Def.~3.1]{multiG}).   
Throughout the entire paper, we will frequently use the equivalence between \eqref{eq:epsi unif} and \eqref{eq:epsi unif2}, e.g.~when integrating such bounds over some spectral parameters as done in Sections \ref{sec:proofmaster} and \ref{sec:proofreduc}.

We can now formulate our main results of the present section, Theorem \ref{thm:multiGll} and Theorem~\ref{thm:singleGopt}, in the language of our basic control quantities $\Psi_k^{\rm av/iso}$.
\begin{lemma} {\rm (Estimates on $\Psi^{\rm av/iso}_1$ and $\Psi^{\rm av/iso}_2$)} \label{lem:multiGll}
	For any $\epsilon > 0$ and $\kappa > 0$ we have
	\begin{equation*}
		\Psi_1^{\rm av} + \Psi_1^{\rm iso} \prec 1 \qquad \text{and} \qquad 	\Psi_2^{\rm av}  + 	\Psi_2^{\rm iso}\prec \sqrt{N \eta}
	\end{equation*}
	$(\epsilon, \kappa)$-uniformly in regular matrices (i.e.~for spectral parameters $w_j\in \mathbf{D}^{(\epsilon, \kappa)}$ with $ 1 \ge \eta:= \min_j |\Im w_j|$).
\end{lemma}
\begin{proof}[Proof of Theorems \ref{thm:multiGll} and \ref{thm:singleGopt}]
	These immediately follow from Lemma \ref{lem:multiGll}.
\end{proof}
The rest of the proof is structured as follows: First, in Section \ref{subsec:masterandred}, we state the \emph{master inequalities} and corresponding \emph{reduction inequalities} on the $\Psi^{\rm av/iso}_k$ parameters, which we then use in Section \ref{subsec:proof Psi est} to prove Lemma \ref{lem:multiGll}. Afterwards, in Section \ref{sec:proofmaster}, we prove the master inequalities and, finally, the proof of the reduction inequalities is presented in Section \ref{sec:proofreduc}.
\subsection{Master inequalities and reduction lemma} \label{subsec:masterandred}
We now state the relevant part of a non-linear infinite hierarchy of coupled master inequalities for $\Psi^{\rm av}_k$ and $\Psi^{\rm iso}_k$. In fact, for our purposes, it is sufficient to have only the inequalities for $k \in [2]$, but with fairly more effort (despite closely following the arguments in Section \ref{sec:proofmaster}) it is possible to obtain analogous estimates for general $k \in \N$.
\begin{proposition}{\rm (Master inequalities)} \label{prop:master} Assume that
	\begin{equation} \label{eq:apriori Psi}
		\Psi_i^{\rm av/iso} \prec \psi_i^{\rm av/iso}\,, \quad i \in [4]\,,
	\end{equation}
	$\ell$-uniformly (i.e.~for spectral parameters $w_j \in \mathbf{D}^{(\epsilon_0, \kappa_0)}_\ell$ and $1 \ge \min_j |\Im w_j|$) in regular matrices, for some deterministic control parameters $\psi_i^{\rm av/iso}$, which are independent of the spectral parameters $w_j$. Then it holds that
	\begin{subequations}
		\begin{align}
			\label{eq:masterineq Psi 1 av}
			\Psi_1^{\rm av}  & \prec 1 + \frac{\psi_1^{\rm av}}{N \eta} + \frac{\psi_1^{\rm iso} +  (\psi_2^{\rm av})^{1/2}  }{(N \eta)^{1/2}} + \frac{(\psi_2^{\rm iso})^{1/2}}{(N \eta)^{1/4}}	\,,                                                                                                            \\
			\label{eq:masterineq Psi 1 iso}
			\Psi_1^{\rm iso} & \prec 1 + \frac{\psi_1^{\rm iso} + \psi_1^{\rm av} }{(N \eta)^{1/2}} + \frac{(\psi_2^{\rm iso})^{1/2}}{(N \eta)^{1/4}} \,,                                                                                                                                                       \\
			\label{eq:masterineq Psi 2 av}
			\Psi_2^{\rm av}  & \prec 1 + \frac{(\psi_1^{\rm av})^{2} + {(\psi_1^{\rm iso})^2} + \psi_2^{\rm av}}{N \eta}+ \frac{\psi_2^{\rm iso} + (\psi_4^{\rm av})^{1/2}  }{(N \eta)^{1/2}}  + \frac{(\psi_3^{\rm iso})^{1/2} + (\psi_4^{\rm iso})^{1/2}}{(N \eta)^{1/4}}    \,,                              \\
			\label{eq:masterineq Psi 2 iso}
			\Psi_2^{\rm iso} & \prec 1 + \psi_1^{\rm iso} + \frac{ \psi_1^{\rm av} \psi_1^{\rm iso} + (\psi_1^{\rm iso})^2 }{N \eta}+ \frac{\psi_2^{\rm iso} + (\psi_1^{\rm iso} \psi_3^{\rm iso})^{1/2}}{(N \eta)^{1/2}}  +   \frac{(\psi_3^{\rm iso})^{1/2} + (\psi_4^{\rm iso})^{1/2}}{(N \eta)^{1/4}}   \,,
		\end{align}
	\end{subequations}
	now $(\ell+1)$-uniformly (i.e.~for spectral parameters $w_j\in \mathbf{D}^{(\epsilon_0, \kappa_0)}_{\ell+1}$ with $ 1 \ge \eta:= \min_j |\Im w_j|$) in regular matrices.
\end{proposition}
As shown in the above proposition, resolvent chains of length $k$ are estimated by resolvent chains up to length $2k$. In order to avoid the indicated infinite hierarchy of master inequalities with higher and higher $k$ indices, we will need the following \emph{reduction lemma}.

\begin{lemma} \label{lem:reduction} {\rm (Reduction inequalities)}
	Assume that
	$\Psi_n^{\rm av/iso} \prec \psi_n^{\rm av/iso}$ holds for $1 \le n \le 4$, $\ell$-uniformly (i.e.~for spectral parameters $w_j\in \mathbf{D}^{(\epsilon_0, \kappa_0)}_{\ell}$ with $ 1 \ge \eta:= \min_j |\Im w_j|$) in regular matrices (cf.~Definition~\ref{def:epsi unif}). Then we have
	\begin{equation} \label{eq:reduction av}
		\Psi_4^{\rm av} \prec (N\eta)^2 + (\psi_2^{\rm av})^2\,,
	\end{equation}
	on the same domain. Furthermore, we have
	\begin{equation} \label{eq:reduction iso}
		\begin{split}
			\Psi_3^{\rm iso} &\prec N \eta \left( 1+ \frac{\psi_2^{\rm iso}}{\sqrt{N\eta}} \right) \left( 1 + \frac{\psi_2^{\rm av}}{N\eta} \right)^{1/2}\,,\\
			\Psi_4^{\rm iso} &\prec (N \eta)^{3/2}\left( 1+ \frac{\psi_2^{\rm iso}}{\sqrt{N\eta}} \right) \left( 1 + \frac{\psi_2^{\rm av}}{N\eta} \right)\,
		\end{split}
	\end{equation}
	again uniformly in $w_j\in \mathbf{D}^{(\epsilon_0, \kappa_0)}_{\ell}$ and in regular matrices.
\end{lemma}


\subsection{Proof of Lemma \ref{lem:multiGll}} \label{subsec:proof Psi est}

Within the proof, we repeatedly use a simple argument, which we call \emph{iteration}.
\begin{lemma}{\rm (Iteration)} \label{lem:iteration}
	For every $D > 0$, $\nu > 0$, and $\alpha \in (0,1)$, there exists some $K = K(D,\nu, \alpha)$, such that whenever (i) $X \prec N^D$ on $\mathbf{D}^{(\epsilon_0, \kappa_0)}_1$ and (ii) $X \prec x$ on $\mathbf{D}^{(\epsilon_0, \kappa_0)}_\ell$ for some $\ell \in \N$, implies that
	\begin{equation*}
		X \prec A + \frac{x}{B} + x^{1-\alpha}C^\alpha \qquad \text{on} \quad \mathbf{D}^{(\epsilon_0, \kappa_0)}_{\ell+1}
	\end{equation*}
	for some constants $B \ge N^\nu$ and $A,C > 0$, it also holds that
	\begin{equation*} 
		X \prec A+ C \qquad \text{on} \quad \mathbf{D}^{(\epsilon_0, \kappa_0)}_{\ell+K}\,.
	\end{equation*}
\end{lemma}

We can now turn to the proof of Lemma \ref{lem:multiGll}.

\begin{proof}[Proof of Lemma \ref{lem:multiGll}]
	Assume that
	\begin{equation*}
		\Psi_j^{\rm av/iso} \prec \psi_j^{\rm av/iso}\,, \quad j \in [4]\,,
	\end{equation*}
	$\ell$-uniformly, for some fixed $\ell>0$, i.e.~it holds on the domain $\mathbf{D}^{(\epsilon_0, \kappa_0)}_\ell$. Then, by \eqref{eq:masterineq Psi 1 av}--\eqref{eq:masterineq Psi 2 iso}, using that $N\eta\ge 1$ to remove some lower order terms, we immediately obtain
	\begin{equation}
		\label{eq:step1it}
		\begin{split}
			\Psi_1^{\rm av}+\Psi_1^{\rm iso}&\prec 1+\frac{\psi_1^{\rm av}+\psi_1^{\rm iso}}{(N \eta)^{1/2}}+\frac{(\psi_2^{\rm av})^{1/2}+(\psi_2^{\rm iso})^{1/2}}{(N\eta)^{1/4}}\\
			\Psi_2^{\rm av}+\Psi_2^{\rm iso}&\prec 1+\psi_1^{\rm iso} +\frac{(\psi_1^{\rm av})^2+(\psi_1^{\rm iso})^2}{N\eta}+\frac{\psi_2^{\rm av}+\psi_2^{\rm iso}}{(N \eta)^{1/2}} \\
			& \hspace{3cm}+\frac{(\psi_4^{\rm av})^{1/2}}{(N \eta)^{1/2}}+ \frac{(\psi_1^{\rm iso} \psi_3^{\rm iso})^{1/2}}{(N \eta)^{1/2}}  +   \frac{(\psi_3^{\rm iso})^{1/2} + (\psi_4^{\rm iso})^{1/2}}{(N \eta)^{1/4}}
		\end{split}
	\end{equation}
	on the domain $ \mathbf{D}^{(\epsilon_0, \kappa_0)}_{\ell+1}$. We point out that to get the second bound in \eqref{eq:step1it} we also used $\psi_1^{\rm iso}\psi_1^{\rm av}\le(\psi_1^{\rm iso})^2+(\psi_1^{\rm av})^2$. Now, given the first estimate in \eqref{eq:step1it} we obtain a bound for $\Psi_1^{\rm av}+\Psi_1^{\rm iso}$ which is better than the original a priori bound \eqref{eq:apriori Psi}. We can thus replace the $\psi_1^{\rm av/\rm iso}$ from \eqref{eq:apriori Psi} with the rhs. of the first line of \eqref{eq:step1it}.  Using iteration in both lines, we thus get
	\begin{equation}
		\label{eq:step2it}
		\begin{split}
			\Psi_1^{\rm av}+\Psi_1^{\rm iso}&\prec 1+\frac{(\psi_2^{\rm av})^{1/2}+(\psi_2^{\rm iso})^{1/2}}{(N\eta)^{1/4}}\,,\\
			\Psi_2^{\rm av}+\Psi_2^{\rm iso}&\prec 1+\frac{(\psi_4^{\rm av})^{1/2}}{\sqrt{N\eta}}+ \frac{(\psi_2^{\rm av})^{1/4}+(\psi_2^{\rm iso})^{1/4}}{(N\eta)^{1/8}}\cdot\frac{(\psi_3^{\rm iso})^{1/2}}{(N \eta)^{1/2}}  +   \frac{(\psi_3^{\rm iso})^{1/2} + (\psi_4^{\rm iso})^{1/2}}{(N \eta)^{1/4}}\,,
		\end{split}
	\end{equation}
	on the domain $ \mathbf{D}^{(\epsilon_0, \kappa_0)}_{\ell+K}$, for some $K$ as in Lemma~\ref{lem:iteration}. We now use the reduction inequalities from Lemma~\ref{lem:reduction} in the second line of \eqref{eq:step2it}:
	\begin{equation}
		\label{eq:step3it}
		\begin{split}
			&\Psi_1^{\rm av}+\Psi_1^{\rm iso}\prec 1+\frac{(\psi_2^{\rm av})^{1/2}+(\psi_2^{\rm iso})^{1/2}}{(N\eta)^{1/4}}\\
			&\Psi_2^{\rm av}+\Psi_2^{\rm iso}\prec (N \eta)^{1/2}+\frac{\psi_2^{\rm av}}{\sqrt{N\eta}}+(N\eta)^{1/4}(\psi_2^{\rm iso})^{1/2}+(\psi_2^{\rm av})^{1/2}+\frac{(\psi_2^{\rm av}\psi_2^{\rm iso})^{1/2}}{(N\eta)^{1/4}}\,, \\
			& \ +\left( (N\eta)^{1/4}+\frac{(\psi_2^{\rm av})^{1/4}+(\psi_2^{\rm iso})^{1/4}}{(N\eta)^{1/8}}\right)\left(1+\frac{(\psi_2^{\rm iso})^{1/2}}{(N\eta)^{1/4}}+\frac{(\psi_2^{\rm av})^{1/4}}{(N\eta)^{1/8}}+\frac{(\psi_2^{\rm iso})^{1/2}(\psi_2^{\rm av})^{1/4}}{(N\eta)^{3/8}}\right)\,,
		\end{split}
	\end{equation}
	on the domain $ \mathbf{D}^{(\epsilon_0, \kappa_0)}_{\ell+K}$. Next, using iteration once again in the second line of \eqref{eq:step3it}, we obtain
	\begin{equation*}
		\Psi_1^{\rm av}+\Psi_1^{\rm iso}\prec 1+\frac{(\psi_2^{\rm av})^{1/2}+(\psi_2^{\rm iso})^{1/2}}{(N\eta)^{1/4}}\,, \quad \qquad
		\Psi_2^{\rm av}+\Psi_2^{\rm iso}\prec (N \eta)^{1/2}
	\end{equation*}
	on the domain $ \mathbf{D}^{(\epsilon_0, \kappa_0)}_{\ell+K+K'}$, for some $K'$ as in Lemma~\ref{lem:iteration}. We point out that here we used Schwarz and Young inequalities for several terms. Finally, using iteration one last time we conclude
	\begin{equation*}
		\Psi_1^{\rm av}+\Psi_1^{\rm iso}\prec 1\,, \qquad\quad \Psi_2^{\rm av}+\Psi_2^{\rm iso}\prec (N \eta)^{1/2}
	\end{equation*}
	on the domain $ \mathbf{D}^{(\epsilon_0, \kappa_0)}_{\ell+K+K'+K''}$, for some $K''$ as in Lemma~\ref{lem:iteration}. This concludes the proof.
\end{proof}
\begin{remark} \label{rmk:L=L(epsilon)}
	We observe that in every application of Lemma \ref{lem:iteration} during the proof of Lemma \ref{lem:multiGll}, the parameter $D$ is uniformly bounded by, say, $D \le 10$, as follows by estimating every resolvent in $\Psi_k^{\rm av/iso}$ by norm and using the trivial $1/\eta$-bound on inverse of the stability operator in the iterative definition of $M(w_1, ... , w_k)$ given in Definition \ref{def:Mdef}. A further quick inspection of the above proof shows, that $\alpha$ can be chosen as fixed $\alpha = 1/2$. Finally, the parameter $\nu$ is lower bounded by (some universal positive constant times) $\epsilon$, since $N \eta \ge N^{\epsilon/2}$ by construction of the initial domain \eqref{eq:Omega00}. Hence, the constants $K$, $K'$, and $K''$ only depend on $\epsilon$ and therefore also the maximal number $L = L(\epsilon)$ of domain shrinkings.
\end{remark}


\section{Proof of the master inequalities, Proposition \ref{prop:master}} \label{sec:proofmaster}
Before going into the proofs of the master inequalities, we state a simple lemma, which will frequently be used in the following. Recall that the deformation $\Defo \in \C^{N \times N}$ is fixed and hence omitted from the notation.
\begin{lemma} \label{lem:intrepG^2}{\rm (Integral representations for products of resolvents)} \\ 
	Let $k \in \N$ and $w_1, ... , w_k \in \C\setminus \R$ be spectral parameters, whose imaginary parts have equal sign, i.e.~$\sgn(\Im w_1) = ... = \sgn(\Im w_k) =: \tau$. Then, for any $J \subset \R$ being a union of compact intervals such that $\Re w_i \in \mathring{J}$ (the interior) for all $i \in [k]$ and $0< \tilde{\eta} < \eta := \min_j |\Im w_j|$, we have the integral representation
	\begin{equation} \label{eq:intrepG^2}
		\prod_{j=1}^{k} G(w_j) =  \frac{1}{2\pi \I} \int_{\Gamma} G(z) \prod_{j=1}^{k} \frac{1}{z - w_j} \, \D z\,,
	\end{equation}
	where the contour $\Gamma$ from \eqref{eq:intrepG^2} is defined as (see Figure \ref{fig:J})
	\begin{equation} \label{eq:contour}
		\Gamma \equiv \Gamma^\tau_{\tilde{\eta}}(J) := \begin{cases}
			\partial \big( J\times [\I \tilde{\eta}, \I \infty) \big) \quad      & \text{if} \quad \tau = + \\
			\partial \big( J \times (- \I \infty, - \I \tilde{\eta}] \big) \quad & \text{if} \quad \tau = -
		\end{cases}
	\end{equation}
	and the boundary is parameterised in counter-clockwise orientation.
\end{lemma}
\begin{proof}
	This easily follows from residue calculus. {
	For example, for $k=2$, we have that 
	\begin{equation*}
\frac{1}{\lambda_i - w_1} \frac{1}{\lambda_i - w_2} = \frac{1}{2 \pi \ii} \int_\Gamma \frac{1}{\lambda_i - z} \frac{1}{z-w_1} \frac{1}{z-w_2} \dif z 
	\end{equation*}
	for every eigenvalue $\lambda_i$ of $H$ by definition of the contour. This implies \eqref{eq:intrepG^2} for $k=2$ by spectral decomposition for $H$; the argument for general $k$ is analogous.}
\end{proof}
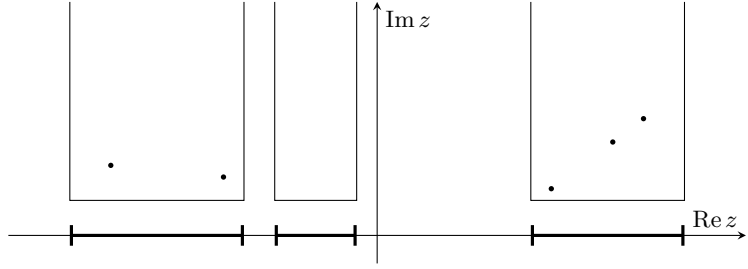
\begin{figure}[htbp]
	\centering
	\begin{tikzpicture}
		\begin{axis}[no markers,
				width=0.9\textwidth,
				xlabel=$\Re z$,
				ylabel=$\Im z$,
				height=0.4\textwidth,
				axis lines=middle,
				ymin=-.12,ymax=1,
				xmin=-3.6,xmax=3.6,
				ytick={0},yticklabels={0},
				xtick={0},
				xticklabels={0},
				axis on top=true, color=black]
			\addplot[mark=none] coordinates {(-3,1) (-3,.15) (-1.3,.15) (-1.3,1)};
			\addplot[mark=none] coordinates {(-1.,1) (-1.,.15) (-.2,.15) (-.2,1)};
			\addplot[mark=none] coordinates {(1.5,1) (1.5,.15) (3,.15) (3,1)};
			\draw[|-|,very thick] (-3,0) -- (-1.3,0);
			\draw[|-|,very thick] (-1,0) -- (-.2,0);
			\draw[|-|,very thick] (1.5,0) -- (3,0) node [midway,below] {
			};
			\fill[black] (1.7,.2) circle (1pt);
			\fill[black] (2.6,.5) circle (1pt);
			\fill[black] (2.3,.4) circle (1pt);
			\fill[black] (-2.6,.3) circle (1pt);
			\fill[black] (-1.5,.25) circle (1pt);
		\end{axis}
	\end{tikzpicture}
	\caption{Depicted is the scenario from Lemma \ref{lem:intrepG^2} with five spectral parameters represented as dots in the upper half plane. Moreover, we indicated the union of compact intervals $J$ on the real axis and the contour $\Gamma$ as described in \eqref{eq:contour}. Note that one of the three intervals constituting $J$ does not contain any $\Re w_j$.}
	\label{fig:J}
\end{figure}

We recall the definition of the \emph{second order renormalisation}, denoted by underline, from \cite{ETHpaper}. For functions $f(W), g(W)$ of the random matrix $W$ (see \eqref{eq:herm}), we define
\begin{equation} \label{eq:underline}
	\underline{f(W) W g(W)} := f(W) W g(W) - \widetilde{\mathbb{E}} \big[  (\partial_{\widetilde{W}}f)(W)\widetilde{W} g(W) + f(W)  \widetilde{W} (\partial_{\widetilde{W}}g)(W) \big]\,,
\end{equation}
where $\partial_{\widetilde{W}}$ denotes the directional derivative in the direction of
\begin{equation*}
	\widetilde{W} := \begin{pmatrix}
		0 & \widetilde{X} \\ \widetilde{X}^* & 0
	\end{pmatrix}\,,
\end{equation*}
where $\widetilde{X}$ is a complex Ginibre matrix that is independent of $W$. The expectation is taken w.r.t.~the matrix $\widetilde{X}$. Note that, if $W$ itself consists of a complex Ginibre matrix $X$, then $\E \underline{f(W)Wg(W)} = 0$, while for $X$ with a general distribution this expectation is independent of the first two moments of $X$. In other words, the underline renormalises the product $f(W)W g(W)$ to second order.  We remark that underline \eqref{eq:underline} is a well-defined notation, if the `middle' $W$ to which the renormalisation refers is unambiguous. This is always be the case in all our proofs, since the functions $f, g$ will be products of resolvents, never involving explicitly monomials in $W$.

We note that {for any deterministic matrix $R\in \C^{2N\times 2N}$ we have}
\begin{equation*} 
	\widetilde{\E} \widetilde{W} R \widetilde{W} = 2 \langle R E_2 \rangle E_1 + 2 \langle R E_1 \rangle E_2 = \sum_\sigma \sigma \langle R E_\sigma \rangle E_\sigma = \mathcal{S}[R]
\end{equation*}
and furthermore, that the directional derivative of the resolvent is given by
\begin{equation*}
	\partial_{\widetilde{W}} G = -G \widetilde{W} G\,.
\end{equation*}
For example, in the special case $f(W) = 1$ and $g(W) = (W + \hat{\Defo} -w)^{-1} = G$, we thus have
\begin{equation*}
	\underline{WG} = WG + \mathcal{S}[G] G
\end{equation*}
by definition of the underline in \eqref{eq:underline}.

Using this underline notation in combination with the identity $G (W + \hat{\Defo} - w) = E_+$ and the defining equation \eqref{eq:MDE} for $M$, we have
\begin{equation} \label{eq:start}
	G = M - M \underline{W G} + M \mathcal{S}[G-M] G  = M - \underline{GW}M + G \mathcal{S}[G-M]M\,.
\end{equation}
Recall that $\langle GE_-\rangle=0$ (see below~\eqref{eq:chiral}) which 
immediately yields that
$\mathcal{S}[G] = \sum_\sigma \sigma \langle G E_\sigma\rangle E_\sigma =  \langle G \rangle$. Moreover, we have that $\mathcal{S}[M] = \langle M\rangle$, as follows from \eqref{eq:Mu}--\eqref{eq:mde}, and hence $\mathcal{S}[\cdot]$ effectively acts like a trace on $G$ and $M$, i.e.
\begin{equation} \label{eq:GMtrace}
	\mathcal{S}[G-M] = \langle G- M \rangle \,.
\end{equation}

Now, similarly to \cite{multiG}, the key idea of the proof of Proposition~\ref{prop:master} is using \eqref{eq:start} for some $G_j$ in a chain $G_1 A_1 \, \cdots \, A_{k} G_{k+1}$ and extending the renormalisation \eqref{eq:underline} to the whole product at the expense of adding resolvent products of shorter length. This will be done for each of the four estimates from Proposition~\ref{prop:master} separately and presented in an \emph{underlined lemma} in the beginning of each of the following subsections. Afterwards, the renormalisation of the whole product will be handled by cumulant expansion, exploiting that its expectation vanishes up to second order. While the proofs of the underlined lemmas for $\Psi_1^{\rm av/iso}$ are presented in detail, we defer the analogous arguments for $\Psi_2^{\rm av/iso}$ to Appendix \ref{app:underlinedproofs}.



\subsection{Proof of the first master inequality \eqref{eq:masterineq Psi 1 av}} \label{subsec:proofmaster1}
Let $w \equiv w_1 $ be a spectral parameter in $\mathbf{D}_{\ell+1}^{(\epsilon_0, \kappa_0)}$ (in particular in the bulk of the scDos, recall \eqref{eq:Omegal}) and $A \equiv A_1$ a $(w, w)$-\emph{regular} matrix (cf.~Definition \ref{def:regobs}). We use the notation $w = e + \I \eta$ and we assume without loss of generality~(by conjugation with $E_-$, see \eqref{eq:chiral}) that $1 \ge \eta > 0$. We also assume that \eqref{eq:apriori Psi} holds (in this subsection we will need it only for $\Psi_1^{\rm av}$ and $\Psi_2^{\rm av}$).
\begin{lemma} {\rm (Representation as full underlined)} \label{lem:underlined1} \\
	For any regular matrix $A = \mathring{A}$ we have that
	\begin{equation} \label{eq:underlined Psi 1 av}
		\langle (G-M)\mathring{A} \rangle = - \langle \underline{W G\mathring{A}'} \rangle + \mathcal{O}_\prec\big( \mathcal{E}_1^{\rm av} \big)
	\end{equation}
	for some other regular matrix $A' = \mathring{A}'$, which linearly depends on $A$ (see \eqref{eq:A'1av def} for the precise formula for $A'$). For the error term in \eqref{eq:underlined Psi 1 av}, we used the shorthand notation
	\begin{equation} \label{eq:E1av}
		\mathcal{E}_1^{\rm av} := \frac{1}{N \eta^{1/2}}\left(1 + \frac{\psi_1^{\rm av}}{N \eta}\right)\,.
	\end{equation}
\end{lemma}
Having this approximate representation of $\langle (G-M)\mathring{A} \rangle$ as a full underlined term at hand, we turn to the proof of \eqref{eq:masterineq Psi 1 av} via a (minimalistic) cumulant expansion: {For fixed indices $a,b$ and a smooth function $f(W)$ we have 
\begin{equation}
	\E w_{ab} f(W) = \sum_{l_1+l_2\ge 1} \frac{1}{l_1! l_2!} \kappa(\{w_{ab}\}^{l_1+1}, \{w_{ba}\}^{l_2}) \E \partial_{ab}^{l_1} \partial_{ba}^{l_2} f(W) ,
\end{equation}
where $\kappa(\{w_{ab}\}^{l_1+1}, \{w_{ba}\}^{l_2})$ is the cumulant $l_1+1$ copies of the random variable $w_{ab}$ and $l_2$ copies of the random variable $w_{ba}$, and $\partial_{ab}^{l_1} \partial_{ba}^{l_2}$ denotes the $l_1$-th derivative in the $ab$-entry and the $l_2$-th derivative in the $ba$-entry.
}
\begin{proof}[Proof of \eqref{eq:masterineq Psi 1 av}] 
	Let $p \in \N$. Then, starting from \eqref{eq:underlined Psi 1 av}, {and recalling the second order renormalisation \eqref{eq:underline}, we have
		\begin{equation}\label{cum exp intro}
			\begin{split}
				&\E \braket{\underline{WG} A}\braket{(G-M)A}^{2p-1} = \frac{1}{N} \sum_{ab} \E \underline{w_{ab}(GA)_{ba}} \braket{(G-M)A}^{2p-1}\\
				&\quad= \frac{1}{N}\sum_{ab} \E (GA)_{ba}\Bigl(\kappa(w_{ab}, w_{ba}) \partial_{ba} \braket{(G-M)A}^{2p-1} +    \kappa(w_{ab}, w_{ab}) \partial_{ab} \braket{(G-M)A}^{2p-1} \Bigr)\\
				&\qquad + \frac{1}{N} \sum_{ab} \sum_{l_1+l_2\ge 2} \frac{1}{l_1! l_2!} \kappa(\{w_{ab}\}^{l_1+1}, \{w_{ba}\}^{l_2}) \E \partial_{ab}^{l_1} \partial_{ba}^{l_2}\big[ (GA)_{ba} \braket{(G-M)A}^{2p-1}\big].
			\end{split}
		\end{equation}
		By computing the resolvent derivatives explicitly as 
		\begin{equation}
			\partial_{ab}\braket{(G-M)A}=-\frac{1}{N} (G AG)_{ba}
		\end{equation}
		and using that $\kappa(w_{ab},w_{ba})=R_{ab}/N$, the first term in the middle line of \eqref{cum exp intro} simplifies to 
		\begin{equation*}
			\frac{1}{N^3} \sum_{ab}  (GA)_{ba} (GAG)_{ab} = \frac{\braket{E_1 GA E_2 GAG}+\braket{E_2 GA E_1 GAG}}{N^2}
		\end{equation*}
		(up to the factor of $\braket{(G-M)A}^{2p-2}$) ,
		the second term being similar up to an additional transposition. Here \[
			R_{ab} := \mathbf{1}(a \le N, b \ge N+1 \  \text{or} \  b \le N, a \ge N+1 )
		\] is the indicator function for the off-diagonal blocks of $W$.
		For the remaining term in~\eqref{cum exp intro} we simply estimate the cumulants by their size 
		$|\kappa(\{w_{ab}\}^{l_1+1}, \{w_{ba}\}^{l_2})| \le N^{-(l_1+l_2+1)/2} R_{ab}$ to obtain
	}
	\begin{align}
		         & \mathbf{E} |\langle (G-M)A \rangle |^{2p} \nonumber                                                                                                                                                                                                                   \\
		=        & \left\vert  - \mathbf{E}   \langle \underline{WG} {A'} \rangle  \langle (G-M)A \rangle^{p-1}   \langle (G-M)^*A^* \rangle^{p}   \right\vert   + \mathcal{O}_\prec\big((\mathcal{E}_1^{\rm av})^{2p}\big) \label{eq:min exp 1av}                                       \\
		\lesssim & \, \mathbf{E} \,  \frac{\big| \sum_{\sigma } \sigma \langle G E_\sigma G {A'} E_\sigma G A \rangle \big| + \big| \sum_{\sigma } \sigma \langle G^* E_\sigma G {A'} E_\sigma G^* A^* \rangle \big| }{N^2} \big\vert  \langle (G-M)A \rangle \big\vert^{2p-2} \nonumber \\
		         & + \sum_{|\boldsymbol{l}| + \sum(J \cup J_*) \ge 2} \mathbf{E} \, \Xi_1^{\rm av}(\boldsymbol{l}, J, J_*) \big\vert  \langle (G-M)A \rangle \big\vert^{2p-1 - |J \cup J_*|} + \mathcal{O}_\prec\big((\mathcal{E}_1^{\rm av})^{2p}\big) \nonumber\,,
	\end{align}
	where $\Xi_1^{\rm av}(\boldsymbol{l}, J, J_*)$ is defined as
	\begin{equation} \label{eq:Xi 1 av def}
		\Xi_1^{\rm av} := N^{-(|\boldsymbol{l}| + \sum(J \cup J_*) + 3)/2} \sum_{ab} R_{ab}|\partial^{\boldsymbol{l}} (G {A'})_{ba}| \prod_{\boldsymbol{j} \in J} |\partial^{\boldsymbol{j}} \langle GA \rangle |   \prod_{\boldsymbol{j} \in J_*} |\partial^{\boldsymbol{j}} \langle G^*A^*  \rangle |
	\end{equation}
	and the summation in the last line of  \eqref{eq:min exp 1av} is taken over tuples\footnote{{Note that the role played by $(l_1,l_2)$ here is slightly different than in~\eqref{cum exp intro} above. Here the derivatives are applied to the individual factors according to Leibniz' rule, resulting in $\bm l, J, J_\ast$, and $\bm l$ encodes only 
	the derivatives hitting the $(GA')_{ab}$ factor.}} $\boldsymbol{l} \in \Z^2_{\ge 0}$ and multisets of tuples $J, J_* \subset \Z^2_{\ge 0} \setminus \{(0,0)\}$.
	Moreover, we set ${\partial^{\bm l}=}\partial^{(l_1,l_2)} := \partial_{ab}^{l_1} \partial_{ba}^{l_2}$, {similarly 
	$\partial^{\bm j}=\partial^{(j_1,j_2)} := \partial_{ab}^{j_1} \partial_{ba}^{j_2}$ and we define}
	$|{\bm l}|=|(l_1, l_2)| = l_1 + l_2$, $\sum J = \sum_{\boldsymbol{j} \in J} |\boldsymbol{j}|$. 
	In the remainder of the proof, we need to analyze the rhs.~of the inequality derived in \eqref{eq:min exp 1av}. We begin with the third line and study the terms involving $\Xi_1^{\rm av}$ from \eqref{eq:Xi 1 av def} afterwards.

	Before going into the proof, we note that, due to the cumulant expansion in \eqref{eq:min exp 1av}, there are chains of resolvents $G$ and deterministic matrices $A$ appearing, where some of the $A$'s are \emph{not} necessarily regular w.r.t.~the spectral parameters of the surrounding $G$'s. The principal idea is to decompose such $A$ with the aid of Lemma \ref{lem:regularbasic} and carefully track the resulting errors. As a rule of thumb, potentially small denominators resulting from {resolvent identities \eqref{eq:resolid}} or the integral representation in Lemma \ref{lem:intrepG^2} are balanced with the linear perturbative estimates from Lemma \ref{lem:regularbasic}.  See also Remark \ref{rmk:strategy} below.
	\\[2mm]
	{\bf \underline{Gaussian contribution: third line of \eqref{eq:min exp 1av}.}} 
	In order to do so, we need to analyze in total four terms, each of which carries a factor of
	\begin{equation*}
		\langle G E_\sigma G {A'} E_\sigma G A \rangle \quad \text{or} \quad  \langle G^* E_\sigma G {A'} E_\sigma G^* A^* \rangle\,, \quad \text{for} \quad \sigma = \pm \,.
	\end{equation*}
	Since their treatment is very similar, we focus on the two exemplary terms
	\begin{equation} \label{eq:twoterms1av}
		\text{(i)} \quad  \langle G G {A'}  G A \rangle \qquad \text{and} \qquad  \text{(ii)} \quad \langle G^*  G {A'}  G^* A^* \rangle\ \,.
	\end{equation}
	In the analysis of the Gaussian contribution in Section \ref{subsec:proofmaster2}, we will discuss the analogs of the other two terms in more detail.
	\\[1mm]
	\emph{\underline{First term.}} For the first term in \eqref{eq:twoterms1av}, we apply the integral representation from Lemma \ref{lem:intrepG^2} to $GG$ with
	\[
		\tau = +\,, \quad   J= \mathbf{B}_{ \ell \kappa_0}\,, \quad \text{and} \quad  \tilde{\eta} = \frac{\ell}{\ell +1} \eta\,,
	\]
	for which we recall that $w \in \mathbf{D}_{\ell +1}^{(\epsilon_0, \kappa_0)}$, i.e.~in particular $\eta \ge (\ell +1) N^{-1+\epsilon_0}$
	and hence $\tilde{\eta} \ge \ell N^{-1+\epsilon_0}$. {The fact that $J$ is a union of compact intervals follows from the fact the support of the density of $H^\Lambda$ has finitely many components.}
	In particular,  $\Gamma \equiv \Gamma^\tau_{\tilde{\eta}}(J) \subset \mathbf{D}_{\ell}^{(\epsilon_0, \kappa_0)}$. Now, we split the contour $\Gamma$ in three parts,\footnote{\label{ftn:furthersplit}In the case of several $w_1, ... , w_k$, the second part might require a further decomposition: If the spectral parameters of the resolvents which are \emph{not} involved in such an integral representation have spectral parameters with imaginary parts of absolute value greater than one, we need to split $\Gamma_2$ according to $|\Im z| \le 1$ and $|\Im z| > 1$. While the former will be treated exactly as $\Gamma_2$ here, the latter shall be estimated by means of the $\eta> 1$-laws, which we discussed after Remark \ref{rmk:sqrteta}.} i.e.
	\begin{equation} \label{eq:contourdecomp}
		\Gamma = \Gamma_1 \plus \Gamma_2 \plus \Gamma_3\,.
	\end{equation}
	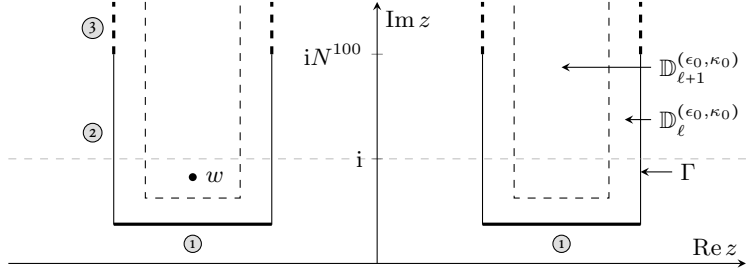
\begin{figure}[htbp]
		\centering
		\begin{tikzpicture}
			\begin{axis}[no markers,
				width=0.9\textwidth,
				xlabel=$\Re z$,
				ylabel=$\Im z$,
				height=0.4\textwidth,
				axis lines=middle,
				ymin=0,ymax=1,
				xmin=-3.5,xmax=3.5,
				ytick={0,.4,.8},yticklabels={0,$\mathrm{i}$,$\mathrm{i}N^{100}$},
				xtick={0},
				xticklabels={0},
				axis on top=true, color=black]
				\addplot[mark=none, very thick] coordinates {(-2.5,.15) (-1.,.15)};
				\addplot[mark=none] coordinates {(-2.5,.8) (-2.5,.15) (-1,.15) (-1,.8)};
				\addplot[mark=none,dashed] coordinates {(-2.2,1) (-2.2,.25) (-1.3,.25) (-1.3,1)};
				\addplot[mark=none, very thick, dashed] coordinates {(-1,.8) (-1,1)};
				\addplot[mark=none, very thick, dashed] coordinates {(-2.5,.8) (-2.5,1)};
				\draw[dashed,lightgray] (-3.5,.4) -- (3.5,.4);
				\addplot[mark=none, very thick] coordinates {(2.5,.15) (1.,.15)};
				\addplot[mark=none] coordinates {(2.5,.8) (2.5,.15) (1,.15) (1,.8)};
				\addplot[mark=none,dashed] coordinates {(2.2,1) (2.2,.25) (1.3,.25) (1.3,1)};
				\addplot[mark=none, very thick, dashed] coordinates {(1,.8) (1,1)};
				\addplot[mark=none, very thick, dashed] coordinates {(2.5,.8) (2.5,1)};
				\draw[dashed,lightgray] (2.5,.4) -- (2.5,.4);
				\node[circle,draw=black,inner sep=1pt,fill=lightgray!50] at (axis cs: 1.75,.075) {\footnotesize 1};
				\node[circle,draw=black,inner sep=1pt,fill=lightgray!50] at (axis cs: -1.75,.075) {\footnotesize 1};
				\node[circle,draw=black,inner sep=1pt,fill=lightgray!50] at (axis cs: -2.7,.5) {\footnotesize 2};
				\node[circle,draw=black,inner sep=1pt,fill=lightgray!50] at (axis cs: -2.7,.9) {\footnotesize 3};
				\draw[stealth-] (1.75,.75) -- (2.6,.75) node [right] {$\mathbb D_{\ell+1}^{(\epsilon_0,\kappa_0)}$};
				\draw[stealth-] (2.35,.55) -- (2.6,.55) node [right] {$\mathbb D_{\ell}^{(\epsilon_0,\kappa_0)}$};
				\draw[stealth-] (2.5,.35) -- (2.8,.35) node [right] {$\Gamma$};
				\node[circle,fill=black,inner sep=0pt,minimum size=3pt,label=right:{$w$}] at (-1.75,.33) {};
			\end{axis}
		\end{tikzpicture}
		\caption{The contour $\Gamma$ is split into three parts (see \eqref{eq:contourdecomp}). In case of multiple spectral parameters, the second part might require a further decomposition at the level
			indicated by the dashed horizontal line (see Footnote \ref{ftn:furthersplit}).
			Depicted is the situation, where the bulk $\mathbf{B}_{\ell \kappa_0}$ consists of  two components.}
		\label{fig:decomp}
	\end{figure}
	As depicted in Figure \ref{fig:decomp}, the first part of the contour consists of the entire horizontal part of $\Gamma$. The second part, $\Gamma_2$, covers the vertical components up to
	$|\Im z| \le N^{100}$. Finally, $\Gamma_3$ consists of the remaining part with $|\Im z| > N^{100}$.

	{Now, the contribution coming from $\Gamma_3$ can easily be estimated by one via a trivial norm bound on $G$. For $z \in \Gamma_2$, we use that $\mathbf{1}_\delta^\pm(z,w) = 0$ for every $w \in \mathbf{D}_{\ell + 1}^{(\epsilon_0, \kappa_0)}$ (recall \eqref{kappabulkreg} and \eqref{eq:deltachoice}) and hence every matrix is $(z,w)$-regular. Therefore, after adding and subtracting the corresponding deterministic approximation, we can bound this part by $(1 + \psi_2^{\rm av}/(N \eta))$ with the aid of Lemma \ref{lem:Mbound}.} Hence, after splitting the contour integral and bounding each contribution as just described, we find
	\begin{equation} \label{eq:Gaussian1av1}
		\left| \langle G G {A'}  G A \rangle\right| \prec  \left(1 + \frac{\psi_2^{\rm av}}{N \eta}\right) +  \int_{ \mathbf{B}_{ \ell \kappa_0}} \frac{\left\vert \langle G(x + \I \tilde{\eta}) A' G(e+\I \eta) A\rangle \right\vert}{(x -e)^2 +  \eta^2}  \, \D x\,.
	\end{equation}

	Next, we decompose $A = \mathring{A} = \mathring{A}^{{e+\I \eta, e + \I \eta}}$ and $A'= \mathring{A}' = \mathring{(A')}^{{e+\I \eta, e + \I \eta}}$ according to Lemma \ref{lem:regularbasic} as
	\begin{align*}
		\mathring{A}^{{e+\I \eta, e + \I \eta}}    & = \mathring{A}^{{ e+ \I \eta, x + \I \tilde{\eta}}} + \mathcal{O}\big(|x-e|+ \eta\big) E_+ + \mathcal{O}\big(|x-e| + \eta\big) E_-\,,      \\[1mm]
		\mathring{(A')}^{{e+\I \eta, e + \I \eta}} & = \mathring{(A')}^{{ x + \I \tilde{\eta}, e+ \I \eta, }} + \mathcal{O}\big(|x-e|+ \eta\big) E_+ + \mathcal{O}\big(|x-e| + \eta\big) E_-\,.
	\end{align*}
	Plugging this into \eqref{eq:Gaussian1av1}, we obtain several terms contributing to the integral. 	{Adding and subtracting the deterministic approximation, the leading term accounts for
	\begin{equation*}
		\int_{ \mathbf{B}_{ \ell \kappa_0}} \frac{\big\vert \langle G(x + \I \tilde{\eta}) \mathring{(A')}^{{ x + \I \tilde{\eta}, e+ \I \eta }} G(e+\I \eta) \mathring{A}^{{ e+ \I \eta, x + \I \tilde{\eta}}}\rangle \big\vert}{(x -e)^2 +  \eta^2}  \, \D x \prec \frac{1}{\eta} \left(1 + \frac{\psi_2^{\rm av}}{N \eta}\right)
	\end{equation*}
	by means of Lemma \ref{lem:Mbound}.
Here the ``$1$'' on the right-hand side is due to the contribution of the deterministic approximation $\langle M(x+\ii\tilde\eta,(A')^\circ,e+\ii\eta) \mathring{A}^{{ e+ \I \eta, x + \I \tilde{\eta}}} \rangle $, while the ``$\psi_2^{\rm av}/(N\eta)$'' is due to the definition of $\Psi_2^\mathrm{av} $ and the bound $\Psi_2^\mathrm{av} \prec \psi_2^\mathrm{av}$.}
	The error terms can be dealt with by simple \cred{resolvent identities \eqref{eq:resolid}} in combination with the usual single-resolvent local law, Theorem \ref{thm:singleG}, proving them to be bounded by $\eta^{-1}$. Indeed, for a generic $B \in \C^{2N \times 2N}$, we consider the exemplary term
	\begin{align*}
		\int_{ \mathbf{B}_{ \ell \kappa_0}}\left\vert \langle G(x + \I \tilde{\eta})  E_+ G(e+\I \eta) B\rangle \right\vert & \frac{|x-e| + \eta}{(x -e)^2 +  \eta^2}  \, \D x                                                                                                                                                  \\
		\lesssim                                                                                                            & \int_{ \mathbf{B}_{ \ell \kappa_0}} \frac{\left\vert \big\langle \big(G(x + \I \tilde{\eta}) -  G(e+\I \eta)\big) B\big\rangle \right\vert}{(x -e)^2 +  \eta^2}  \, \D x  \prec \frac{1}{\eta}\,.
	\end{align*}
	\emph{\underline{Second term.}} The second term in \eqref{eq:twoterms1av} is much simpler than the first. After writing $G G^* = \Im G/ \eta$, it suffices to realise that, by means of Lemma \ref{lem:regularbasic},
	\begin{equation*}
		A' = \mathring{(A')}^{{e+\I \eta, e - \I \eta}}\,, \quad \mathring{(A')}^{{e-\I \eta, e - \I \eta}} = A' + \mathcal{O}(|e|) E_-\,, \quad \text{and} \quad A^* = \mathring{(A^*)}^{{e-\I \eta, e \pm \I \eta}}
	\end{equation*}
	in order to bound
	\[
		\left\vert \langle G^*  G {A'}  G^* A^* \rangle \right\vert \prec \frac{1}{\eta} \left(1 + \frac{\psi_2^{\rm av}}{N \eta} \right) + \frac{|e|}{\eta} \, \frac{\big| \langle [G(-e + \I \eta) - G(e - \I \eta)] A^*E_- \rangle \big|}{|e| + \eta} \prec \frac{1}{\eta} \left(1 + \frac{\psi_2^{\rm av}}{N \eta} \right)
	\]
	with the aid of Lemma \ref{lem:Mbound}, the chiral symmetry \eqref{eq:chiral}, a {resolvent identity \eqref{eq:resolid}} and Theorem \ref{thm:singleG}.
	\\[2mm]
	This finishes the estimate for the Gaussian contribution from the third line of \eqref{eq:min exp 1av}, for which we have shown that
	\begin{equation} \label{eq:Gaussian1av final}
		\frac{1}{N^2}\sum_{\sigma }\left(\big|  \langle G E_\sigma G {A'} E_\sigma G A \rangle \big| + \big|  \langle G^* E_\sigma G {A'} E_\sigma G^* A^* \rangle \big|\right)   \prec \frac{1}{N^2 \eta} \left(1 + \frac{\psi_2^{\rm av}}{N \eta}\right)\,.
	\end{equation}
	We are now left with the terms from the last line \eqref{eq:min exp 1av} resulting from higher order cumulants.
	\\[2mm]
	{\bf \underline{Higher order cumulants and conclusion.}} The terms stemming from higher order cumulants are estimated in Section \ref{subsec:higherorder}, the precise bound being given in \eqref{eq:higherorder1av}. Indeed, plugging \eqref{eq:Gaussian1av final} and \eqref{eq:higherorder1av} into \eqref{eq:min exp 1av} we obtain
	\begin{align*}
		\mathbf{E} & |\langle (G-M)A \rangle |^{2p} \prec (\mathcal{E}_1^{\rm av})^{2p}                                                                                                                                                                                                   \\
		           & + \sum_{m=1}^{p} \left[ \frac{1}{N \eta^{1/2}} \left( 1 + \frac{\psi_1^{\rm iso} + (\psi_2^{\rm av})^{1/2}}{(N \eta)^{1/2}} + \frac{(\psi_2^{\rm iso})^{1/8}}{(N \eta)^{1/8}} \right) \right]^m \left( \E |\langle (G-M)A \rangle |^{2p}  \right)^{1-m/2p} \nonumber
	\end{align*}
	and get the appropriate estimate $\E |\ldots|^{2p}$ using Young inequalities. Since $p$ was arbitrary, it follows that
	\begin{equation*}
		\Psi_1^{\rm av} \prec 1 + \frac{\psi_1^{\rm av}}{N \eta} + \frac{\psi_1^{\rm iso} +  (\psi_2^{\rm av})^{1/2}  }{(N \eta)^{1/2}} + \frac{(\psi_2^{\rm iso})^{1/4}}{(N \eta)^{1/8}}\,.
	\end{equation*}
	The bound given in Proposition \ref{prop:master} is an immediate consequence after a further trivial Young inequality.
\end{proof}
\begin{remark} \label{rmk:strategy}
	Although the proof of the first master inequality \eqref{eq:masterineq Psi 1 av} is rather short, it already revels a general strategy for dealing with a generic (not strictly) alternating chain
	\begin{equation} \label{eq:chainstrategy}
		\cdots \, GG A G A G E_- A  G E_- G A \, \cdots
	\end{equation}
	of resolvents $G$ and deterministic matrices $A$.
	\begin{itemize}
		\item[(i)] Apply {resolvent identites \eqref{eq:resolid}} and the integral representation from Lemma \ref{lem:intrepG^2} in order to reduce a product of resolvents to a linear combination (discrete or continuous, respectively). For terms of the form $G E_- G$ instead of $GG$ this additionally requires an application of the chiral symmetry \eqref{eq:chiral}.
		\item[(ii)] In the resulting strictly alternating chain, decompose every deterministic $A$ according to the regularisation from Definition \ref{def:regobs} w.r.t.~the spectral parameters of its surrounding resolvents by using Lemma~\ref{lem:regularbasic}.
		\item[(iii)] Estimate the regular parts coming from this decomposition in terms of $\Psi_k^{\rm av/iso} \prec \psi_k^{\rm av/iso}$. Carefully track the resulting errors stemming from the other parts.
	\end{itemize}
	These steps shall be applied repeatedly until the entire chain \eqref{eq:chainstrategy} has been examined. {The first two steps outlined above will be performed mechanically without any complication.} However, the third step is non-trivial and requires careful analysis on a case-by-case basis.

	We have already mentioned that, as a rule of thumb, potentially small denominators resulting from Step (i) are balanced with the linear perturbative numerators from Step (ii).
\end{remark}

It remains to give a proof of Lemma \ref{lem:underlined1}.

\begin{proof}[Proof of Lemma \ref{lem:underlined1}] Similarly as in \eqref{eq:underlined Psi 1 av}, we suppress the indices of $G \equiv G_1$, $M \equiv M_1$ etc.

	We start with the first identity in \eqref{eq:start}, such that, after defining the one-body stability operator
	\begin{equation*}
		\mathcal{B}:= 1- M \mathcal{S}[\cdot] M
	\end{equation*}
	we find
	\begin{equation*}
		\mathcal{B}[G-M] = - M \underline{WG} + M \mathcal{S}[G-M] (G-M)
	\end{equation*}
	and consequently, by inversion, multiplication by $A = \mathring{A}$ (in the sense of \eqref{eq:circ def}, see also \eqref{eq:1G traceless}) and taking a trace
	\begin{equation} \label{eq:1G basic exp}
		\langle (G-M) A \rangle = - \langle \underline{WG} \mathcal{X}[A]M \rangle + \langle \mathcal{S}[G-M] (G-M) \mathcal{X}[A]M \rangle\,,
	\end{equation}
	where we introduced the linear operator
	\begin{equation*}
		\mathcal{X}[B] :=\big((\mathcal{B}^*)^{-1}[B^*]\big)^*  = \big(1 - \mathcal{S}[M \, \cdot \, M]\big)^{-1}[B]\quad \text{for} \quad B \in \C^{2N \times 2N}\,.
	\end{equation*}

	Then, it is important to note that the condition $\mathbf{1}_\delta^+ \langle \Im M A \rangle = 0$ (the first of the two imposed via \eqref{eq:1G traceless}; recall the definition of the cutoff function $\mathbf{1}_\delta^+$ from \eqref{eq:case regulation} and \eqref{eq:case regulation2}), is stable under the linear operation $A \mapsto \mathcal{X}[A] M $.
	\begin{lemma} \label{lem:Psi1av stable}
		For a generic $B \in \C^{2N \times 2N}$,  we find
		\begin{equation} \label{eq:stability}
			\langle \mathcal{X}[B] M \Im M \rangle = \langle B \mathcal{B}^{-1}[M\Im M] \rangle = \frac{\I }{2} \frac{\langle B \Im M \rangle}{\langle \Im M \rangle} + \mathcal{O}(\eta)\,.
		\end{equation}
	\end{lemma}
	\begin{proof}
		Using \eqref{eq:MWard}, we compute
		\begin{equation*}
			\mathcal{B}^{-1}[M \Im M] = \frac{\mathcal{B}^{-1}[M^2 - M M^*]}{2 \I} = \frac{\I}{2} \frac{\Im M}{\eta + \langle \Im M\rangle} + \frac{1}{2 \I} \frac{1 - \langle M M^* \rangle}{1 - \langle M^2 \rangle} M^2\,.
		\end{equation*}
		Now, by means of Lemma \ref{lem:Mbasic} and Lemma \ref{lem:eigendecomp}, we find that
		\begin{equation*}
			\left| 1 - \langle M M^* \rangle \right| = \mathcal{O}(\eta) \quad \text{and} \quad \left| 1 - \langle M^2 \rangle \right| \gtrsim 1\,, \quad \text{respectively}. \qedhere
		\end{equation*}
	\end{proof}

	Recall from \eqref{eq:GMtrace} that $\mathcal{S}[G - M] = \langle G - M \rangle$.
	Therefore, by means of the usual averaged local law, Theorem \ref{thm:singleG}, which in particular shows that $\vert \langle \underline{WG} B \rangle \vert \prec \frac{1}{N\eta}$ for arbitrary $\Vert B \Vert \lesssim 1$ (see also Appendix~\ref{app:locallaw} and \cite{slowcorr}), we can write \eqref{eq:1G basic exp} as
	\begin{align} \nonumber
		\langle (G-M) A \rangle = & - \langle \underline{WG} (\mathcal{X}[A]M)^\circ \rangle + \langle G - M \rangle\langle  (G-M)( \mathcal{X}[A]M)^\circ \rangle                     \\
		                          & -\mathbf{1}_\delta^-  c_-(\mathcal{X}[A]M)    \langle \underline{WG} E_- \rangle + \mathcal{O}_\prec \big( N^{-1}\big)\,, \label{eq:1G basic exp2}
	\end{align}
	where in the underlined term, we used that the $E_+$ component
	of the regularisation of $\mathcal{X}[A]M$ is negligible thanks to Lemma \ref{lem:Psi1av stable}
	and the regularity of $A$, and we introduced the short hand notation
	\begin{equation*}
		c_-(\mathcal{X}[A]M) := \frac{\langle M \mathcal{X}[A] M M E_- \rangle}{\langle M E_- M E_- \rangle } \,.
	\end{equation*}

	Next, with the aid of $WG = I -\hat{\Defo}  G +w G$ and using $\langle G E_- \rangle = 0$ from \eqref{eq:GMtrace}, we undo the underline in the second to last term, such that we infer
	\begin{align*}
		\langle \underline{WG} E_- \rangle = - \langle G E_- \hat{\Defo}\rangle =-  \langle (G - M ) E_- \hat{\Defo}\rangle = - \langle (G - M ) (E_- \hat{\Defo})^\circ \rangle\,.
	\end{align*}
	In the second equality, we used that $ \langle M E_- \hat{\Defo}\rangle = 0$, which follows by a simple computation using the explicit form of $M$ given in \eqref{eq:Mu}--\eqref{eq:mde}. For the last equality, we note that
	\[
		(E_- \hat{\Defo})^\circ = E_- \hat{\Defo} - \mathbf{1}_\delta^+ \frac{\langle \Im M E_- \hat{\Defo} \rangle }{\langle \Im M \rangle } E_+ - \mathbf{1}_\delta^- \frac{\langle M E_- \hat{\Defo} M E_- \rangle }{\langle ME_- M E_- \rangle} E_- = E_- \hat{\Defo}\,,
	\]
	which again follows after a simple computation using the fact that $\hat{\Defo}$ is off-diagonal together with \eqref{eq:Mu}--\eqref{eq:mde}.

	We can now write \eqref{eq:1G basic exp2} for $ A = \mathring{A} = (E_- \hat{\Defo})^\circ = E_- \hat{\Defo}$ and solve the resulting equation for $\langle (G- M) E_- \hat{\Defo} \rangle$. Plugging this back into \eqref{eq:1G basic exp2} yields
	\begin{align}
		\langle (G-M) A \rangle= & - \langle \underline{WG} (\mathcal{X}[A]M)^\circ \rangle + \langle G - M \rangle\langle  (G-M)( \mathcal{X}[A]M)^\circ \rangle + \mathcal{O}_\prec \big( N^{-1}\big) \nonumber                                               \\
		                         & \hspace{1cm}+ \frac{\mathbf{1}_\delta^- \, c_-(\mathcal{X}[A]M)}{1 - \mathbf{1}_\delta^- \, c_-(\mathcal{X}[E_- \hat{\Defo}]M)}  \bigg[  - \langle \underline{WG} (\mathcal{X}[E_- Z]M)^\circ \rangle \label{eq:1G av final} \\
		                         & \hspace{3.5cm}+ \langle G - M \rangle\langle  (G-M)( \mathcal{X}[E_- Z]M)^\circ \rangle   + \mathcal{O}_\prec \big( N^{-1}\big) \bigg] \,. \nonumber
	\end{align}

	Since $\Vert \mathcal{X}[\mathring{A}] \Vert \lesssim 1$ (see Lemma \ref{lem:boundedpert}), the only thing left to check is, that the denominator in \eqref{eq:1G av final} is bounded away from zero.
	\begin{lemma} \label{lem:splittingstable 1av}
		For small enough $\delta > 0$, we have that
		\begin{equation*}
			\big|  1 - \mathbf{1}_\delta^-(w, w)  \, c_-(\mathcal{X}[E_- \hat{\Defo}]M)\big| \gtrsim 1\,.
		\end{equation*}
	\end{lemma}
	\begin{proof} The statement is trivial for $\mathbf{1}_\delta^-(w, w)  = 0$ and we hence focus on the case where $\lambda := \mathbf{1}_\delta^-(w, w) \in (0,1]$. First, we note that $\mathcal{X}[E_- \hat{\Defo}] = E_-\hat{\Defo}$, which follows from the explicit form of $M$ given in \eqref{eq:Mu}--\eqref{eq:mde} using the fact that $\hat{\Defo}$ is purely off-diagonal. Next, we use the MDE \eqref{eq:MDE}, the chiral symmetry \eqref{eq:chiralM}, and Lemma \ref{lem:Mbasic}~(a) to infer
		\begin{equation*}
			1 -  c_-(\mathcal{X}[E_- \hat{\Defo}]M)= 1 - \frac{\langle M E_- \hat{\Defo} M M E_- \rangle}{\langle M E_- M E_- \rangle} = \frac{1}{2} \left[ 1 - \frac{w+m}{m} \langle M^2\rangle\right]\,.
		\end{equation*}
		Now, specialising to $w = \I \eta$ with sufficiently small $\eta$, we find that, to leading order,
		\begin{equation} \label{eq:splitstablerealpart}
			\Re \left[  1 - \frac{\eta+\Im m}{\Im m} \langle M^2\rangle \right] \sim\Re \left[ 1 - \langle M^2 \rangle \right] = 1 - \langle MM^*\rangle + 2 \langle (\Im M)^2\rangle \ge 2 \langle \Im M\rangle^2 \gtrsim 1
		\end{equation}
		by direct computation. Using Lipschitz continuity of this expression in $w$, this principal lower bound on $\Re \big[1 -  c_-(\mathcal{X}[E_- \hat{\Defo}]M)\big]$ of order one persists after a small perturbation of $w$ allowing for a non-zero real part, but as long as $\lambda = \mathbf{1}_\delta^-(w,w)> 0$ for some $\delta > 0$ small enough. Hence, we conclude the lower bound
		\begin{equation}
			\big|  1 - \lambda \, c_-(\mathcal{X}[E_- \hat{\Defo}]M)\big| \ge (1 - \lambda ) 1 + \lambda \Re \big[	1 -  c_-(\mathcal{X}[E_- \hat{\Defo}]M) \big] \gtrsim 1
		\end{equation}
		for the convex combination, by separately considering smaller and larger values of $\lambda \in (0,1]$.
	\end{proof}

	From the expansion \eqref{eq:1G av final} it is apparent, that the main terms for understanding the size of $\langle (G - M) A \rangle $ are the underlined ones, the rest carrying additional $\langle G - M \rangle$-factors, hence they will become negligible errors. In fact, summarizing our investigations, we have shown that
	\begin{equation*}
		\langle (G-M) \mathring{A} \rangle = - \langle \underline{W G \mathring{A}'} \rangle +  \mathcal{O}_\prec\big(\mathcal{E}_1^{\rm av}\big)\,,
	\end{equation*}
	where we used the shorthand notation
	\begin{equation} \label{eq:A'1av def}
		\mathring{A}' := (\mathcal{X}[\mathring{A}]M)^\circ + \frac{\mathbf{1}_\delta^- c_-(\mathcal{X}[\mathring{A}]M)}{1 - \mathbf{1}_\delta^- \, c_-(\mathcal{X}[E_- \hat{\Defo}]M)} (\mathcal{X}[E_- \hat{\Defo}]M)^\circ
	\end{equation}
	in the underlined term. Using the usual averaged local law \eqref{eq:single G} and \eqref{eq:apriori Psi}, we collected all the error terms from \eqref{eq:1G av final} in $\mathcal{E}_1^{\rm av}$, defined in \eqref{eq:E1av}.
\end{proof}

\subsection{Proof of the second master inequality \eqref{eq:masterineq Psi 1 iso}} \label{subsec:proofmaster2}
Let $w_j\in \mathbf{D}_{\ell +1}^{(\epsilon_0, \kappa_0)}$ for $j \in [2]$ be spectral parameters and $A_1$ a regular matrix w.r.t.~the pair of spectral parameters $(w_1, w_2)$ (see Definition \ref{def:regobs}). By conjugation with $E_-$, we will assume w.l.o.g.~that $\Im w_1 > 0$ and $\Im w_2< 0$. Moreover, we use the notations $e_j \equiv \Re w_j$, $\eta_j := |\Im w_j|$ for $j \in [2]$ and define $1 \ge \eta := \min_j |\Im w_j|$. We also assume that \eqref{eq:apriori Psi} holds.
\begin{lemma} {\rm (Representation as full underlined)} \label{lem:underlined2} \\
	For $\Vert \boldsymbol{x} \Vert , \Vert \boldsymbol{y} \Vert \le 1$ and any $(w_1, w_2)$-regular matrix $A_1  = \mathring{A}_1$, we have that
	\begin{equation} \label{eq:underlined Psi 1 iso}
		\big(G_1 \mathring{A}_1 G_2 - M(w_1, \mathring{A}_1, w_2)\big)_{\boldsymbol{x} \boldsymbol{y}} = - \big(\underline{G_1 \mathring{A}_1' WG_2}\big)_{\boldsymbol{x} \boldsymbol{y}} + \mathcal{O}_\prec\big(\mathcal{E}_1^{\rm iso}\big)
	\end{equation}
	for some $(w_1, w_2)$-regular matrix $A_1' = \mathring{A}_1'$, which linearly depends on $A_1 = \mathring{A}_1$ (see \eqref{eq:A'1iso def}). For the error term in \eqref{eq:underlined Psi 1 iso}, we used the shorthand notation
	\begin{equation} \label{eq:E1iso}
		\mathcal{E}_1^{\rm iso} := \frac{1}{\sqrt{N \eta^{2}}}\left(1 + \frac{\psi_1^{\rm av}}{(N \eta)^{1/2}} + \frac{\psi_1^{\rm iso}}{N \eta} \right)\,.
	\end{equation}
\end{lemma}
Note that unlike in Section~\ref{subsec:proofmaster1}, now
in \eqref{eq:underlined Psi 1 iso} the second resolvent $G_2$ was expanded instead of $G_1$
rendering the $W$ factor in the middle of the underlined term.
This prevents the emergence of resolvent chains in the proof of \eqref{eq:masterineq Psi 1 iso},
which are `too long' to be handled within our hierarchical framework
of master inequalities (e.g., a chain involving four resolvents would appear in $\widetilde{\Xi}_1^{\rm iso}$ defined below).

Having this approximate representation of $\big(G_1 \mathring{A}_1 G_2 - M(w_1, \mathring{A}_1, w_2)\big)_{\boldsymbol{x} \boldsymbol{y}}$ as a full underlined term at hand, we turn to the proof of \eqref{eq:masterineq Psi 1 iso} via a (minimalistic) cumulant expansion.
\begin{proof}[Proof of \eqref{eq:masterineq Psi 1 iso}]
	Let $p \in \N$. Then, starting from \eqref{eq:underlined Psi 1 iso} and using the same notations as in the proof of \eqref{eq:masterineq Psi 1 av}, we obtain
	\begin{align}
		         & \mathbf{E} \big|\big(G_1 \mathring{A}_1 G_2-M(w_1, \mathring{A}_1, w_2)\big)_{\boldsymbol{x}\boldsymbol{y}}  \big|^{2p} \label{eq:min exp 1iso}                                                                                                                                                         \\
		\lesssim & \, \mathbf{E} \,  \widetilde{\Xi}_1^{\rm iso} \,  \big\vert  \big(G_1 \mathring{A}_1 G_2-M(\ldots)\big)_{\boldsymbol{x}\boldsymbol{y}} \big\vert^{2p-2} \nonumber                                                                                                                                       \\
		         & + \sum_{|\boldsymbol{l}| + \sum(J \cup J_*) \ge 2} \mathbf{E} \, \Xi_1^{\rm iso}(\boldsymbol{l}, J, J_*) \big\vert  \big(G_1 \mathring{A}_1 G_2-M(\ldots)\big)_{\boldsymbol{x}\boldsymbol{y}}  \big\vert^{2p-1 - |J \cup J_*|} + \mathcal{O}_\prec\big((\mathcal{E}_1^{\rm iso})^{2p}\big) \nonumber\,,
	\end{align}
	where
	\begin{align*}
		\widetilde{\Xi}_1^{\rm iso} := & \frac{\sum_\sigma\left[\left\vert \big(  G_1 \mathring{A}_1'E_\sigma G_1 \mathring{A}_1 G_2\big)_{\boldsymbol{x}\boldsymbol{y}} \big( G_1 E_\sigma G_2 \big)_{\boldsymbol{x}\boldsymbol{y}} \right\vert + \left\vert\big( G_1 \mathring{A}_1' E_\sigma G_2 \big)_{\boldsymbol{x}\boldsymbol{y}} \big(G_1 \mathring{A}_1 G_2 E_\sigma G_2  \big)_{\boldsymbol{x}\boldsymbol{y}}\right\vert\right]}{N} \nonumber \\[2mm]
		                               & + \frac{\sum_\sigma \left[\left\vert \big( G_1 \mathring{A}_1'E_\sigma G_2^*(\mathring{A}_1)^* G_1^*\big)_{\boldsymbol{x}\boldsymbol{x}} \big(G_2^* E_\sigma G_2\big)_{\boldsymbol{y}\boldsymbol{y}} \right\vert + \left|  \big(G_1 \mathring{A}_1'E_\sigma G_1^*\big)_{\boldsymbol{x}\boldsymbol{x}} \big( G_2^*(\mathring{A}_1)^* G_1^* E_\sigma G_2\big)_{\boldsymbol{y}\boldsymbol{y}}\right|\right]}{N}
	\end{align*}
	and	$\Xi_1^{\rm iso}(\boldsymbol{l}, J, J_*)$ is defined via
	\begin{align} \label{eq:Xi 1 iso def}
		\Xi_1^{\rm iso}:=N^{-(|\boldsymbol{l}| + \sum(J \cup J_*) + 1)/2} \sum_{ab} & R_{ab}\big|\partial^{\boldsymbol{l}} \big[(G_1 \mathring{A}'_1)_{\boldsymbol{x}a}\big(G_2\big)_{b \boldsymbol{y}} \big]\big|                                                                                                                                                                         \\
		                                                                            & \times \prod_{\boldsymbol{j} \in J} \big|\partial^{\boldsymbol{j}}  \big(G_1 \mathring{A}_1 G_2\big)_{\boldsymbol{x}\boldsymbol{y}}\big|   \prod_{\boldsymbol{j} \in J_*} \big|\partial^{\boldsymbol{j}} \big(G_2^* (\mathring{A}_1)^* G_2^*\big)_{\boldsymbol{y}\boldsymbol{x}}  \big|\,. \nonumber
	\end{align}
	In the remainder of the proof, we need to analyze the rhs.~of the inequality derived in \eqref{eq:min exp 1iso}. Following the general strategy outlined in Remark \ref{rmk:strategy}, we begin with the second line and study the terms involving $\Xi_1^{\rm iso}$ from \eqref{eq:Xi 1 iso def} afterwards.
	\\[2mm]
	{\bf \underline{Gaussian contribution: third line of \eqref{eq:min exp 1iso}.}} 
	In order to do so, following Remark \ref{rmk:strategy}, we need to analyze in total eight terms, each of which carries one of the summands in the definition of $\widetilde{\Xi}_1^{\rm iso}$ as a factor. Since their treatment is very similar, we focus on the two exemplary terms
	\begin{equation} \label{eq:twoterms1iso}
		\text{(i)} \ \big(  G_1 \mathring{A}_1'E_- G_1 \mathring{A}_1 G_2\big)_{\boldsymbol{x}\boldsymbol{y}} \big( G_1 E_- G_2 \big)_{\boldsymbol{x}\boldsymbol{y}} \,, \quad \text{(ii)}\ \big(G_1 \mathring{A}_1'E_- G_1^*\big)_{\boldsymbol{x}\boldsymbol{x}} \big( G_2^*(\mathring{A}_1)^* G_1 E_- G_2\big)_{\boldsymbol{y}\boldsymbol{y}}\,,
	\end{equation}
	{showcasing the key difficulties. 
	Recall that, in the analysis of the Gaussian term in Section \ref{subsec:proofmaster1} we discussed analogs of the above terms with the choice $\sigma = +$.}
	\\[2mm]
	\underline{\emph{Term (i) in \eqref{eq:twoterms1iso}.}} For the first term, we decompose, similarly to Lemma \ref{lem:regularbasic},
	\begin{equation} \label{eq:Psi 1 iso Gaussian split}
		(\mathring{A}'_1)^{{1,2}} E_- = \big((\mathring{A}_1')^{{1,2}} E_-\big)^{\circ_{1,1}} + \mathcal{O}\big(|e_1 + e_2| + |\eta_1 -  \eta_2|\big) E_+ + \mathcal{O}\big(|e_1 +e_2| + |\eta_1 -  \eta_2|\big) E_-\,.
	\end{equation}
	Inserting this into the first term in \eqref{eq:twoterms1iso} and using Lemma \ref{lem:Mbound}, we find
	\begin{align}\label{ggg}
		\left|  \big(  G_1 \mathring{A}_1'E_- G_1 \mathring{A}_1 G_2\big)_{\boldsymbol{x}\boldsymbol{y}} \right| \prec \frac{1}{\eta}\left(1 + \frac{\psi_2^{\rm iso}}{\sqrt{N \eta}}\right) + \big( |e_1 + e_2| + |\eta_1 - \eta_2|\big) \sum_\sigma \left|  \big(  G_1 E_\sigma G_1 \mathring{A}_1 G_2\big)_{\boldsymbol{x}\boldsymbol{y}} \right|\,.
	\end{align}
	In the last term, we focus on $\sigma = - $, while $\sigma =+$ can be dealt with by Lemma \ref{lem:intrepG^2}. In fact, using \eqref{eq:chiral} and a {resolvent identity \eqref{eq:resolid}}, we obtain
	\begin{align*}
		\left\vert \big(  G_1 E_- G_1 \mathring{A}_1 G_2\big)_{\boldsymbol{x}\boldsymbol{y}} \right\vert = \left\vert \frac{1}{w_1} \big(  [G(-w_1)- G(w_1)] \mathring{A}_1^{w_1, w_2} G(w_2)\big)_{(E_-\boldsymbol{x})\boldsymbol{y}} \right\vert \prec \frac{1}{\eta^2} \left( 1 + \frac{\psi_1^{\rm iso}}{\sqrt{N \eta}}\right)\,,
	\end{align*}
	where in the last step we used Lemma \ref{lem:Mbound} and the trivial approximation
	$$\mathring{A}_1^{-w_1, w_2} = \mathring{A}_1^{w_1,w_2} + \mathcal{O}(1) E_+ + \mathcal{O}(1)E_-\,.$$

	For the second factor in the first term in \eqref{eq:twoterms1iso}, we use \eqref{eq:chiral} and employ the integral representation from Lemma~\ref{lem:intrepG^2} with
	\[
		\tau  = +\,, \quad J =  \mathbf{B}_{ \ell \kappa_0}\,, \quad \text{and} \quad \tilde{\eta} = \frac{\ell}{\ell +1} \eta\,,
	\]
	for which we recall that $w_j \in \mathbf{D}_{\ell +1}^{(\epsilon_0, \kappa_0)}$, i.e.~in particular $\eta \ge (\ell +1) N^{-1+\epsilon_0}$ and hence $\tilde{\eta} \ge \ell N^{-1+\epsilon_0}$. After splitting the contour integral and estimating the contribution as described around \eqref{eq:contourdecomp}, we find, with the aid of Lemma \ref{lem:Mbound} and absorbing logarithmic corrections into `$\prec$', that
	\begin{align}
		\left| \big(G_1 E_- G_2\big)_{\boldsymbol{x}\boldsymbol{y}} \right| & \prec 1 + \int_{ \mathbf{B}_{ \ell \kappa_0}} \frac{\big| \big(G(x + \I \tilde{\eta})\big)_{\boldsymbol{x}(E_- \boldsymbol{y})} \big|}{\left| \big( x - e_1 - \I (\eta_1 - \tilde{\eta})\big) \, \big(x + e_2 - \I (\eta_2 - \tilde{\eta})\big) \right|} \D x \nonumber \\
		                                                                    & \prec 1 + \frac{1}{|e_1 + e_2| + \eta_1  + \eta_2} \label{eq:G1E-G2 iso}
	\end{align}
	where in the last step we used the usual single resolvent local law from Theorem \ref{thm:singleG}.
	Notice the key cancellation of the $|e_1+e_2|$ factor in~\eqref{ggg} and~\eqref{eq:G1E-G2 iso}.
	Collecting all the estimates, we have shown that
	\begin{equation} \label{eq:Psi1isofirsttermfinal}
		\big| \eqref{eq:twoterms1iso} \, \text{(i)} \, \big| \prec \frac{1}{\eta^2}\left( 1 + \frac{\psi_1^{\rm iso}}{\sqrt{N \eta}} + \frac{\psi_2^{\rm iso}}{\sqrt{N \eta}} \right) \,.
	\end{equation}
	\underline{\emph{Term (ii) in \eqref{eq:twoterms1iso}.}} In the first factor in the second term in \eqref{eq:twoterms1iso}, we again employ the decomposition \eqref{eq:Psi 1 iso Gaussian split} to find
	\begin{equation} \label{eq:Psi1iso secondterm1}
		\big| \big(G_1 \mathring{A}_1'E_- G_1^*\big)_{\boldsymbol{x}\boldsymbol{x}} \big| \prec \frac{1}{\eta^{1/2}} \left( 1 + \frac{\psi_1^{\rm iso}}{\sqrt{N \eta}}\right) + \frac{|e_1 + e_2| + |\eta_1 - \eta_2|}{\eta}
	\end{equation}
	with the aid of Theorem \ref{thm:singleG} and Lemma \ref{lem:Mbound} as well as a {resolvent identity \eqref{eq:resolid}} and Lemma \ref{lem:intrepG^2} for the $E_+$ and $E_-$ in \eqref{eq:Psi 1 iso Gaussian split}, respectively.

	In the second factor, similarly to \eqref{eq:G1E-G2 iso} above, we use Lemma \ref{lem:intrepG^2} together with the decomposition\footnote{\label{ftn:Osigma} The notation $\sum_\sigma \mathcal{O}_\sigma (\alpha) E_\sigma $ means that, for each $\sigma$, the term $\mathcal{O}_\sigma(\alpha)$ is a scalar function $g_\sigma(\alpha)$ of order $\mathcal{O}(\alpha)$.}
	\[
		(\mathring{A}_1^{{w}_1, {w}_2})^* = \mathring{(A_1^*)}^{\bar{w}_2, \bar{w}_1} = \mathring{(A_1^*)}^{\bar{w}_2, {w}_1} = \mathring{(A_1^*)}^{\bar{w}_2, x + \I \tilde{\eta}} + \sum_\sigma \mathcal{O}_\sigma(|x - e_1| + |\eta_1 - \tilde{\eta}|) E_\sigma
	\]
	from Lemma \ref{lem:regularbasic} for arbitrary $x$ to find
	\begin{align}
		\left| \big( G_2^*(\mathring{A}_1)^* G_1 E_- G_2\big)_{\boldsymbol{y}\boldsymbol{y}} \right| \prec & \frac{1}{\eta^{1/2}}  \left(1 + \frac{\psi_1^{\rm iso}}{\sqrt{N \eta}} \right) \nonumber                                                                                                                                                                                                                                            \\
		    & +	\int_{ \mathbf{B}_{ \ell \kappa_0}}  \frac{\big| \big(G(\bar{w}_2) \mathring{(A_1^*)}^{\bar{w}_2, x + \I \tilde{\eta}} G(x + \I \tilde{\eta}) \big)_{\boldsymbol{y}(E_- \boldsymbol{y})} \big|}{\left| \big( x - e_1 - \I (\eta_1 - \tilde{\eta})\big) \, \big(x + e_2 - \I (\eta_2 - \tilde{\eta})\big) \right|} \D x  \nonumber \\
		    & + \int_{ \mathbf{B}_{ \ell \kappa_0}}  \frac{\sum_\sigma \big| \big(G(\bar{w}_2) E_\sigma G(x + \I \tilde{\eta}) \big)_{\boldsymbol{y}(E_- \boldsymbol{y})} \big|}{\left| x + e_2 - \I (\eta_2 - \tilde{\eta}) \right|} \D x \label{eq:Psi1iso secondterm2}   \\
		\prec & \frac{1}{\eta^{1/2}} \left(1 + \frac{\psi_1^{\rm iso}}{\sqrt{N \eta}}\right) \left(1 + \frac{1}{|e_1 + e_2| + \eta_1 + \eta_2}\right) + \frac{1}{\eta} \nonumber\,.
	\end{align}
	Now, combining \eqref{eq:Psi1iso secondterm1} and \eqref{eq:Psi1iso secondterm2}, we obtain
	\begin{equation} \label{eq:Psi1isosecondtermfinal}
		\big| \eqref{eq:twoterms1iso} \, \text{(ii)} \, \big| \prec \frac{1}{\eta^2} \left(  1 + \frac{\psi_1^{\rm iso}}{\sqrt{N \eta}}\right)^2\,.
	\end{equation}
	\\[1mm]
	This finishes the estimate for the Gaussian contribution from the third line of \eqref{eq:min exp 1iso}, for which we have shown that
	\begin{equation} \label{eq:Gaussian1iso}
		\widetilde{\Xi}_1^{\rm iso} \prec \frac{1}{N \eta^2} \left( 1 + \frac{(\psi_1^{\rm iso})^2}{N \eta} + \frac{\psi_2^{\rm iso}}{\sqrt{N \eta}} \right)
	\end{equation}
	as easily follows by combining \eqref{eq:Psi1isofirsttermfinal} with \eqref{eq:Psi1isosecondtermfinal} and using a Schwarz inequality.

	We are now left with the terms from the last line \eqref{eq:min exp 1iso} resulting from higher order cumulants.
	\\[2mm]
	{\bf \underline{Higher order cumulants and conclusion.}} The estimate stemming from higher order cumulants is given in \eqref{eq:higherorder1iso}. Then, plugging \eqref{eq:Gaussian1iso} and \eqref{eq:higherorder1iso} into \eqref{eq:min exp 1iso}, we find, similarly to Section \ref{subsec:proofmaster1}, that
	\begin{equation*}
		\Psi_1^{\rm iso} \prec 1 + \frac{\psi_1^{\rm iso}}{N \eta} + \frac{\psi_1^{\rm iso} + \psi_1^{\rm av} }{(N \eta)^{1/2}} + \frac{(\psi_2^{\rm iso})^{1/2}}{(N \eta)^{1/4}} + \frac{(\psi_2^{\rm iso})^{1/4}}{(N \eta)^{1/8}}\,.
	\end{equation*}
	The bound given in Proposition \ref{prop:master} is an immediate consequence after a trivial Young inequality.
\end{proof}
It remains to give a proof of Lemma \ref{lem:underlined2}. This is much more involved than for the previous underlined Lemma \ref{lem:underlined1}. The proof of Lemma \ref{lem:underlined1} crucially used that the orthogonality $\langle \Im M A \rangle = 0$ is (almost) preserved under the operation $A \mapsto \mathcal{X}[A]M$ (see Lemma \ref{lem:Psi1av stable}). This is simply not available here, since we deal with two spectral parameters $w_1, w_2$.
\begin{proof}[Proof of Lemma \ref{lem:underlined2}]
	We denote $A_1 \equiv \mathring{A}_1$, except we wish to emphasise $A_1$ being regular. Just as in Section~\ref{subsec:proofmaster1}, we start with
	\begin{equation*}
		G_2 = M_2 - M_2 \underline{W G_2} + M_2 \mathcal{S}[G_2-M_2] G_2\,,
	\end{equation*}
	such that we get
	\begin{equation*}
		G_1 \tilde{A}_1 G_2 = G_1 \tilde{A}_1 M_2 - G_1 \tilde{A}_1 M_2 \underline{W G_2} + G_1 \tilde{A}_1 M_2 \mathcal{S}[G_2-M_2] G_2
	\end{equation*}
	for  $\tilde{A}_1 = \mathcal{X}_{12}[A_1]$ and $A_1 = \mathring{A}_1$ (note that $\Vert\mathcal{X}_{12}[\mathring{A}_1] \Vert \lesssim 1$ by Lemma \ref{lem:boundedpert}), where we introduced the linear operator
	\begin{equation} \label{eq:X12def}
		\mathcal{X}_{12}[B] := \big(1 - \mathcal{S}[M_1 \, \cdot \, M_2]\big)^{-1}[B]\quad \text{for} \quad B \in \C^{2N \times 2N}\,.
	\end{equation}
	Extending the underline to the whole product, we obtain
	\begin{align*}
		G_1 \tilde{A}_1 G_2 = & M_1 \tilde{A}_1 M_2 + (G_1 - M_1) \tilde{A}_1 M_2 - \underline{G_1 \tilde{A}_1 M_2 W G_2}     \\
		                      & + G_1 \tilde{A}_1 M_2 \mathcal{S}[G_2 - M_2]G_2 + G_1 \mathcal{S}[G_1 \tilde{A}_1 M_2] G_2\,,
	\end{align*}
	from which we conclude that
	\begin{align*}
		G_1 \big(\tilde{A}_1 - \mathcal{S}[M_1 \tilde{A}_1 M_2] \big) G_2 = & \ M_1 \tilde{A}_1 M_2 + (G_1 - M_1) \tilde{A}_1 M_2 - \underline{G_1 \tilde{A}_1 M_2 W G_2}                 \\
		                                                                    & + G_1 \tilde{A}_1 M_2 \mathcal{S}[G_2 - M_2]G_2 + G_1 \mathcal{S}[(G_1-M_1) \tilde{A}_1 M_2] G_2  \nonumber
	\end{align*}
	and thus
	\begin{align} \label{eq:Psi 1 iso before decomp}
		G_1 A_1 G_2 = & M_1 \mathcal{X}_{12}[A_1] M_2 + (G_1 - M_1) \mathcal{X}_{12}[A_1] M_2 - \underline{G_1 \mathcal{X}_{12}[A_1]M_2 W G_2}             \\
		              & + G_1 \mathcal{X}_{12}[A_1] M_2 \mathcal{S}[G_2 - M_2]G_2 + G_1 \mathcal{S}[(G_1-M_1) \mathcal{X}_{12}[A_1] M_2] G_2 \,. \nonumber
	\end{align}
	We note that $\Vert \mathcal{X}_{12}[\mathring{A}_1] \Vert \lesssim1 $ by means of Lemma \ref{lem:boundedpert}.

	Then, we need to further decompose $\mathcal{X}_{12}[A_1] M_2$  in the last three terms in \eqref{eq:Psi 1 iso before decomp} as
	\begin{equation} \label{eq:decomp Psi 1 iso}
		\mathcal{X}_{12}[A_1] M_2 = \big(\mathcal{X}_{12}[A_1] M_2\big)^\circ + \sum_{\sigma} \mathbf{1}_\delta^\sigma\,
		c_{\sigma}(\mathcal{X}_{12}[A_1]M_2) E_\sigma\,,
	\end{equation}
	where 
	we suppressed the spectral parameters (and the relative sign of their imaginary parts, which has been fixed by $\Im w_1 > 0$ and $\Im w_2 <0$) in the notation for the linear functionals $c_\sigma(\cdot)$ on $\C^{2N \times 2N}$ defined as
	\begin{equation} \label{eq:csigmadef Psi 1 iso}
		c_+(B) := \frac{\langle M_1 B M_{2}  \rangle}{\langle M_1  M_{2} \rangle} \qquad \text{and} \qquad c_{- }(B) := \frac{\langle M_1 B M^*_{2} E_{- } \rangle}{\langle M_1 E_{- } M^*_{2} E_{- } \rangle}\,.
	\end{equation}

	Plugging \eqref{eq:decomp Psi 1 iso} into \eqref{eq:Psi 1 iso before decomp} we find $G_1 A_1 G_2$ to equal
	\begin{align} \label{eq:Psi 1 iso after decomp}
		 & M_1 \mathcal{X}_{12}[A_1] M_2 + (G_1 - M_1) \mathcal{X}_{12}[A_1] M_2 - \underline{G_1 \big(\mathcal{X}_{12}[A_1] M_2 \big)^\circ W G_2}                                                                                           \\
		 & + G_1 \big(\mathcal{X}_{12}[A_1] M_2 \big)^\circ \mathcal{S}[G_2 - M_2]G_2 + G_1 \mathcal{S}\big[(G_1-M_1) \big(\mathcal{X}_{12}[A_1] M_2 \big)^\circ\big] G_2  \nonumber                                                          \\
		 & + \sum_{\sigma} \mathbf{1}_\delta^\sigma \, c_\sigma(\mathcal{X}_{12}[A_1]M_2) \left[ - \underline{G_1 E_\sigma W G_2} + G_1 E_\sigma \mathcal{S}[G_2 - M_2 ] G_2 + G_1 \mathcal{S}[(G_1 - M_1) E_\sigma] G_2 \right] \nonumber\,.
	\end{align}
	Recall that the regular component is defined w.r.t.~the pair of spectral parameters $(w_1, w_2)$. In particular, $\big(\mathcal{X}_{12}[A_1] M_2 \big)^\circ = \big(\mathcal{X}_{12}[A_1] M_2 \big)^{\circ_{1,2}}$ in the last term in the second line of \eqref{eq:Psi 1 iso after decomp} is \emph{not} regular as defined via the conditions with one resolvent \eqref{eq:1G traceless}.

	In the last line of \eqref{eq:Psi 1 iso after decomp} we now undo the underline and find the bracket $\big[ \cdots\big]$ to equal (the negative of)
	\begin{align*}
		  & G_1 E_\sigma W G_2 + G_1 E_\sigma \mathcal{S}[M_2] G_2 + G_1 \mathcal{S}[M_1 E_\sigma] G_2                               \\
		= & G_1 E_\sigma + G_1 \big(E_\sigma (w_2 - \hat{\Defo} +  \mathcal{S}[M_2]) + \mathcal{S}[M_1 E_\sigma]\big) G_2            \\
		= & G_1 E_\sigma -  G_1 \big(E_\sigma M_2^{-1} - \mathcal{S}[M_1 E_\sigma]\big) G_2 =: G_1 E_\sigma - G_1 \Phi_\sigma G_2\,,
	\end{align*}
	where we used $WG_2 = E_+ + w_2 G_2 - \hat{\Defo} G_2$ in the first step and the MDE \eqref{eq:MDE} in the second step. Moreover, we introduced the shorthand notation
	\begin{equation} \label{eq:Psidef1iso}
		\Phi_\sigma := E_\sigma \frac{1}{M_2} - \mathcal{S}[M_1E_\sigma]\,.
	\end{equation}

	From the expansion \eqref{eq:Psi 1 iso after decomp} it is apparent (and it can also be checked by hand using the explicit form of \eqref{eq:Psidef1iso}) that
	\begin{equation*}
		M_1 E_\sigma = M_1 \big(E_\sigma M_2^{-1}\big)M_2 = M_1 \mathcal{X}_{12}[\Phi_\sigma] M_2  = M(w_1, \Phi_\sigma, w_2)\,,
	\end{equation*}
	where in the last step we used \eqref{eq:Mexample}. This finally yields that $G_1 A_1 G_2$ equals
	\begin{align} \label{eq:Psi 1 iso after decomp2}
		 & M(w_1, A_1, w_2) + (G_1 - M_1) \mathcal{X}_{12}[A_1] M_2 - \underline{G_1 \big(\mathcal{X}_{12}[A_1] M_2 \big)^\circ W G_2}                                                             \\
		 & + G_1 \big(\mathcal{X}_{12}[A_1] M_2 \big)^\circ \mathcal{S}[G_2 - M_2]G_2 + G_1 \mathcal{S}\big[(G_1-M_1) \big(\mathcal{X}_{12}[A_1] M_2 \big)^\circ\big] G_2  \nonumber               \\
		 & + \sum_{\sigma} \mathbf{1}_\delta^\sigma \, c_\sigma(\mathcal{X}_{12}[A_1]M_2) \left[ - (G_1-M_1)E_\sigma + \big(G_1 \Phi_\sigma G_2- M(w_1,\Phi_\sigma, w_2)\big) \right] \nonumber\,.
	\end{align}

	The last term in the last line of \eqref{eq:Psi 1 iso after decomp2} requires further decomposition of $\Phi_\sigma$ from \eqref{eq:Psidef1iso} (completely analogous to \eqref{eq:decomp Psi 1 iso} and \eqref{eq:csigmadef Psi 1 iso}) as
	\begin{equation*}
		\Phi_\sigma = \mathring{\Phi}_\sigma + \sum_{\tau} \mathbf{1}_\delta^\tau \, c_\tau(\Phi_\sigma) E_\tau\,.
	\end{equation*}
	Using the explicit form of $\Phi_\sigma$, we further observe that
	\begin{equation} \label{eq:orthogonality Psi 1 iso}
		c_\tau(\Phi_\sigma) \sim \delta_{\sigma, \tau} \quad \text{and} \quad c_\tau(\mathcal{X}_{12}[\Phi_\sigma] M_2) \sim \delta_{\sigma, \tau}\,.
	\end{equation}
	Therefore, by means of the first relation in \eqref{eq:orthogonality Psi 1 iso}, the expansion \eqref{eq:Psi 1 iso after decomp2} can be carried out further as
	\begin{align} \label{eq:Psi 1 iso before stab}
		 & M(w_1, A_1, w_2) + (G_1 - M_1) \mathcal{X}_{12}[A_1] M_2 - \underline{G_1 \big(\mathcal{X}_{12}[A_1] M_2 \big)^\circ W G_2}                                                                          \\
		 & + G_1 \big(\mathcal{X}_{12}[A_1] M_2 \big)^\circ \mathcal{S}[G_2 - M_2]G_2 + G_1 \mathcal{S}\big[(G_1-M_1) \big(\mathcal{X}_{12}[A_1] M_2 \big)^\circ\big] G_2  \nonumber                            \\
		 & + \sum_{\sigma} \mathbf{1}_\delta^\sigma \, c_\sigma(\mathcal{X}_{12}[A_1]M_2) \  \bigg[ - (G_1-M_1)E_\sigma + \big(G_1 \mathring{\Phi}_\sigma G_2- M(w_1,\mathring{\Phi}_\sigma, w_2)\big)\nonumber \\
		 & \hspace{5cm} + c_\sigma(\Phi_\sigma)\big( G_1E_\sigma G_2 - M(w_1, E_\sigma, w_2) \big) \bigg] \nonumber\,.
	\end{align}

	Next, we write \eqref{eq:Psi 1 iso before stab} for both, $A_1 = \mathring{A}_1 = \mathring{\Phi}_+$ and $A_1= \mathring{A}_1 = \mathring{\Phi}_-$, and solve the two resulting linear equations for $G_1 \mathring{\Phi}_\pm G_2 - M(w_1, \mathring{\Phi}_\pm, w_2)$. Observe that by means of the second relation in \eqref{eq:orthogonality Psi 1 iso} the original system of linear equations boils down to two separate ones. Thus, plugging the solutions for $G_1 \mathring{\Phi}_\pm G_2 - M(w_1, \mathring{\Phi}_\pm, w_2)$ back into \eqref{eq:Psi 1 iso before stab} we arrive at
	\begin{align} \label{eq:Psi 1 iso final}
		G_1 A_1 G_2 = & M(w_1, A_1, w_2) + (G_1 - M_1) \mathcal{X}_{12}[A_1] M_2 - \underline{G_1 \big(\mathcal{X}_{12}[A_1] M_2 \big)^\circ W G_2}                                                                                                                                                                                                                     \\
		              & + G_1 \big(\mathcal{X}_{12}[A_1] M_2 \big)^\circ \mathcal{S}[G_2 - M_2]G_2 + G_1 \mathcal{S}\big[(G_1-M_1) \big(\mathcal{X}_{12}[A_1] M_2 \big)^\circ\big] G_2  \nonumber                                                                                                                                                                       \\
		              & + \sum_{\sigma} \frac{\mathbf{1}_\delta^\sigma \, c_\sigma(\mathcal{X}_{12}[A_1]M_2)}{1 - \mathbf{1}_\delta^\sigma \, c_\sigma(\mathcal{X}_{12}[\mathring{\Phi}_\sigma]M_2)} \  \bigg[ (G_1 - M_1) \mathcal{X}_{12}[\mathring{\Phi}_\sigma] M_2 - \underline{G_1 \big(\mathcal{X}_{12}[\mathring{\Phi}_\sigma] M_2 \big)^\circ W G_2} \nonumber \\
		              & + G_1 \big(\mathcal{X}_{12}[\mathring{\Phi}_\sigma] M_2 \big)^\circ \mathcal{S}[G_2 - M_2]G_2 + G_1 \mathcal{S}\big[(G_1-M_1) \big(\mathcal{X}_{12}[\mathring{\Phi}_\sigma] M_2 \big)^\circ\big] G_2 \nonumber                                                                                                                                  \\
		              & - (G_1-M_1)E_\sigma  + c_\sigma(\Phi_\sigma)\big( G_1E_\sigma G_2 - M(w_1, E_\sigma, w_2) \big) \bigg]\,.  \nonumber
	\end{align}

	We now need to check that the denominators in \eqref{eq:Psi 1 iso final} are bounded away from zero.
	\begin{lemma} \label{lem:splitting stable 1iso} For small enough $\delta > 0$, we have that
		\begin{equation*}
			\left| 1 - \mathbf{1}_\delta^\sigma(w_1, w_2) \, c_\sigma(\mathcal{X}_{12}[\mathring{\Phi}_\sigma]M_2) \right| \gtrsim 1 \qquad \text{for} \quad \sigma = \pm\,.
		\end{equation*}
	\end{lemma}
	\begin{proof} The statements are trivial for $\mathbf{1}_\delta^\sigma(w_1, w_2)  = 0$ and we hence focus on cases where $\lambda^\sigma:= \mathbf{1}_\delta^\sigma(w_1, w_2) \in (0,1]$. First, we compute
		\begin{align} \label{eq:denom Psi 1 iso}
			1 - c_+(\mathcal{X}_{12}[\mathring{\Phi}_+] M_2) & = \langle M_1 \rangle \frac{\langle M_1 M_2 M_2 \rangle}{\langle M_1 M_2 \rangle^2}\qquad \text{and}                                                                                                                                             \\
			1 - c_-(\mathcal{X}_{12}[\mathring{\Phi}_-] M_2) & = \frac{\langle M_1 E_- M_2^* M_2^{-1} E_- \rangle + \langle M_1 \rangle \langle M_1 E_- M_2^* E_- \rangle }{1 + \langle M_1 E_- M_2 E_- \rangle } \ \frac{\langle M_1 E_- M_2 M_2^* E_- \rangle}{\langle M_1 E_- M_2^* E_- \rangle^2} \nonumber
		\end{align}
		for arbitrary spectral parameters $w_1, w_2$. Recall that we assumed the two spectral parameters to be on different halfplanes, i.e.~$\mathfrak{s}_1 = - \sgn(\Im w_1 \Im w_2) = +$, hence we shall specialise (i) the first expression in \eqref{eq:denom Psi 1 iso} to $w_2 = \bar{w}_1$ and (ii) the second expression in \eqref{eq:denom Psi 1 iso} to $w_2 = - w_1$.

		In this case, for the first expression in \eqref{eq:denom Psi 1 iso}, using Lemma~\ref{lem:Mbasic} and $\Im M_1 \Im w_1 > 0$, we obtain
		\begin{align} \label{eq:splitstableabsval}
			\big\vert1 - c_+(\mathcal{X}_{12}[\mathring{\Phi}_+] M_2) \big\vert = \big| \langle M_1 \rangle \frac{\langle \Im M_1 M_1^* \rangle}{\langle\Im M_1\rangle^2} (\langle \Im M_1 \rangle + \Im w_1) \big| \ge \langle \Im M_1 \rangle^2 \gtrsim 1 \qquad 
		\end{align}
		in the bulk of the spectrum. This principal lower bound of order one persists after a small perturbation of $w_2$ around the special case $w_2 = \bar{w}_1$, but as long as $\lambda^+ = 1$ (for some $\delta > 0$ small enough), which proves the claim for $\sigma = +$ and $\lambda^+ = 1$. A further direct computation by estimating real and imaginary part of $1 - c_+(\mathcal{X}_{12}[\mathring{\Phi}_+] M_2)$ instead of its absolute value in \eqref{eq:splitstableabsval}, similarly to \eqref{eq:splitstablerealpart} shows that also the convex combination
		\begin{equation*}
			(1 - \lambda^+) 1 + \lambda^+ \big[ 1 - c_+(\mathcal{X}_{12}[\mathring{\Phi}_+] M_2)\big]
		\end{equation*}
		is bounded away from zero (in absolute value), by separately considering small and large values of $\lambda^+ \in (0,1)$.
		For the second expression in \eqref{eq:denom Psi 1 iso}, the argument is similar and hence omitted.
	\end{proof}

	Next, we take the scalar product of \eqref{eq:Psi 1 iso final} with two deterministic vectors $\boldsymbol{x}, \boldsymbol{y}$ satisfying $\Vert \boldsymbol{x} \Vert, \Vert \boldsymbol{y} \Vert \le 1$. In the resulting expression,
	there are two particular terms, namely the ones of the form
	\begin{align}
		 & \big(G_1 \mathcal{S}\big[(G_1-M_1) \mathring{A}_1^{{1,2}}\big] G_2\big)_{\boldsymbol{x}\boldsymbol{y}} \quad \text{and} \label{eq:Psi 1 iso firstcrit} \\ &c_\sigma(\mathcal{X}_{12}[\mathring{A}_1^{{1,2}}] M_2) c_\sigma(\Phi_\sigma) \big(G_1 E_\sigma G_2 - M(w_1, E_\sigma, w_2)\big)_{\boldsymbol{x}\boldsymbol{y}}\,, \label{eq:Psi 1 iso secondcrit}
	\end{align}
	whose direct (naive) estimates are $1/(N \eta^2)$ and $1/\eta$, respectively, and thus do not match the target size. Hence, they have to be discussed in more detail. In our notation, we emphasised that the regularisation is defined w.r.t.~the spectral parameters $(w_1, w_2)$, i.e., in particular, $A_1^\circ = A_1^{\circ_{1,2}}$.
	\\[2mm]
	\emph{\underline{Estimating \eqref{eq:Psi 1 iso firstcrit}.}} For the term \eqref{eq:Psi 1 iso firstcrit}, we expand
	\begin{equation} \label{eq:firstcrit 1}
		\big(G_1 \mathcal{S}\big[(G_1-M_1) \mathring{A}_1^{{1,2}}\big] G_2\big)_{\boldsymbol{x}\boldsymbol{y}} = \sum_{\sigma} \sigma \langle (G_1 - M_1) \mathring{A}_1^{{1,2}} E_\sigma \rangle \big(G_1 E_\sigma G_2\big)_{\boldsymbol{x}\boldsymbol{y}}
	\end{equation}
	and observe that, by definition of $\cdot^\circ$ in \eqref{eq:circ def}, we have, similarly to Lemma \ref{lem:regularbasic} (see also \eqref{eq:Psi 1 iso Gaussian split}),
	\begin{equation} \label{eq:firstcrit 2}
		\mathring{A}_1^{{1,2}} E_\sigma = \big(\mathring{A}_1^{{1,2}} E_\sigma\big)^{\circ_{1,1}} + \mathcal{O}\big(|e_1 - \sigma e_2| + |\eta_1 -  \eta_2|\big) E_+ + \mathcal{O}\big(|e_1 - \sigma e_2| + |\eta_1 -  \eta_2|\big) E_-\,.
	\end{equation}

	Now, in the second term in \eqref{eq:firstcrit 1} for $\sigma = +$ and $E_\sigma = E_+$, we use a {resolvent identity \eqref{eq:resolid}} and the usual isotropic local law \eqref{eq:single G} to estimate it as
	\begin{equation} \label{eq:Psi 1 iso G1G2}
		\big\vert \big(G_1  G_2\big)_{\boldsymbol{x}\boldsymbol{y}} \big\vert \prec 1+ \frac{1}{|e_1 - e_2| +  \eta_1 + \eta_2}\,.
	\end{equation}

	Furthermore, in the second term in \eqref{eq:intrepG^2} for $\sigma = - $ and $E_\sigma = E_-$, we employ the integral representation from Lemma \ref{lem:intrepG^2} in combination with the
	usual isotropic local law \eqref{eq:single G} (see also \eqref{eq:G1E-G2 iso}) to infer
	\begin{equation} \label{eq:Psi 1 iso G1E-G2}
		\big\vert \big(G_1  E_- G_2\big)_{\boldsymbol{x}\boldsymbol{y}} \big\vert \prec 1 + \frac{1}{|e_1 + e_2| +  \eta_1 + \eta_2}\,.
	\end{equation}
	Combining \eqref{eq:Psi 1 iso G1G2} and \eqref{eq:Psi 1 iso G1E-G2} with the decomposition \eqref{eq:firstcrit 2} and the usual averaged local law \eqref{eq:single G}, we find that \eqref{eq:firstcrit 1} can be bounded by
	\begin{equation*}
		\sum_{\sigma} \left( \left\vert \langle (G_1 - M_1) \big(\mathring{A}_1^{{1,2}} E_\sigma\big)^{\circ_{1,1}}\rangle  \right\vert + \frac{|e_1 - \sigma e_2| + \vert \eta_1 - \eta_2 \vert}{N\eta_1}\right) \left(1 + \frac{1}{|e_1 - \sigma e_2| + \eta_1 + \eta_2 }\right)\,.
	\end{equation*}
	Using the definition of $\Psi_1^{\rm av}$ in \eqref{eq:Psi avk} and the apriori bound $\Psi_1^{\rm av} \prec \psi_1^{\rm av}$, this immediately implies the estimate
	\begin{equation} \label{eq:Psi 1 iso firstcrit final}
		\vert  \eqref{eq:Psi 1 iso firstcrit} \vert \prec	\frac{1}{N \eta} + \frac{1}{\sqrt{N}\eta} \frac{\psi_1^{\rm av}}{(N \eta)^{1/2}}\,.
	\end{equation}
	\emph{\underline{Estimating \eqref{eq:Psi 1 iso secondcrit}.}} For the term \eqref{eq:Psi 1 iso secondcrit}, we first note that the two prefactors $c_\sigma(\mathcal{X}_{12}[A_1^{\circ_{1,2}}] M_2) $ and $c_\sigma(\Phi_\sigma)$ are bounded. However, in each of the two cases $\sigma = \pm $, the bound on \emph{one} of the prefactors needs to be improved: In the first case, $\sigma = +$, we use \eqref{eq:saturation} and compute
	\begin{equation*}
		c_+(\Phi_+) = \frac{\langle M_1 \rangle \big( 1 - \langle M_1 M_2 \rangle  \big)}{\langle M_1 M_2 \rangle } = \mathcal{O} \big( |e_1 - e_2| + \eta_1 + \eta_2 \big)
	\end{equation*}
	from \eqref{eq:csigmadef Psi 1 iso} and \eqref{eq:Psidef1iso}. Combining this with the bound
	\begin{equation*}
		\left\vert \big(G_1 G_2 - M(w_1, E_+, w_2)\big)_{\boldsymbol{x}\boldsymbol{y}} \right\vert \prec \left(\frac{1}{\sqrt{N \eta_1}} + \frac{1}{\sqrt{N \eta_2}}\right) \cdot \frac{1}{|e_1 - e_2| +  \eta_1 + \eta_2}
	\end{equation*}
	which is obtained completely analogous to \eqref{eq:Psi 1 iso G1G2}, we conclude that \eqref{eq:Psi 1 iso secondcrit} for $\sigma = +$ can be estimated by $1/\sqrt{N \eta}$ (recall $\eta := \min\{ \eta_1, \eta_2 \}$). Similarly, in the second case, $\sigma = -$, we perform a computation similar to the one leading to \eqref{eq:stability} and use \eqref{eq:saturation} in order to obtain that $c_-(\mathcal{X}_{12}[\mathring{A}_1^{{1,2}}] M_2)$ equals
	\begin{equation*}
		\frac{\I}{2} \frac{\langle M_1 \mathring{A}_1^{{1,2}} M_2^* E_- \rangle }{\langle M_1 E_- M_2^* E_- \rangle} + \frac{1}{2 \I} \frac{\langle M_1 \mathring{A}_1^{{1,2}} M_2 E_- \rangle }{\langle M_1 E_- M_2^* E_- \rangle} \frac{1 + \langle M_1 E_- M_2^* E_- \rangle }{1 + \langle M_1 E_- M_2 E_- \rangle } =  \mathcal{O}\big(|e_1 + e_2| + \eta_1 + \eta_2\big)
	\end{equation*}
	Combining this with the bound
	\begin{equation*}
		\left\vert \big(G_1 E_- G_2 - M(w_1, E_-, w_2)\big)_{\boldsymbol{x}\boldsymbol{y}} \right\vert \prec \frac{1}{\sqrt{N \eta}} \cdot \frac{1}{|e_1 + e_2| +  \eta_1 + \eta_2}
	\end{equation*}
	which is obtained completely analogous to \eqref{eq:Psi 1 iso G1E-G2}, we conclude that \eqref{eq:Psi 1 iso secondcrit} can be estimated by $1/\sqrt{N \eta}$ -- now in both cases $\sigma = \pm$.
	\\[2mm]
	\underline{\emph{Conclusion}.}
	Summarizing our investigations, 
	we have shown that
	\begin{equation*}
		\big(G_1 \mathring{A}_1 G_2 - M(w_1, \mathring{A}_1, w_2)\big)_{\boldsymbol{x} \boldsymbol{y}} = - \big(\underline{G_1 \mathring{A}_1' W G_2}\big)_{\boldsymbol{x} \boldsymbol{y}} + \mathcal{O}_\prec\big(\mathcal{E}_1^{\rm iso}\big)\,,
	\end{equation*}
	where we used the shorthand notation
	\begin{equation} \label{eq:A'1iso def}
		\mathring{A}_1' :=   \big(\mathcal{X}_{12}[\mathring{A}_1] M_2 \big)^\circ +   \sum_{\sigma} \frac{\mathbf{1}_\delta^\sigma c_\sigma(\mathcal{X}_{12}[\mathring{A}_1]M_2)}{1 - \mathbf{1}_\delta^\sigma c_\sigma(\mathcal{X}_{12}[\mathring{\Phi}_\sigma]M_2)} \big(\mathcal{X}_{12}[\mathring{\Phi}_\sigma] M_2 \big)^\circ
	\end{equation}
	in the underlined term. Combining \eqref{eq:Psi 1 iso firstcrit final} and the bound on \eqref{eq:Psi 1 iso secondcrit} established above  with the usual single resolvent local laws \eqref{eq:single G} and the bounds on deterministic approximations in Lemma \ref{lem:Mbound}, we collected all the error terms from \eqref{eq:Psi 1 iso final} in \eqref{eq:E1iso}.
\end{proof}

\subsection{Proof of the third master inequality \eqref{eq:masterineq Psi 2 av}} \label{subsec:proofmaster3}
Let $w_j\in \mathbf{D}_{\ell +1}^{(\epsilon_0, \kappa_0)}$ for $j \in [2]$ be spectral parameters and $A_1$ a regular matrix
w.r.t.~$(w_1, w_2)$ and $A_2$ a regular matrix w.r.t.~$(w_2, w_1)$ (see Definition~\ref{def:regobs}).
By conjugation with $E_-$, we again assume w.l.o.g.~that $\Im w_1 > 0$ and $\Im w_2< 0$.
Just as in Section~\ref{subsec:proofmaster2}, we use the notations $e_j \equiv \Re w_j$, $\eta_j := |\Im w_j|$ for $j \in [2]$ and define $1 \ge \eta := \min_j |\Im w_j|$. We also assume that \eqref{eq:apriori Psi} holds.
\begin{lemma} {\rm (Representation as full underlined)} \label{lem:underlined3} \\
	For any $(w_1, w_2)$-regular matrix $A_1  = \mathring{A}_1$ and $(w_2, w_1)$-regular matrix $A_2 = \mathring{A}_2$, we have that
	\begin{equation} \label{eq:underlined Psi 2 av}
		\big\langle \big(G_1 \mathring{A}_1 G_2 - M(w_1, \mathring{A}_1, w_2)\big) \mathring{A}_2 \big\rangle= - \big\langle\underline{W G_1 \mathring{A}_1 G_2 \mathring{A}_2'}\big\rangle + \mathcal{O}_\prec\big(\mathcal{E}_2^{\rm av}\big)
	\end{equation}
	for some $(w_2, w_1)$-regular matrix $A_2' = \mathring{A}_2'$, which linearly depends on $A_2 = \mathring{A}_2$ (analogously to \eqref{eq:A'1iso def}, see \eqref{eq:A'2av def} for an explicit formula). For the error term in \eqref{eq:underlined Psi 2 av}, we used the shorthand notation
	\begin{equation} \label{eq:E2av}
		\mathcal{E}_2^{\rm av} := \frac{1}{N \eta}\left(1 
		+ \frac{(\psi_1^{\rm av})^2}{N \eta} + \frac{\psi_2^{\rm av}}{N \eta}\right)\,.
	\end{equation}
\end{lemma}
Note that similarly to Lemma \ref{lem:underlined1} but contrary to Lemma \ref{lem:underlined2}, we again expanded the first resolvent $G_1$.
Otherwise, the proof of Lemma \ref{lem:underlined3}, given in Appendix \ref{app:underlinedproofs}, is very
similar to the one of Lemma \ref{lem:underlined2}. We only mention that the quadratic error
$(\psi_1^{\rm av})^2$ stems from terms of the form
\begin{equation*}
	\langle  \mathcal{S}[G_1 \mathring{A}^{{1,2}}_1 G_2] (G_2 - M_2) \mathring{A}_2^{{2,1}} \rangle\,,
\end{equation*}
appearing in the analogue of \eqref{eq:Psi 1 iso final} (see \eqref{eq:Psi 2 av firstcrit} in Appendix \ref{app:underlinedproofs}).
Having the approximate representation \eqref{eq:underlined Psi 2 av}, 
we turn to the proof of \eqref{eq:masterineq Psi 2 av} via cumulant expansion of
the  full underlined term. 
\begin{proof}[Proof of \eqref{eq:masterineq Psi 2 av}]
	Let $p \in \N$. Starting from \eqref{eq:underlined Psi 1 av}, we obtain, as in the proofs of \eqref{eq:masterineq Psi 1 av} and \eqref{eq:masterineq Psi 1 iso},
	\begin{align}
		         & \mathbf{E} \big|\langle (G_1 \mathring{A}_1 G_2-M(w_1, \mathring{A}_1, w_2))\mathring{A}_2 \rangle \big|^{2p}  \label{eq:min exp 2av}                                                                                                                                                       \\
		\lesssim & \, \mathbf{E} \,  \widetilde{\Xi}_2^{\rm av}\,  \big\vert  \langle (G_1 \mathring{A}_1 G_2-M(\ldots))\mathring{A}_2 \rangle  \big\vert^{2p-2} \nonumber                                                                                                                                     \\
		         & + \sum_{|\boldsymbol{l}| + \sum(J \cup J_*) \ge 2} \mathbf{E} \, \Xi_2^{\rm av}(\boldsymbol{l}, J, J_*) \big\vert  \langle (G_1 \mathring{A}_1 G_2-M(\ldots))\mathring{A}_2 \rangle \big\vert^{2p-1 - |J \cup J_*|} + \mathcal{O}_\prec\big((\mathcal{E}_2^{\rm av})^{2p}\big) \nonumber\,,
	\end{align}
	where
	\begin{equation*}
		\widetilde{\Xi}_2^{\rm av}:= \frac{1}{N^2} \sum_\sigma \big|  \langle G_1 \mathring{A}_1 G_2 \mathring{A}_2 G_1 E_\sigma G_1 \mathring{A}_1 G_2 \mathring{A}_2' E_\sigma \rangle \big| + \cdots
	\end{equation*}
	with the other terms being analogous, just $1$ and $2$ in the first half $G_1 \mathring{A}_1 G_2 \mathring{A}_2 G_1$ of the chain interchanged or the entire half taken as adjoint, and $\Xi_2^{\rm av}(\boldsymbol{l}, J, J_*)$ is defined as
	\begin{align} \label{eq:Xi 2 av def}
		\Xi_2^{\rm av} := N^{-(|\boldsymbol{l}| + \sum(J \cup J_*) + 3)/2} \sum_{ab} & R_{ab}|\partial^{\boldsymbol{l}} (G_1 \mathring{A}_1 G_2\mathring{A}_2')_{ba}|                                                                                                                                                                                  \\
		                                                                             & \times \prod_{\boldsymbol{j} \in J} |\partial^{\boldsymbol{j}} \langle G_1 \mathring{A}_1 G_2\mathring{A}_2 \rangle |   \prod_{\boldsymbol{j} \in J_*} |\partial^{\boldsymbol{j}} \langle G_2^* \mathring{A}_2^* G_1^* \mathring{A}_1^* \rangle |\,.  \nonumber
	\end{align}
	As in Sections \ref{subsec:proofmaster1} and \ref{subsec:proofmaster2}, in the remainder of the proof, we need to
	analyze the rhs. of~\eqref{eq:min exp 2av}.
	We begin with the second line and study the terms involving $\Xi_2^{\rm av}$ from \eqref{eq:Xi 2 av def} afterwards.
	\\[2mm]
	{\bf \underline{Gaussian contribution: second line of \eqref{eq:min exp 2av}.}} 
	Along the principal strategy outlined in Remark \ref{rmk:strategy}, we need to analyze in total eight terms, each of which carries one of the summands in the definition of $\widetilde{\Xi}_2^{\rm av}$ as a factor. Since their treatment is very similar, we focus on the exemplary term
	\begin{equation} \label{eq:oneterm2av}
		\langle G_1 \mathring{A}_1^{w_1, w_2} G_2 \mathring{A}_2^{w_2, w_1} G_1 G_1 \mathring{A}_1^{w_1, w_2} G_2 (\mathring{A}_2' )^{w_2, w_1} \rangle \,.
	\end{equation}

	Now, we represent $G_1 G_1$ via the integral representation from Lemma \ref{lem:intrepG^2} with
	\[
		\tau = +\,, \quad   J =  \mathbf{B}_{ \ell \kappa_0}\,, \quad \text{and} \quad  \tilde{\eta} = \frac{\ell}{\ell +1} \eta\,,
	\]
	for which we recall that $w \in \mathbf{D}_{\ell +1}^{(\epsilon_0, \kappa_0)}$, i.e.~in particular $\eta \ge (\ell +1) N^{-1+\epsilon_0}$ and hence $\tilde{\eta} \ge \ell N^{-1+\epsilon_0}$. After splitting the contour integral and bounding the individual contributions as described in \eqref{eq:contourdecomp}, we obtain, with the aid of Lemma \ref{lem:Mbound},
	\begin{equation*}
		\big| \eqref{eq:oneterm2av} \big| \prec \frac{1}{\eta^2} \left( 1 + \frac{\psi_4^{\rm av}}{N \eta}  \right) + \int_{ \mathbf{B}_{ \ell \kappa_0}} \frac{\big|  \langle G_1 \mathring{A}_1^{w_1, w_2} G_2 \mathring{A}_2^{w_2, w_1} G(x + \I \tilde{\eta}) \mathring{A}_1^{w_1, w_2} G_2 (\mathring{A}_2' )^{w_2, w_1} \rangle  \big|}{(x - e_1)^2 + \eta_1^2} \D x\,.
	\end{equation*}

	Next, we decompose $\mathring{A}_2^{w_2, w_1}$ and $\mathring{A}_1^{w_1, w_2}$ in the integrand as {(recall the notation in Footnote \ref{ftn:Osigma})}
	\begin{equation} \label{eq:Psi2avGaussdecomp}
		\begin{split}
			\mathring{A}_2^{w_2, x + \I \tilde{\eta}} &= \mathring{A}_2^{w_2, w_1} + \sum_\sigma \mathcal{O}_\sigma(|x - e_1| + |\eta_1 - \tilde{\eta}|) E_\sigma \\
			\mathring{A}_1^{ x + \I \tilde{\eta}, w_2} &= \mathring{A}_1^{w_1, w_2} + \sum_\sigma \mathcal{O}_\sigma(|x - e_1| + |\eta_1 - \tilde{\eta}|) E_\sigma\,.
		\end{split}
	\end{equation}

	While the properly regularised term contributes an $\eta^{-2}\big(1 + \psi_4^{\rm av}/(N \eta)\big)$-error, a typical cross term shall be estimated as
	\begin{equation} \label{eq:avtoiso}
		\int_{ \mathbf{B}_{ \ell \kappa_0}}\frac{\big|  \big\langle G_1 \mathring{A}_1^{w_1, w_2} G_2 \mathring{A}_2^{w_2, x + \I \tilde{\eta}} [G(x + \I \tilde{\eta}) -  G_2] (\mathring{A}_2' )^{w_2, w_1} \big\rangle  \big|}{\big(|x - e_1| + \eta_1\big) \, \big(|x - e_2| + \eta_2\big)} \prec \frac{1}{\eta^2} \left( 1 + \frac{\psi_2^{\rm iso}}{\sqrt{N \eta}} \right)
	\end{equation}
	where in the second step we wrote out the averaged trace and estimated each summand in isotropic form with the aid of Lemma \ref{lem:Mbound}, using $\psi_2^{\rm iso}$ instead of $\psi_3^{\rm av}$.


	Finally, for  `error $\times$ error'-type terms are bounded by $\eta^{-2}$, simply by using
	a trivial Schwarz inequality in combination with a Ward identity and the usual local law from Theorem~\ref{thm:singleG} to infer
	\begin{equation*}
		\big| \langle G_1 B_1 G_2 B_2 \big| \le \sqrt{\langle G_1 B_1 B_1^* G_1^* \rangle \langle G_2 B_2 B_2^* G_2^* \rangle} \le \frac{1}{\eta} \sqrt{\langle \Im G_1 B_1 B_1^*  \rangle \langle \Im G_2 B_2 B_2^*  \rangle} \prec \frac{1}{\eta}\,,
	\end{equation*}
	which is valid for arbitrary bounded matrices $\Vert B_1 \Vert, \Vert B_2 \Vert \lesssim 1$.

	This finishes the estimate for the Gaussian contribution from the second line of \eqref{eq:min exp 2av}, for which, collecting the above estimates, we have shown that
	\begin{equation} \label{eq:Gaussian2av}
		\widetilde{\Xi}_2^{\rm av} \prec \frac{1}{N^2 \eta^2} \left( 1 + \frac{\psi_2^{\rm iso}}{\sqrt{N \eta}} + \frac{\psi_4^{\rm av}}{N \eta} \right)\,.
	\end{equation}

	We are now left with the terms from the last line of \eqref{eq:min exp 2av} resulting from higher order cumulants.
	\\[2mm]
	{\bf \underline{Higher order cumulants and conclusion.}} The estimate stemming from higher order cumulants is given in \eqref{eq:higherorder2av} in Section \ref{subsec:higherorder}. Then, plugging \eqref{eq:Gaussian2av} and \eqref{eq:higherorder2av} into \eqref{eq:min exp 2av}, we find, similarly to Section~\ref{subsec:proofmaster1}, that
	\begin{equation*}
		\Psi_2^{\rm av} \prec 1 + \frac{(\psi_1^{\rm av})^{2} + {(\psi_1^{\rm iso})^2} + \psi_2^{\rm av}}{N \eta}+ \frac{\psi_2^{\rm iso} + (\psi_4^{\rm av})^{1/2}  }{(N \eta)^{1/2}} + \frac{(\psi_2^{\rm iso})^{1/2}}{(N \eta)^{1/4}} + \frac{(\psi_3^{\rm iso})^{3/8} + (\psi_4^{\rm iso})^{3/8}}{(N \eta)^{3/16}}  \,.
	\end{equation*}
	The bound given in Proposition \ref{prop:master} is an immediate consequence after a trivial Young inequality.
\end{proof}

\subsection{Proof of the fourth master inequality \eqref{eq:masterineq Psi 2 iso}} \label{subsec:proofmaster4}   Let $w_j\in \mathbf{D}_{\ell +1}^{(\epsilon_0, \kappa_0)}$ for $j \in [3]$ be spectral parameters and $A_1$ a regular matrix w.r.t.~$(w_1, w_2)$ and $A_2$ a regular matrix w.r.t.~$(w_2, w_3)$ (see Definition~\ref{def:regobs}). By conjugation with $E_-$, we will assume w.l.o.g.~that $\Im w_1 > 0$,
$\Im w_2< 0$, and $\Im w_3 > 0$. As before, we use the notations $e_j \equiv \Re w_j$, $\eta_j := |\Im w_j|$ for $j \in [3]$ and define $1 \ge \eta := \min_j |\Im w_j|$. We also assume that \eqref{eq:apriori Psi} holds.
\begin{lemma} {\rm (Representation as full underlined)} \label{lem:underlined4} \\
	For $\Vert \boldsymbol{x} \Vert , \Vert \boldsymbol{y} \Vert \le 1$ and any $(w_1, w_2)$-regular matrix $A_1  = \mathring{A}_1$ and $(w_2, w_3)$-regular matrix $A_2 = \mathring{A}_2$, we have that
	\begin{equation} \label{eq:underlined Psi 2 iso}
		\big(G_1 \mathring{A}_1 G_2 \mathring{A}_2 G_3 - M(w_1, \mathring{A}_1, w_2, \mathring{A}_2, w_3)\big)_{{\bm x}{\bm y}}  = - \big(\underline{G_1 \mathring{A}_1' W G_2 \mathring{A}_2 G_3}\big)_{{\bm x}{\bm y}}  + \mathcal{O}_\prec\big(\mathcal{E}_2^{\rm iso}\big)
	\end{equation}
	for some other $(w_1, w_2)$-regular matrix $A_1' = \mathring{A}_1'$, which linearly depends on $A_1 = \mathring{A}_1$ (analogously to \eqref{eq:A'1iso def}, see \eqref{eq:A'2iso def} for an explicit formula). For the error term in \eqref{eq:underlined Psi 2 iso}, we used the shorthand notation
	\begin{equation} \label{eq:E2iso}
		\mathcal{E}_2^{\rm iso} := \frac{1}{\sqrt{N \eta^3}}\left(1 + \psi_1^{\rm iso}  + \frac{\psi_1^{\rm av} \psi_1^{\rm iso}}{N \eta} + \frac{\psi_2^{\rm iso}}{N \eta}\right)\,.
	\end{equation}
\end{lemma}
Note that similarly to \eqref{eq:underlined Psi 1 iso}, we again expanded the second resolvent.
The proof of Lemma \ref{lem:underlined4}, given in Appendix \ref{app:underlinedproofs}, is very similar to the one of Lemma \ref{lem:underlined2}. We only mention that the errors carrying $\psi_1^{\rm iso} \psi_1^{\rm av}$ and $\psi_1^{\rm iso}$ stem from terms of the form
\begin{align*}
	                                               & \big(G_1 \mathcal{S}[(G_1-M_1) A_1^{\circ_{1,2}}] G_2 \mathring{A}_2 G_3 \big)_{\boldsymbol{x}\boldsymbol{y}
	} \quad \text{and}                                                                                                                                                                              \\[1mm]
	c_\sigma(\mathcal{X}_{12}[\mathring{A}_1] M_2) & c_\sigma(\Phi_\sigma) \big(G_1 E_\sigma G_2 \mathring{A}_2 G_3 - M(w_1, E_\sigma, w_2, \mathring{A}_2, w_3)\big)_{\boldsymbol{x}\boldsymbol{y}
		}\,,
\end{align*}
respectively, appearing in the analogue of \eqref{eq:Psi 1 iso final} (see \eqref{eq:Psi 2 iso firstcrit} and \eqref{eq:Psi 2 iso thirdcrit} in Appendix \ref{app:underlinedproofs}).
Having the  representation \eqref{eq:underlined Psi 2 iso}
we turn to the proof of \eqref{eq:masterineq Psi 2 iso} via cumulant expansion of the underlined term.
\begin{proof}[Proof of \eqref{eq:masterineq Psi 2 iso}]
	Let $p \in \N$. Then, starting from \eqref{eq:underlined Psi 2 iso}, we obtain
	\begin{align}
		         & \mathbf{E} \big|\big(G_1 \mathring{A}_1 G_2\mathring{A}_2 G_3-M(w_1, \mathring{A}_1, w_2, \mathring{A}_2, w_3)\big)_{\boldsymbol{x}\boldsymbol{y}}  \big|^{2p} \label{eq:min exp 2iso}                                                                         \\
		\lesssim & \, \mathbf{E} \,  \widetilde{\Xi}_2^{\rm iso} \,  \big\vert  \big(G_1 \mathring{A}_1 G_2\mathring{A}_2 G_3-M(\ldots)\big)_{\boldsymbol{x}\boldsymbol{y}} \big\vert^{2p-2} + \mathcal{O}_\prec\big((\mathcal{E}_1^{\rm iso})^{2p}\big) \nonumber                \\
		         & + \sum_{|\boldsymbol{l}| + \sum(J \cup J_*) \ge 2} \mathbf{E} \, \Xi_2^{\rm iso}(\boldsymbol{l}, J, J_*) \big\vert  \big(G_1 \mathring{A}_1 G_2\mathring{A}_2 G_3 -M(\ldots)\big)_{\boldsymbol{x}\boldsymbol{y}}  \big\vert^{2p-1 - |J \cup J_*|} \nonumber\,,
	\end{align}
	where
	\begin{align*}
		\widetilde{\Xi}_2^{\rm iso} := & \frac{\sum_\sigma \sum_{j=1}^{3}\left\vert \big(  G_1 \mathring{A}_1'E_\sigma G_j \mathring{A}_j \ldots G_3\big)_{\boldsymbol{x}\boldsymbol{y}} \big( G_1 \mathring{A}_1 \ldots \mathring{A}_{j-1}G_j E_\sigma G_2 \mathring{A}_2 G_3 \big)_{\boldsymbol{x}\boldsymbol{y}} \right\vert }{N} \nonumber         \\[2mm]
		                               & + \frac{\sum_\sigma \sum_{j=1}^{3} \left\vert \big( G_1 \mathring{A}_1'E_\sigma  G_j^* \mathring{A}_{j-1}^* \ldots \mathring{A}_1^* G_1^* \big)_{\boldsymbol{x}\boldsymbol{x}} \big( G_3^* \ldots \mathring{A}_j^* G_j^* E_\sigma G_2 \mathring{A}_2 G_3 \big)_{\boldsymbol{y}\boldsymbol{y}} \right\vert}{N}
	\end{align*}
	and	$\Xi_2^{\rm iso}(\boldsymbol{l}, J, J_*)$ is defined as
	\begin{align} \label{eq:Xi 2 iso def}
		\Xi_2^{\rm iso} := N^{-(|\boldsymbol{l}| + \sum(J \cup J_*) + 1)/2} & \sum_{ab} R_{ab}\big|\partial^{\boldsymbol{l}} \big[(G_1 \mathring{A}'_1)_{\boldsymbol{x}a}\big(G_2\mathring{A}_2G_3\big)_{b \boldsymbol{y}} \big]\big|                                                                                                                                                                                    \\
		                                                                    & \times \prod_{\boldsymbol{j} \in J} \big|\partial^{\boldsymbol{j}}  \big(G_1 \mathring{A}_1 G_2\mathring{A}_2G_3\big)_{\boldsymbol{x}\boldsymbol{y}}\big|   \prod_{\boldsymbol{j} \in J_*} \big|\partial^{\boldsymbol{j}} \big(G_3^* \mathring{A}_2^* G_2^* \mathring{A}_1^* G_2^*\big)_{\boldsymbol{y}\boldsymbol{x}}  \big|\,. \nonumber
	\end{align}
	We need to analyze the rhs.~of the inequality derived in \eqref{eq:min exp 2iso}. We begin with the second line.
	\\[2mm]
	{\bf \underline{Gaussian contribution: second line of \eqref{eq:min exp 2iso}.}} Following Remark \ref{rmk:strategy}, we need to analyze in total twelve terms, each of which carries one of the summands in the definition of $\widetilde{\Xi}_2^{\rm iso}$ as a factor. Again, using Lemma \ref{lem:regularbasic} for the $A$'s, we pick two exemplary terms
	\begin{align}
		\big(G_1 \mathring{A}_1^{w_1, w_2} G_2 \mathring{A}_2^{w_2, w_3} G_3 E_- G_2 \mathring{A}_2^{w_2, w_3} G_3\big)_{\boldsymbol{x} \boldsymbol{y}}    \big(  G_1 \mathring{(A_1')}^{w_1, w_2} E_-  G_3\big)_{\boldsymbol{x} \boldsymbol{y}} \label{eq:twoterms2iso1} \\[1mm]
		\big( G_1 (\mathring{A}_1')^{w_1, w_2} G_2^* (\mathring{A_1^*})^{{\bar{w}_2, \bar{w}_1}} G_1^* \big)_{\boldsymbol{x}\boldsymbol{x}} \big( G_3^* \mathring{(A_2^*)}^{{\bar{w}_3, \bar{w}_2}} G_2^* G_2 \mathring{A}_2^{w_2, w_3} G_3 \big)_{\boldsymbol{y}\boldsymbol{y}} \label{eq:twoterms2iso2}
	\end{align}
	which shall be treated in more detail. The other terms are analogous and hence omitted. 
	\\[2mm]
	\noindent \underline{\emph{The term \eqref{eq:twoterms2iso1}.}} In the first factor, we use \eqref{eq:chiral}, Lemma \ref{lem:regularbasic}, Lemma \ref{lem:Mbound} and Lemma \ref{lem:intrepG^2} with parameters
	\[
		\tau  = +\,, \quad J =  \mathbf{B}_{ (\ell+ \frac{1}{2}) \kappa_0}\,, \quad \text{and} \quad \tilde{\eta} = \frac{2\ell}{2\ell +1} \eta\,,
	\]
	(in order to have some flexibility before approaching the boundary of the domain $\mathbf{D}_\ell^{(\epsilon_0, \kappa_0)}$) to bound it as
	\begin{align*} 
		\big| \big(G_1 \mathring{A}_1^{w_1, w_2} & G_2 \mathring{A}_2^{w_2, w_3} G_3 E_- G_2 \mathring{A}_2^{w_2, w_3} G_3\big)_{\boldsymbol{x} \boldsymbol{y}}   \big| \prec \frac{1}{\eta^{3/2}} \left( 1 + \frac{\psi_3^{\rm iso}}{\sqrt{N \eta}}\right)                                                                                                                 \\[1mm]
		                                         & + \int_{\mathbf{B}_{ (\ell+ \frac{1}{2}) \kappa_0}} \frac{\big| \big(G_1 \mathring{A}_1^{w_1, w_2} G_2 \mathring{A}_2^{w_2, w_3} G(x + \I \tilde{\eta}) \mathring{(E_- A_2)}^{-w_2, w_3} G_3\big)_{\boldsymbol{x} \boldsymbol{y}} \big|}{\big(|x - e_3| + \eta_3\big) \, \big( |x + e_2|+\eta_2 \big)} \D x \nonumber\,.
	\end{align*}

	Next, we decompose $\mathring{A}_2^{w_2, w_3}$ and $\mathring{(E_- A_2)}^{-w_2, w_3}$ according to the integration variable with the aid of Lemma \ref{lem:regularbasic}~(iii), analogously to \eqref{eq:Psi2avGaussdecomp}. This leaves us with four terms, which shall be estimated separately. While the fully regularised term gives
	\begin{equation*}
		\frac{1}{\eta^{3/2}} \left( 1 + \frac{\psi_3^{\rm iso}}{\sqrt{N \eta}}\right) \left( 1 + \frac{1}{|e_2 + e_3| + \eta_2 + \eta_3} \right) \,,
	\end{equation*}
	the cross terms can be estimated as
	\begin{equation*}
		\frac{1}{\eta^2} \left( 1 + \frac{\psi_2^{\rm iso}}{\sqrt{N \eta}} \right)\,,
	\end{equation*}
	analogously to \eqref{eq:avtoiso}.
	As an exemplary error term, we consider
	\begin{equation}\label{eq:Psi2isoGauss3}
		\int_{\mathbf{B}_{ (\ell+ \frac{1}{2}) \kappa_0}} \big| \big(G_1 \mathring{A}_1^{w_1, w_2} G_2 E_+ G(x + \I \tilde{\eta}) E_- G_3\big)_{\boldsymbol{x} \boldsymbol{y}} \big|\D x
	\end{equation}
	and use Lemma \ref{lem:intrepG^2} with new parameters
	\[
		\tau = -\,, \quad  J = \mathbf{B}_{ \ell \kappa_0}\,, \quad \tilde{\eta} = \frac{\ell}{\ell +1} \eta\,,
	\]
	to find, dropping the integration domains for ease of notation,
	\begin{align*}
		\big| \eqref{eq:Psi2isoGauss3} \big| \prec \frac{1}{\eta^{1/2}}\left( 1 + \frac{\psi_1^{\rm iso}}{\sqrt{N \eta}} \right) + \int \mathrm{d} x \int \D y & \frac{\big|  \big( G_1 \mathring{A}_1^{w_1, w_2} G(y - \I \tilde{\eta}) \big)_{\boldsymbol{x} (E_- \boldsymbol{y})} \big|}{\big( |y - e_2| + \eta_2 \big)\, \big( |y +x| + \eta \big)\, \big( |y + e_3| + \eta_3 \big)} \\[1mm]
		                                                                                                                                                       & \prec \frac{1}{\eta^{3/2}} \left( 1 + \frac{\psi_1^{\rm iso}}{\sqrt{N \eta}} \right) \left( 1 + \frac{1}{|e_2 + e_3| + \eta_2 + \eta_3} \right)\,,
	\end{align*}
	where in the last step we used Lemma \ref{lem:regularbasic} for decomposing $\mathring{A}_1^{w_1, w_2}$ accordingly, and Lemma~\ref{lem:Mbound}.

	This finishes the bound on the first factor in \eqref{eq:twoterms2iso1}. The second factor can easily be estimated as
	\begin{equation*}
		\big|  \big(  G_1 \mathring{(A_1')}^{w_1, w_2} E_-  G_3\big)_{\boldsymbol{x} \boldsymbol{y}}  \big| \prec \frac{1}{\eta^{1/2}} \left(  1 + \frac{\psi_1^{\rm iso}}{\sqrt{N \eta}} \right) + \frac{|e_2 + e_3| + \eta_2 + \eta_3}{\eta}
	\end{equation*}
	using \eqref{eq:chiral}, Lemma \ref{lem:regularbasic}, and Lemma \ref{lem:Mbound}.
	Notice  the cancellation of $|e_2+e_3|$ between the two factors.
	\\[2mm]
	\underline{\emph{The term \eqref{eq:twoterms2iso2}.}} For the first factor in \eqref{eq:twoterms2iso2}, we realise that $(\mathring{A}_1')^{w_1, w_2} =(\mathring{A}_1')^{w_1, \bar{w}_2}$, which without approximation immediately yields that
	\[
		\big| \big( G_1 (\mathring{A}_1')^{w_1, w_2} G_2^* (\mathring{A_1^*})^{{\bar{w}_2, \bar{w}_1}} G_1^* \big)_{\boldsymbol{x}\boldsymbol{x}} \big| \prec \frac{1}{\eta} \left( 1 + \frac{\psi_2^{\rm iso}}{\sqrt{N \eta}} \right)
	\]
	with the aid of Lemma \ref{lem:Mbound}.

	In the second factor, we apply a Ward identity to $G_2^* G_2$ and again use that the regularisation is insensitive to complex conjugation in the second spectral parameter. In this way, and decomposing
	\[
		\mathring{A}_2^{w_2, w_3} = \mathring{A}_2^{\bar{w}_2, w_3} + \mathcal{O}\big( |e_2 - e_3| + |\eta_2 - \eta_3| \big) E_+ + \mathcal{O}\big( |e_2 + e_3| + |\eta_2 - \eta_3| \big) E_-
	\]
	by means of Lemma \ref{lem:regularbasic}~(ii), we find that the second factor is stochastically dominated by
	\begin{align*}
		\frac{1}{\eta^2} \left( 1 + \frac{ \psi_1^{\rm iso}+ \psi_2^{\rm iso}}{\sqrt{N \eta}} \right) \,.
	\end{align*}

	This finishes the estimate for the Gaussian contribution from the second line of \eqref{eq:min exp 2iso}, for which, collecting the above estimates, we have shown that
	\begin{equation} \label{eq:Gaussian2iso}
		\widetilde{\Xi}_2^{\rm iso} \prec \frac{1}{N \eta^3} \left[\left( 1 + \frac{ \psi_3^{\rm iso}}{\sqrt{N \eta}}  \right)\left( 1 + \frac{ \psi_1^{\rm iso}}{\sqrt{N \eta}}  \right) + \left(1 + \frac{\psi_1^{\rm iso} + \psi_2^{\rm iso}}{\sqrt{N \eta}}\right)^2\right]\,.
	\end{equation}

	We are now left with the terms from the last line of \eqref{eq:min exp 2iso} resulting from higher order cumulants.
	\\[2mm]
	{\bf \underline{Higher order cumulants and conclusion.}}
	The estimate stemming from higher order cumulants is given in \eqref{eq:higherorder2iso} in Section \ref{subsec:higherorder}. Then, plugging \eqref{eq:Gaussian2iso} and \eqref{eq:higherorder2iso} into \eqref{eq:min exp 2iso}, we find, similarly to Section~\ref{subsec:proofmaster1}, that
	\begin{equation*}
		\Psi_2^{\rm iso} \prec 1 + \psi_1^{\rm iso} + \frac{ \psi_1^{\rm av} \psi_1^{\rm iso} + (\psi_1^{\rm iso})^2 + \psi_2^{\rm iso}}{N \eta}+ \frac{\psi_2^{\rm iso} + (\psi_1^{\rm iso} \psi_3^{\rm iso})^{1/2}}{(N \eta)^{1/2}}  +   \frac{(\psi_3^{\rm iso})^{3/8} + (\psi_4^{\rm iso})^{3/8}}{(N \eta)^{3/16}}
	\end{equation*}
	The bound given in Proposition \ref{prop:master} is an immediate consequence after a trivial Young inequality.
\end{proof}

\subsection{Contributions from higher order cumulants} \label{subsec:higherorder}
The goal of the present section is to estimate the terms originating from higher order cumulants in \eqref{eq:min exp 1av}, \eqref{eq:min exp 1iso}, \eqref{eq:min exp 2av}, and \eqref{eq:min exp 2iso}. In order to do so, we assume that \eqref{eq:apriori Psi} holds.
\begin{lemma} \label{lem:higherorder}For any $J, J_* \subset \Z^2_{\ge 0} \setminus \{ (0,0) \}$, $\bm{l} \in \Z^2_{\ge 0} $ with $\vert \bm{l}\vert + \sum (J \cup J_*) \ge 2$ it holds that
	\begin{subequations}
		\begin{align}
			\big(\Xi_1^{\rm av} \big)^{1/(1 + \sum (J \cup J_*))}  & \prec \frac{1}{N \eta^{1/2}} \left( 1 + \frac{\psi_1^{\rm iso}}{(N \eta)^{1/2}} + \frac{(\psi_2^{\rm iso})^{1/4}}{(N \eta)^{1/8}} \right) \label{eq:higherorder1av} \,,                                                                        \\[1mm]
			\big(\Xi_1^{\rm iso} \big)^{1/(1 + \sum (J \cup J_*))} & \prec  \frac{1}{\sqrt{N \eta^2}}  \left( 1 + \frac{\psi_1^{\rm iso}}{(N \eta)^{1/2}} + \frac{(\psi_2^{\rm iso})^{1/4}}{(N \eta)^{1/8}} \right) \label{eq:higherorder1iso} \,,                                                                  \\[1mm]
			\big(\Xi_2^{\rm av} \big)^{1/(1 + \sum (J \cup J_*))}  & \prec  \frac{1}{N \eta}  \left( 1 + \frac{(\psi_1^{\rm iso})^2}{N \eta} + \frac{\psi_2^{\rm iso}}{(N \eta)^{1/2}} + \frac{(\psi_3^{\rm iso})^{3/8} + (\psi_4^{\rm iso})^{3/8}}{(N \eta)^{3/16}} \right) \label{eq:higherorder2av} \,,          \\[1mm]
			\big(\Xi_2^{\rm iso} \big)^{1/(1 + \sum (J \cup J_*))} & \prec  \frac{1}{\sqrt{N \eta^3}} \left( 1 + \frac{(\psi_1^{\rm iso})^2}{N \eta} + \frac{\psi_2^{\rm iso}}{(N \eta)^{1/2}} + \frac{(\psi_3^{\rm iso})^{3/8} + (\psi_4^{\rm iso})^{3/8}}{(N \eta)^{3/16}} \right) \,. \label{eq:higherorder2iso}
		\end{align}
	\end{subequations}
\end{lemma}

For \(k=1,2\), \(\bm l\in \Z_{\ge 0}^2\) and a multiset \(J\subset \Z_{\ge 0}^2\setminus\set{(0,0)}\) we now define slightly (notationally) simplified versions of $\Xi_k^{\rm av/iso}$, namely
\begin{align}\label{Xi1av def}
	\Xi_k^{\rm av} (\bm l,J) & := N^{-(\abs{\bm l} +\sum J+3)/2}\sum_{ab} \abs*{\partial^{\bm l} ((GA)^{k-1}GA')_{ba}} \prod_{\bm j\in J} \abs*{\partial^{\bm j}\braket{(GA)^k}}\,,                           \\\label{Xi1iso def}
	\Xi_k^{\rm iso}(\bm l,J) & := N^{-(\abs{\bm l} +\sum J+1)/2}\sum_{ab} \abs*{\partial^{\bm l} [(GA)_{\bm x a}(G(AG)^{k-1})_{b\bm y}]} \prod_{\bm j\in J} \abs*{\partial^{\bm j}((GA)^kG)_{\bm x\bm y}} \,,
\end{align}
where \(\sum J:=\sum_{\bm j\in J}\abs{\bm j}\), \(\abs{(j_1,j_2)}:=j_1+j_2\) and \(\partial^{(j_1,j_2)}:=\partial_{ab}^{j_1}\partial_{ba}^{j_2}\). Here, for notational simplicity, we do not carry the dependence on the spectral parameters of the resolvents but assume that implicitly each resolvent has its own spectral parameter and that each \(A\) is correctly regularised with respect to its neighboring resolvents. In particular compared to \eqref{eq:Xi 1 av def}, \eqref{eq:Xi 1 iso def}, \eqref{eq:Xi 2 av def}, and \eqref{eq:Xi 2 iso def}, it is not necessary to distinguish the sets \(J,J^\ast\).
\begin{proof}[Proof of Lemma \ref{lem:higherorder}]
	Throughout the proof, we denote $\phi_k := \psi_k^{\rm iso}/\sqrt{N \eta}$. The \emph{naive estimate} for the derivatives simply is
	\begin{equation}\label{eq naive}
		\begin{split}
			\abs{\partial^{\bm l}((GA)^{k-1}GA')_{ba}} &\prec \eta^{-(k-1)/2}\Bigl(1+\phi_{k-1}\Bigr)\,, \\
			\abs{\partial^{\bm j}\braket{(GA)^k}} &\prec  \frac{1}{N\eta^{k/2}}\sum_{k_1+k_2+\cdots=k}\prod_i\Bigl(1+\phi_{k_i}\Bigr)
		\end{split}
	\end{equation}
	due to~\eqref{eq:Mboundnorm} and recalling~\eqref{eq:Psi isok}. {Here a single derivative just splits the chain with $k$ resolvents into two chains with a total of $k+1$ resolvents, leaving the estimates on the main and the error term unchanged.} Using~\eqref{eq naive} in~\eqref{Xi1av def} we obtain
	\begin{equation}
		\begin{split}
			\abs{\Xi_1^{\rm av}} &\prec (N\eta^{1/2})^{-1-\abs{J}} N^{(2-\abs{\bm l}-\sum J)} \sqrt{N\eta} \, \Bigl(1+\phi_1\Bigr)^{\abs{J}}\,,\\
			\abs{\Xi_2^{\rm av}} &\prec (N\eta)^{-1-\abs{J}} N^{(2-\abs{\bm l}-\sum J)} \sqrt{N\eta} \, \Bigl(1+\phi_1\Bigr)\Bigl(1+\phi_2+\phi_1^2\Bigr)^{\abs{J}}\,,\\
			\abs{\Xi_1^{\rm iso}} &\prec (\sqrt{N}\eta)^{-1-\abs{J}} \eta^{1+\abs{J}/2} N^{(4-\abs{\bm l}+\abs{J}-\sum J)/2} \Bigl(1+\phi_1\Bigr)^{\abs{J}}\,, \\
			\abs{\Xi_2^{\rm iso}} &\prec (\sqrt{N}\eta^{3/2})^{-1-\abs{J}} \eta^{1+\abs{J}/2} N^{(4-\abs{\bm l}+\abs{J}-\sum J)/2} \Bigl(1+\phi_1\Bigr)\Bigl(1+\phi_2+\phi_1^2\Bigr)^{\abs{J}} \,,
		\end{split}
	\end{equation}
	and therefore have proved \eqref{eq:higherorder1av} and \eqref{eq:higherorder2av} in all cases except \(\abs{\bm l}+\sum J=2\) and \eqref{eq:higherorder1iso} and \eqref{eq:higherorder2iso} in all cases except \(\abs{\bm l}+ \sum J-\abs{J}< 4\). For the remaining cases we need a more refined estimate using the following \emph{Ward lemma}:
	\begin{lemma}\label{lemma Ward}
		Let \(\vx\) be any deterministic vector of bounded norm, let \(w_1,\ldots,w_k \in \mathbf{D}_{\ell+1}^{(\epsilon_0, \kappa_0)}\) be spectral parameters and \(A_1,\ldots,A_{k}\) deterministic matrices of bounded norm. Then for \(G_i=G(w_i)\) it holds that
		\begin{equation*}
			\frac{1}{N}\sum_a \abs*{(G_1 \mathring{A}_1^{w_1,w_2}\cdots \mathring{A}_{k-1}^{w_{k-1},w_k}G_k A_{k} )_{\vx a}}\prec \frac{1}{\sqrt{N\eta}}\frac{1}{\eta^{(k-1)/2}} \Bigl(1+\phi_1+\cdots+\phi_{2k}\Bigr)^{1/2} \,,
		\end{equation*}
		which improves upon the term-wise bound by a factor of \((N\eta)^{-1/2}\) at the expense of replacing \(1+\phi_k\) by \(1+\sqrt{\phi_1+\cdots+\phi_{2k}}\).
	\end{lemma}
	The proof of the above \emph{Ward lemma} is largely based on yet another more general estimate.
	\begin{lemma}\label{lemma general chains}
		Let \(\vx,\vy\) be normalised vectors, let \(w_1,\ldots,w_{k+1} \in \mathbf{D}_{\ell+1}^{(\epsilon_0, \kappa_0)}\) be spectral parameters and \(A_{1},\ldots,A_{k}\) be deterministic matrices of bounded norm such that \(a\) of them are regular, i.e.\ \(\mathring{A}_i^{w_i,w_{i+1}}=A_i\) for all \(i \in \mathcal{I} \) for some \(\mathcal{I} \subset [k]\) of cardinality $a$. Then with \(G_i=G(w_i)\) it holds that
		\begin{equation}\label{eq all but one regular}
			\abs*{(G_1 A_1 G_2\cdots A_k G_{k+1})_{\vx\vy}} \prec \frac{1}{\eta^{k-a/2}}\Bigl(1+\phi_1 + \cdots+\phi_a\Bigr)\,.
		\end{equation}
	\end{lemma}
	We defer the proof of Lemma \ref{lemma general chains} to the end of this section.
	\begin{proof}[Proof of Lemma \ref{lemma Ward}]
		By Cauchy-Schwarz and the norm bound on the middle \(A_k\) we have
		\begin{equation*}
			\begin{split}
				&\Bigl(\frac{1}{N}\sum_a \abs*{(G_1 \mathring{A}_1^{w_1,w_2}\cdots \mathring{A}_{k-1}^{w_{k-1},w_k}G_k A_k )_{\vx a}}\Bigr)^2\\
				&\quad\lesssim \frac{1}{N} \Bigl(G_1 \mathring{A}_1^{w_1,w_2}\cdots \mathring{A}_{k-1}^{w_{k-1},w_k} G_kG_k^* \mathring{A}_{k-1}^{\bar{w}_{k},\bar{w}_{k-1}} \cdots \mathring{A}_1^{\bar{w}_2,\bar{w}_1)}G_1^* \Bigr)_{\vx\vx}\\
				&\quad\prec \frac{1}{N\eta^k}  \Bigl(1+\phi_1+\cdots+\phi_{2k}\Bigr)
			\end{split}
		\end{equation*}
		due to Lemma~\ref{lemma general chains} for \(2k\) resolvents and \(a=2k-2\) regularised \(A\)-matrices.
	\end{proof}
	The rest of the proof is split into several cases.
	\\[2mm]
	\emph{Treatment of \eqref{eq:higherorder1av} and \eqref{eq:higherorder2av} for \(\abs{\bm l}+\sum J=2\):}
	For the case \(\abs{\bm l}+\sum J=2\) we either have \(\abs{\bm l}\in\set{0,2}\) or \(\sum J=1=\abs{J}\). In the former case an off-diagonal resolvent is guaranteed to be present in the first factor of~\eqref{Xi1av def} (by parity) and in the latter case the second factor consists of a single off-diagonal resolvent chain. In either case we may use Lemma~\ref{lemma Ward} to gain a factor of \(1/\sqrt{N\eta}\) compared to~\eqref{eq naive} and obtain
	\begin{equation}\label{Xi av improved}
		\begin{split}
			\abs{\Xi_1^{\rm av}} &\prec (N\eta^{1/2})^{-1-\abs{J}}(1+\phi_1)^{(\abs{J}-1)_+}\Bigl(1+\phi_1+\bm1(\abs{J}\ge 1)\phi_2^{1/2}\Bigr)\,,\\
			\abs{\Xi_2^{\rm av}} &\prec (N\eta)^{-1-\abs{J}}(1+\phi_1^2+\phi_2)^{(\abs{J}-1)_+}\Bigl(1 + \phi_1^3 + \phi_2^{3/2} + \bm1(\abs{J}\ge 1) (\phi_3+\phi_4)^{3/4}\Bigr)\,,
		\end{split}
	\end{equation}
	where we used the fact that for \(\abs{J}=0\) only a single factor of \((1+\phi_1)\) needs to be replaced by a factor of \((1+(\phi_1+\phi_2)^{1/2})\) for \(\Xi_2^{\rm av}\) and no factor needs to be replaced for \(\Xi_1^{\rm av}\). Moreover, we used \(\phi_1(\phi_3+\phi_4)^{1/2}+\phi_1^2\phi_2^{1/2}\lesssim \phi_1^3+\phi_2^{3/2}+(\phi_3+\phi_4)^{3/4}\) by a simple Young inequality. Now~\eqref{Xi av improved} implies~\eqref{eq:higherorder1av} and \eqref{eq:higherorder2av} by another simple Young inequality.
	\\[2mm]
	\emph{Treatment of \eqref{eq:higherorder1iso} and \eqref{eq:higherorder2iso} for \(\abs{\bm l}+\sum J-\abs{J}\in\set{2,3}\):}     In this case we can simply use Lemma \ref{lemma Ward} for the two resolvent chains in the first factor of~\eqref{Xi1iso def} involving \(\vx,\vy\) to gain a factor of \((N\eta)^{-1}\) compared to~\eqref{eq naive} at the expense of replacing \(1+\phi_1\) by \(1+\phi_1^{1/2}+\phi_2^{1/2}\) in case of \(\Xi_2^\iso\) which proves \eqref{eq:higherorder1iso} and \eqref{eq:higherorder2iso} in this case.
	\\[2mm]
	\emph{Treatment of \eqref{eq:higherorder1iso} and \eqref{eq:higherorder2iso} for \(\abs{\bm l}+\sum J-\abs{J}=0\):}
	In this case we necessarily have \(\abs{\bm l}=0\) and \(\abs{J}\ge 2\) and \(\abs{\bm j}=1\) for all \(\bm j\in J\). In particular all factors of~\eqref{Xi1iso def} consist of two resolvent chains evaluated in \((\bm x,a),(\bm y,b)\) or \((\bm x,b),(\bm y,a)\), respectively. This allows to use~Lemma \ref{lemma Ward} four times (twice for the \(a\)- and twice for the \(b\)-summation) to gain a factor of \((N\eta)^{-2}\) compared to~\eqref{eq naive} at the expense of replacing
	\begin{equation*}
		\text{one factor of}\quad (1+\phi_1)\quad\text{by}\quad (1+(\phi_1+\phi_2)^{1/2})
	\end{equation*}
	in case of \(\Xi_1^\iso\) and
	\begin{equation}\label{Xi 1 repl}
		\text{one factor of}\quad (1+\phi_1)(1+\phi_1^2+\phi_2)\quad\text{by}\quad (1+(\phi_1+\phi_2)^{1/2})(1+\phi_1+\phi_2+(\phi_3+\phi_4)^{1/2})
	\end{equation}
	in case of \(\Xi_2^\iso\). This concludes the proof in case of \(\Xi_1^\iso\) and together with
	\begin{equation*}
		(1+(\phi_1+\phi_2)^{1/2})(1+\phi_1+\phi_2+(\phi_3+\phi_4)^{1/2}) \lesssim 1+(\phi_1+\phi_2)^{3/2} + (\phi_3+\phi_4)^{3/4}
	\end{equation*}
	also in case of \(\Xi_2^\iso\).
	\\[2mm]
	\emph{Treatment of \eqref{eq:higherorder1iso} and \eqref{eq:higherorder2iso} for \(\abs{\bm l}+\sum J-\abs{J}=1\):}
	In this case we necessarily have \(\abs{J}\ge 1\) and either \(\abs{\bm l}=0\) or \(\abs{\bm j}=1\) for all \(\bm j\in J\). In either case we can use Lemma~\ref{lemma Ward} twice for the first factor and once for some other factor in~\eqref{Xi1iso def} to gain a factor of \((N\eta)^{-3/2}\) compared to~\eqref{eq naive} at the expense of replacing~\eqref{Xi 1 repl} in case of \(\Xi_1^\iso\) and
	\begin{equation*}
		\text{one factor of}\quad (1+\phi_1)(1+\phi_1^2+\phi_2)\quad\text{by}\quad (1+(\phi_1+\phi_2)^{1/2})((1+\phi_1)(1+\phi_1+\phi_2)^{1/2}+(\phi_3+\phi_4)^{1/2})
	\end{equation*}
	in case of \(\Xi_2^\iso\). Together with
	\begin{equation*}
		(1+(\phi_1+\phi_2)^{1/2})((1+\phi_1)(1+\phi_1+\phi_2)^{1/2}+(\phi_3+\phi_4)^{1/2}) \lesssim 1+(\phi_3+\phi_4)^{3/4} + \phi_2^{3/2} + \phi_1^2
	\end{equation*}
	this concludes the proof also in this case.
\end{proof}
It remains to give the proof of Lemma \ref{lemma general chains}.
\begin{proof}[Proof of Lemma \ref{lemma general chains}]
	The proof is via induction, i.e.\ we assume that~\eqref{eq all but one regular} has been established for resolvent chains of up to \(k\) resolvents. For \(k+1\) resolvents and \(a=k\), i.e.\ in case when all deterministic matrices are regular, the claim follow by definition of \(\psi_k^\iso\). Therefore we may assume that some \(A_j\) is not regular which we decompose into its regular component \(\mathring{A}_j^{w_j,w_{j+1}}\) and a linear combination of \(E_\pm\). By linearity it thus suffices to check~\eqref{eq all but one regular} for the cases \(A_j=E_\pm\), and moreover, by chiral symmetry \(G_j E_-G_{j+1}=-E_- G(-w_j) E_+ G_{j+1}\) and \(\mathring{A}^{w_{j-1},w_j} E_-=\mathring{A}^{w_{j-1},-w_j} \) (recall Lemma \ref{lem:regularbasic}) the estimate for \(E_-\) follows from the estimate for \(E_+\) upon replacing \(w_j\) by \(-w_{j}\). Therefore it suffices to check~\eqref{eq all but one regular} in case \(A_j=E_+\).

	If \(\mathfrak{s}_j = -\sgn (\Im w_j\Im w_{j+1}) = +\), i.e.\ the adjacent spectral parameters lie in opposite half-planes, then we use the {resolvent identity \eqref{eq:resolid}} to write
	\begin{equation*}
		A_{j-1}G_j E_+ G_{j+1}A_{j+1} G_{j+2} = A_{j-1}\frac{G_{j}-G_{j+1}}{w_{j}-w_{j+1}} A_{j+1} G_{j+2}\,.
	\end{equation*}
	We discuss each of the two resulting summands separately. For the summand involving $G_{j+1}$, if \(A_{j-1}\) was not counted as regularised, i.e.\ \(j-1\not\in  \mathcal{I}\), then the claim follows by induction and the trivial estimate \(\abs{w_j-w_{j+1}}\ge \eta\) since \(k\) has been reduced by one, while \(a\) has been preserved. On the other hand, if \(A_{j-1} \) was correctly regularised, then we use Lemma~\ref{lem:regularbasic} to write
	\begin{equation}\label{Aj-1}
		\mathring{A}_{j-1}^{w_{j-1}, w_j}  = \mathring{A}_{j-1}^{w_{j-1},\bar{w}_{j}}= \mathring{A}^{w_{j-1},w_{j+1}}_{j-1} + \mathcal{O}(\abs{\bar{w}_j-w_{j+1}}) E_+  + \mathcal{O}(\abs{\bar{w}_j-w_{j+1}}) E_-\,.
	\end{equation}
	Inserting~\eqref{Aj-1} into \(A_{j-1}G_{j+1}A_{j+1}G_{j+2}/(w_j-w_{j+1})\) the claimed bound follows from induction since for the \(\mathring{A}_{j-1}^{w_{j-1},w_{j+1}}\)-term \(a\) has been preserved and \(k\) has been reduced by one compensating for \(\abs{w_j-w_{j+1}}\ge \eta\), while for \(E_\pm\) both \(k,a\) have been reduced by one and \(\abs{\bar{w}_j-w_{j+1}}/\abs{w_j-w_{j+1}}\le 1\). Next, for the summand involving $G_j$, the argument is completely analogous, apart from the two error terms in
	\begin{align}\label{Aj+1}
		\mathring{A}^{w_j, w_{j+1}}_{j+1}= \mathring{A}_{j+1}^{w_{j},w_{j+2}} & + \mathcal{O}(|w_j - \bar{w}_{j+1}| + |{w}_{j}- \mathfrak{s}_{j+1}{w}_{j+2}|) E_{\mathfrak{s}_{j+1}}                   \\
		                                                                      & +  \mathcal{O}(|w_j - \bar{w}_{j+1}| +|{w}_{j}+ \mathfrak{s}_{j+1}\bar{w}_{j+2}|) E_{-\mathfrak{s}_{j+1}} \nonumber\,,
	\end{align}
	appearing for an $A_{j+1} = \mathring{A}_{j+1}^{w_{j+1}, w_{j+2}}$, which has been correctly regularised. Here, we applied Lemma~\ref{lem:regularbasic} and denoted, as usual, $\mathfrak{s}_{j+1} = - \sgn(\Im w_{j+1} \Im w_{j+2})$. Now, for the error terms, we assume that the second summand in each $\mathcal{O}(...)$ is non-zero (otherwise we are back to \eqref{Aj-1}) and argue by induction: Indeed, using \eqref{eq:chiral} and applying a {resolvent identity \eqref{eq:resolid}}, we find
	\begin{align} \label{threeenergies}
		 & \frac{|w_j - \bar{w}_{j+1}| + |{w}_{j}- \mathfrak{s}_{j+1}{w}_{j+2}|}{w_j - w_{j+1}}   G_j E_{\mathfrak{s}_{j+1}} G_{j+2}                                                                                                                               \\
		 & \hspace{2cm}=  \frac{|w_j - \bar{w}_{j+1}| + |{w}_{j}- \mathfrak{s}_{j+1}{w}_{j+2}|}{(w_j - w_{j+1})\, ({w}_{j}- \mathfrak{s}_{j+1}{w}_{j+2})} \mathfrak{s}_{j+1}  \big(G(w_j) - G(\mathfrak{s}_{j+1} w_{j+2})\big) E_{\mathfrak{s}_{j+1}}\,, \nonumber
	\end{align}
	such that, in the resulting chain we have reduced $k$ by \emph{two} and $a$ by one, and the prefactor in \eqref{threeenergies} is bounded by $1/\eta$. The argument for the second error in \eqref{Aj+1} is completely analogous, after realizing that $\big(|w_j - \bar{w}_{j+1}| + |{w}_{j}+ \mathfrak{s}_{j+1}\bar{w}_{j+2}|\big)/\big(\vert w_j - w_{j+1}\vert \, \vert {w}_{j}+  \mathfrak{s}_{j+1}{w}_{j+2}\vert \big) \le 1/\eta$.


	On the contrary, if \(\mathfrak{s}_j = -\sgn (\Im w_j\Im w_{j+1}) = -\), i.e.\ the adjacent spectral parameters lie the same half-plane (without loss of generality the upper one), then we use the integral representation from Lemma~\ref{lem:intrepG^2} to write
	\begin{equation}\label{Gj E+ Gj+1 opp}
		A_{j-1}G_j E_+ G_{j+1}A_{j+1} =\frac{1}{2\pi\ii} \int_{\Gamma} \frac{A_{j-1}G(z)A_{j+1}}{(z-w_j)(z-w_{j+1})}\dif z\,,
	\end{equation}
	where $\Gamma$ is an appropriately chosen contour.
	If \(j-1,j+1\not\in \mathcal{I}\), i.e.\ both \(A_{j-1},A_{j+1}\) were not counted as regularised, then the claim follows by induction and estimating the integral by \(\eta^{-1}\) (up to log factors) since \(k\) has been reduced by one, and \(a\) has been preserved. On the other hand, if both \(A_{j-1},A_{j+1}\) were counted as regularised, then we use Lemma~\ref{lem:regularbasic} to write them
	as
	\begin{equation}\label{Ajpm1 exp}
		\begin{split}
			\mathring{A}^{w_{j-1},w_j}_{j-1}& = \mathring{A}_{j-1}^{w_{j-1},z} + \mathcal{O}(\abs{w_j-z}) E_++ \mathcal{O}(\abs{w_j-z}) E_- \,, \\[1mm]
			\mathring{A}_{j+1}^{w_{j+1}, w_{j+2}}& = \mathring{A}_{j+1}^{z,w_{j+2}} + \mathcal{O}(\abs{w_{j+1}-z}) E_+ + + \mathcal{O}(\abs{w_{j+1}-z}) E_-\,.
		\end{split}
	\end{equation}
	The resulting term with \(\mathring{A}^{w_{j-1},z}_{j-1},\mathring{A}_j^{z,w_{j+2}}\) can be estimated by induction since \(k\) has been reduced by one, \(a\) has been preserved and the integral may be estimated by \(\eta^{-1}\). The other terms with either one or two \(E_\pm\) can also be estimated by induction since the integral is at most logarithmically divergent, \(k\) has been reduced by one and \(a\) by at most two. Finally, if in~\eqref{Gj E+ Gj+1 opp} one of \(A_{j-1},A_{j+1}\) were counted as regularised, then we use the relevant expansion from~\eqref{Ajpm1 exp}, so that for the resulting term with \(\mathring{A}\), \(k\) has been reduced by one, and \(a\) has been preserved, so that the \(\eta^{-1}\) estimate on the integral is affordable. The other term with \(E_\pm\) can also be estimated by induction with both \(a,k\) reduced by one, and the integral being at most logarithmically divergent. This concludes the proof.
\end{proof}

\section{Proof of the reduction inequalities, Lemma \ref{lem:reduction}} \label{sec:proofreduc}

During the proof of Lemma \ref{lem:reduction}, we will heavily rely on the following integral representation for the absolute value $|G|$ of a resolvent (see also \cite[Lemma~5.1]{multiG}).
\begin{lemma} {\rm (Integral representation for the absolute value of a resolvent)} \\
	Let $w = e + \I \eta \in \C \setminus \R$. Then the absolute value of the resolvent $G(w)$ can be represented as
	\begin{equation} \label{eq:intrep|G|}
		|G(e+\I \eta)| = \frac{2}{\pi}\int_{0}^{\infty} \Im G (e + \I \sqrt{\eta^2 + s^2}) \frac{\D s}{\sqrt{\eta^2 + s^2}}\,.
	\end{equation}
\end{lemma}
\begin{proof}
	This immediately follows from the functional calculus for $H$ and the identity
	\begin{equation*}
		\frac{1}{|x - \I \eta|} = \frac{1}{\I \pi} \int_{0}^{\infty} \left( \frac{1}{x - \I (\eta^2 + s^2)^{1/2}}  - \frac{1}{x + \I (\eta^2 + s^2)^{1/2}}\right)\frac{\D s}{\sqrt{\eta^2 + s^2}}\,. \qedhere
	\end{equation*}
\end{proof}
\begin{proof}[Proof of Lemma \ref{lem:reduction}]

	To keep the notation simpler within this proof we may often denote
	\[
		A_i=\mathring{A}_i=\mathring{A}_i^{w_i,w_{i+1}}\,, 
	\]
	i.e. 
	sometimes we
	drop the spectral parameters $w_i = e_i+ \ii \eta_i$.


	We start with the proof of \eqref{eq:reduction av}, for which, similarly to \cite[Lemma 3.6]{multiG}, we get
	\begin{equation}
		\label{eq:avred}
		\Psi_4^{\rm av} \lesssim N \eta + N^2 \eta^2 \, \Big(\langle |G_1| A_1 |G_2|A_1^* \rangle \langle |G_2| A_2 |G_3|A_2^* \rangle \langle |G_3| A_3 |G_4|A_3^* \rangle \langle |G_4| A_4 |G_1|A_4^* \rangle\Big)^{1/2}\,,
	\end{equation}
	by Lemma \ref{lem:Mbound}, spectral decomposition, and a Schwarz inequality. Next, we use \eqref{eq:intrep|G|} to write
	\begin{equation}
		\label{eq:absexp}
		\langle |G_1| A_1 |G_2|A_1^* \rangle=\frac{4}{\pi^2}\iint_0^\infty \langle \Im G(w_{1,s})\mathring{A}_1^{w_1,w_2} \Im G(w_{2,t})(\mathring{A}_1^{w_1,w_2})^*  \rangle\, \frac{\D s \D t}{\sqrt{\eta_1^2+s^2}\sqrt{\eta_2^2+t^2}}\,,
	\end{equation}
	where we defined $w_{i,s}:=e_i+\ii \sqrt{\eta_i^2+s^2}$. The very large $s,t$--regimes in \eqref{eq:absexp} can be easily shown to be negligible (e.g. see \cite[Proof of Lemma~5.1]{multiG}), i.e. even if not stated explicitly we assume that the upper integration limit can be replaced by $N^{100}$. Additionally, we can restrict to the case when $\eta:=\min_j|\Im w_j|\le 1$, when this is not the case we use the local law in the regime $\eta > 1$ from Theorems~\ref{thm:multiGll}--\ref{thm:singleGopt} (see \cite[Proof of Lemma~5.1]{multiG} for a detailed argument). We remark that this argument is not circular since in the proof of the local law for $\eta> 1$ sketched below Remark \ref{rmk:sqrteta} one does not use the reduction inequalities in \eqref{eq:reduction av}--\eqref{eq:reduction iso}.

	In order to estimate the rhs. of \eqref{eq:absexp} we write $ \Im G=\frac{1}{2\ii} (G-G^*)$
	for both $\Im G$ to obtain four terms
	with two resolvents;
	to keep the presentation concise we only present the estimate for one of them.
	From now on we consider only the term $\langle |G_1| A_1 |G_2|A_1^* \rangle$, the bound for all the other terms in the last line of \eqref{eq:avred} is completely analogous and so omitted. In the following we will often use the approximations from Lemma~\ref{lem:regularbasic} (omitting the trivial $\wedge 1$ in the errors for notational simplicity):
	\begin{equation}
		\label{eq:usefrel}
		\begin{split}
			\mathring{A}^{w_1,w_2}&=\mathring{A}^{w_{1,s},w_{2,t}}+\mathcal{O}(|\sqrt{\eta_1^2+s^2}-\eta_1|+|\sqrt{\eta_2^2+t^2}-\eta_2|)E_+ \\
			&\quad+\mathcal{O}(|\sqrt{\eta_1^2+s^2}-\eta_1|+|\sqrt{\eta_2^2+t^2}-\eta_2|)E_- \,, \\[1mm]
			(\mathring{A}^{w_1,w_2})^*&=(\mathring{A}^*)^{w_{2,t},w_{1,s}} +\mathcal{O}(|e_1-e_2|+\sqrt{\eta_1^2+s^2}+\sqrt{\eta_2^2+t^2})E_+ \\
			&\quad+\mathcal{O}(|e_1+e_2|+\sqrt{\eta_1^2+s^2}+\sqrt{\eta_2^2+t^2})E_-\,.
		\end{split}
	\end{equation}
	We point out that when taking the adjoint of the first formula to arrive at the second we used that for any
	$w_1,w_2$ it holds $(\mathring{A}^{w_1,w_2})^*=(\mathring{A}^*)^{\overline{w_2},\overline{w_1}}$, see Lemma~\ref{lem:regularbasic}.
	Recall that within this proof we always assume that $\eta\le 1$. From now on for the error terms we will always use the bounds
	\begin{equation}
		\label{eq:enoughb}
		|\sqrt{\eta_1^2+s^2}-\eta_1|\lesssim s\,, \qquad \sqrt{\eta_1^2+s^2}\le\eta_1+s\,,
	\end{equation}
	and a similar bound with $\eta_1,s$ replaced with $\eta_2,t$. The first bound is not optimal for small $\eta_1$, but good enough for our estimates. Then using \eqref{eq:usefrel} we write 
	\begin{equation}
		\begin{split}
			\label{eq:bigsplit}
			&\iint_0^\infty \langle G(w_{1,s})\mathring{A}_1^{w_1,w_2} G(w_{2,t})(\mathring{A}_1^{w_1,w_2})^*  \rangle\, \frac{\D s  \D t}{\sqrt{\eta_1^2+s^2}\sqrt{\eta_2^2+t^2}} \\
			&=\iint_0^\infty \langle  G(w_{1,s})\mathring{A}_1^{w_{1,s},w_{2,t}} G(w_{2,t})(\mathring{A}_1^*)^{w_{2,t},w_{1,s}}  \rangle\, \frac{\D s  \D t}{\sqrt{\eta_1^2+s^2}\sqrt{\eta_2^2+t^2}} \\
			&+ \sum_{\sigma\in \{+,-\}}\iint_0^\infty \langle  G(w_{1,s})E_\sigma G(w_{2,t})(\mathring{A}_1^*)^{w_{2,t},w_{1,s}}  \rangle \mathcal{O}(\eta_1+\eta_2+s+t)\, \frac{\D s  \D t}{\sqrt{\eta_1^2+s^2}\sqrt{\eta_2^2+t^2}} \\
			&+ \sum_{\sigma\in \{+,-\}}\iint_0^\infty \langle  G(w_{1,s})\mathring{A}_1^{w_{1,s},w_{2,t}} G(w_{2,t})E_\sigma  \rangle \mathcal{O}(\eta_1+\eta_2+s+t)\, \frac{\D s  \D t}{\sqrt{\eta_1^2+s^2}\sqrt{\eta_2^2+t^2}} \\
			&+ \sum_{\sigma,\tau\in \{+,-\}}\iint_0^\infty \langle  G(w_{1,s})E_\sigma G(w_{2,t})E_\tau \rangle \mathcal{O}(\eta_1^2+\eta_2^2+s^2+t^2)\, \frac{\D s \D t}{\sqrt{\eta_1^2+s^2}\sqrt{\eta_2^2+t^2}} \\
			&+ \iint_0^\infty \langle  G(w_{1,s})\big[\sum_\sigma\mathcal{O}(|e_1-\sigma e_2|)E_\sigma\big]G(w_{2,t})(\mathring{A}_1^*)^{w_{2,t},w_{1,s}} \rangle \, \frac{\D s  \D t}{\sqrt{\eta_1^2+s^2}\sqrt{\eta_2^2+t^2}} \\
			&+ \iint_0^\infty \langle  G(w_{1,s})\mathring{A}_1^{w_{1,s},w_{2,t}} G(w_{2,t})[\mathcal{O}(|e_1-e_2|)E_++\mathcal{O}(|e_1+e_2|)E_-] \rangle\, \frac{\D s  \D t}{\sqrt{\eta_1^2+s^2}\sqrt{\eta_2^2+t^2}} \\
			&+ \iint_0^\infty \langle  G(w_{1,s})\big[\sum_\sigma\mathcal{O}(|e_1-\sigma e_2|)E_\sigma\big] G(w_{2,t})\big[\sum_\tau\mathcal{O}(|e_1-\tau e_2|)E_\tau\big]\rangle  \, \frac{\D s \D t}{\sqrt{\eta_1^2+s^2}\sqrt{\eta_2^2+t^2}}\,.
		\end{split}
	\end{equation}

	We now estimate the terms in the rhs. of \eqref{eq:bigsplit} one by one. In the following estimates we will always omit $\log N$-factors. We start with
	\begin{equation*}
		\left|\iint_0^\infty \langle  G(w_{1,s})\mathring{A}_1^{w_{1,s},w_{2,s}} G(w_{2,t})(\mathring{A}_1^*)^{w_{2,t},w_{1,s}}  \rangle\, \frac{\D s  \D t}{\sqrt{\eta_1^2+s^2}\sqrt{\eta_2^2+t^2}}\right|\prec 1+\frac{\psi_2^{\rm av}}{N\eta}\,,
	\end{equation*}
	which readily follows by the definition of $\Psi_2^{\rm av}$ in \eqref{eq:Psi avk} and from the assumption  $\Psi_2^{\rm av}\prec\psi_2^{\rm av}$. For the third to the fifth line in \eqref{eq:bigsplit} we use the bound
	\begin{equation}
		\begin{split}
			\label{eq:step2}
			&\left| \iint_0^\infty \langle  G(w_{1,s}) E_\sigma G(w_{2,t})B \rangle \mathcal{O}(\eta_1+\eta_2+s+t)\, \frac{\D s  \D t}{\sqrt{\eta_1^2+s^2}\sqrt{\eta_2^2+t^2}}\right| \\
			&\qquad\qquad\quad\prec \iint_0^\infty\left(\frac{1}{\sqrt{\eta_1^2+s^2}}\wedge\frac{1}{\sqrt{\eta_2^2+t^2}}\right)\big[\eta_1+\eta_2+s+t\big]\, \frac{\D s  \D t}{\sqrt{\eta_1^2+s^2}\sqrt{\eta_2^2+t^2}}\lesssim 1\,,
		\end{split}
	\end{equation}
	for any deterministic norm bounded matrices $B$ and for $\sigma\in \{+,-\}$.
	For the fifth line of  \eqref{eq:bigsplit} we used the bound $(s^2+t^2)\wedge 1\le (s+t)\wedge 1$ (recall that $\wedge 1$ is omitted in the error terms in \eqref{eq:bigsplit} for notational simplicity). Note that here we used:
	\begin{equation}
		\label{eq:needbne}
		|\langle G(w_{1,s})E_\sigma G(w_{2,t})B\rangle|\prec  \frac{1}{\sqrt{\eta_1^2+s^2}}\wedge\frac{1}{\sqrt{\eta_2^2+t^2}}\,,
	\end{equation}
	which holds uniformly in matrices with $\lVert B\rVert\lesssim 1$. We point out that to obtain the bound \eqref{eq:needbne} we used spectral decomposition of the resolvents and that $\langle \boldsymbol{w}_{i}, E_\sigma  \boldsymbol{w}_{j}\rangle = \delta_{i, \sigma j}$ to bound\footnote{{We point out that here $\boldsymbol{w}_i$ denotes an eigenvector of $H^\Lambda$, and it should not be confused with the spectral parameters $w_{1,s},w_{2,t}$.}}
	\[
		\begin{split}
			|\langle G(w_{1,s})E_\sigma G(w_{2,t})B\rangle|&=\left|\frac{1}{2N}\sum_i \frac{\langle \boldsymbol{w}_i,B \boldsymbol{w}_{\sigma i} \rangle}{(\lambda_i-w_{1,s})(\lambda_i-\sigma w_{2,t})}\right| \\
			&\lesssim \frac{1}{N}\sum_i \frac{1}{|\lambda_i-w_{1,s}||\lambda_i-\sigma w_{2,t}|}\\
			&\prec \frac{1}{|\Im w_{1,s}|\vee |\Im w_{2,t}| }\,,
		\end{split}
	\]
	where in the last inequality we used the single resolvent local law. 



	Finally, for the last three lines in \eqref{eq:bigsplit} we use that for any norm bounded matrix $B$, by {resolvent identity \eqref{eq:resolid}}, we have (recall that $E_+=I$)
	\begin{equation}
		\label{eq:resid}
		| \langle  G(w_{1,s})B G(w_{2,t}) \rangle|\prec \frac{1}{|w_{1,s}-w_{2,t}|}\,, \qquad | \langle  G(w_{1,s})B G(w_{2,t})E_- \rangle|\prec \frac{1}{|w_{1,s}+w_{2,t}|}\,,
	\end{equation}
	which after the integration in \eqref{eq:bigsplit} gives a bound of order one, as a consequence of
	\[
		\frac{|e_1\pm e_2|}{|w_{1,s}\pm w_{2,t}|}\lesssim 1\,.
	\]
	Note that here it is important
	that the error terms in \eqref{eq:bigsplit}  involving $|e_1-e_2|$ are always multiplied with the matrix $E_+$, while
	errors of order $|e_1+e_2|$ are in the direction of $E_-$.

	Combining the computations in \eqref{eq:bigsplit}--\eqref{eq:resid} we conclude that
	\begin{equation}
		\label{eq:finfirststep}
		|\langle |G_1| A_1 |G_2|A_1^* \rangle|\prec 1+\frac{\psi_2^{\rm av}}{N\eta},
	\end{equation}
	which, after plugging it in the rhs. of \eqref{eq:avred}, clearly implies \eqref{eq:reduction av}\,.

	For \eqref{eq:reduction iso} for $\Psi_3^{\rm iso}$, we find
	\begin{equation}
		\label{eq:isored3}
		\Psi_3^{\rm iso} \lesssim \sqrt{N\eta} + N \eta^2 \Big( \big(G_1 A_1 |G_2| A_1^* G_1^*\big)_{\boldsymbol{x}\boldsymbol{x}} \big(G_4^* A_3^* |G_3| A_3 G_4\big)_{\boldsymbol{y}\boldsymbol{y}} \langle |G_2| A_2 |G_3| A_2^* \rangle \Big)^{1/2}\,,
	\end{equation}
	again by Lemma \ref{lem:Mbound}, spectral decomposition, and a Schwarz inequality. Then, using again the integral representation \eqref{eq:intrep|G|}, we find that
	\begin{equation*}
		\big(G_1 A_1 |G_2| A_1^* G_1^*\big)_{\boldsymbol{x}\boldsymbol{x}}=\frac{2}{\pi}\int_0^\infty \big(G_1 A_1 \Im G(w_{2,s}) A_1^* G_1^*\big)_{\boldsymbol{x}\boldsymbol{x}}\, \frac{\D s}{\sqrt{\eta_2^2+s^2}}\,,
	\end{equation*}
	recalling the notation $w_{2,s}=e_2+\ii \sqrt{\eta_2^2+s^2}$. The estimate for this term is fairly similar to the one in \eqref{eq:absexp}, hence we present only the main differences and skip the details; actually the current case is easier since we now have only one $|G|$.

	After splitting $\Im G = \frac{1}{2\ii} (G-G^*)$ and handling both terms separately,  we can write,
	similarly to \eqref{eq:bigsplit} and using \eqref{eq:usefrel}--\eqref{eq:enoughb}, the following approximation:
	\begin{equation}
		\begin{split}
			\label{eq:newdisplay}
			&\int_0^\infty \big(G_1 A_1  G(w_{2,s}) A_1^* G_1^*\big)_{\boldsymbol{x}\boldsymbol{x}}\, \frac{\D s}{\sqrt{\eta_2^2+s^2}} \\
			&=\int_0^\infty \big(G_1 \mathring{A}_1^{w_1,w_{2,s}}  G(w_{2,s}) (\mathring{A}_1^*)^{w_{2,s},w_1} G_1^*\big)_{\boldsymbol{x}\boldsymbol{x}}\, \frac{\D s}{\sqrt{\eta_2^2+s^2}}+\mathcal{E}\,.
		\end{split}
	\end{equation}
	Here $\mathcal{E}$ is an error coming from all the errors in \eqref{eq:usefrel}. For the first term in the second line of \eqref{eq:newdisplay} we use the bound
	\begin{equation}
		\label{eq:step1iso}
		\left|\int_0^\infty \big(G_1 \mathring{A}_1^{w_1,w_{2,s}}  G(w_{2,s}) (\mathring{A}_1^*)^{w_{2,s},w_1} G_1^*\big)_{\boldsymbol{x}\boldsymbol{x}}\, \frac{\D s}{\sqrt{\eta_2^2+s^2}}\right|\prec \frac{1}{\eta}\left(1+\frac{\psi_2^{\rm iso}}{\sqrt{N\eta}}\right)\,,
	\end{equation}
	which follows by the definition of $\Psi_2^{\rm iso}$. For the error term we do not write the details, since once we replace \eqref{eq:needbne}--\eqref{eq:resid} with (here $B,B_1,B_2$ are deterministic norm bounded matrices)
	\begin{equation}\label{triv}
		\begin{split}
			|\big(G_1 B_1 G(w_{2,s}) B_2 G_1^*\big)_{\boldsymbol{x}\boldsymbol{x}}|&\le \big(G_1 B_1B_1^*G_1^*\big)_{\boldsymbol{x}\boldsymbol{x}}^{1/2}\big(G_1B_2^*G(w_{2,s})G(w_{2,s})^* B_2 G_1^*\big)_{\boldsymbol{x}\boldsymbol{x}}^{1/2}\prec \frac{1}{\eta\sqrt{\eta_2^2+s^2}} \\
			|\big(G_1 E_\sigma G(w_{2,s}) B G_1^*\big)_{\boldsymbol{x}\boldsymbol{x}}|&\prec \frac{1}{\eta|w_1-w_{2,s}|}\,,
		\end{split}
	\end{equation}
	respectively, the estimate
	\begin{equation}
		\label{eq:esterrterm}
		|\mathcal{E}|\prec \frac{1}{\eta}
	\end{equation}
	follows completely analogously.
	The estimates~\eqref{triv} follow by repeated applications of the {resolvent identity \eqref{eq:resolid}}
	(after commuting $E_\sigma$ with $G$ in case of the second formula), the trivial bound $\| G\|\le 1/\eta$ and
	the single resolvent local law.
	Combining, \eqref{eq:step1iso}--\eqref{eq:esterrterm} we conclude
	\begin{equation}
		\label{eq:stepalliso}
		\big|\big(G_1 A_1 |G_2| A_1^* G_1^*\big)_{\boldsymbol{x}\boldsymbol{x}}\big|\prec \frac{1}{\eta}\left(1+\frac{\psi_2^{\rm iso}}{\sqrt{N\eta}}\right)\,.
	\end{equation}
	The bound in \eqref{eq:stepalliso}, together with \eqref{eq:finfirststep} to estimate the averaged term in \eqref{eq:isored3}, concludes the proof \eqref{eq:reduction iso} for $\Psi_3^{\rm iso}$.

	Analogously to \eqref{eq:isored3}, for $\Psi_4^{\rm iso}$ we find that
	\begin{equation*}
		\begin{split}
			\Psi_4^{\rm iso}&\lesssim \sqrt{N\eta} + N \eta^{5/2} \Big( \big(G_1 A_1 |G_2| A_1^* G_1^*\big)_{\boldsymbol{x}\boldsymbol{x}} \big(G_5^* A_4^* |G_4| A_4 G_5\big)_{\boldsymbol{y}\boldsymbol{y}} \langle |G_2| A_2 G_3 A_3 |G_4| A_3^* G_3^* A_2^* \rangle \Big)^{1/2} \\
			& \lesssim  \sqrt{N\eta} + N^{3/2} \eta^{5/2} \Big( \big(G_1 A_1 |G_2| A_1^* G_1^*\big)_{\boldsymbol{x}\boldsymbol{x}} \big(G_5^* A_4^* |G_4| A_4 G_5\big)_{\boldsymbol{y}\boldsymbol{y}}\Big)^{1/2} \\
			&\qquad \times\Big( \langle|G_2|A_2|G_3|A_2^*\rangle \langle|G_3|A_3|G_4|A_3^*\rangle \langle|G_4|A_3^*|G_3|A_3\rangle \langle|G_3|A_2^*|G_2|A_2\rangle \Big)^{1/4}
		\end{split}
	\end{equation*}
	where in the last inequality we used spectral decomposition and a bound as in \cite[Proof of Lemma~3.6]{multiG} to bound the trace with four $G$'s and four $A$'s in terms of a product of traces containing only two $G$'s and two $A$'s. Finally, using the bounds \eqref{eq:finfirststep}, \eqref{eq:stepalliso}, we conclude the proof of \eqref{eq:reduction iso} for $\Psi_4^{\rm iso}$ as well.
\end{proof}

\appendix
\section{Motivating derivations of the regularisation} \label{app:motivation}
In this appendix, we shall motivate and derive the regularisation \eqref{eq:circ} introduced in Definition \ref{def:reg obs1} by considering
two basic examples. We also use these  examples  to present two different approaches
to guess the right regularisation.  Before the details, we give an informal summary of these two model calculations.

First, in Section~\ref{subsec:variance}, we compute
\begin{equation} \label{eq:variance motivation}
	\E  \big| \langle \underline{ W G(\I \eta) A} \rangle \big|^2,
\end{equation}
which is the leading contribution to $\langle (G-M)B\rangle$ in the single-resolvent local law,
with $A=\mathcal{X}[B]M$, see~\eqref{eq:1G basic exp}. We will
show that, in order to be able to reduce its naive size $1/(N \eta)^2$ to the target $1/(N^2 \eta)$, we \emph{need} that $\langle A, V_\pm \rangle = 0$, i.e.~we need $A \in \C^{2N \times 2N}$ to be orthogonal to two certain directions $V_\pm$ in $\C^{2N \times 2N}$. For simplicity, we chose the spectral parameter $w = \I \eta$ to be on the imaginary axis, assuming that $0 \in \mathbf{B}_\kappa$ for some $\kappa > 0$. In this case, both cutoff functions \eqref{eq:case regulation2} in the actual definition of the regularisation satisfy $\mathbf{1}_\delta^\pm(\I \eta, \I \eta) = 0$ for $\eta> 0$ small enough. Hence, at least \emph{a posteriori}, we really catch both directions $V_\pm$ and not only one.
This  calculation is rather \emph{foundational} and unambiguously reveals two directions $V_\pm$, for which we need that $\langle A, V_\pm \rangle = 0$, in order to reduce the naive size of \eqref{eq:variance motivation}.

Second, in Section \ref{subsec:generalchain}, we consider the averaged chain with two resolvents
\begin{equation} \label{eq:chain motivation}
	\langle G^{\Defo_1}(w_1)A_1 G^{\Defo_2}(w_2)A_2 \rangle \,,
\end{equation}
where the resolvents are even allowed to have generally \emph{different}\footnote{All results in the current
	paper concern the $\Lambda_1=\Lambda_2$ case; the generalisation $\Lambda_1\ne\Lambda_2$ is mentioned only to stress that
	our method is  also valid beyond the scope of the current paper.} deformations, $\Defo_1$ and  $\Defo_2$.
Let $M_1:= M^{\Defo_1}(w_1)$ and $M_2:= M^{\Defo_2}(w_2)$.
For simplicity,
we will assume that 
the stability operators \begin{equation} \label{eq:stabopgeneral}
	{\mathcal{B}}_{m^{(*)}n^{(*)}} := 1 - M_m^{(*)} \mathcal{S}[\cdot] M_n^{(*)}\,, \quad m,n \in [2]\,,
\end{equation}
for all constellations of adjoints, have at most one \emph{critical} eigenvalue $\beta_{m^{(*)}n^{(*)}}$ which is \emph{not} of order one (with associated right and left eigenvectors $R_{m^{(*)}n^{(*)}}$ and $L_{m^{(*)}n^{(*)}}$, respectively, cf.~\eqref{eq:crit triple} later).
As we will show in Lemma~\ref{lem:eigendecomp}~(c), this is the case, e.g., if $\Defo \equiv \Defo_1 = \Defo_2$ and $\Re w_1, \Re w_2 \in \mathbf{B}_\kappa^\Defo$, and actually remains true for other more general random matrix models with a \emph{flat} \cite{firstcorr} self-energy operator $\mathcal{S}[\cdot]$. Recall that $\mathcal{S}[\cdot]$ is flat if
\begin{equation} \label{eq:flatness}
	c \langle R \rangle \le \mathcal{S}[R] \le C \langle R \rangle
\end{equation}
for some constants $c,C >0$ and any positive semi-definite matrix $R \ge 0$.

Again, the main question is what  special property $A_1, A_2$ must have so that~\eqref{eq:chain motivation} be
smaller than its naive size of order $1/\eta$ obtained from a simple Schwarz inequality.
Similarly to~\eqref{eq:variance motivation}, we could
directly compute the second moment of the corresponding
underline term (see Lemma \ref{lem:underlined3}), but for pedagogical reason we present
an alternative argument.
Quite \emph{pragmatically}, we
start the usual  proof via cumulant expansion
for a bound on \eqref{eq:chain motivation} and find that certain deterministic
terms are too big for general $A_1, A_2$.
We shall see that there exist two
matrices $\tilde{V}_\pm \in \C^{2N \times 2N}$ (which turn out to be certain
right eigenvectors $R_{m^{(*)}n^{(*)}}$ of \eqref{eq:stabopgeneral},
see \eqref{eq:Vchoice} and \eqref{eq:Vchoice2} later), such that, if $\langle A_i, \tilde{V}_\pm \rangle = 0$, these critical terms are smaller.
This suggests  a \emph{pragmatic ansatz}  of the form~\eqref{eq:circ} on the regularisation.
We will observe
that, for the situation $\Defo_1 = \Defo_2$ and $w_1 = w_2 = \I \eta$, the expressions for $\tilde{V}_\pm$ in fact \emph{coincide} with those for $V_\pm$ obtained in Section \ref{subsec:variance}. Notice that
the imaginary part of the single-resolvent setup leading to~\eqref{eq:variance motivation} 
is a special case of the two-resolvent setup~\eqref{eq:chain motivation} since
$$
	\Im \langle G B\rangle =  \langle (\Im G) B\rangle =   \eta\langle  G B G^* E_+\rangle
$$
for self-adjoint $B$. Hence the regularity of $B$ tested against $\Im G(\ii \eta)$
is the same as the regularity of $B$ between  $G(\ii\eta )$ and $G^*(\ii\eta)=G(-\ii \eta)$. This shows, at least in this special case,
that the \emph{foundational} and the \emph{pragmatic}
approaches lead to the same regularisation. Similar conclusion about the equivalence of both approaches
holds in general.

Finally, in Section \ref{subsec:exactform}, motivated by the previous tandem of \emph{foundational} and \emph{pragmatic}
computations in Sections \ref{subsec:variance} and \ref{subsec:generalchain},
respectively, we list generally valid (i.e.~for arbitrary $w_1, w_2$ also away from the imaginary axis)
explicit formulas for the directions $V_\pm$ regularising~\eqref{eq:chain motivation}  in case that $\Defo_1 = \Defo_2$.
These explicit formulas are identical to those used in the regularisation introduced in Definition \ref{def:reg obs1}.

\subsection{Variance calculation of \eqref{eq:variance motivation}} \label{subsec:variance} In the following, we simply write $G = G(\I \eta)$ for ease of notation. Then, using a cumulant expansion and neglecting cumulants of order at least three (or assuming that $X$ is Ginibre), one gets
\begin{align} \nonumber
	\mathbf{E} & \big\vert  \langle \underline{WG A} \rangle \big\vert^2  = \frac{1}{N} \sum_{ab} R_{ab} \mathbf{E} \langle \Delta^{ab} G {A} \rangle \partial_{ba} \langle \underline{{A}^* G^* W} \rangle                                                                        \\
	           & = \frac{1}{N} \sum_{ab} R_{ab} \mathbf{E} \langle \Delta^{ab} G {A} \rangle \langle  G {A}^* G^*  \Delta^{ba}\rangle \label{eq:cumex}                                                                                                                             \\
	           & \hspace{1cm}+ \frac{1}{N^2} \sum_{abcd} R_{ab} R_{cd} \mathbf{E}\langle \Delta^{ab} G \Delta^{dc} G {A} \rangle \langle {A}^* G^* \Delta^{ba} G^* \Delta^{cd} \rangle \nonumber                                                                                   \\
	           & = \frac{1}{N^2} \sum_\sigma \sigma  \mathbf{E}  \langle E_\sigma G {A} E_\sigma {A}^* G^* \rangle \, + \, \frac{1}{N^2} \sum_{\sigma \tau} \, \sigma \tau\mathbf{E}  \langle E_\sigma G^* E_\tau GA \rangle \langle E_\sigma G E_\tau (GA)^* \rangle\,. \nonumber
\end{align}
The rescaled cumulant $R_{ab} := N \kappa(ab, ba)$ has been introduced {above \eqref{eq:min exp 1av}} and $\Delta^{ab} \in \mathbf{C}^{2N \times 2N}$ contains only one non-zero entry at position $(a,b)$, i.e.~$(\Delta^{ab})_{cd} = \delta_{ac} \delta_{bd}$.

As we will show, the cumulant expansion \eqref{eq:cumex} yields that (up to a constant)
\begin{equation} \label{eq:cumexresult}
	\E \big\vert  \langle \underline{WG A} \rangle \big\vert^2 \approx \frac{\E \big| \langle \Im G A \rangle\big|^2}{(N\eta)^2} + \frac{\E \big| \langle \Im G A E_- \rangle\big|^2}{(N\eta)^2} + \mathcal{O}\left(\frac{1}{N^2\eta}\right).
\end{equation}
Indeed, the first summand in the last line of \eqref{eq:cumex} is estimated by $1/(N^2 \eta)$, the target size, with the aid of a trivial Schwarz inequality and a Ward identity using Theorem \ref{thm:singleG}. By writing out the summation in the last summand, we get in total four terms. Since their treatment is very similar, we focus on two exemplary terms with $\sigma = \tau = +$ (analogous to $\sigma = \tau = -$) and $\sigma = - \tau = -$ (analogous to $\sigma = -\tau = +$).

For the former, we apply a Ward identity and find it to be given by
\begin{equation} \label{eq:variance++}
	\frac{\E \big| \langle \Im G A \rangle\big|^2}{(N\eta)^2}\,,
\end{equation}
which, without any further information on $A$, using that $\langle GA \rangle \sim 1$ from Theorem \ref{thm:singleG}, is too big, compared to the targeted $1/(N^2 \eta)$-size. However, this drastically improves if $\langle \Im M, A \rangle = 0$ (recall that $\Im M $ is self adjoint): Since $\langle (G-M)A \rangle$ and $\langle \underline{WGA} \rangle$ are roughly of the same size (see \eqref{eq:1G basic exp}), the contribution \eqref{eq:variance++} basically becomes a lower-order correction. We have thus identified the first of the two directions $V_\pm$, to which $A$ has to be orthogonal to in order to reduce the naive size of \eqref{eq:variance motivation}, namely
\begin{equation} \label{eq:V+variance}
	V_+ = \alpha_+ \, \Im M \quad \text{for some non-zero} \quad \alpha_+ \in \C\,.
\end{equation}

The latter case, $\sigma = - \tau = -$, is slightly more involved due to the asymmetry of the two factors in the last summand in the last line of \eqref{eq:cumex}: For the first factor, again a Ward identity is sufficient. In the second factor, we use \eqref{eq:chiral} together with the integral representation \cite[Eq.~(3.14)]{multiG}
\begin{equation*}
	G^* G^* = \int_\R \frac{\Im G(x + \I \eta/2)}{(x + \I \eta/2)^2} \D x\,,
\end{equation*}
similar to Lemma \ref{lem:intrepG^2}, in the approximate form $G^* G^* \sim \Im G/\eta$. This follows (at least as an effective upper bound) by replacing the Cauchy kernel in the integral
\begin{equation*}
	| \langle G^* G^* E_- A^* \rangle | \le	\int_\R  \frac{| \langle \Im G(x + \I \eta/2) E_- A^* \rangle |}{x^2 + (\eta/2)^2} \D x \sim \frac{\Im G(\I \eta)}{\eta}
\end{equation*}
by a $\delta$-distribution.
Overall, this leaves us (roughly) with
\begin{equation}
	\label{eq:variance-+}
	\frac{\E \big| \langle \Im G A E_- \rangle\big|^2}{(N\eta)^2}
\end{equation}
for the second case. Hence, arguing for \eqref{eq:variance-+} completely analogous as done for \eqref{eq:variance++}, we find the second direction $V_-$, to which $A$ has to be orthogonal to, in order to reduce the naive size of \eqref{eq:variance motivation}, namely
\begin{equation} \label{eq:V-variance}
	V_- = \alpha_- \, \Im M E_- \quad \text{for some non-zero} \quad \alpha_- \in \C\,.
\end{equation}

We point out that the first term in \eqref{eq:cumexresult} would have worked in the exact same way also for spectral parameters $w = e + \I \eta$ with $e \neq 0$. However, the second direction $V_-$ would \emph{not} have been visible in this scenario, since the second term in \eqref{eq:cumexresult} would have been replaced by (at least for an upper bound)
\begin{equation*}
	\frac{\E \big| \langle \Im G(e+\I \eta) A E_- \rangle  \big|^2}{N^2 \eta \,  (|e| + \eta)} + \frac{\E  \big| \langle \Im G(e + \I \eta) A E_- \rangle \langle \Im G(-e + \I \eta) E_- A^* \rangle\big|}{N^2 \eta \,  (|e| + \eta)}\,.
\end{equation*}
\subsection{General structural regularisation in \eqref{eq:chain motivation}} \label{subsec:generalchain} We begin with the general rather \emph{structural} regularizing decomposition of a matrix $A$ (recall \eqref{eq:circ}), which shall be conducted as (dropping the tilde, which has been temporarily introduced below \eqref{eq:stabopgeneral})
\begin{equation} \label{eq:circapp}
	\boxed{A^\circ \equiv \mathring{A} := A - \langle V_+, A \rangle U_+ - \langle V_-, A \rangle U_- }
\end{equation}
for some $U_\sigma, V_\sigma \in \mathbf{C}^{2N \times 2N}$ to be determined but subject to the conditions $\langle V_\sigma, U_\tau \rangle = \delta_{\sigma, \tau}$ and $\langle U_\sigma, U_\sigma \rangle  = 1$. We point out, that the following calculations are largely
insensitive to the form of the self-energy operator $\mathcal{S}[\cdot]$  (but see Footnote~\ref{foot1})
and hence the conclusions for $U_\sigma$ and $V_\sigma$ derived in this section are valid beyond our concrete model
(up to the fact that, due to the chiral symmetry \eqref{eq:chiral}, the regularisation involves a \emph{two}-dimensional projection).

The goal of the present subsection is to \emph{show} that $V_\pm$ must be chosen as certain right eigenvectors $R_{m^{(*)}n^{(*)}}$ of \eqref{eq:stabopgeneral}. This follows by expanding \eqref{eq:chain motivation} and identifying several terms, whose size is too big for general deterministic matrices. Now, these terms can be neutralised, if $\langle A_i , R_{m^{(*)}n^{(*)}} \rangle = 0$ for certain right eigenvectors.  However, as already mentioned in Section \ref{sec:proofmain}, for the directions $U_\pm$ there are \emph{a priori} no further constraints or conditions (apart from orthogonality and normalisation). Hence, as it turns out to be convenient for our proofs, we will choose the matrices $U_\sigma$ in such a way, that a resolvent identity, i.e.~the transformation of a product into a difference,
\begin{equation*} 
	\boxed{G^{\Defo_1}(w_1) U_\sigma G^{\Defo_2}(w_2) \approx \big(G^{\Defo_1}(w_1) - G^{\Defo_2}(\sigma w_2)\big) U_\sigma\,,}
\end{equation*}
can be applied (here, the symbol `$\approx$' neglects lower order terms).
Finally, the condition $\langle V_\sigma, U_\tau \rangle = \delta_{\sigma, \tau}$ will guarantee that the regularisation is idempotent, i.e.~$(\mathring{A})^\circ = \mathring{A}$. Note that our general ansatz \eqref{eq:circapp} is restricted to the non-degenerate situation, where $U_\sigma$ and $V_\sigma$ are non-orthogonal, $\langle V_\sigma, U_\sigma \rangle \sim 1 $. This is guaranteed for our concrete model with deformations $\Defo_1 = \Defo_2$ (see Section \ref{subsec:exactform}) but requires some non-trivial arguments in more general cases.

Although the regularisation is inherently two-dimensional (at least for our model), we also define
\begin{equation*} 
	\mathring{A}^\sigma = A^{\circ_\sigma}:= A - \langle V_\sigma, A \rangle U_\sigma\,, \quad \sigma \in \{+,- \}\,,
\end{equation*}
and refer to $A^{\circ_\sigma}$ as the \emph{$\sigma$-regular component} (or \emph{$\sigma$-regularisation}) of $A$ and to $\langle V_\sigma, A \rangle U_\sigma$ as its \emph{$\sigma$-singular component}. Note that $(A^{\circ_+})^{\circ_-} = (A^{\circ_-})^{\circ_+} = \mathring{A}$, since $\langle V_\sigma, U_\tau \rangle = \delta_{\sigma, \tau}$.



As usual, we use the common notation $\eta_i := |\Im w_i|$ for $i \in [2]$ and abbreviate (see \eqref{eq:sign})
\begin{equation} \label{eq:signapp}
	\mathfrak{s} _i := - \sgn(\Im w_i \Im w_{i+1})\,, \quad i \in [2]\,,
\end{equation}
where the indices are understood cyclically modulo $2$ (cf.~Definition \ref{def:regobs}). This means that, in particular, $\mathfrak{s}_1 = \mathfrak{s}_2$ due to the short length of the chain \eqref{eq:chain motivation}. In the following, we will drop the arguments by writing, e.g.,~$M_1 = M^{\Defo_1}(w_1)$ and $G_2 = G^{\Defo_2}(w_2)$.
Moreover, we take $A_1 = \mathring{A}_1$ and $A_2 = \mathring{A}_2$ to be regular, i.e.~orthogonal to some yet to be specified $V_\pm$.

Now, by means of
\begin{equation*}
	G_1 = M_1 - M_1 \underline{W G_1} + M_1 \mathcal{S}[G_1-M_1] G_1\,,
\end{equation*}
we immediately find
\begin{equation*}
	G_1 A_1 G_2 = M_1 A_1 G_2 - M_1 \underline{W G_1} A_1 G_2 + M_1 \mathcal{S}[G_1 - M_1] G_1 A_1 G_2\,,
\end{equation*}
from which we conclude that
\begin{align*}
	{\mathcal{B}}_{12}[G_1 A_1 G_2] = \  & M_1 A_1 M_2 + M_1 A_1 (G_2 - M_2) - M_1 \underline{WG_1 A_1 G_2}                       \\
	                                     & + M_1 \mathcal{S}[G_1 - M_1] G_1 A_1 G_2 + M_1 \mathcal{S}[G_1 A_1 G_2] (G_2 - M_2)\,.
\end{align*}
This implies
\begin{align*}
	\langle (G_1 A_1 G_2 - M_{12}^{A_1}) A_2 \rangle  =  \  & \langle M_1 A_1 (G_2 - M_2) {\mathcal{X}}_{21}[A_2] \rangle  - \langle M_1 \underline{WG_1 A_1 G_2} {\mathcal{X}}_{21}[A_2] \rangle \\
	                                                        & + \langle M_1 \mathcal{S}[G_1 - M_1] G_1 A_1 G_2 {\mathcal{X}}_{21}[A_2] \rangle                                                    \\
	                                                        & + \langle M_1 \mathcal{S}[G_1 A_1 G_2] (G_2 - M_2) {\mathcal{X}}_{21}[A_2] \rangle
\end{align*}
where we defined
\begin{equation} \label{eq:det approx mot}
	M_{12}^{A_1}:= {\mathcal{B}}_{12}^{-1}[M_1 A_1 M_2] = M_1 \mathcal{X}_{12}[A_1] M_2 = M(w_1, A_1, w_2)
\end{equation}
(recall \eqref{eq:Mexample} and Definition \ref{def:Mdef}) and used the shorthand notation
\begin{equation*} 
	{\mathcal{X}}_{mn}[B] = \big(({\mathcal{B}}_{nm}^*)^{-1}[B^*]\big)^* = ({\mathcal{B}}_{m^*n^*}^{-1})^*[B]\,, \quad B \in \C^{2N \times 2N}\,.
\end{equation*}
The adjoint of ${\mathcal{B}}_{nm}$ is understood with respect to the standard (normalised) inner product $\langle S, T \rangle := \langle S^* T \rangle$ for  $S, T \in \mathbf{C}^{2N \times 2N}$, which is given by
\begin{equation} \label{eq:stabop adjoint}
	{\mathcal{B}}^* \equiv {\mathcal{B}}^*(w_1,w_2) [\cdot] := 1-\mathcal{S}[(M(w_1))^* \, \cdot \,  (M(w_2))^*]\,.
\end{equation}

So far, the regularisation of $A_1$ and $A_2$ has been rather \emph{structural}. To make it more concrete, we must allow $V_\sigma$ and $U_\sigma$ to be potentially different depending on which of the $A_i$ is regularised. In order to do so, we also temporarily
introduce the additional index $i$, referring to the considered $A_i$. That is, we will write $V_{\sigma, i}$ instead of $V_\sigma$.

The matrices $V_{\mathfrak{s}_i,i}$ (recall \eqref{eq:signapp} for the definition of $\mathfrak{s}_i$) shall be determined by requiring that
\begin{equation*}
	\Vert M_{12}^{A_1}\Vert  = \Vert {M_1 \mathcal{X}}_{12}[ A_1 ]M_2\Vert \lesssim \Vert A_1 \Vert  \quad  \text{for}\ \ i=1 \quad \text{and} \quad  \Vert {\mathcal{X}}_{21}[A_2] \Vert \lesssim \Vert A_2 \Vert \quad  \text{for} \ \ i=2\,,
\end{equation*}
meaning that the (adjoint of the) stability operator has a bounded inverse on regular observables (i.e.~subtracting the $\mathfrak{s}_i$-singular component amounts to removing the `bad direction' of the stability operators $ \mathcal{X}_{12}$ and $\mathcal{X}_{12}$, respectively). From this condition, we find the characterisation of $V_{\mathfrak{s}_1,1}$ and $V_{\mathfrak{s}_2,2}$, namely
\begin{equation} \label{eq:Vchoice}
	\boxed{ V_{\mathfrak{s}_1,1} = R_{1^*2^*} = (R_{21})^* \qquad \text{and} \qquad   V_{\mathfrak{s}_2,2} = R_{2^* 1^*} = (R_{12})^*\,,}
\end{equation}
up to a normalisation constant, which can be specified only after determining $U_\sigma$ (recall that $\langle V_\sigma, U_\tau \rangle = \delta_{\sigma, \tau}$ and $\langle U_\sigma, U_\sigma \rangle  = 1$). Recall from \eqref{eq:stabopgeneral}, that we denote by $R_{m^{(*)}n^{(*)}}$ and $L_{m^{(*)}n^{(*)}}$ the (normalised) right and left eigenvectors of ${\mathcal{B}}_{m^{(*)}n^{(*)}}$ corresponding to the (potentially) \emph{critical eigenvalue} ${\beta}_{m^{(*)}n^{(*)}}$.

Indeed, in order to verify that \eqref{eq:Vchoice} is the right choice for $V_{\mathfrak{s}_i,i}$, we use the decomposition
\begin{equation} \label{decompX}
	\mathcal{X}_{mn} = 	({\mathcal{B}}_{m^{*}n^{*}}^{-1})^* = \frac{1}{\bar{\beta}_{m^{*}n^{*}}} \ket{L_{m^{*}n^{*}}} \bra{R_{m^{*}n^{*}}} + \mathcal{O}(1)\,,
\end{equation}
where $\mathcal{O}(1)$ is a shorthand notation for a linear operator
$\mathcal{E}:\C^{2N \times 2N} \to \C^{2N \times 2N}$ satisfying $\Vert \mathcal{E}[B] \Vert \lesssim \Vert B \Vert$.
This linear operator is represented by a contour integration of the form
$$
	\frac{1}{2\pi \ii}  \oint \frac{\D z}{z- {\mathcal{B}}^*_{m^{*}n^{*}}}
$$
where the contour encircles all non-critical eigenvalues of ${\mathcal{B}}^*_{m^{*}n^{*}}$
and remains at an order one distance from the entire spectrum. Note that for general non-Hermitian operators
the resolvent $(z- {\mathcal{B}}^*_{m^{*}n^{*}})^{-1}$ would not necessarily be bounded (independently  of $N$)
just because $z$ is well away from the eigenvalues. However, the explicit form of $\mathcal{S}$ (see~\eqref{Sop})
implies\footnote{This is the only place in Section~\ref{subsec:generalchain} where the special form of  $\mathcal{S}$ is currently used.
	For more general $\mathcal{S}$ operator an appropriate generalisation of the symmetrised (saturated)
	self-energy operator~\cite[Def. 4.5]{firstcorr} to two different spectral parameters is needed, see~\cite[Eq. (2.30)]{LandonLopattoSosoe}
	in the commutative case.
	\label{foot1}}
that ${\mathcal{B}}^*_{m^{*}n^{*}}= 1 + T$ where $T$ is a rank-two operator. For such operators
elementary linear algebra shows that
$$
	\left\| \frac{1}{z- {\mathcal{B}}^*_{m^{*}n^{*}}} \right\| \lesssim
	\big[\mbox{dist}\big(z, \mbox{Spec} ({\mathcal{B}}^*_{m^{*}n^{*}} )\big) \big]^{-2},
$$
i.e. the non-Hermitian instability only affects a two-dimensional subspace.

Using~\eqref{decompX} we
find
\begin{equation*}
	\mathcal{X}_{12}[\mathring{A}_1^{\mathfrak{s}_1}] =\frac{1}{\bar{\beta}_{1^{*}2^{*}}} \big( \langle R_{1^*2^*}, A_1\rangle - \langle V_{\mathfrak{s}_1,1}, A_1 \rangle \langle R_{1^* 2^*}, U_{\mathfrak{s}_1, 1} \rangle   \big)L_{1^*2^*}+ \mathcal{O}(1)[A_1]
\end{equation*}
for the decomposition of $A_1$ and
\begin{equation*}
	{\mathcal{X}}_{21}[\mathring{A}_2^{\mathfrak{s}_2}] = \frac{1}{\bar{\beta}_{2^{*}1^{*}}} \big(    \langle R_{2^*1^*}, A_2 \rangle - \langle V_{\mathfrak{s}_2,2}, A_2 \rangle \langle R_{2^*1^*}, U_{\mathfrak{s}_2, 2} \rangle\big) L_{2^*1^*} + \mathcal{O}(1)[A_2]\,,
\end{equation*}
for the decomposition of $A_2$. This implies that for $\big(\cdots\big)$ to be vanishing for every $\mathring{A}_i^{{\mathfrak{s}_i}}$, the matrix $V_{\mathfrak{s}_i,i}$ has to be chosen according to \eqref{eq:Vchoice} (recall $\langle V_{\sigma,i}, U_{\tau,i} \rangle = \delta_{\sigma, \tau}$).\footnote{\label{ftn:tolerance}In case that $\Defo_1 = \Defo_2$, by the lower bound \eqref{eq:beta pm lowerbound}, the choices in \eqref{eq:Vchoice} not necessarily have to be made \emph{exact}, but tolerate an error of the order given in the rhs.~of \eqref{eq:beta pm lowerbound}. Having such a  tolerance might be important if one treats the $\Defo_1 \neq \Defo_2$ case (contrary to $\Defo_1 = \Defo_2$ as done in this paper) and still has to satisfy the constraints $\langle V_\sigma, U_\tau \rangle = \delta_{\sigma, \tau}$ and $\langle U_\sigma, U_\sigma \rangle  = 1$.} Overall, subtracting the $\mathfrak{s}_i$-singular component already accounts for removing the `bad direction' of a involved stability operator and thus -- in particular -- reduces the naive size of the deterministic approximation \eqref{eq:det approx mot}.

However, removing the $\mathfrak{s}_i$-singular component is not sufficient: Although $\langle V_{\mathfrak{s}_i,i}, U_{-\mathfrak{s}_i,i} \rangle = 0$ and thus $U_{-\mathfrak{s}_i,i}$ is $\mathfrak{s}_i$-regular, we observe that
\begin{equation} \label{eq:chain motivation 2}
	\langle G_1 U_{-\mathfrak{s}_1, 1} G_2 U_{-\mathfrak{s}_2, 2}\rangle
\end{equation}
still (potentially) has large fluctuations: In our concrete i.i.d.~model, take $z \equiv z_1 = z_2$ (to be suppressed from the notation) and $w \equiv w_1 = -w_2$ with $e = \Re w_1$ and $ \eta = \Im w_1 > 0$ w.l.o.g., which implies that $\mathfrak{s}_1 = \mathfrak{s}_2 = +$ and $U_\sigma = E_\sigma$ for $\sigma = \pm$ (see the discussion below \eqref{eq:resolventid}). In this situation, we use \eqref{eq:chiral} and thus \eqref{eq:chain motivation 2} takes the form
\begin{equation*}
	\langle G(e+\I \eta) E_- G(-e - \I\eta) E_- \rangle = - \langle G(e+ \I\eta ) G(e+\I \eta) \rangle\,.
\end{equation*}
By construction of $V_{\mathfrak{s}_i, i}$, the corresponding deterministic approximation \eqref{eq:det approx mot} is bounded by one, but this is dominated by the fluctuation of order $1/(N \eta^2)$ in the relevant small regime $\eta \sim  N^{-1+\epsilon}$. This example shows again, what we have already established in Section \ref{subsec:variance}: For our concrete model, at least close to the imaginary axis, the regularisation \eqref{eq:circ} is \emph{necessarily a two-dimensional operation}.

For determining the other directions $V_{-\mathfrak{s}_i, i}$, we note that the regularisation should be designed in such a way, that it covers also the cases where one (or both) of the resolvents $G_1, G_2$ are taken as an adjoint (see, e.g., \eqref{eq:twoterms1av} and \eqref{eq:finfirststep}). Hence, requiring that the same arguments leading to \eqref{eq:Vchoice} should also be followed for (i) $\langle G_1 A_1 G_2^* A_2 \rangle$ and (ii) $\langle G_1^* A_1 G_2 A_2 \rangle$ (considering $\langle G_1^* A_1 G_2^* A_2 \rangle$ would again lead to a conclusion for $V_{\mathfrak{s}_i, i}$ as the relative sign of imaginary parts is preserved), we find that $V_{-\mathfrak{s}_1, 1} = (R_{2^*1})^*$ and $V_{-\mathfrak{s}_2, 2} = (R_{12^*})^*$ in case (i), and $V_{-\mathfrak{s}_1, 1} = (R_{21^*})^*$ and $V_{-\mathfrak{s}_2, 2} = (R_{1^*2})^*$ in case~(ii). In general, the right eigenvectors for these two cases are not the same. However, as pointed out in Footnote~\ref{ftn:tolerance}, there is a certain \emph{tolerance} in choosing the $V_\pm$. Therefore, within this tolerance and in order to have a consistent and conceptually simple choice, we take $V_{-\mathfrak{s}_1, 1}$ from case (i) and $V_{-\mathfrak{s}_2, 2}$ from case~(ii), i.e.
\begin{equation} \label{eq:Vchoice2}
	\boxed{	V_{-\mathfrak{s}_1,1} = R_{1^*2} = (R_{2^*1})^* \qquad \text{and} \qquad   V_{-\mathfrak{s}_2,2} = R_{2^* 1} = (R_{1^*2})^*\,.}
\end{equation}
Here, in both situations the spectral parameter being the \emph{right} neighbour of $A_i$ receives a complex conjugate. In comparison, if we took $V_{-\mathfrak{s}_1, 1}$ from case (ii) and $V_{-\mathfrak{s}_2, 2}$ from case (i), we would have ended up with the alternative regularisation from Footnote \ref{ftn:alternreg}, where the \emph{left} neighbor of $A_i$ received a complex conjugate.
Again, the relations in \eqref{eq:Vchoice2} are understood up to a normalizing constant, which can be specified only after determining $U_\sigma$.

Now, it is very important to observe that, for our concrete model with $\Defo_1 = \Defo_2$ and $w_1 = w_2 = \I \eta$ (in particular, $\mathfrak{s}_1 = \mathfrak{s}_2 = -$), our choices for $V_\pm$ in \eqref{eq:Vchoice} and \eqref{eq:Vchoice2} \emph{agree} with those in \eqref{eq:V+variance} and \eqref{eq:V-variance} obtained from a variance calculation with only a single resolvent. This follows from the explicit formulas for the critical right eigenvector given later in \eqref{eq:crit triple}, Lemma \ref{lem:Mbasic}~(a), and \eqref{eq:chiralM}


\subsection{Explicit formulas for our concrete model and $\Defo_1 = \Defo_2$} \label{subsec:exactform}
In this subsection, we will give explicit formulas for $V_{\pm}$ and $U_\pm$ for our concrete model with one fixed deformation $\Defo$. In fact, for $\Defo_1 = \Defo_2$, the so far unspecified matrices $U_\sigma$ can be characterised by requiring that, jointly with the symmetry relation $ E_- G^z(-w) E_- = - G^z(w)$, a {resolvent identity (see \eqref{eq:resolid} for the standard resolvent identity)} can be applied to $G_2 U_\sigma G_1$. This yields, together with the normalisation $\langle U_\sigma , U_\sigma \rangle = 1$, that\footnote{Note that the assignment of $\pm$ is \emph{a priori} not determined, but we chose it in that way. This is also reflected in \eqref{eq:Vchoice} and \eqref{eq:Vchoice2}.}
\begin{equation*} \label{eq:U choice}
	U_+ = E_+ \quad \text{and} \quad U_- = E_-\,.
\end{equation*}

The singular (or critical) eigenvectors of the stability operators characterizing $V_{\mathfrak{s}_i, i}$ can also be explicitly calculated. Using \eqref{eq:Vchoice} and \eqref{eq:Vchoice2}, we infer, by means of \eqref{eq:crit triple} and the normalisation/orthogonality condition $\langle V_{\sigma,i}, U_{\tau,i} \rangle = \delta_{\sigma, \tau}$, that
\begin{equation} \label{eq:Vpmexplicit}
	\begin{alignedat}{2}
		V_{\mathfrak{s}_1,1} &= \frac{M_2 E_{\mathfrak{s}_1}M_1}{\langle M_2 E_{\mathfrak{s}_1}M_1 E_{\mathfrak{s}_1}\rangle}\,, \qquad &&V_{-\mathfrak{s}_1,1} = \frac{M_2^* E_{-\mathfrak{s}_1} M_1}{\langle M_2^* E_{-\mathfrak{s}_1} M_1 E_{-\mathfrak{s}_1} \rangle }\,,   \\  V_{\mathfrak{s}_2,2} &= \frac{M_1E_{\mathfrak{s}_2}M_2}{\langle M_1 E_{\mathfrak{s}_2}M_2E_{\mathfrak{s}_2} \rangle}\,, \qquad &&V_{-\mathfrak{s}_2,2} = \frac{M_1^* E_{-\mathfrak{s}_2}M_2}{\langle M_1^* E_{-\mathfrak{s}_2} M_2 E_{-\mathfrak{s}_2} \rangle }\,,
	\end{alignedat}
\end{equation}
matching the definition of the regularisation given in \eqref{eq:circ def} and \eqref{eq:reg A1A2}. The normalisation is obvious and the orthogonality readily follows from \eqref{eq:chiralM} in combination with Lemma \ref{lem:Mbasic}.

Finally, we remark that in order to define the regularisation \eqref{eq:reg A1A2} and work with \eqref{eq:Vchoice} and \eqref{eq:Vchoice2}, it is \emph{not} necessary to have the explicit forms for $V_{\sigma, i}$ at hand. Instead, the single instance of relevant \emph{explicit} formulas is the proof of Theorem \ref{thm:main}, more precisely, the bound in Proposition~\ref{prop:2Gav}, where one needs that for $| \Im w_1| \sim N^{-1+\epsilon}$, e.g., $(R_{1^*1})^*$ is close to $\Im M_1$  (up to a normalisation). But this is true beyond our model, as easily follows after taking the imaginary part of the general \emph{matrix Dyson equation} (see \cite{MDEreview})
\begin{equation*}
	-\frac{1}{M} = w - A + \mathcal{S}[M]\,, \qquad \Im w \cdot\Im M > 0
\end{equation*}
with self-adjoint \emph{matrix of expectations} $A = A^*$ and (flat) \emph{self-energy operator} $\mathcal{S}[\cdot]$. In fact, this yields
\begin{equation*}
	\big(1 - M \mathcal{S}[\cdot]M^*\big) (\Im M) = ( \Im w) \, M M^* \,,
\end{equation*}
i.e.~for $|\Im w| \ll 1$ very small, $\Im M$ is an approximate right eigenvector of the stability operator $1 - M \mathcal{S}[\cdot]M^*$ corresponding to the \emph{critical} eigenvalue (recall the discussion below \eqref{eq:stabopgeneral}).

\section{Properties of the MDE and the stability operator: Proof of Lemma~\ref{lem:regularbasic}} \label{app:stabop}

In the first part of this appendix, we derive several elementary properties of the MDE
\begin{equation} \label{eq:MDEapp}
	- \frac{1}{M} = w - \hat{\Defo} + \mathcal{S}[M]\,, \qquad w \in \C\setminus \R\,,
\end{equation}
(recall \eqref{eq:MDE}) and its unique solution $M$ (under the usual constraint $\Im M \cdot \Im w > 0$)
where the  operator $\mathcal{S}$ was given in~\eqref{Sop} and $\hat{\Defo} \in \C^{2N \times 2N}$ is from \eqref{eq:Defohat}.
Afterwards, in the second part, we turn to the associated two-body stability operator
\begin{equation} \label{eq:stabop}
	{\mathcal{B}} \equiv {\mathcal{B}}(w_1,w_2)[\cdot] := 1-M(w_1) \mathcal{S}[\cdot] M(w_2)
\end{equation}
and its adjoint ${\mathcal{B}}^*$ (see \eqref{eq:stabop adjoint}).
Moreover, we also explain the relation between the regularisation from Definition \ref{def:reg obs1} and the stability operator.

Finally, after proving and combining Lemmas \ref{lem:MDE} and \ref{lem:Mbasic} with Lemma \ref{lem:boundedpert}
on $M$ and  $\mathcal{B}$, respectively, we will complete the proof of Lemma \ref{lem:regularbasic}.

\subsection{The Matrix Dyson Equation \eqref{eq:MDEapp} and its solution}\label{sec:A1}

Existence and uniqueness of the solution $M=M(w)$ to \eqref{eq:MDEapp}
with $\Im M \cdot \Im w > 0$ has already been shown in
\cite{Helton}.
By \cite[Prop. 2.1]{firstcorr}, this solution can also be represented as the Stieltjes transform
of a compactly supported semi-definite
matrix-valued probability measure on $\R$, which has the immediate consequence that $\Vert M(w)\Vert \le |\Im w|^{-1}$.
\begin{lemma} \label{lem:MDE}
	Let $M$ be the unique solution to \eqref{eq:MDEapp} and write its $2\times 2$-block representation as
	\begin{equation} \label{eq:Mapp}
		M = \begin{pmatrix}
			M_{11} & M_{12} \\
			M_{21} & M_{22}
		\end{pmatrix}\,.
	\end{equation}
	Then we have the following:
	\begin{itemize}
		\item[(a)]  The average trace $\langle M\rangle$ coincides with the solution $m$ of~\eqref{eq:selfconsm}, $\langle M(w)\rangle= m(w)$,
			and the blocks in~\eqref{eq:Mapp} are given by~\eqref{eq:Mu}--\eqref{eq:mde}. We have $M^*(w)=M(\bar w)$.
		\item[(b)] The solution has a continuous extension to the real line from the upper half plane, denoted by
			$M(e): = \lim_{\eta\downarrow 0} M(e+\ii \eta)$; the limit from the lower half plane is $M^*(e)$.
			The self-consistent density of states of the MDE, defined as $\rho(e)=\frac{1}{\pi}\langle \Im M(e)\rangle$, is identical to the
			free convolution of $\mu_{\hat\Defo} \boxplus\mu_{sc}$ from~\eqref{freec}. Both $\rho$  and its Stieltjes transform $m$
			are H\"older continuous with a small universal exponent $c$, i.e.
			$$
				|\rho(e_1)-\rho(e_2)|\le C|e_1-e_2|^c, \qquad e_1, e_2\in \R,
			$$
			and
			\begin{equation}\label{HC}
				|m(w_1)-m(w_2)|\le C' |w_1-w_2|^c, \qquad w_1, w_2\in \C_+,
			\end{equation}
			where $C, C'$ depend only  on $\|\Lambda\|$.

		\item[(c)] We have the chiral symmetry
			\begin{equation} \label{eq:chiralMapp}
				E_- M(w)= - M(-w) E_- \,.
			\end{equation}
			In particular, for purely imaginary spectral parameter, $w = \I \Im w $, it holds that $m = \I \Im m$ as well as $M_{11} = \I \Im M_{11}$ and $M_{22} = \I \Im M_{22}$. Moreover, the off-diagonal blocks of $\Im M$ are vanishing on the imaginary axis.
		\item[(d)] Fix $\kappa > 0$. For any spectral parameter in the $\kappa$-bulk, $w \in \C \setminus \R$ with $\Re w \in \mathbf{B}_\kappa$,
			we have 
			\begin{equation} \label{eq:Mbounded}
				\Vert M(w)\Vert \le C(\kappa, \|\Lambda\|)\,
			\end{equation}
			for some constant depending only on $\kappa$ and an upper bound on the norm $\|\Lambda\|$.
			Moreover, $\rho(e)$ is real analytic on $\mathbf{B}_\kappa$ with derivatives controlled uniformly
			\begin{equation}\label{realanal}
				\max\{ |\partial^k \rho(e)| \; : \;  e\in\mathbf{B}_\kappa\} \le C(k,\kappa, \|\Lambda\|)
			\end{equation}
			with a constant $C(k,\kappa, \|\Lambda\|)$ for any $k\in\N$.
	\end{itemize}
\end{lemma}
\begin{proof}
	For part (a),  a direct computation shows that $M$ from \eqref{eq:Mapp} with the  blocks given
	in~\eqref{eq:Mu}--\eqref{eq:mde}   indeed solves \eqref{eq:MDEapp}
	if $m$ is replaced with $\langle M \rangle$ in these formulas. The calculation uses the simple  observation
	that $\langle M_{11} \rangle = \langle M_{22}\rangle$ from~\eqref{eq:mde}, hence $\mathcal{S}[M]=\langle M \rangle $.
	Furthermore, the MDE also implies that $\langle M \rangle$ solves~\eqref{eq:selfconsm}, but this
	equation has a unique solution by the theory of free convolutions with a semicircular density, hence $m= \langle M \rangle$.
	Finally $M^*(w)=M(\bar w)$ follows from $\bar m(w)= m(\bar w)$. This
	proves (a). 

	For part (b), since $\mathcal{S}[M]=\langle M \rangle $, we observe that
	$M$ solves
	\begin{equation*}
		- \frac{1}{M} = w - \hat{\Defo} + \langle M \rangle\,,
	\end{equation*}
	which is exactly the MDE for a deformed Wigner matrix model.\footnote{That is, a matrix $H = W + \hat{\Defo}$, where $W$ is a Hermitian matrix with normalised i.i.d.~ (up to the symmetry) entries of variance $1/(2N)$.}
	The point is that the Hermitised $H$ from~\eqref{eq:herm} does not satisfy the uniform lower bound
	in the \emph{flatness}
	condition on the self-energy operator, i.e. $\mathcal{S}[T]\ge c\langle T \rangle$ does not hold in general. Nevertheless, for the purpose of
	computing $M$ we can replace $H$ with the deformed Wigner model $W+\hat\Defo$ with self-energy given
	$\mathcal{S}[T]=\langle T \rangle$ and which is  \emph{flat}.
	Thus we can use several  results from the analysis of the MDE with flatness condition.
	The H\"older-continuity of the scDos  was proven in~\cite[Prop.~2.2]{firstcorr},
	which easily extends to the H\"older-continuity of its Stieltjes transform $m$, see e.g.~\cite[Lemma A.7]{1506.05095}.
	In particular $\langle M(w) \rangle$
	extends continuously to the real line and thus the scDos $\rho(e):=\frac{1}{\pi}\langle\Im M(e)\rangle = \frac{1}{\pi}\Im m(e)$
	is well defined. Since it has the same Stieltjes transform as the free convolution~\eqref{freec} by part (a), we proved that
	the scDos defined via MDE is the same as the free convolution~\eqref{freec}.

	The continuous extension of $M$ (and not only its trace) requires an additional argument.
	For any open interval  $I\in \R$ define
	\begin{equation*}
		\| M\|_I: = \sup \{ \| M(e+\ii\eta)\| \; : \;  e\in I, \eta>0\} \,. 
	\end{equation*}
	Suppose for some  open $I\in \R$ we have $\| M\|_I<\infty$, then
	we have the Lipschitz continuity
	\begin{equation*}
		\| M(w_1)-M(w_2)\| \le \| M\|_I^2 |w_1-w_2| \,, \qquad \Re w_1, \Re w_2\in I
	\end{equation*}
	following from the resolvent identity applied to $M(w)=(\hat\Lambda-w-m)^{-1}$.
	Thus $M(w)$ continuously extends to any $e\in I$.

	So the key question for the extension (and for many other results on the MDE) is the boundedness $\| M\|_I<\infty$.
	In the bulk spectrum, i.e. for any $e\in \R$ with $\rho(e)>0$, we can use the bound
	$$\| M(w)\|\le |\Im m(w)+\Im w|^{-1}
	$$
	that is obtained by taking the imaginary part of \eqref{eq:MDEapp}, yielding
	\begin{equation*}
		\Im M = (\Im w + \langle \Im M \rangle) M M^*\,,
	\end{equation*}
	and using $\Vert M M^* \Vert = \Vert M\Vert^2$ and $\Vert \Im M \Vert \le \Vert M \Vert$.
	By the H\"older continuity~\eqref{HC}
	in small neighborhood $I$ of $e$ (whose size depend on the lower bound on $\rho(e)$)
	we obtain $\| M \|_I\lesssim \rho(e)^{-2}<\infty$. Thus $M$ continuously extends to $I$ with the same bound
	and it is locally Lipschitz continuous with a Lipschitz constant of order $\rho(e)^{-2}$.
	In the entire $\kappa$-bulk this extension is controlled by a constant depending only on $\kappa$
	and $\|\Lambda\|$ (via~\eqref{HC}). This proves~\eqref{eq:Mbounded}.

	Near the spectral edges we have only  an $N$-independent upper bound for $\| M\|$.
	Using the spectral decomposition
	of $\hat\Defo$ with eigenvalues $\nu_i$ and normalised eigenvectors ${\bm y}_i$, $i\in \pm [N]$, we have
	\begin{equation}\label{Mspec}
		M(w) = \sum_i \frac{|{\bm y}_i\rangle \langle {\bm y}_i|}{\nu_i -w-m(w)}, \quad \mbox{thus}
		\quad   \|M(w)\| \le \frac{2N}{\min_i |\nu_i -w-m(w)|}\,.
	\end{equation}
	On the other hand the imaginary part of \eqref{eq:selfconsm} implies
	$$
		\Im m = \frac{1}{2N} \sum_i \frac{\Im m +\Im w}{| \nu_i - w-m|^2}
	$$
	thus
	$$
		\frac{1}{2N} \sum_i \frac{1}{| \nu_i - w-m|^2 } = \frac{\Im m}{ \Im m +\Im w} \le 1
	$$
	so $|\nu_i - w-m|\ge 1/\sqrt{2N}$.  From~\eqref{Mspec} this gives the uniform bound
	\begin{equation*}
		\|M(w)\| \le  (2N)^{3/2}, \qquad w\in \C\setminus \R,
	\end{equation*}
	which guarantees the continuous extension of $M$ to the real line with a uniform Lipschitz constant
	$(2N)^{3/2}$. As we have seen, in the bulk this regularity can be improved.\footnote{We remark that under some extra condition on $\Lambda$
	further improvements  away from the bulk are possible for $m$ but not for $M$. For example,
	if the singular values $\nu_i$ of $\Lambda$ are 1/2-H\"older  continuous
	in the sense that $|\nu_i- \nu_j |\le C_0[ |i-j|/N]^{1/2}$, then $m$ is also uniformly bounded
	and 1/3-H\"older continuous with
	a constant depending on $C_0$, see Section 11.4 of~\cite{1506.05095}.}


	For part (c), the symmetry  $\rho(e)=\rho(-e)$ immediately implies the symmetry $m(w)=-m(-w)$ for its Stieltjes transform.
	%
	Then \eqref{eq:chiralMapp} is an immediate consequences of the formulas~\eqref{eq:Mu}--\eqref{eq:mde}.

	Finally, for part (d), the bound~\eqref{eq:Mbounded} was already proven above.
	The real analyticity of $\rho$ and $m$ in the bulk
	with the bounds on the derivative~\eqref{realanal}
	follows from taking  derivatives in~\eqref{eq:selfconsm} and using again the lower bound on $\Im m$.
\end{proof}

Finally, we prove some regularity property of the $\kappa$-bulk, see~\eqref{kappabulkreg}.
\begin{lemma} Let $0 < \kappa' < \kappa$ be two small constants, then
	\begin{equation}\label{kappabulkreg1}
		\mathrm{dist}(\partial \mathbf{B}_{\kappa'}, \mathbf{B}_\kappa) \ge \mathfrak{c} (\kappa - \kappa')
	\end{equation}
	with some $N$-independent constant $ \mathfrak{c}= \mathfrak{c}(\Vert \Defo \Vert) > 0$. Moreover,
	$\mathbf{B}_\kappa$ is a finite  union of disjoint compact intervals; the number of  these \emph{components}
	depends only on $\kappa$ and $\Vert \Defo \Vert$.
\end{lemma}

\begin{proof} As in the proof of Lemma~\ref{lem:MDE}, we interpret $\mathbf{B}_\kappa$ as the $\kappa$-bulk
	of the deformed Wigner matrix $W+\hat\Defo$, i.e. a model with the flatness condition.
	The statement on the number of components directly follows from the real analyticity of $\rho$ and~\eqref{realanal}.

	The same argument would also imply~\eqref{kappabulkreg1} with a constant $\mathfrak{c}$ that depends on $\kappa$
	and an upper bound on $\|\Lambda\|$. To remove the $\kappa$-dependence, we need to use the detailed \emph{shape
		analysis}  for $\rho$ from~\cite{1804.07752}. In particular, the flatness condition and $\| M\|_I<C(\kappa)$
	for any interval $I\subset \mathbf{B}_\kappa$
	(equivalent to \cite[Eq. (4.16)]{1804.07752}) implies that Assumption 4.5 in~\cite{1804.07752} holds.
	Therefore Theorem 7.2 in~\cite{1804.07752} applies to our case. This theorem says that
	in the regime where $\rho$ is small, it is approximately given by explicit 1/3-H\"older continuous functions,
	moreover $\rho$ itself is 1/3-H\"older continuous with H\"older constant depending only on
	the so-called model parameters of the problem, which in our case is just an upper bound on $\Lambda$
	(note that~\cite{1804.07752} was written for much more complicated self-energy operators
	to include the MDE analysis for random matrices with correlated entries).
	Noticing the $\kappa^{1/3}$ power in the definition of $\mathbf{B}_\kappa$ in~\eqref{eq:bulk}, this means
	that the boundaries of
	$\mathbf{B}_\kappa$ are  Lipschitz continuous functions of $\kappa$ when $\kappa$ is small
	with a Lipschitz constant depending only on an upper bound on $\|\Lambda\|$.
\end{proof}

\begin{remark}\label{rem:cdependence}
	Note that the proof of the independence of $ \mathfrak{c}= \mathfrak{c}(\Vert \Defo \Vert)$ of $\kappa$ required
	a much more sophisticated analysis. However, for our main proof, $ \mathfrak{c}= \mathfrak{c}(\kappa, \Vert \Defo \Vert)>0$
	in~\eqref{kappabulkreg1}
	is sufficient, note that~\eqref{kappabulkreg1} is only used in choosing $\delta$ in~\eqref{eq:deltachoice} appropriately. More precisely, for fixed $L = L(\epsilon)$ and $\kappa_0 > 0$, given the family $(\ell \kappa_0)_{\ell \in [L]}$ of parameters for the domains $\mathbf{D}_\ell^{(\epsilon_0, \kappa_0)}$, we would have that $\mathrm{dist}(\partial \mathbf{B}_{(\ell - 1) \kappa_0}, \mathbf{B}_{\ell \kappa_0}) \ge \mathfrak{c}(\ell \kappa_0, \Vert \Defo \Vert) \kappa_0$. Now, the cutoff parameter $\delta$ in \eqref{eq:deltachoice} is chosen much smaller than $\mathfrak{c}(\ell \kappa_0, \Vert \Defo \Vert) \kappa_0$ for every $\ell \le L(\epsilon)$. 
\end{remark}


\subsection{The stability operator \eqref{eq:stabop} and its spectral properties} Throughout the entire paper, the two-body stability operator \eqref{eq:stabop} and its adjoint \eqref{eq:stabop adjoint} play a crucial role. These operators depend on two (a priori) different spectral parameters $w_1, w_2$ via the solutions $M_1 = M(w_1)$ and $M_2= M(w_2)$ of the MDE \eqref{eq:MDEapp}. For these solutions, we have the following basic lemma.
\begin{lemma} \label{lem:Mbasic}
	Let $w_1, w_2 \in \C\setminus \R$ be two spectral parameters and $M_1 = M(w_1), M_2 = M(w_2)$ the corresponding solutions to \eqref{eq:MDEapp}.
	\begin{itemize}
		\item[(a)] Then we have the \emph{$M$-Ward identity},
			\begin{equation}\label{eq:MWard}
				M_1 -M_2 = \big[ (w_1-w_2)+ (\langle M_1 \rangle - \langle M_2 \rangle) \big] \, M_2 M_1\,.
			\end{equation}
			In particular, $M_1$ and $M_2$ commute and it holds that
			\begin{equation} \label{eq:saturation}
				\left( 1 - \langle M M^* \rangle  \right) \langle \Im M\rangle = \Im w \,  \langle M M^* \rangle \,.
			\end{equation}
		\item[(b)] Fix $\kappa > 0$ and let $\Re w_1, \Re w_2 \in \mathbf{B}_\kappa$. Then, for $\Im w_1 \Im w_2 > 0$, we have the perturbative estimate
			\begin{equation*}
				\left\Vert M(w_1) -  M(w_2) \right\Vert = \mathcal{O}(|w_1 -w_2|\wedge 1)\,.
			\end{equation*}
	\end{itemize}

\end{lemma}
\begin{proof}
	Part (a) is an immediate consequence of the MDE \eqref{eq:MDEapp} using the fact that
	\begin{equation*}
		M = \big( \hat{\Defo} - (w+m) \big)^{-1}
	\end{equation*}
	is a resolvent of $\hat{\Defo}$. The special case \eqref{eq:saturation} follows from \eqref{eq:MWard} with $w_1 = w$ and $w_2 = \bar{w}$, and taking a trace.

	For part (b), we focus on the case of small imaginary parts for the spectral parameters (the complementary regime being trivial) and use that $M$ is analytic away from the real axis and differentiate \eqref{eq:MDEapp} w.r.t.~$w$, such that we find
	\begin{equation*}
		\partial_w M = \frac{1}{1- \langle M^2 \rangle} M^2
	\end{equation*}
	by means of $\mathcal{S}[M^2] = \langle M^2 \rangle$ as follows from the explicit form of $M$ in \eqref{eq:Mu}--\eqref{eq:mde}. Next, using \eqref{eq:saturation}, the denominator is lower bounded as
	\begin{equation} \label{eq:1-M2}
		\big| 1 - \langle M^2 \rangle \big| = \big| \big(1 - \langle M M^* \rangle\big) - 2 \I \langle M \Im M  \rangle \big| \ge 2 \big| \langle (\Im M)^2 \rangle \big| \ge 2 \langle \Im M \rangle^2\,,
	\end{equation}
	which shows that $\Vert \partial_w M \Vert \lesssim 1$ in the bulk. Now the claim follows from the fundamental theorem of calculus together with the boundedness of $M$, see \eqref{eq:Mboundedmain}.
\end{proof}

Armed with this information, we can now turn to the following lemma, collecting several basic spectral properties stability operator $\mathcal{B}$. Its proof will be given at the end of this section.

\begin{lemma} \label{lem:eigendecomp} Let $w_1, w_2 \in \C\setminus \R$ and $M_1, M_2$ be the respective solutions of \eqref{eq:MDEapp}.
	\begin{itemize}
		\item[(a)]	The associated two-body stability operator
			\begin{equation*}
				\mathcal{B} = 1 - M_1 \mathcal{S}[\cdot]M_2
			\end{equation*}
			has two non-trivial eigenvalues $\beta_\pm$ (the other $(2N)^2-2$ are equal to one), given by
			\begin{equation} \label{eq:beta pm}
				\beta_\pm = 1 \mp \langle M_1 E_\pm M_2 E_\pm \rangle \,.
			\end{equation}
			The corresponding right- and left-eigenvectors
			\[
				{\mathcal{B}}[R_\pm] = \beta_\pm R_\pm\,, \qquad {\mathcal{B}}^*[L_\pm^*] = \bar{\beta}_\pm L_\pm^* \,,
			\]
			take the explicit form
			\begin{equation} \label{eq:RL pm}
				R_\pm = M_1 E_\pm M_2 \,, \qquad L_\pm = E_\pm\,,
			\end{equation}
			up to a normalisation ensuring that $\langle L_\pm, R_\pm \rangle = 1$.
		\item[(b)] The eigenvalues \eqref{eq:beta pm} can be lower bounded as
			\begin{equation} \label{eq:beta pm lowerbound}
				| \beta_\pm| \gtrsim \big( \vert \Re w_1 \mp\Re w_2\vert + |\Im w_1| + |\Im w_2|\big)\wedge 1\,.
			\end{equation}
			In particular, the inverse stability operator $\mathcal{B}^{-1}$ exists.
		\item[(c)] Fix $\kappa > 0$ and denote $\mathfrak{s}  := - \sgn(\Im w_1 \, \Im w_2)$. Then, for $\Re w_1, \Re w_2 \in \mathbf{B}_\kappa$, we have that $|\beta_{-\mathfrak{s} }| \gtrsim 1$.
	\end{itemize}
\end{lemma}
By the last item, given $\mathfrak{s}  := - \sgn(\Im w_1 \, \Im w_2)$, we will always refer to
\begin{equation} \label{eq:crit triple}
	\big(\beta := 1 -\mathfrak{s}  \langle M_1 E_\mathfrak{s}  M_2 E_\mathfrak{s}  \rangle\,, \ R := M_1 E_\mathfrak{s}  M_2\,, \ L := E_\mathfrak{s} \big)
\end{equation}
as the \emph{critical eigentriple} (and accordingly $\beta$ as the \emph{critical eigenvalue} etc.), consisting of the eigenvalue and the corresponding right- and left-eigenvector. Moreover, the estimate \eqref{eq:beta pm lowerbound} shows that, if we have (recall \eqref{eq:case regulation})
\[
	\mathbf{1}_\delta^\pm(w_1, w_2) := \phi_\delta(\Re w_1 \mp\Re w_2 ) \ \phi_\delta(\Im w_1) \ \phi_\delta(\Im w_2) = 0
\]
for some $\delta > 0$, then the inverse stability operator $\mathcal{B}^{-1}$ is bounded and none of the eigenvalues $\beta_\pm$ is really critical. In the complementary regime, $\mathbf{1}_\delta^\pm(w_1, w_2) = 1$, and $\Re w_1, \Re w_2 \in \mathbf{B}_\kappa$, we shall now explain the interplay between the critical eigentriple \eqref{eq:crit triple} and the regularisation \eqref{eq:reg A1A2}.

\begin{lemma} \label{lem:boundedpert} Let $w_1, w_2 \in \C\setminus \R$ with $\Re w_1, \Re w_2 \in \mathbf{B}_\kappa$ for some fixed $\kappa > 0$ and denote the relative sign of imaginary parts by $\mathfrak{s} := - \sgn(\Im w_1 \, \Im w_2)$.
	Moreover, let $M_1 = M(w_1), M_2= M(w_2)$ be the respective solutions of~\eqref{eq:MDEapp} and $A \in \C^{2N \times 2N}$ a bounded deterministic matrix.
	\begin{itemize}
		\item[(a)] If $\mathbf{1}_\delta^\mathfrak{s} (w_1, w_2) =1$ for some $\delta > 0$ small enough, the critical left- and right-eigenvectors \eqref{eq:crit triple} are normalised as $\langle L, R \rangle \sim 1$. In particular, if $\mathbf{1}_\delta^\pm(w_1, w_2) = 1$, the respective denominator in the regularisation $\mathring{A}^{w_1, w_2}$ (see \eqref{eq:reg A1A2}) is bounded away from zero.
		\item[(b)] The operator $\mathcal{X}_{12}$, acting as
			\begin{equation*}
				\mathcal{X}_{12}[B] := \big((\mathcal{B}_{12}^*)^{-1}[B^*]\big)^* = \big(1 - \mathcal{S}[M_1 \cdot M_2]\big)^{-1}[B]\,, \quad B \in \C^{2N \times 2N}\,,
			\end{equation*}
			where $\mathcal{B}_{12}:= 1 - M_1 \mathcal{S}[\cdot]M_2$, is well defined and bounded on the $\mathfrak{s}$-regular component $\mathring{A}^\mathfrak{s}$ (w.r.t.~the pair of spectral parameters $(w_1, w_2)$) of any bounded $A$. This means, for
			\begin{equation} \label{eq:circsigma bdd}
				\mathring{A}^\mathfrak{s}  := A - \mathbf{1}_\delta^\mathfrak{s} (w_1, w_2)\frac{\langle M_1 A M_2 E_\mathfrak{s} \rangle}{\langle M_1 E_\mathfrak{s}  M_2 E_\mathfrak{s}  \rangle} E_\mathfrak{s}
			\end{equation}
			it holds that $\Vert \mathcal{X}_{12}[\mathring{A}^\mathfrak{s} ]\Vert \lesssim 1$.
	\end{itemize}
\end{lemma}
In particular, combining Lemma \ref{lem:Mbasic}~(b) with Lemma \ref{lem:boundedpert}~(a), \eqref{eq:chiralM}, and Lemma~\ref{lem:Mbasic}~(a), we conclude the perturbative statements from Lemma \ref{lem:regularbasic}. 
\begin{proof}[Proof of Lemma \ref{lem:boundedpert}] For part (a), similarly to the proof of Lemma \ref{lem:eigendecomp}~(c) given below, we focus on the extreme case $w_2 = \mathfrak{s} \bar{w}_1$, where the critical eigentriple is given by
	\begin{equation} \label{eq:crit triple special}
		\big(\beta = 1 - \mathfrak{s}  \langle M(w_1)  E_\mathfrak{s}  M(\mathfrak{s} \bar{w}_1) E_\mathfrak{s}  \rangle\,, \ R = M(w_1) E_\mathfrak{s}  M(\mathfrak{s}\bar{w}_1) \,, \ L = E_\mathfrak{s} \big)\,.
	\end{equation}
	Now by means of the chiral symmetry \eqref{eq:chiralM}, we readily obtain
	\begin{equation*}
		\langle L, R \rangle = \mathfrak{s} \langle M_1 M_1^* \rangle = \mathfrak{s}  \frac{\langle \Im M_1 \rangle}{\Im w_1 + \langle \Im M_1 \rangle} \sim 1\,,
	\end{equation*}
	where we used \eqref{eq:saturation} in the second step. This principal normalisation of order persists after small perturbation of $w_2$ around the extreme case, but as long as $\mathbf{1}_\delta^\mathfrak{s} (w_1, w_2) =1$. Our claim for the denominators in the regularisation \eqref{eq:reg A1A2} follows immediately from the representation in \eqref{eq:crit triple special}.

	For part (b), we first note that, by means of Lemma \ref{lem:eigendecomp}, the statement is trivial for constellations of spectral parameters $w_1, w_2$ satisfying $\mathbf{1}_\delta^\mathfrak{s} (w_1, w_2) = 0$ and we can hence focus on the complementary extreme case $\mathbf{1}_\delta^\mathfrak{s} (w_1, w_2) = 1$. Then it follows from the explicit form
	\begin{equation*}
		\mathcal{X}_{12}[B] = B + \sum_\sigma \sigma \frac{\langle M_1 B M_2 E_\sigma \rangle}{1 - \sigma \langle M_1 E_\sigma M_2 E_\sigma \rangle} E_\sigma
	\end{equation*}
	and Lemma \ref{lem:eigendecomp} that
	\begin{equation} \label{eq:X12bdd}
		\mathcal{X}_{12}[B] = \mathfrak{s}  \frac{1}{\beta} \langle M_1 B M_2 E_\mathfrak{s} \rangle E_\mathfrak{s} + \mathcal{O}(1)[B]\,,
	\end{equation}
	where $\mathcal{O}(1)$ is a shorthand notation for a linear operator $\mathcal{E}:\C^{2N \times 2N} \to \C^{2N \times 2N}$ satisfying $\Vert \mathcal{E}[B] \Vert \lesssim \Vert B \Vert$. Now, plugging $\mathring{A}^\mathfrak{s} $ from \eqref{eq:circsigma bdd} into \eqref{eq:X12bdd} yields the desired.
\end{proof}

It remains to give the proof of Lemma \ref{lem:eigendecomp}.
\begin{proof}[Proof of Lemma \ref{lem:eigendecomp}]
	For (a), we first observe that, due to the simple structure of $\mathcal{S}[\cdot]$, indeed $(2N)^2-2$ of the $(2N)^2$ eigenvalues of $\mathcal{B}$ are equal to one. The expressions \eqref{eq:beta pm} and \eqref{eq:RL pm} can be verified by direct computation, invoking Lemma \ref{lem:Mbasic} in combination with the chiral symmetry \eqref{eq:chiralM}.

	For (b) with $w_1 \neq \pm w_2$, we first find that
	\begin{equation} \label{eq:betalower1}
		\frac{1}{\beta_{\pm}} = \frac{1}{1 \mp \langle M_1 E_\pm M_2 E_\pm \rangle} = 1 + \frac{\langle M_1 \rangle \mp \langle M_2 \rangle }{w_1 \mp w_2}
	\end{equation}
	as a consequence of Lemma \ref{lem:Mbasic}~(a) and the chiral symmetry. Now, using that $\big|\langle M \rangle\big| \le \langle M M^* \rangle^{1/2} < 1$, which follows from $M M^* = \Im M /(\Im w + \langle \Im M \rangle)$ (see Lemma \ref{lem:Mbasic}~(a)), we conclude that
	\begin{equation} \label{eq:betalower11}
		| \beta_\pm| \gtrsim   |\Re w_1 \mp \Re w_2| \wedge 1
	\end{equation}
	by application of a triangle inequality in \eqref{eq:betalower1}. Next, we estimate
	\begin{equation} \label{eq:betalower2}
		\min\big\{ | \beta_+|\,,  |\beta_-|  \big\}\ge \big| 1 - \langle M_1 M_1^* \rangle^{1/2} \langle M_2 M_2^* \rangle^{1/2} \big| \gtrsim \big(|\Im w_1| + |\Im w_2|\big)\wedge 1\,,
	\end{equation}
	where in the first step we used $\langle M M^* \rangle < 1$ together with a Schwarz inequality, and \eqref{eq:saturation} in the second step. Combining \eqref{eq:betalower11} and \eqref{eq:betalower2} yields the claim.

	Finally, for (c), we consider the case of small imaginary parts for the spectral parameters (the complementary regime being trivial) and focus on the extreme case $ w_1 = -\mathfrak{s} w_2$. Then, using \eqref{eq:chiralM} and \eqref{eq:1-M2}, we obtain
	\begin{equation}
		\big|\beta_{-\mathfrak{s} }\big| = \big|1 - \langle M_1^2 \rangle\big| \ge 2 \langle \Im M_1 \rangle^2 \gtrsim 1\,.
	\end{equation}
	This principal lower bound persists after small perturbations of $w_2$, and the complementary regime can be dealt with by \eqref{eq:beta pm lowerbound}.
\end{proof}

\section{Proof of Theorem \ref{thm:singleG}} \label{app:locallaw}
In this appendix, we give a short proof of the usual single resolvent local law in the bulk given in Theorem \ref{thm:singleG}. In the literature, bulk local laws are established under the usual \emph{flatness} assumption (see \cite[Assumption~E]{slowcorr}) on the self-energy operator $\mathcal{S}$ (recall \eqref{eq:flatness}). However, for our model, the stability operator $\mathcal{S}[R] = \sum_\sigma \sigma \langle R E_\sigma \rangle E_\sigma$ \emph{violates} the lower bound in the flatness condition \eqref{eq:flatness}, which is why we need to provide a separate argument. The main idea is that lacking of the lower bound in \eqref{eq:flatness} is compensated by the orthogonality relation $\langle GE_- \rangle = \langle ME_- \rangle = 0$ as a consequence of \eqref{eq:GMtrace}.

The following argument heavily relies on  \cite[Theorem~4.1]{slowcorr}, where a general high-moment bound on the underlined term in
\begin{equation} \label{eq:1G basic exp app}
	\langle (G-M)B \rangle = - \langle \underline{WG}\mathcal{X}[B]M \rangle + \langle G - M \rangle \langle (G-M)  \mathcal{X}[B]M\rangle
\end{equation}
and its isotropic counterpart (see \eqref{eq:1G basic exp app iso} below) has been shown. We stress that this estimate from \cite{slowcorr} does \emph{not} require the lower bound in \eqref{eq:flatness} for the self-energy operator $\mathcal{S}$. As usual, we suppressed the spectral parameter $w \in \C\setminus \R$ satisfying $\Re w \in \mathbf{B}_\kappa$ for some fixed $\kappa > 0$ from the notation. The expansion \eqref{eq:1G basic exp app} for an arbitrary deterministic matrix $B \in \C^{2N \times 2N}$ has already been established in \eqref{eq:1G basic exp}, where we introduced the linear operator $\mathcal{X}[B]:= \big(1 - \mathcal{S}[M \cdot M]\big)^{-1}[B]$ acting on matrices.

For given $B$, we now decompose it into its $(-)$-regular and $(-)$-singular component (see \eqref{eq:circsigma bdd}, the cutoff function being irrelevant here),
\begin{equation*}
	B = \mathring{B}^- + \frac{\langle M B M E_-\rangle }{\langle ME_-M E_- \rangle }E_- \,,
\end{equation*}
respectively. For the second summand, we note that $\langle G E_- \rangle = \langle M E_- \rangle = 0$, and we can hence focus on the regular component, i.e.~assume that $B = \mathring{B}^-$ is $(-)$-regular.

In this case, for a bounded deterministic $\Vert B \Vert \lesssim 1$ we thus have $\Vert \mathcal{X}[B]\Vert \lesssim 1$ from Lemma~\ref{lem:boundedpert}. With the high-moment bound on the underlined term from \cite[Theorem 4.1, part (b)]{slowcorr} one can conclude the proof of Theorem \ref{thm:singleG} in the averaged case, $|\langle (G-M)B \rangle| \prec (N\eta)^{-1}$, by a standard \emph{bootstrap} argument (see, e.g., \cite[Sections 5.3 and 5.4]{slowcorr}).

In the isotropic case, we evaluate \eqref{eq:1G basic exp app} for $B = 2 N \ket{\boldsymbol{y}} \bra{\boldsymbol{x}}$, where $\boldsymbol{x}, \boldsymbol{y} \in \C^{2N}$ are deterministic vectors in  with $\Vert \boldsymbol{x} \Vert, \Vert \boldsymbol{y}\Vert \lesssim 1$. More precisely, we subtract its $(-)$-singular component (which can be dealt with separately as explained above) and insert
\[
	B = \mathring{B}^- = 2N \ket{\boldsymbol{y}} \bra{\boldsymbol{x}} - \frac{\langle \boldsymbol{x}, M E_- M \boldsymbol{y}\rangle}{\langle M E_- M E_- \rangle} E_-
\]
in the expansion \eqref{eq:1G basic exp app}, which leaves us with
\begin{align} \label{eq:1G basic exp app iso}
	\big(G-M\big)_{\boldsymbol{x}\boldsymbol{y}} = & - \big(\underline{WG}\big)_{\boldsymbol{x} (M \boldsymbol{y})} + \langle G-M \rangle \big(G-M\big)_{\boldsymbol{x} (M \boldsymbol{y})}                                                                                                                                                                      \\
	                                               & + \left[\frac{\langle \boldsymbol{x}, M E_- M \boldsymbol{y}\rangle}{\langle M E_- M E_- \rangle} + \frac{\langle \boldsymbol{x}, M^2 \boldsymbol{y}\rangle }{1 - \langle M^2 \rangle}\right]  \big[\langle \underline{WG} E_-M \rangle - \langle G-M\rangle \langle (G-M)E_- M \rangle\big]\,.   \nonumber
\end{align}
After realizing that the denominators in \eqref{eq:1G basic exp app iso} are bounded away from zero (see Lemma \ref{lem:eigendecomp} and Lemma~\ref{lem:boundedpert}), the proof of Theorem \ref{thm:singleG} in the isotropic case, $\big| \big(G-M\big)_{\boldsymbol{x}\boldsymbol{y}}\big| \prec (N \eta)^{-1/2}$, can be concluded again by a standard \emph{bootstrap} argument, now using the high-moment bound from \cite[Theorem 4.1, part (a)]{slowcorr} and the already proven averaged law $|\langle (G-M)B \rangle| \prec (N\eta)^{-1}$ with $\Vert B \Vert \lesssim 1$ as an input.

\section{Bounds on the deterministic approximations: Proof of Lemma \ref{lem:Mbound}} \label{app:Mbound}
The goal of this appendix is to prove the bounds from Lemma \ref{lem:Mbound} on the deterministic approximation
\begin{equation*}
	M(w_1, B_1, w_2, ... , B_{k-1}, w_k)
\end{equation*}
to a resolvent chain
\begin{equation*}
	G(w_1) B_1 G(w_2) \cdots B_{k-1} G(w_k)\,.
\end{equation*}
While $M(w_1, ..., w_k)$ has been introduced for an arbitrary number $k$ of spectral parameters $w_1, ... , w_k$ in Definition \ref{def:Mdef}, the bounds in Lemma \ref{lem:Mbound} shall be proven for $k$ at most five and the deterministic matrices $B_1, ... , B_{k-1}$ being regular w.r.t.~to the surrounding spectral parameters.

As a preparation for the proof of Lemma~\ref{lem:Mbound}, we shall now show that $M(w_1, ... , w_k)$ from \eqref{eq:M_definitionapp} satisfies multiple recursive relations, called  {\it recursive Dyson equations},
by using a so-called \emph{meta argument}, that relies on the fact that $M(w_1, ... , w_k)$ actually approximates a chain of products of resolvents. In fact, we only picked one of the recursive relations (namely \eqref{eq:recursion1} with $j=1$) for actually \emph{defining} $M(w_1, ... , w_k)$ in Definition~\ref{def:Mdef}. Although the second recursion relation \eqref{eq:recursion2} will not be used in the proof of Lemma \ref{lem:Mbound}, it is obtained completely analogous to \eqref{eq:recursion1} and we hence give it for completeness. A similar meta argument has been done several times, see e.g.~\cite{metaargument}. For convenience of the reader we repeat it in our setup.
\begin{lemma} {\rm (Recursive Dyson equations for $M(w_1, ... , w_k)$, see \cite[Lemma~4.1]{multiG})} \label{lem:recurel} \\
	Fix $k \in \N$. Let $w_1, ... , w_{k} \in \C\setminus \R$ be spectral parameters and $B_1, ... , B_{k-1} \in \C^{2N \times 2N}$ deterministic matrices. Then for any $1\le j \le k$ we have the relations
	\begin{align}
		M(w_1, & ... , w_{k}) = M(w_1, ... , w_{j-1}, B_{j-1} M(w_j) B_{j}, w_{j+1}, ... , w_{k}) \label{eq:recursion1}                                                                               \\
		       & + \sum_{\sigma = \pm} \sum_{l = 1}^{j-1} \sigma M(w_1, ... ,B_{l-1}, w_l, E_\sigma, w_j, B_{j}, ... , w_{k}) \langle M(w_l, ... , w_{j-1}) B_{j-1} M(w_j) E_\sigma \rangle \nonumber \\
		       & + \sum_{\sigma = \pm} \sum_{l = j+1}^{k} \sigma M(w_1, ... ,  B_{j-1} M(w_j) E_\sigma, w_l, B_{l}... , w_{k}) \langle M(w_j,  ... , w_l) E_\sigma \rangle \nonumber
	\end{align}
	and
	\begin{align}
		M(w_1, & ... , w_{k}) = M(w_1, ... , w_{j-1}, B_{j-1} M(w_j) B_{j}, w_{j+1}, ... , w_{k}) \label{eq:recursion2}                                                                                   \\
		       & + \sum_{\sigma = \pm} \sum_{l = 1}^{j-1} \sigma M(w_1, ... ,  B_{l-1}, w_l, E_\sigma M(w_j) B_j , ... ,  w_{k}) \langle M(w_l, ... , w_{j}) E_\sigma \rangle \nonumber                   \\
		       & + \sum_{\sigma = \pm} \sum_{l = j+1}^{k} \sigma M(w_1, ... ,B_{j-1}, w_j, E_\sigma, w_l, B_{l}, ... , w_{k}) \langle M(w_j) B_{j} M(w_{j+1}, ... , w_{l})  E_\sigma \rangle \nonumber\,.
	\end{align}
	If $j=1$ or $j= k$, we define $B_0 = E_+$ resp.~$B_{k} = E_+$ in \eqref{eq:recursion1} and \eqref{eq:recursion2}.
\end{lemma}
The formulas \eqref{eq:recursion1} and \eqref{eq:recursion2} shall be derived by expanding the $j^{\rm th}$ resolvent $G_j$ in the resolvent chain $G_1B_1 \ \cdots G_j B_j \ \cdots \  B_{k-1} G_k$ corresponding to $M(w_1, ... , w_k)$ in an underlined term, once to the right (for \eqref{eq:recursion1}, see \eqref{eq:rightexp}) and once to the left (for \eqref{eq:recursion2}, see \eqref{eq:leftexp}). Altogether, this yields $2k$ different recursions for $M(w_1, ...,w_k)$, which are listed in the above lemma. Moreover, it would be possible to prove directly that all these different recursions define the same $M(w_1, ... , w_k)$. This strategy has been used in a much simpler setup \cite{thermalisation} dealing with Wigner matrices. Here, we find it simpler to use the alternative meta argument.
\begin{proof} The principal idea is to derive the respective relations \eqref{eq:recursion1} and \eqref{eq:recursion2} on the level of resolvent chains $G_1 B_1 \cdots B_{k-1} G_k$, which, after taking the expectation and using that $G_i \approx M_i$ from Theorem~\ref{thm:singleG}, yields the same relation on the level of the deterministic approximations. For the purpose of proving identities about $M(w_1, ..., w_k)$, we may use the most convenient distribution for $X$, namely Gaussian. For the sake of this proof, we thus assume the single entry distribution $\chi$ of $X$ to be a standard complex Gaussian $\chi = \mathcal{N}_\C(0,1)$, i.e.~$X$ in Assumption \ref{ass:iid} is a complex Ginibre matrix, in which case it holds that (recall the discussion below \eqref{eq:underline})
	\begin{equation} \label{eq:zeroexp}
		\E \underline{f(W)W g(W)} = 0 \,.
	\end{equation}
	Let  $w_1, ..., w_k \in \C\setminus\R$ be arbitrary (but fixed!) spectral parameters. We now conduct the \emph{meta argument}, consisting of three steps.
	\\[1mm]
	\underline{\textbf{Step 1.}} We consider the resolvent chain
	\begin{equation} \label{eq:Gchainmeta}
		G_1 B_1 \ \cdots \  B_{k-1} G_k\,.
	\end{equation}
	Expanding $G_1$ via the identity
	\begin{equation*} 
		G_1 = M_1 - M_1 \underline{W G_1} + M_1 \mathcal{S}[G_1-M_1] G_1
	\end{equation*}
	and using $\mathcal{S}[G_1-M_1] = \langle G_1 - M_1\rangle$ from \eqref{eq:GMtrace}, we find that
	\begin{align}
		  & G_1 B_1 \ \cdots  \ B_{k-1} G_k \nonumber                                                                                                                                                               \\
		= & M_1 B_1  \ \cdots  \ B_{k-1}G_k -  M_1 \underline{W G_1} B_{1} \ \cdots \ B_{k-1} G_k + \langle G_1-M_1\rangle \,  M_1 G_1 B_1 \ \cdots  \ B_{k-1} G_k \nonumber                                        \\
		= & M_1 B_1  \ \cdots  \ B_{k-1}G_k  + \sum_{\sigma = \pm} \sum_{l=2}^{k-1}  \sigma M_1 \langle G_1 B_1 \ \cdots \ B_{l-1} G_l E_\sigma \rangle E_\sigma G_l B_l \ \cdots \ B_{k-1} G_k \label{eq:metathm1} \\
		  & -M_1 \underline{W G_1  B_1 \ \cdots \ B_{k-1} G_k }+ \langle G_1-M_1\rangle \,  M_1 G_1 B_1  \ \cdots \ B_{k-1}G_k + M_1 \mathcal{S}[G_1 B_1 \ \cdots  \ B_{k-1} G_k ] M_k\,,\nonumber
	\end{align}
	where in the last step we distributed the derivatives coming from the definition of the underline in \eqref{eq:underline} according to the Leibniz rule. Now, \eqref{eq:metathm} can be rewritten as
	\begin{align}
		G_1 B_1 & \ \cdots  \ B_{k-1} G_k \nonumber                                                                                                                                                                                            \\
		=       & (\mathcal{B}_{1k})^{-1} \bigg[ M_1 B_1  \ \cdots  \ B_{k-1}G_k  + \sum_{\sigma = \pm} \sum_{l=2}^{k-1}  \sigma M_1 \langle G_1 B_1 \ \cdots \ B_{l-1} G_l E_\sigma \rangle E_\sigma G_l B_l \ \cdots \ B_{k-1} G_k \nonumber \\
		        & -M_1 \underline{W G_1  B_1 \ \cdots \ B_{k-1} G_k }+ \langle G_1-M_1\rangle \,  M_1 G_1 B_1  \ \cdots \ B_{k-1}G_k \bigg] \,.\label{eq:E=0line}
	\end{align}
	Apart from the last two terms in \eqref{eq:E=0line}, this is the exact same relation on the level of resolvents as in Definition \ref{def:Mdef} for $M(w_1, ..., w_k)$.
	\\[1mm]
	\underline{\textbf{Step 2.}} Let the original matrix size $N$ be fixed. For any $d \in \N$, we consider the $dN \times dN$ Ginibre random matrix $\bm{X}^{(d)}$ with entries having variance $1/(dN)$, and the deformation $\bm{\Defo}^{(d)} := \Defo \otimes I_d \in \C^{dN \times dN}$, where $I_d \in \C^{d \times d}$ is the identity matrix. Analogously to \eqref{eq:Defohat} and \eqref{eq:herm}, we also define the Hermitisations $\hat{\bm{\Defo}}^{(d)}$ and $\bm{W}^{(d)}$, as well as the resolvents $\bm{G}^{(d)}_i = \bm{G}^{(d)}(w_i):= (\bm{W}^{(d)} + \hat{\bm{\Defo}}^{(d)} - w_i)^{-1}$. It is crucial to observe that the correspondingly modified MDE
	\begin{equation*}
		- \frac{1}{\bm{M}^{(d)}} = w - \hat{\bm{\Defo}}^{(d)} + \mathcal{S}^{(d)}[\bm{M}^{(d)}]
	\end{equation*}
	under the usual $\Im w \, \Im \bm{M}^{(d)} >0$ constraint with
	\begin{equation*}
		\mathcal{S}^{(d)}[R] :=	\widetilde{\E} \widetilde{\bm{W}}^{(d)} R \widetilde{\bm{W}}^{(d)} =  \sum_\sigma \sigma \langle R \, \bm{E}^{(d)}_\sigma\rangle \bm{E}^{(d)}_\sigma \,, \quad \text{where} \quad \bm{E}^{(d)}_\sigma := E_\sigma \otimes I_d\,,
	\end{equation*}
	has the \emph{unique solution} $\bm{M}^{(d)} = M \otimes I_d$, where $M$ is the unique solution of the MDE \eqref{eq:MDE} on $\C^{2N \times 2N}$. In particular, if we define $\bm{B}_i^{(d)} := B_i \otimes I_d$ for all $i \in [k]$, then it holds that \eqref{eq:M_definitionapp} defined with $\bm{M}^{(d)}_i$ and $\bm{B}_i^{(d)} $ as inputs, also satisfies $\bm{M}^{(d)}(w_1, \bm{B}^{(d)}_1, ... , \bm{B}^{(d)}_{k-1}, w_k) = M(w_1, B_1, ..., B_{k-1}, w_k) \otimes I_d$.

	We now multiply the analogue of \eqref{eq:E=0line} in boldface matrices by some $\bm{B}^{(d)}_k = B_k \otimes I_d$ with $B_k \in \C^{2N \times 2N}$ and take the averaged trace. Next, by means of \eqref{eq:zeroexp}, taking the expectation of the resulting expression removes the underlined term. Hence, using the one-to-one correspondence between the terms in the second line of \eqref{eq:E=0line} and the terms on the rhs.~of \eqref{eq:M_definitionapp}, mentioned below \eqref{eq:E=0line}, it follows by telescopic replacement and a simple induction on the length $k$ of the chain,  that
	\begin{equation} \label{eq:EG=M}
		\lim\limits_{d \to \infty}\E \big\langle \bm{G}^{(d)}_1 \bm{B}^{(d)}_1 \ \cdots \ \bm{G}^{(d)}_k \bm{B}^{(d)}_k \big\rangle = \langle M(w_1, B_1, ... , w_k) B_k \rangle
	\end{equation}
	by means of the usual \emph{global law} \cite[Theorem~2.1]{slowcorr} for the last term on the rhs.~of \eqref{eq:E=0line}. In fact, due to the tensorisation, we have that $|\langle \bm{G}_1^{(d)} - \bm{M}_1^{(d)}\rangle| \prec 1/(Nd)$ since $|\Im w_1| \gtrsim 1$, where the implicit constant potentially depends on $N$ but not on $d$.

	We emphasise that the tensorisation by $I_d$ is indeed a necessary step, since the matrices $M_i$ and $B_i$ are $N$-dependent and hence one cannot take the limit $N \to \infty$ in \eqref{eq:EG=M} for $d=1$.
	\\[1mm]
	\underline{\textbf{Step 3.}} Having \eqref{eq:EG=M} at hand, the recursive relations in \eqref{eq:recursion1} and \eqref{eq:recursion2} can be proven as follows: For \eqref{eq:recursion1}, let $1 \le j \le k$ and expand $G_j$ in \eqref{eq:Gchainmeta} according to
	\begin{equation} \label{eq:rightexp}
		G_j = M_j - M_j \underline{W G_j} + M_j \mathcal{S}[G_j-M_j] G_j\,,
	\end{equation}
	which yields, analogously to \eqref{eq:metathm1},
	\begin{align}
		G_1 \ \cdots \ B_{j-1} & G_j B_{j} \ \cdots \ G_k  =  G_1 \ \cdots \ B_{j-1} M_j B_{j} \ \cdots \ G_k \label{eq:metathm}                                                                             \\
		                       & + \sum_{\sigma = \pm}\sum_{l=1}^{j-1} \sigma G_1  \ \cdots \ B_{l-1} G_l \langle G_l \ \cdots \ G_{j-1} B_{j-1}M_j E_\sigma\rangle E_\sigma G_j B_j \ \cdots \ G_k\nonumber \\
		                       & + \sum_{\sigma = \pm}\sum_{l=j+1}^{k} \sigma G_1  \ \cdots \ B_{j-1} M_j \langle G_j B_j \ \cdots \ B_{l-1} G_l E_\sigma \rangle E_\sigma G_l B_l \ \cdots \ G_k \nonumber  \\
		                       & - \underline{G_1  \ \cdots \ B_{j-1} M_j W G_j B_{j} \ \cdots \ G_k }+ \langle G_j-M_j\rangle \,  G_1  \ \cdots \ B_{j-1}  M_j G_j  B_{j} \ \cdots \ G_k\,. \nonumber
	\end{align}
	Hence, after taking the trace against some arbitrary $B_k \in \C^{2N \times 2N}$, by performing the tensorisation from \textbf{Step 2}, taking an expectation, and using \eqref{eq:EG=M}, we obtain \eqref{eq:recursion1}, but in a trace against $B_k$. However, since $B_k$ was arbitrary, we conclude the desired.

	For the second recursion \eqref{eq:recursion2}, the argument is identical except from the fact that we expand $G_j$ in \eqref{eq:Gchainmeta} according to
	\begin{equation} \label{eq:leftexp}
		G_j = M_j -  \underline{ G_j W} M_j + G_j \mathcal{S}[G_j-M_j] M_j\,.
	\end{equation}
\end{proof}
The recursive relations from Lemma \ref{lem:recurel} can be used to show the bounds from Lemma \ref{lem:Mbound} on the deterministic counterparts in the definition of $\Psi_k^{\rm av/iso}$ in \eqref{eq:Psi avk} resp.~\eqref{eq:Psi isok} for $k \le 4$. Recall that all deterministic matrices $A_i$ appearing in the respective averaged or isotropic chain are regular in the sense of Definition \ref{def:regobs}.

\begin{proof}[Proof of Lemma \ref{lem:Mbound}]  In the following, we will distinguish the two regimes $\eta \le 1$ and $\eta> 1$ and argue for each of them separately, iteratively using Lemma \ref{lem:recurel}. Before going into the iteration, recall that $\Vert M(w_1)\Vert \lesssim \min(1, \tfrac{1}{|\Im w_1|})$ from Lemma \ref{lem:MDE}, which immediately yields \eqref{eq:Mboundtrace} for $k=1$.
	\\[1mm]
	\noindent\underline{Regime $\eta\le1$.} Using \eqref{eq:recursion1} for $k=j=2$, we find that
	\begin{equation} \label{eq:M2}
		M(w_1, A_1, w_2) = M(w_1) \mathcal{X}_{12}[A_1] M(w_2) = \mathcal{B}_{12}^{-1}[M(w_1)A_1 M(w_2)]\,,
	\end{equation}
	where $\mathcal{X}_{12}[B] := \big( 1 - \mathcal{S}[M(w_1) \, \cdot \, M(w_2)] \big)^{-1}[B]$ for $B \in \C^{2N \times 2N}$. Since $A_1$ is regular, we conclude \eqref{eq:Mboundnorm} for $k=1$ (by means of Lemma \ref{lem:boundedpert}~(b)), which immediately translates to \eqref{eq:Mboundtrace} for $k=2$.

	Next, for \eqref{eq:Mboundnorm} and $k=2$, we again use \eqref{eq:recursion1} with $j=2$, such that we obtain
	\begin{align} \label{eq:M3}
		M(w_1, A_1, w_2, A_2, w_3) = & M(w_1, \mathcal{X}_{12}[A_1] M(w_2)A_2, w_3)                                                                                   \\
		                             & + \sum_\sigma \sigma M(w_1, \mathcal{X}_{12}[A_1] M(w_2) E_\sigma, w_3) \langle M(w_2, A_2, w_3) E_\sigma \rangle\,. \nonumber
	\end{align}
	Moreover, using \eqref{eq:Mboundtrace} for $k=2$ in combination with  \eqref{eq:M2} and the lower bound \eqref{eq:beta pm lowerbound} on the eigenvalues of the stability operator $\mathcal{B}$, \eqref{eq:Mboundnorm} for $k=2$ readily follows.

	For \eqref{eq:Mboundtrace} and $k=3$ we need a different representation of $M(w_1, A_1, w_2, A_2, w_3)$ as
	\begin{equation*}
		\mathcal{B}_{13}^{-1}\big[ M(w_1) A_1 M(w_2, A_2, w_3) + \sum_\sigma \sigma M(w_1) E_\sigma M(w_2, A_2, w_3) \langle M(w_1, A_1, w_2) E_\sigma \rangle  \big]\,,
	\end{equation*}
	which follows from \eqref{eq:recursion1} with $j=1$ (or simply by Definition \ref{def:Mdef}). This implies
	\begin{equation*}
		\langle \mathcal{B}_{13}^{-1}[\cdots] A_3 \rangle = \langle [\cdots] \mathcal{X}_{31}[A_3] \rangle
	\end{equation*}
	and thus, since $\Vert [\cdots] \Vert \lesssim 1$ from \eqref{eq:Mboundnorm} with $k=1$ and $\Vert \mathcal{X}_{31}[A_3] \Vert \lesssim 1$ (recall Lemma \ref{lem:boundedpert}~(b)), we have proven \eqref{eq:Mboundtrace} for $k=3$.

	In order to see \eqref{eq:Mboundnorm} for $k=3$, we first need to show that \eqref{eq:Mboundnorm} for $k=2$ remains valid, if only \emph{one} of the two involved matrices $A_1, A_2$ is regular. Henceforth, we will assume that $A_1 = \mathring{A}_1$ and $A_2$ is arbitrary, the other case being similar and hence omitted. We start with \eqref{eq:M3} and use the lower bound \eqref{eq:beta pm lowerbound} on the eigenvalues of $\mathcal{B}$ in the first term in \eqref{eq:M3}, such that the remaining terms to be investigated are in the last line of \eqref{eq:M3}, where we study each factor separately. Thereby, we focus on the case $\Im w_1 > 0$ and $\mathfrak{s}_1 = \mathfrak{s}_2 = +$ (recall \eqref{eq:sign}), other constellations being completely analogous. Now, in the second factor in the last line of \eqref{eq:M3} we use
	\begin{equation*}
		\big\vert \langle M(w_2, A_2, w_3) E_- \rangle \big\vert = \big\vert \langle M(w_2) A_2 M(w_3) \mathcal{X}_{32}[E_-] \rangle \big\vert \lesssim 1
	\end{equation*}
	for $\sigma = -$. For $\sigma = +$, we find, using cyclicity of the trace, that $\big\vert \langle M(w_2, A_2, w_3) E_+ \rangle \big\vert$ equals
	\begin{equation*}
		\big\vert \langle A_2 M(w_3, E_+, w_2) \rangle \big\vert = \frac{1}{|w_3 - w_2|} \big\vert \langle A_2 \big(M(w_3)- M(w_2)\big) \rangle \big\vert \lesssim 1 + \frac{1}{|w_3 - w_2|}\,.
	\end{equation*}
	In the first factor in the last line of \eqref{eq:M3}, we use the usual bound \eqref{eq:beta pm lowerbound} for $\sigma = -$ and conclude the desired estimate together with the bound on the second factor for $\sigma = -$. However, for $\sigma = +$, the argument is slightly more involved: Using the usual notations $e_j = \Re w_j$ and $\eta_j = |\Im w_j|$, recall from the proof of Lemma \ref{lem:underlined2} (see the estimate of \eqref{eq:Psi 1 iso secondcrit}) that
	\begin{equation*}
		\langle M_1 \mathcal{X}_{12}[A_1^{\circ_{1,2}}] M_2 M_2^* E_- \rangle = \mathcal{O}\big( |e_1 + e_2| +  \eta_1 + \eta_2 \big)\,,
	\end{equation*}
	which readily implies that
	\begin{equation} \label{eq:Mbound stability}
		\langle M_1 \mathcal{X}_{12}[A_1^{\circ_{1,2}}] M_2 M_3 E_- \rangle = \mathcal{O}\big( |e_2 - e_3| + |e_1 + e_2| +  \eta_1 + \eta_2+ \eta_3\big)
	\end{equation}
	by means of Lemma \ref{lem:Mbasic}~(b).
	Employing the associated decomposition in the first factor in the last line of \eqref{eq:M3} (and using the analogous $c_\tau(...)$-notation as in \eqref{eq:csigmadef Psi 1 iso}), we find it being equal to
	\begin{equation*}
		M\big(w_1, \big(\mathcal{X}_{12}[A_1] M(w_2)\big)^{\circ_{1,3}} , w_3\big) + \sum_\tau c_\tau(\mathcal{X}_{12}[A_1^{\circ_{1,2}}] M_2) M(w_1, E_\tau , w_3)\,.
	\end{equation*}
	The first summand is easily bounded by one, as follows from \eqref{eq:Mboundnorm} for $k=1$. Using \eqref{eq:M2}, the term with $\tau = +$ is also bounded by one. The remaining term with $\tau = -$ can be estimated with the aid of \eqref{eq:Mbound stability} as
	\begin{equation*}
		\frac{	|e_2 - e_3| + |e_1 + e_2| +  \eta_1 + \eta_2+\eta_3}{|w_1 + w_3|}\,.
	\end{equation*}
	Collecting all the estimates from above, we find that $\Vert M(w_1, \mathring{A}_1, w_2, A_2, w_3) \Vert$ is bounded by
	\begin{equation*}
		\frac{1}{\eta} + \left( 1 + \frac{|e_1 + e_3| + |e_2 - e_3| + \eta_1 + \eta_2 + \eta_3}{|e_1 + e_3| + \eta_1 + \eta_3} \right) \left( 1 + \frac{1}{|e_3 - e_2| + \eta_2 + \eta_3} \right) \lesssim \frac{1}{\eta}\,,
	\end{equation*}
	which shows that \eqref{eq:Mboundnorm} remains valid if only one of the two involved matrices $A_1$, $A_2$ is regular.

	Having this at hand, we can now turn to the proof of \eqref{eq:Mboundnorm} for $k=3$. In fact, by \eqref{eq:recursion1} for $k=4$, we find
	\begin{align} \label{eq:M4}
		M(w_1, .. , w_4) = & M(w_1, \mathcal{X}_{12}[A_1] M(w_2), A_2, w_3, A_3, w_4)                                                                                 \\
		                   & + \sum_\sigma \sigma M(w_1, \mathcal{X}_{12}[A_1] M(w_2) E_\sigma, w_3, A_3, w_4) \langle M(w_2, A_2, w_3)E_\sigma \rangle \nonumber     \\
		                   & + \sum_\sigma \sigma M(w_1, \mathcal{X}_{12}[A_1] M(w_2) E_\sigma, w_4) \langle M(w_2, A_2, w_3, A_3, w_4) E_\sigma \rangle \nonumber\,,
	\end{align}
	where the first and second line of \eqref{eq:M4} are bounded by $\frac{1}{\eta}$ and we can thus focus on the last line. Structurally, this term is the analog of the last line in \eqref{eq:M3} and also proving it being bounded by $\frac{1}{\eta}$ is completely analogous to the arguments above. This concludes the proof of \eqref{eq:Mboundnorm} for $k=3$, from which \eqref{eq:Mboundtrace} for $k=4$ immediately follows.

	Finally, we turn to the proof of \eqref{eq:Mboundnorm} for $k=4$. By \eqref{eq:recursion1} for $j=1$ (or simply by Definition~\ref{def:Mdef}) we find the different representation
	\begin{align*}
		M(w_1, ... , w_5) = & \mathcal{B}_{15}^{-1} \bigl[ M(w_1) A_1 M(w_2, ... , w_5)                                               \\
		                    & + \sum_\sigma \sigma M(w_1) E_\sigma M(w_2, ... , w_5) \langle M(w_1, A_1, w_2) E_\sigma \rangle        \\
		                    & + \sum_\sigma \sigma M(w_1)E_\sigma M(w_3,..., w_5) \langle M(w_1, ... , w_3) E_\sigma \rangle          \\
		                    & + \sum_\sigma \sigma M(w_1)E_\sigma M(w_4,A_4, w_5) \langle M(w_1, ... , w_4) E_\sigma \rangle\bigr]\,.
	\end{align*}
	Combining $\Vert [\cdots]\Vert \lesssim \eta^{-1}$, as follows from \eqref{eq:Mboundnorm} for $k \in [3]$ and \eqref{eq:Mboundtrace} for $k \in [4]$,  with the usual bound \eqref{eq:beta pm lowerbound}, we conclude the desired. This finishes the proof in the first regime where $\eta \le 1$.
	\\[1mm]
	\noindent\underline{Regime $\eta > 1$.} In this second regime, we note that all inverses of stability operators are bounded (see \eqref{eq:beta pm lowerbound}). Moreover, it easily follows from \eqref{eq:recursion1} that every summand in the definition of $M(w_1, ..., w_k)$ carries at least $k$ factors of (different) $M(w_i)$. Now, as mentioned in the beginning of the proof, we have $\Vert M(w_i) \Vert \lesssim 1/\eta$, which implies the desired bound.
\end{proof}

\section{Proof of Lemmas \ref{lem:underlined3} and \ref{lem:underlined4}} \label{app:underlinedproofs}
In this appendix, we carry out the proofs of the two
Lemmas \ref{lem:underlined3} and \ref{lem:underlined4}.

\begin{proof}[Proof of Lemma \ref{lem:underlined3}]
	Similarly to the proof of Lemma \ref{lem:underlined2}, we get from Appendix \ref{app:motivation} and \eqref{eq:Mexample} that
	\begin{align}
		      & \langle (G_1 A_1 G_2 - M_1 \mathcal{X}_{12}[A_1] M_2)A_2 \rangle \label{eq:Psi 2 av after decomp}                                                                                 \\
		=  \  & \langle M_1 A_1 (G_2 - M_2) {\mathcal{X}}_{21}[A_2] \rangle  - \langle M_1 \underline{WG_1 A_1 G_2} {\mathcal{X}}_{21}[A_2] \rangle \nonumber                                     \\
		      & + \langle M_1 \mathcal{S}[G_1 - M_1] G_1 A_1 G_2 {\mathcal{X}}_{21}[A_2] \rangle + \langle M_1 \mathcal{S}[G_1 A_1 G_2] (G_2 - M_2) {\mathcal{X}}_{21}[A_2] \rangle \,. \nonumber
	\end{align}
	We note that $\Vert \mathcal{X}_{12}[\mathring{A}_1] \Vert \lesssim1$ and $\Vert \mathcal{X}_{21}[\mathring{A}_2] \Vert \lesssim 1 $ by means of Lemma \ref{lem:boundedpert}.


	Then, analogously to \eqref{eq:decomp Psi 1 iso}, 
	we need to further decompose $\mathcal{X}_{21}[A_2] M_1$  in the last three terms in  \eqref{eq:Psi 1 iso before decomp}  as
	\begin{equation} \nonumber
		{\mathcal{X}}_{21}[\mathring{A}_2] M_1 = ({\mathcal{X}}_{21}[\mathring{A}_2] M_1)^{\circ} + \sum_{\sigma }\mathbf{1}_\delta^\sigma\, c_\sigma({\mathcal{X}}_{21}[\mathring{A}_2] M_1) E_\sigma\,,
	\end{equation}
	where  we again suppressed the spectral parameters (and the relative sign of their imaginary parts, which has been fixed by $\Im w_1 > 0$ and $\Im w_2 <0$) in the notation for the linear functionals $c_\sigma(\cdot)$ on $\C^{2N \times 2N}$ defined as
	\begin{equation} \label{eq:csigmadef Psi 2 av}
		c_+(B) := \frac{\langle M_2 B M_{1}  \rangle}{\langle M_2  M_{1} \rangle} \qquad \text{and} \qquad c_{- }(B) := \frac{\langle M_2 B M^*_{1} E_{- } \rangle}{\langle M_2 E_{- } M^*_{1} E_{- } \rangle}\,.
	\end{equation}
	Continuing the expansion of \eqref{eq:Psi 2 av after decomp}, we arrive at
	\begin{align*}
		 & \langle  M_1 \mathring{A}_1 (G_2 - M_2) {\mathcal{X}}_{21}[\mathring{A}_2] \rangle  - \langle  \underline{WG_1 \mathring{A}_1 G_2} ({\mathcal{X}}_{21}[\mathring{A}_2] M_1)^{\circ} \rangle                                                   \\
		 & + \langle  \mathcal{S}[G_1 - M_1] G_1 \mathring{A}_1 G_2 ({\mathcal{X}}_{21}[\mathring{A}_2] M_1)^\circ \rangle + \langle  \mathcal{S}[G_1 \mathring{A}_1 G_2] (G_2 - M_2) ({\mathcal{X}}_{21}[\mathring{A}_2] M_1)^\circ \rangle             \\
		 & + \sum_{\sigma } \mathbf{1}_\delta^\sigma\,c_\sigma({\mathcal{X}}_{21}[\mathring{A}_2] M_1) \big[   - \langle  \underline{WG_1 \mathring{A}_1 G_2} U_\sigma \rangle + \langle  \mathcal{S}[G_1 - M_1] G_1 \mathring{A}_1 G_2 E_\sigma \rangle \\
		 & \hspace{7cm}+ \langle  \mathcal{S}[G_1 \mathring{A}_1 G_2] (G_2 - M_2) E_\sigma \rangle   \big]\,.
	\end{align*}
	We emphasise that, in case of $\mathring{A}_2$ and its linear dependents, the regular component is defined w.r.t.~the pair of spectral parameters $(w_2, w_1)$.

	Next, analogously to the proof of Lemma \ref{lem:underlined2}, we undo the underline in $\big[\cdots\big]$, such that our expansion of \eqref{eq:Psi 2 av after decomp} becomes
	\begin{align}
		 & \langle (G_1 \mathring{A}_1 G_2 - M_1 \mathcal{X}_{12}[\mathring{A}_1] M_2)\mathring{A}_2 \rangle   \nonumber                                                                                                                                                                                                   \\ = & \  \langle  M_1 \mathring{A}_1 (G_2 - M_2) {\mathcal{X}}_{21}[\mathring{A}_2 ] \rangle  - \langle  \underline{WG_1 \mathring{A}_1 G_2} ({\mathcal{X}}_{21}[\mathring{A}_2] M_1)^{\circ} \rangle \label{eq:fin}\\
		 & + \langle  \mathcal{S}[G_1 - M_1] G_1 \mathring{A}_1 G_2 ({\mathcal{X}}_{21}[\mathring{A}_2] M_1)^{\circ} \rangle + \langle  \mathcal{S}[G_1 \mathring{A}_1 G_2] (G_2 - M_2) ({\mathcal{X}}_{21}[\mathring{A}_2] M_1)^{\circ} \rangle \nonumber                                                                 \\
		 & + \sum_{\sigma }\mathbf{1}_\delta^\sigma\,c_\sigma({\mathcal{X}}_{21}[\mathring{A}_2] M_1) \big[   - \langle \mathring{A}_1 G_2 E_\sigma \rangle + \langle G_1 \mathring{A}_1 G_2 \mathring{\Phi}_\sigma \rangle  + c_\sigma (\Phi_\sigma) \langle G_1 \mathring{A}_1 G_2 E_\sigma \rangle   \big]\,, \nonumber
	\end{align}
	where
	\begin{equation} \label{eq:Phidef}
		\Phi_\sigma := E_\sigma \frac{1}{M_1} - \mathcal{S}[M_2E_\sigma]
	\end{equation}
	was further decomposed with the aid of $ c_\sigma(\Phi_\tau )  \sim \delta_{\sigma, \tau}$ and we used the notation \eqref{eq:csigmadef Psi 2 av}.

	We can now write \eqref{eq:fin} for both, $\mathring{A}_2 = \mathring{\Phi}_+$ and $\mathring{A}_2 = \mathring{\Phi}_-$, and solve the two resulting equation for $\langle G_1 \mathring{A}_1 G_2 \mathring{\Phi}_\sigma \rangle$ and $\langle G_1 \mathring{A}_1 G_2 \mathring{\Phi}_- \rangle$. Observe that by means of
	\[
		c_\tau(\mathcal{X}_{21}[\mathring{\Phi}_\sigma] M_1) \sim \delta_{\sigma, \tau}\,,
	\]
	the original \emph{system} of linear equations boils down to two separate ones. Thus, plugging the solutions for $\langle G_1 \mathring{A}_1 G_2 \mathring{\Phi}_\pm \rangle$ back into \eqref{eq:fin} we arrive at
	\begin{align}
		     & \langle (G_1 \mathring{A}_1 G_2 - M_{1}\mathcal{X}_{12}[\mathring{A}_1]M_2)\mathring{A}_2 \rangle   \nonumber                                                                                                                                                                                                                                  \\
		= \  & - \langle  \underline{WG_1 \mathring{A}_1 G_2} ({\mathcal{X}}_{21}[\mathring{A}_2] M_1)^{\circ} \rangle + \langle G_1 - M_1 \rangle \langle  G_1 \mathring{A}_1 G_2 ({\mathcal{X}}_{21}[\mathring{A}_2] M_1)^{\circ} \rangle  \nonumber                                                                                                        \\
		     & + \langle  M_1 \mathring{A}_1 (G_2 - M_2) {\mathcal{X}}_{21}[\mathring{A}_2 ] \rangle+ \langle  \mathcal{S}[G_1 \mathring{A}_1 G_2] (G_2 - M_2) ({\mathcal{X}}_{21}[\mathring{A}_2] M_1)^{\circ} \rangle   \label{eq:Psi 2 av av cancel 1}                                                                                                     \\
		     & +\sum_{\sigma }  \frac{\mathbf{1}_\delta^\sigma\,c_\sigma({\mathcal{X}}_{21}[\mathring{A}_2] M_1) }{1 - \mathbf{1}_\delta^\sigma\,c_\sigma({\mathcal{X}}_{21}[\mathring{\Phi}_\sigma] M_1)} \bigg[   - \langle  \underline{WG_1 \mathring{A}_1 G_2} ({\mathcal{X}}_{21}[\mathring{\Phi}_\sigma] M_1)^{\circ} \rangle \label{eq:Psi 2 av denom} \\
		     & + \langle G_1 - M_1 \rangle \langle  G_1 \mathring{A}_1 G_2 ({\mathcal{X}}_{21}[\mathring{\Phi}_\sigma] M_1)^{\circ} \rangle   + \langle  M_1 \mathring{A}_1 (G_2 - M_2) {\mathcal{X}}_{21}[\mathring{\Phi}_\sigma ] \rangle  \nonumber                                                                                                        \\
		     & + \langle  \mathcal{S}[G_1 \mathring{A}_1 G_2] (G_2 - M_2) ({\mathcal{X}}_{21}[\mathring{\Phi}_\sigma] M_1)^{\circ} \rangle  \label{eq:Psi 2 av av cancel 2}                                                                                                                                                                                   \\
		     & - \langle \mathring{A}_1 (G_2- M_2) E_\sigma \rangle   + c_\sigma( \Phi_\sigma) \langle (G_1 \mathring{A}_1 G_2 - M_{12}^{\mathring{A}_1}) E_\sigma \rangle  \bigg]\,. \label{eq:VPhicancel2}
	\end{align}

	We now need to check that the denominators in \eqref{eq:Psi 2 av denom} are bounded away from zero.
	\begin{lemma} \label{lem:splitting stable 2av} For small enough $\delta > 0$, we have that
		\begin{equation*}
			\left| 1 -  \mathbf{1}_\delta^\sigma(w_2, w_1)\, c_\sigma(\mathcal{X}_{21}[\mathring{\Phi}_\sigma]M_1) \right| \gtrsim 1 \qquad \text{for} \quad \sigma = \pm\,.
		\end{equation*}
	\end{lemma}
	\begin{proof}
		Completely analogous to Lemma \ref{lem:splitting stable 1iso}.
	\end{proof}

	Next, 
	there are two particular terms, namely the ones of the form
	\begin{equation} \label{eq:Psi 2 av firstcrit}
		\langle  \mathcal{S}[G_1 \mathring{A}^{{1,2}}_1 G_2] (G_2 - M_2) \mathring{A}_2^{{2,1}} \rangle\,,
	\end{equation}
	appearing in \eqref{eq:Psi 2 av av cancel 1} and \eqref{eq:Psi 2 av av cancel 2}, and
	\begin{equation} \label{eq:Psi 2 av secondcrit}
		c_\sigma({\mathcal{X}}_{21}[\mathring{A}^{{2,1}}_2] M_1) c_\sigma(\Phi_\sigma) \langle (G_1 \mathring{A}^{{1,2}}_1 G_2 - M_{1}\mathcal{X}_{12}[\mathring{A}^{{1,2}}_1] M_2) E_\sigma \rangle\,,
	\end{equation}
	appearing in \eqref{eq:VPhicancel2}, whose naive size $1/(N \eta^2)$ does not match the target. Hence, they have to be discussed in more detail. In \eqref{eq:Psi 2 av firstcrit} and \eqref{eq:Psi 2 av secondcrit}, we emphasised the pair of spectral parameters with respect to which the regularisation has been conducted. Moreover, for the following estimates, we recall the a priori bounds \eqref{eq:apriori Psi}.
	\\[2mm]
	\underline{\emph{Estimating \eqref{eq:Psi 2 av firstcrit}.}}
	We begin by expanding
	\begin{equation} \label{eq:Psi 2 av firstcrit start}
		\langle  \mathcal{S}[G_1 \mathring{A}_1^{{1,2}} G_2] (G_2 - M_2) \mathring{A}_2^{{2,1}} \rangle = \sum_{\sigma} \sigma \, \langle G_1 \mathring{A}_1^{{1,2}} G_2 E_\sigma \rangle \langle   (G_2 - M_2) \mathring{A}_2^{{2,1}}E_\sigma  \rangle
	\end{equation}
	and note that, analogously to \eqref{eq:firstcrit 2},
	\begin{equation} \label{eq:Psi 2 av firstcrit 1}
		\mathring{A}_i^{{i,j}} E_\sigma = \big(\mathring{A}_i^{{i,j}} E_\sigma\big)^{\circ_{i,i}} + \mathcal{O}\big(|e_i - \sigma e_j| + |\eta_i -  \eta_j|\big) E_+ + \mathcal{O}\big(|e_i - \sigma e_j| + |\eta_i -  \eta_j|\big) E_-
	\end{equation}
	as well as
	\begin{equation} \label{eq:Psi 2 av firstcrit 2}
		\mathring{A}_i^{{i,j}} E_\sigma = \big(\mathring{A}_i^{{i,j}} E_\sigma\big)^{\circ_{j,j}} + \mathcal{O}\big(|e_i - \sigma e_j| + |\eta_i -  \eta_j|\big)E_+ + \mathcal{O}\big(|e_i - \sigma e_j| + |\eta_i -  \eta_j|\big) E_-
	\end{equation}
	for $i \neq j \in [2]$ and $\sigma = \pm$.

	In the first term in \eqref{eq:Psi 2 av firstcrit start}, for $\sigma = +$ and $E_\sigma = E_+$, we use a \cred{resolvent identity \eqref{eq:resolid}} and the usual averaged local law \eqref{eq:single G} in combination with \eqref{eq:Psi 2 av firstcrit 1}, \eqref{eq:Psi 2 av firstcrit 2} and \eqref{eq:circ def}, in order to bound it as
	\begin{equation} \label{eq:Psi 2 av firstcrit +}
		\big\vert \langle G_1 \mathring{A}_1^{{1,2}} G_2 \rangle \big\vert
		\prec 1 + \frac{1}{|e_1 - e_2| + \eta_1 + \eta_2} \max_{i \in [2]}\vert \langle (G_i - M_i) (\mathring{A}_1^{{1,2}})^{\circ_{i,i}} \rangle \vert\,.
	\end{equation}
	For $\sigma = -$ and $E_\sigma = E_-$, we first add and subtract the corresponding deterministic approximation $\langle M(w_1, \mathring{A}_1^{{1,2}}, w_2) E_- \rangle$, which itself is bounded by means of Lemma \ref{lem:Mbound}. In the difference term, we use \eqref{eq:chiral} and employ the integral representation from Lemma~\ref{lem:intrepG^2} with
	\[
		\tau  = +\,, \quad J = \mathbf{B}_{ \ell \kappa_0}\,, \quad \text{and} \quad \tilde{\eta} = \frac{\ell}{\ell +1} \eta\,,
	\]
	for which we recall that $w_j \in \mathbf{D}_{\ell +1}^{(\epsilon_0, \kappa_0)}$, i.e.~in particular $\eta \ge (\ell +1) N^{-1+\epsilon_0}$ and hence $\tilde{\eta} \ge \ell N^{-1+\epsilon_0}$. Note that Lemma \ref{lem:intrepG^2} is also true on the level of the corresponding deterministic approximations, as can be seen, e.g., by a meta argument similarly to the proof of Lemma \ref{lem:recurel}. Hence, after splitting the contour integral and bounding the individual contributions as described in \eqref{eq:contourdecomp}, we obtain
	\begin{align*}
		      & \big\vert \langle G_1 A_1^{\circ_{1,2}} G_2 E_- \rangle \big\vert \prec 1 +  \int_{\mathbf{B}_{ \ell \kappa_0} } \frac{\big\vert \big\langle \big( G(x + \I \tilde{\eta})  - M(x + \I \tilde{\eta})\big) A_1^{\circ_{1,2}} E_- \big\rangle \big\vert}{\big\vert \big( x - e_1 - \I (\eta_1 - \tilde{\eta}) \big) \, \big( x + e_2 - \I (\eta_2 - \tilde{\eta}) \big)\big\vert} \D x   \\[1mm]
		\prec & 1 + \int_{\mathbf{B}_{ \ell \kappa_0}}  \frac{\big\vert \big\langle \big( G(x + \I \tilde{\eta})  - M(x + \I \tilde{\eta})\big) \big(A_1^{\circ_{1,2}} E_-\big)^{\circ_{x+ \I \tilde{\eta},x + \I \tilde{\eta}}} \big\rangle \big\vert }{\big\vert \big( x - e_1 - \I (\eta_1 - \tilde{\eta}) \big) \, \big( x + e_2 - \I (\eta_2 - \tilde{\eta}) \big)\big\vert } \D x  \nonumber\,,
	\end{align*}
	where in the second step we used \eqref{eq:Psi 2 av firstcrit 1} and \eqref{eq:Psi 2 av firstcrit 2}, and absorbed logarithmic corrections from the integral into `$\prec$'. This finally yields that
	\begin{equation} \label{eq:Psi 2 av firstcrit -}
		\big\vert  \langle G_1 A_1^{\circ_{1,2}} G_2 E_- \rangle  \big\vert \prec 1  + \frac{1}{|e_1 + e_2| + \eta_1 + \eta_2 } \cdot \frac{\psi_1^{\rm av}}{N \eta^{1/2}}\,.
	\end{equation}

	Combining \eqref{eq:Psi 2 av firstcrit +} and \eqref{eq:Psi 2 av firstcrit -} with the estimate
	\begin{equation} \label{eq:Psi 1 av flexible}
		\big\vert \langle   (G_2 - M_2) A_2^{\circ_{2,1}}E_\sigma  \rangle \big\vert \prec \frac{|e_1 - \sigma e_2| + |\eta_1 - \eta_2|}{N\eta} + \frac{\psi_1^{\rm av}}{N \eta^{1/2}}
	\end{equation}
	for the second term in \eqref{eq:Psi 2 av firstcrit start},which readily follows from \eqref{eq:Psi 2 av firstcrit 1} and \eqref{eq:single G}, we find that \eqref{eq:Psi 2 av firstcrit} can be bounded as
	\begin{equation} \label{eq:Psi 2 av firstcrit final}
		\big\vert \langle  \mathcal{S}[G_1 {A}^{\circ_{1,2}}_1 G_2] (G_2 - M_2) A_2^{\circ_{2,1}} \rangle \big\vert \prec 	\frac{1}{N \eta} + \frac{(\psi_1^{\rm av})^2}{(N \eta)^2}\,,
	\end{equation}
	where we used the trivial estimate $\psi_1^{\rm av} \prec \eta^{-1/2}$.
	\\[2mm]
	\underline{\emph{Estimating \eqref{eq:Psi 2 av secondcrit}.}} For the term \eqref{eq:Psi 2 av secondcrit}, we first note that the two prefactors $c_\sigma(\mathcal{X}_{21}[A_2^{\circ_{2,1}}] M_1) $ and $c_\sigma(\Phi_\sigma)$ are bounded. However, completely analogous to the proof of Lemma \ref{lem:underlined2}, in each of the two cases $\sigma = \pm $, the bound on \emph{one} of the prefactors can be improved: In the first case, $\sigma = +$, we use \eqref{eq:saturation} and compute
	\begin{equation*}
		c_+(\Phi_+) = \frac{\langle M_1 \rangle \big( 1 - \langle M_1 M_2 \rangle  \big)}{\langle M_1 M_2 \rangle } = \mathcal{O} \big( |e_1 - e_2| + \eta_1 + \eta_2 \big)\,.
	\end{equation*}
	\begin{equation*}
		\left\vert \langle G_1 \mathring{A}_1G_2 - M(w_1, \mathring{A}_1, w_2)\rangle\right\vert
		\prec \frac{1}{N \eta} + \frac{1}{|e_1 - e_2| + \eta_1 + \eta_2} \max_{i \in [2]}\vert \langle (G_i - M_i) (A_1^{\circ_{1,2}})^{\circ_{i,i}} \rangle \vert
	\end{equation*}
	which is obtained completely analogous to \eqref{eq:Psi 2 av firstcrit +}, we conclude that \eqref{eq:Psi 2 av secondcrit} for $\sigma = +$ can be estimated by $1/(N \eta)$.
	Similarly, in the second case, $\sigma = -$, we perform a computation similar to the one leading to \eqref{eq:stability} and use \eqref{eq:saturation} in order to obtain that $c_-(\mathcal{X}_{12}[A_1^{\circ_{1,2}}] M_2)$ equals
	\begin{equation*}
		\frac{\I}{2} \frac{\langle M_1 A_1^{\circ_{1,2}} M_2^* E_- \rangle }{\langle M_1 E_- M_2^* E_- \rangle} + \frac{1}{2 \I} \frac{\langle M_1 A_1^{\circ_{1,2}} M_2 E_- \rangle }{\langle M_1 E_- M_2^* E_- \rangle} \frac{1 + \langle M_1 E_- M_2^* E_- \rangle }{1 + \langle M_1 E_- M_2 E_- \rangle } =  \mathcal{O}\big(|e_1 + e_2| + \eta_1 + \eta_2\big)
	\end{equation*}
	Combining this with the bound
	\begin{align*}
		\big\vert  \langle \big(G_1 A_1^{\circ_{1,2}} G_2 - M(w_1, A_1^{\circ_{1,2}}, w_2)\big) E_- \rangle  \big\vert
		\prec \frac{1}{N\eta} + \frac{1}{|e_1 + e_2| + \eta_1 + \eta_2 } \cdot \frac{\psi_1^{\rm av}}{N \eta^{1/2}}
	\end{align*}
	which is obtained completely analogous to \eqref{eq:Psi 2 av firstcrit -}, we conclude that \eqref{eq:Psi 2 av secondcrit} can be estimated by $1/(N \eta)$ -- now in both cases $\sigma = \pm$.
	\\[2mm]
	\underline{\emph{Conclusion}.}
	Summarizing our investigations,
	we have shown that
	\begin{equation*}
		\big\langle \big(G_1 \mathring{A}_1 G_2 - M(w_1, \mathring{A}_1, w_2)\big) \mathring{A}_2 \big\rangle= - \big\langle\underline{W G_1 \mathring{A}_1 G_2 \mathring{A}_2'}\big\rangle + \mathcal{O}_\prec\big(\mathcal{E}_2^{\rm av}\big)\,,
	\end{equation*}
	where we used the shorthand notation
	\begin{equation} \label{eq:A'2av def}
		\mathring{A}_2' :=   \big(\mathcal{X}_{21}[\mathring{A}_2] M_1 \big)^\circ +   \sum_{\sigma} \frac{\mathbf{1}_\delta^\sigma \, c_\sigma(\mathcal{X}_{21}[\mathring{A}_2]M_1)}{1 - \mathbf{1}_\delta^\sigma \,  c_\sigma(\mathcal{X}_{21}[\mathring{\Phi}_\sigma]M_1)} \big(\mathcal{X}_{21}[\mathring{\Phi}_\sigma] M_1 \big)^\circ
	\end{equation}
	in the underlined term. Combining \eqref{eq:Psi 2 av firstcrit final} and the bound on \eqref{eq:Psi 2 av secondcrit} established above  with the usual single resolvent local laws \eqref{eq:single G} and the bounds on deterministic approximations in Lemma \ref{lem:Mbound}, we collected all the error terms from the expansion around \eqref{eq:Psi 2 av av cancel 1}--\eqref{eq:VPhicancel2}   in \eqref{eq:E2av}.
\end{proof}

\begin{proof}[Proof of Lemma \ref{lem:underlined4}]

	We denote $A_i \equiv \mathring{A}_i$, except we wish to emphasise $A_i$ being regular. As usual, we use the customary shorthand notations and start with
	\begin{equation*} 
		G_2 = M_2 - M_2 \underline{W G_2} + M_2 \mathcal{S}[G_2-M_2] G_2\,,
	\end{equation*}
	such that we get
	\begin{equation*}
		G_1 \tilde{A}_1 G_2 \mathring{A}_2 G_3 = G_1 \tilde{A}_1 M_2 \mathring{A}_2 G_3- G_1 \tilde{A}_1 M_2 \underline{W G_2} \mathring{A}_2 G_3+ G_1 \tilde{A}_1 M_2 \mathcal{S}[G_2-M_2] G_2\mathring{A}_2 G_3
	\end{equation*}
	for $\tilde{A}_1 = \mathcal{X}_{12}[A_1]$ with $A_1 = \mathring{A}_1$ (note that $\Vert\mathcal{X}_{12}[\mathring{A}_1]\Vert  \lesssim 1$ by Lemma \ref{lem:boundedpert}) and the linear operator $\mathcal{X}_{12}$ has been introduced in \eqref{eq:X12def}. The definition of $\mathcal{X}_{23}$ is completely analogous.

	Extending the underline to the whole product, we obtain
	\begin{align*}
		G_1 \big(\tilde{A}_1  - & \mathcal{S}[M_1 \tilde{A}_1 M_2]\big) G_2 \mathring{A}_2 G_3                                                                                                  \\
		=                       & G_1 \tilde{A}_1 M_2 \mathring{A}_2 G_3- \underline{G_1 \tilde{A}_1 M_2 W G_2 \mathring{A}_2 G_3} + G_1 \tilde{A}_1 M_2 \mathcal{S}[G_2 \mathring{A}_2 G_3]G_3 \\
		                        & + G_1 \tilde{A}_1 M_2 \mathcal{S}[G_2-M_2] G_2\mathring{A}_2 G_3
		+ G_1 \mathcal{S}[(G_1-M_1) \tilde{A}_1 M_2] G_2 \mathring{A}_2 G_3 \,,
	\end{align*}
	which leaves us with
	\begin{align}
		  & G_1 \mathring{A}_1 G_2 \mathring{A}_2 G_3 - M(w_1 , A_1 , w_2, A_2, w_3)  \label{eq:Psi 2 iso before decomp}                                                                                                                      \\
		= & \ \big(G_1\, \big[ \mathcal{X}_{12}[\mathring{A}_1] M_2 (\mathring{A}_2 + \mathcal{S}[M_2\mathcal{X}_{23}[\mathring{A}_2] M_3]) \big]\, G_3 - M(w_1, \big[\cdots\big], w_3)\big) 	\nonumber                                       \\
		  & - \underline{G_1 \mathcal{X}_{12}[\mathring{A}_1] M_2 W G_2 \mathring{A}_2 G_3}  + G_1 \mathcal{X}_{12}[\mathring{A}_1] M_2 \mathcal{S}[G_2-M_2] G_2\mathring{A}_2 G_3 \nonumber                                                  \\
		  & + G_1 \mathcal{S}[(G_1-M_1) \mathcal{X}_{12}[\mathring{A}_1] M_2] G_2 \mathring{A}_2 G_3 + G_1 \mathcal{X}_{12}[\mathring{A}_1] M_2 \mathcal{S}[G_2 \mathring{A}_2 G_3 - M_2\mathcal{X}_{23}[\mathring{A}_2] M_3]G_3\,, \nonumber
	\end{align}
	where we used Lemma \ref{lem:recurel} for assembling the purely deterministic terms on the l.h.s.
	To continue, we first note that $\Vert \mathcal{X}_{12}[\mathring{A}_1] \Vert \lesssim 1 $ and $\Vert \mathcal{X}_{23}[\mathring{A}_2] \Vert \lesssim 1$ (again, the matrices being regular removes the potentially `bad direction' of the stability operators $\mathcal{X}_{12}$ and $\mathcal{X}_{23}$).


	Then, 
	we need to further decompose $\mathcal{X}_{12}[A_1] M_2$ in the last four terms in \eqref{eq:Psi 2 iso before decomp} as
	\begin{equation} \label{eq:decomp Psi 2 iso}
		\mathcal{X}_{12}[A_1] M_2 = \big(\mathcal{X}_{12}[A_1] M_2\big)^{\circ} + \sum_{\sigma} \mathbf{1}_\delta^\sigma \, c_{\sigma}(\mathcal{X}_{12}[A_1]M_2) E_\sigma\,,
	\end{equation}
	where, similarly as for $\cdot^\circ$, we suppressed the spectral parameters $w_1, w_2$ in the notation for the linear functionals $c_\sigma(...)$, which have been defined in see \eqref{eq:csigmadef Psi 1 iso}.
	Now, plugging \eqref{eq:decomp Psi 2 iso} into \eqref{eq:Psi 2 iso before decomp} we find
	\begin{align}
		  & G_1 \mathring{A}_1 G_2 \mathring{A}_2 G_3 - M(w_1 , \mathring{A}_1 , w_2, \mathring{A}_2, w_3)  \label{eq:Psi 2 iso after decomp}                                                                                                                              \\
		= & \ \big(G_1\, \big[ \mathcal{X}_{12}[\mathring{A}_1] M_2 (\mathring{A}_2 + \mathcal{S}[M_2\mathcal{X}_{23}[\mathring{A}_2] M_3]) \big]\, G_3 - M(w_1, \big[\cdots\big], w_3)\big) 	\nonumber                                                                    \\
		  & - \underline{G_1 \big(\mathcal{X}_{12}[\mathring{A}_1] M_2\big)^\circ W G_2 \mathring{A}_2 G_3}  + G_1 \big(\mathcal{X}_{12}[\mathring{A}_1] M_2\big)^\circ \mathcal{S}[G_2-M_2] G_2\mathring{A}_2 G_3 \nonumber                                               \\
		  & + G_1 \mathcal{S}[(G_1-M_1) \big(\mathcal{X}_{12}[\mathring{A}_1] M_2\big)^\circ] G_2 \mathring{A}_2 G_3 + G_1 \big(\mathcal{X}_{12}[\mathring{A}_1] M_2\big)^\circ \mathcal{S}[G_2 \mathring{A}_2 G_3 - M_2\mathcal{X}_{23}[\mathring{A}_2] M_3]G_3 \nonumber \\
		  & + \sum_{\sigma} \mathbf{1}_\delta^\sigma \,c_\sigma(\mathcal{X}_{12}[\mathring{A}_1] M_2) \bigg[  - \underline{G_1 E_\sigma W G_2 \mathring{A}_2 G_3}  + G_1 E_\sigma \mathcal{S}[G_2-M_2] G_2\mathring{A}_2 G_3 \nonumber                                     \\
		  & \hspace{1.5cm}+ G_1 \mathcal{S}[(G_1-M_1) E_\sigma] G_2 \mathring{A}_2 G_3 + G_1 E_\sigma\mathcal{S}[G_2 \mathring{A}_2 G_3 - M_2\mathcal{X}_{23}[\mathring{A}_2] M_3]G_3  \bigg]\,. \nonumber
	\end{align}

	Next, as in the earlier sections (see, e.g., the display above \eqref{eq:Phidef}), in the last line of \eqref{eq:Psi 2 iso after decomp} we now undo the underline and find the bracket $\big[\cdots\big]$ to equal (the negative of)
	\begin{equation*}
		G_1 E_\sigma \big(\mathring{A}_2 + \mathcal{S}[M(w_2, \mathring{A}_2, w_3)]\big)G_3 - G_1 \Phi_\sigma G_2 \mathring{A}_2 G_3\,,
	\end{equation*}
	where we denoted
	\begin{equation*}
		\Phi_\sigma := E_\sigma \frac{1}{M_2} - \mathcal{S}[M_1 E_\sigma]\,.
	\end{equation*}
	It is apparent from the expansion \eqref{eq:Psi 2 iso after decomp} (and it can also be checked by hand) that
	\begin{equation*}
		M(w_1, E_\sigma \mathring{A}_2 + E_\sigma \mathcal{S}[M(w_2, \mathring{A}_2, w_3)], w_3) = M(w_1, \Phi_\sigma, w_2, \mathring{A}_2, w_3) \,,
	\end{equation*}
	which finally yields
	\begin{align}
		  & G_1 \mathring{A}_1 G_2 \mathring{A}_2 G_3 - M(w_1 , \mathring{A}_1 , w_2, \mathring{A}_2, w_3)  \label{eq:Psi 2 iso after decomp2}                                                                                                                                           \\
		= & \ \big(G_1\, \big[ \mathcal{X}_{12}[\mathring{A}_1] M_2 (\mathring{A}_2 + \mathcal{S}[M_2\mathcal{X}_{23}[\mathring{A}_2] M_3]) \big]\, G_3 - M(w_1, \big[\cdots\big], w_3)\big) 	\nonumber                                                                                  \\
		  & - \underline{G_1 \big(\mathcal{X}_{12}[\mathring{A}_1] M_2\big)^\circ W G_2 \mathring{A}_2 G_3}  + G_1 \big(\mathcal{X}_{12}[\mathring{A}_1] M_2\big)^\circ \mathcal{S}[G_2-M_2] G_2\mathring{A}_2 G_3 \nonumber                                                             \\
		  & + G_1 \mathcal{S}[(G_1-M_1) \big(\mathcal{X}_{12}[\mathring{A}_1] M_2\big)^\circ] G_2 \mathring{A}_2 G_3 + G_1 \big(\mathcal{X}_{12}[\mathring{A}_1] M_2\big)^\circ \mathcal{S}[G_2 \mathring{A}_2 G_3 - M_2\mathcal{X}_{23}[\mathring{A}_2] M_3]G_3 \nonumber               \\
		  & + \sum_{\sigma} \mathbf{1}_\delta^\sigma \,c_\sigma(\mathcal{X}_{12}[\mathring{A}_1] M_2) \bigg[  - \big(G_1 E_\sigma \big(\mathring{A}_2 + \mathcal{S}[M(w_2, \mathring{A}_2, w_3)]\big)G_3 - M(w_1, [\cdots]w_3)\big) \nonumber                                            \\
		  & + \big(G_1 \mathring{\Phi}_\sigma G_2 \mathring{A}_2 G_3 - M(w_1, \mathring{\Phi}_\sigma, w_2, \mathring{A}_2, w_3)\big) + \sum_{\sigma} c_\sigma(\Phi_\sigma) \big(G_1 E_\sigma G_2 \mathring{A}_2 G_3 - M(w_1, E_\sigma, w_2, \mathring{A}_2, w_3)\big)\bigg]\,, \nonumber
	\end{align}
	where we further decomposed $\Phi_\sigma$ in the last line of \eqref{eq:Psi 2 iso after decomp2} (while using the first relation in \eqref{eq:orthogonality Psi 1 iso}) just as $\mathcal{X}_{12}[A_1]M_2$ in \eqref{eq:decomp Psi 2 iso}.

	Next, we write \eqref{eq:Psi 2 iso after decomp2} for both, $A_1 = \mathring{A}_1 = \mathring{\Phi}_+$ and $A_1= \mathring{A}_1 = \mathring{\Phi}_-$, and solve the two resulting linear equations for $G_1 \mathring{\Phi}_\pm G_2 - M(w_1, \mathring{\Phi}_\pm, w_2)$. Observe that by means of the second relation in \eqref{eq:orthogonality Psi 1 iso} the original system of linear equations boils down to two separate ones. Thus, plugging the solutions for $G_1 \mathring{\Phi}_\pm G_2 \mathring{A}_2 G_3 - M(w_1 , \mathring{\Phi}_\pm , w_2, \mathring{A}_2, w_3)$ back into \eqref{eq:Psi 2 iso after decomp2}, we arrive at
	\begin{align}
		  & G_1 \mathring{A}_1 G_2 \mathring{A}_2 G_3 - M(w_1 , \mathring{A}_1 , w_2, \mathring{A}_2, w_3)  \label{eq:Psi 2 iso final}                                                                                                                                                                                                                        \\
		= & \ \big(G_1\, \big[ \mathcal{X}_{12}[\mathring{A}_1] M_2 (\mathring{A}_2 + \mathcal{S}[M_2\mathcal{X}_{23}[\mathring{A}_2] M_3]) \big]\, G_3 - M(w_1, \big[\cdots\big], w_3)\big) 	\nonumber                                                                                                                                                       \\
		  & - \underline{G_1 \big(\mathcal{X}_{12}[\mathring{A}_1] M_2\big)^\circ W G_2 \mathring{A}_2 G_3}  + G_1 \big(\mathcal{X}_{12}[\mathring{A}_1] M_2\big)^\circ \mathcal{S}[G_2-M_2] G_2\mathring{A}_2 G_3 \nonumber                                                                                                                                  \\
		  & + G_1 \mathcal{S}[(G_1-M_1) \big(\mathcal{X}_{12}[\mathring{A}_1] M_2\big)^\circ] G_2 \mathring{A}_2 G_3 + G_1 \big(\mathcal{X}_{12}[\mathring{A}_1] M_2\big)^\circ \mathcal{S}[G_2 \mathring{A}_2 G_3 - M_2\mathcal{X}_{23}[\mathring{A}_2] M_3]G_3 \nonumber                                                                                    \\
		  & + \sum_{\sigma} \frac{\mathbf{1}_\delta^\sigma \,c_\sigma(\mathcal{X}_{12}[\mathring{A}_1] M_2)}{1 - \mathbf{1}_\delta^\sigma \,c_\sigma(\mathcal{X}_{12}[\mathring{\Phi}_\sigma] M_2)} \bigg[  - \big(G_1 \big[E_\sigma \big(\mathring{A}_2 + \mathcal{S}[M(w_2, \mathring{A}_2, w_3)]\big)\big]G_3 - M(w_1, \big[\cdots\big]w_3)\big) \nonumber \\
		  & + \big(G_1\, \big[ \mathcal{X}_{12}[\mathring{\Phi}_\sigma] M_2 (\mathring{A}_2 + \mathcal{S}[M_2\mathcal{X}_{23}[\mathring{A}_2] M_3]) \big]\, G_3 - M(w_1, \big[\cdots\big], w_3)\big) 	\nonumber                                                                                                                                               \\
		  & - \underline{G_1 \big(\mathcal{X}_{12}[\mathring{\Phi}_\sigma] M_2\big)^\circ W G_2 \mathring{A}_2 G_3}  + G_1 \big(\mathcal{X}_{12}[\mathring{\Phi}_\sigma] M_2\big)^\circ \mathcal{S}[G_2-M_2] G_2\mathring{A}_2 G_3 \nonumber                                                                                                                  \\
		  & + G_1 \mathcal{S}[(G_1-M_1) \big(\mathcal{X}_{12}[\mathring{\Phi}_\sigma] M_2\big)^\circ] G_2 \mathring{A}_2 G_3 + G_1 \big(\mathcal{X}_{12}[\mathring{\Phi}_\sigma] M_2\big)^\circ \mathcal{S}[G_2 \mathring{A}_2 G_3 - M_2\mathcal{X}_{23}[\mathring{A}_2] M_3]G_3 \nonumber                                                                    \\
		  & + c_\sigma(\Phi_\sigma) \big(G_1 E_\sigma G_2 \mathring{A}_2 G_3 - M(w_1, E_\sigma, w_2, \mathring{A}_2, w_3)\big)\bigg]\,. \nonumber
	\end{align}
	It has been shown in Lemma \ref{lem:splitting stable 1iso} that the denominators are bounded away from zero.

	Next, we take the scalar product of \eqref{eq:Psi 2 iso final} with two deterministic vectors $\boldsymbol{x}, \boldsymbol{y}$ satisfying $\Vert \boldsymbol{x} \Vert, \Vert \boldsymbol{y} \Vert \le 1$. In the resulting expression,  in case that $\mathbf{1}_\delta^\sigma\,(w_1, w_2) = 1$ (as we assumed in \eqref{eq:2dimregass}), there are three particular terms, namely the ones of the form
	\begin{equation} \label{eq:Psi 2 iso firstcrit}
		\big(G_1 \mathcal{S}[(G_1-M_1) A_1^{\circ_{1,2}}] G_2 \mathring{A}_2 G_3 \big)_{\boldsymbol{x}\boldsymbol{y}
			}\,,
	\end{equation}
	as appearing twice, in the fourth and second to last line,
	\begin{equation} \label{eq:Psi 2 iso secondcrit}
		\big(G_1 A_1^{\circ_{1,2}} \mathcal{S}[G_2 \mathring{A}_2 G_3 - M(w_2, \mathring{A}_2, w_3)]G_3\big)_{\boldsymbol{x}\boldsymbol{y}}\,,
	\end{equation}
	as appearing, again twice, in the fourth and second to last line,
	\begin{equation} \label{eq:Psi 2 iso thirdcrit}
		c_\sigma(\mathcal{X}_{12}[\mathring{A}_1] M_2) c_\sigma(\Phi_\sigma) \big(G_1 E_\sigma G_2 \mathring{A}_2 G_3 - M(w_1, E_\sigma, w_2, \mathring{A}_2, w_3)\big)_{\boldsymbol{x}\boldsymbol{y}
			}\,,
	\end{equation}
	as appearing in the last line, whose naive sizes $1/(N \eta^3)$, $1/(N \eta^3)$, and $1/\sqrt{N \eta^4}$ do not match the target. Hence, they have to be discussed in more detail.
	\\[2mm]
	\noindent \underline{\emph{Estimating \eqref{eq:Psi 2 iso firstcrit}.}} For the terms of the first type, we begin by expanding
	\begin{equation*}
		\big(G_1 \mathcal{S}[(G_1-M_1) A_1^{\circ_{1,2}}] G_2 \mathring{A}_2 G_3 \big)_{\boldsymbol{x}\boldsymbol{y}
			} = \sum_\sigma \sigma \langle (G_1-M_1) A_1^{\circ_{1,2}}E_\sigma\rangle \big(G_1 E_\sigma G_2 \mathring{A}_2 G_3 \big)_{\boldsymbol{x}\boldsymbol{y}}
	\end{equation*}
	and recall from \eqref{eq:Psi 1 av flexible} that first factor can be estimated by
	\begin{equation} \label{eq:Psi 2 iso (i) firstfactor}
		\vert \langle (G_1-M_1) A_1^{\circ_{1,2}}E_\sigma\rangle  \vert \prec \frac{|e_1 - \sigma e_2| + |\eta_1 - \eta_2|}{N\eta}  + \frac{\psi_1^{\rm av}}{N \eta^{1/2}}\,.
	\end{equation}

	In the second factor, we distinguish the two cases $\sigma = \pm$.  For $\sigma = +$, we find
	\begin{equation*}
		G_1 G_2 A_2^{\circ_{2,3}} G_3 = \frac{G_1A_2^{\circ_{2,3}} G_3 - G_2 A_2^{\circ_{2,3}}  G_3 }{(e_1 - e_2) + \I (\eta_1 + \eta_2)}
	\end{equation*}
	by a simple {resolvent identity \eqref{eq:resolid}}, which together with
	\begin{align*}
		\mathring{A}_2^{{w_2,w_3}} = \mathring{A}_2^{{w_1,w_3}} & + \mathcal{O}\big(|e_1 - e_2| + |\eta_1 - \eta_2| + |e_1 - e_3| + |\eta_1 - \eta_3|\big) E_+           
		\\
		                                                        & + \mathcal{O}\big(|e_1 - e_2| + |\eta_1 - \eta_2| + |e_1 + e_3| + |\eta_1 - \eta_3|\big) E_- \nonumber
	\end{align*}
	from Lemma \ref{lem:regularbasic} (note the difference between the $E_+$-error and the $E_-$-error!) and the usual isotropic law \eqref{eq:single G} yields the estimate
	\begin{equation} \label{eq:Psi 2 iso (i) secondfactor +}
		\big\vert \big( G_1 G_2 A_2^{\circ_{2,3}} G_3  \big)_{\boldsymbol{x} \boldsymbol{y}} \big\vert \prec \frac{1}{\eta} + \frac{1}{|e_1 - e_2| + \eta_1 + \eta_2} \left( 1 + \frac{\psi_1^{\rm iso}}{\sqrt{N\eta^2}} \right)\,,
	\end{equation}
	where we again used the a priori bound \eqref{eq:apriori Psi}.
	For $\sigma = -$ we employ the integral representation from Lemma \ref{lem:intrepG^2} and argue similarly as for \eqref{eq:Psi 2 av firstcrit -} such that we finally obtain
	\begin{equation}\label{eq:Psi 2 iso (i) secondfactor -}
		\big\vert \big( G_1 E_- G_2 A_2^{\circ_{2,3}} G_3  \big)_{\boldsymbol{x} \boldsymbol{y}} \big\vert\prec \frac{1}{\eta} + \frac{1}{|e_1 + e_2| + \eta_1 + \eta_2} \left( 1 + \frac{\psi_1^{\rm iso}}{\sqrt{N\eta^2}} \right)\,.
	\end{equation}

	Now, combining \eqref{eq:Psi 2 iso (i) firstfactor} with \eqref{eq:Psi 2 iso (i) secondfactor +} and \eqref{eq:Psi 2 iso (i) secondfactor -}, we find
	\begin{equation} \label{eq:Psi 2 iso (i) final}
		\big\vert \big(G_1 \mathcal{S}[(G_1-M_1) A_1^{\circ_{1,2}}] G_2 \mathring{A}_2 G_3 \big)_{\boldsymbol{x}\boldsymbol{y}}\big\vert \prec \frac{1}{\sqrt{N \eta^3}} \left( 1 +  \frac{\psi_1^{\rm av} \psi_1^{\rm iso}}{N \eta}\right)\,,
	\end{equation}
	where we used that $\psi_1^{\rm av} \prec \eta^{-1/2}$ trivially by \eqref{eq:single G}.
	\\[2mm]
	\underline{\emph{Estimating \eqref{eq:Psi 2 iso secondcrit}.}}
	For terms of the second type, we again start by expanding
	\begin{align*}
		\big(G_1 A_1^{\circ_{1,2}} & \mathcal{S}[G_2 \mathring{A}_2 G_3 - M(w_2, \mathring{A}_2, w_3)]G_3\big)_{\boldsymbol{x}\boldsymbol{y}}                                                                                           \\
		                           & = \sum_\sigma \sigma \big\langle \big(G_2 \mathring{A}_2 G_3 - M(w_2, \mathring{A}_2, w_3)\big)E_\sigma \big\rangle \big(G_1 A_1^{\circ_{1,2}} E_\sigma G_3\big)_{\boldsymbol{x}\boldsymbol{y}}\,.
	\end{align*}
	Then, for the first factor, we recall from the estimate of \eqref{eq:Psi 2 av firstcrit} that
	\begin{align*}
		\big\vert  \big\langle \big(G_2 A_2^{\circ_{2,3}} G_3 - M(w_2, A_2^{\circ_{2,3}}, w_3)\big) E_\sigma \big\rangle  \big\vert
		\prec \frac{1}{N\eta} + \frac{1}{|e_2 - \sigma e_3| + \eta_2 + \eta_3 } \cdot \frac{\psi_1^{\rm av}}{N \eta^{1/2}}\,.
	\end{align*}
	Treating the second factor analogously to \eqref{eq:Psi 2 iso (i) secondfactor +} and \eqref{eq:Psi 2 iso (i) secondfactor -} above, we find
	\begin{equation*}
		\big\vert \big(G_1 A_1^{\circ_{1,2}} E_\sigma G_3\big)_{\boldsymbol{x}\boldsymbol{y}}  \big\vert \prec \frac{|e_2 - \sigma e_3| + |\eta_2 - \eta_3|}{\eta} + \left( 1 + \frac{\psi_1^{\rm iso}}{\sqrt{N\eta^2}} \right)\,.
	\end{equation*}
	Combining the two estimates, we have shown that
	\begin{equation} \label{eq:Psi 2 iso (ii) final}
		\big\vert  \big(G_1 A_1^{\circ_{1,2}} \mathcal{S}[G_2 \mathring{A}_2 G_3 - M(w_2, \mathring{A}_2, w_3)]G_3\big)_{\boldsymbol{x}\boldsymbol{y}} \big\vert \prec \frac{1}{\sqrt{N\eta^3}} \left( 1 + \frac{\psi_1^{\rm iso}}{N\eta} + \frac{\psi_1^{\rm av} \psi_1^{\rm iso}}{N \eta}  \right)
	\end{equation}
	where we again used that $\psi_1^{\rm av} \prec \eta^{-1/2}$ trivially by \eqref{eq:single G}.
	\\[2mm]
	\underline{\emph{Estimating \eqref{eq:Psi 2 iso thirdcrit}.}} For the third term, we recall the (improved) estimates
	\begin{align*}
		c_+(\Phi_+)                               & = \mathcal{O}\big(|e_1 - e_2| + \eta_1 + \eta_2\big) \\
		c_-(\mathcal{X}_{12}[\mathring{A}_1] M_2) & = \mathcal{O}\big(|e_1 + e_2| + \eta_1 + \eta_2\big)
	\end{align*}
	on the anyway bounded prefactors, which have been shown in the course of estimating \eqref{eq:Psi 1 iso secondcrit}.  By arguing analogously to \eqref{eq:Psi 2 iso (i) secondfactor +} and \eqref{eq:Psi 2 iso (i) secondfactor -}, we also find
	\begin{equation*}
		\big\vert \big(G_1 E_\sigma G_2 \mathring{A}_2 G_3 - M(w_1, E_\sigma, w_2, \mathring{A}_2, w_3)\big)_{\boldsymbol{x}\boldsymbol{y}} \big\vert \prec \frac{1}{\sqrt{N \eta^3}} + \frac{1}{|e_1 - \sigma e_2| + \eta_2 + \eta_3} \, \frac{\psi_1^{\rm iso}}{\sqrt{N\eta^2}}\,.
	\end{equation*}
	Now, combining these estimates, we conclude
	\begin{equation} \label{eq:Psi 2 iso (iii) final}
		\big\vert  \eqref{eq:Psi 2 iso thirdcrit} \big\vert
		\prec \frac{1}{\sqrt{N\eta^3}}\left(1+ \psi_1^{\rm iso}\right)\,.
	\end{equation}
	\\[2mm]
	\underline{\emph{Conclusion}.}
	Summarizing our investigations, 
	we have shown that
	\begin{equation*}
		{\big(G_1 \mathring{A}_1 G_2\mathring{A}_2 G_3 - M(w_1, \mathring{A}_1, w_2, \mathring{A}_2, w_3)\big)_{\boldsymbol{x} \boldsymbol{y}} = - \big(\underline{G_1 \mathring{A}_1' WG_2 \mathring{A}_2 G_3}\big)_{\boldsymbol{x} \boldsymbol{y}} + \mathcal{O}_\prec\big(\mathcal{E}_2^{\rm iso}\big)\,,}
	\end{equation*}
	where we used the shorthand notation
	\begin{equation} \label{eq:A'2iso def}
		\mathring{A}_1' =   \big(\mathcal{X}_{12}[A_1] M_2 \big)^\circ +   \sum_{\sigma} \frac{ \mathbf{1}_\delta^\sigma \, c_\sigma(\mathcal{X}_{12}[A_1]M_2)}{1 - \mathbf{1}_\delta^\sigma \, c_\sigma(\mathcal{X}_{12}[\mathring{\Phi}_\sigma]M_2)} \big(\mathcal{X}_{12}[\mathring{\Phi}_\sigma] M_2 \big)^\circ
	\end{equation}
	in the underlined term. Combining \eqref{eq:Psi 2 iso (i) final}, \eqref{eq:Psi 2 iso (ii) final}, and \eqref{eq:Psi 2 iso (iii) final} with the usual single resolvent local laws \eqref{eq:single G} and the bounds on deterministic approximations in Lemma \ref{lem:Mbound}, we collected all the error terms from \eqref{eq:Psi 2 iso final} in \eqref{eq:E2iso}.
\end{proof}

\end{document}